\documentclass[11pt]{amsart}
\usepackage{graphicx} 
\input{style.sty}
\title{More Versions of Real Link Floer Homology}
\author{Yonghan Xiao}
\address{Department of Mathematics, School of Mathematical Sciences\\Peking University}
\email{judy\_xyh0530@stu.pku.edu.cn}

\makeatletter
\let\@wraptoccontribs\wraptoccontribs
\makeatother
\contrib[with an appendix joint with]{Zhenkun Li}
\address{Academy of Mathematics and Systems Science\\ Chinese Academy of Sciences}
\email{zhenkun@amss.ac.cn}

\usepackage{fullpage, graphicx, enumitem}
\usepackage{url}
\usepackage{verbatim}
\usepackage{bbm}
\usepackage{amssymb,amsmath}
\usepackage{amsfonts,savesym,graphicx,bm,amsthm}
\usepackage{color}
\usepackage[mathscr]{eucal}
\usepackage[all]{xy}

\date{\today}
\begin{document}
\maketitle
\begin{abstract} 
In this paper, we further develop the real link Floer homology defined by the first author in \cite{YXHFLR}. We introduce a new base-pointing convention that leads to a different version of real link Floer homology and show that this new theory is related to the old one by an exact triangle. We also define a real link Floer theory for multi-based strongly invertible links, which is a strong real Heegaard invariant in the sense of \cite{GM_real_naturality}, and take a first step toward a real link Floer TQFT. A computer implementation for the new theory via grid diagrams was written by Zhenkun Li in \cite{ZhenkunLioythonprgram}. We also include an appendix containing real grid homology of more than 50 small knots.
\end{abstract}
\tableofcontents

\section{Introduction}
In \cite{YXHFLR}, we defined a real link Floer homology for (generalized) strongly invertible links (see Definition~\ref{def:strongly invertible links}) based on the real Heegaard Floer package developed in \cite{guth2025real}, which provides bounds on equivariant slice genus, unknotting number and satisfies both oriented and unoriented skein relations for links in $S^3$. For background on real Heegaard Floer theory and equivariant knots, we refer readers to the first paragraph of \cite{YXHFLR}. To obtain invariance of the homology groups, we fixed auxiliary data on links, which essentially consist of a properly chosen orientation and a decoration of fixed points using $O$ and $X$. It can be seen from examples that this is necessary since the real knot Floer homology has no quasi-stabilization invariance when it is performed along the fixed set and fails to satisfy a K\"unneth principle for equivariant connected sums (see Section~\ref{sub:example of full real complex} and \cite[Section~6.5]{YXHFLR}). In that paper, we required artificially that on each strongly invertible knot component, the two fixed points are decorated by $O$ and $X$, respectively. In this paper, we will further develop the theory by loosening this restriction. More precisely, we will define an alternative version of real link Floer homology which has all fixed points decorated by $O$, investigate its properties and then make a comparison with the existing one from \cite{YXHFLR}, as promised in \cite[Remark~1.2]{YXHFLR}. This is of particular interest, since these two theories are exactly related by quasi-stabilizations along the fixed set which can also be regarded as an equivariant connected sum (see Section~\ref{subsub:Type-changing quasi-stabilizations}).

Most of the discussion in \cite{YXHFLR} and Section~\ref{sec:A different theory for strongly invertible links} depends heavily on real grid diagrams using techniques developed in \cite{OSS2015grid}. This method is convenient for performing computations, but has poor naturality properties. In the classical setting, link Floer homology were first defined by Ozsv\'{a}th-Szab\'{o}~\cite{Holomorphicdisksandknotinvariants} and independently by Rasmussen~\cite{Rasmussen}. Then Zemke \cite{Zemke2019absolutegradinginHFL,Zemkequasistabandbasepointmoving,zemke2019link} upgraded it to a natural link Floer theory. Motivated by this, the second half of this paper is devoted to defining a natural real link Floer invariant and taking a first step toward defining a real link Floer TQFT.

In the first half, we define an alternative version of real link Floer homology and investigate its properties, as mentioned above. 

\begin{theorem}\label{intro-thm:OO-real link Floer homology} 
Let $(Y,\tau)$ be a closed real $3$-manifold and $(L,\fro)$ be an oriented (null-homologous) generalized strongly invertible link in it. 
Then there are three versions of $OO$-real link Floer homology groups associated to the equivariant isotopy class of $(L,\fro)$: 
\begin{itemize}
	\item $\HFLR^-_{OO}(Y,\tau,L,\fro)$, a (bigraded) module over a graded polynomial ring over $\F$ which has one degree $(-2,-1)$ variable for each pair of knot components and two degree $(-1,-1/2)$ variables for each strongly invertible knot component;
	\item $\HFLR^-_{O}(Y,\tau,L,\fro)$, a (bigraded) module over a graded polynomial ring over $\F$ which has one degree $(-2,-1)$ variable for each pair of knot components and one degree $(-1,-1/2)$ variable for each strongly invertible knot component;
	\item $\widehat{\HFLR}_{O}(Y,\tau,L,\fro)$, a (bigraded) $\F$-vector space.
\end{itemize} 
	Throughout this paper, $\F$ denotes the $2$-element field. Here we bracket the word bigraded since we need some mild extra assumption for the gradings to be well-defined.
\end{theorem}

Specializing to $(Y,\tau)=(S^3,\tau_{\mathrm{std}})$, the $3$-sphere equipped with its unique real involution, the $OO$-real link Floer homology shares many nice structural properties with the one defined in \cite{YXHFLR}, which we will refer to as the $OX$-theory in the current paper. 

\begin{proposition}\label{intro-prop:summary of OO properties}
The $OO$-real link Floer homologies in $(S^3,\tau_{\mathrm{std}})$ satisfy the following: 
\begin{enumerate}
	\item They admit a combinatorial description via real grid diagrams, making them algorithmically computable. A computer implementation is provided in \cite{ZhenkunLioythonprgram} and a list of more than 50 examples is provided in Appendix~\ref{app:Calculation results for knots with small crossing numbers}.
	\item There are crossing change maps $C_{A}^{\pm}$ ($C_{B}^{\pm}$) on $\HFLR^-_{OO}(L,\fro)$ and $\HFLR^-_{O}(L,\fro)$ associated to equivariant crossing changes of type $A$ (type $B$) (see \cite{boyle2025equivariantunknottingnumbersstrongly} or \cite[Definition~5.1]{YXHFLR}). 
	\item There are saddle maps $\sigma$ and $\mu$ on $\HFLR^-_{O}(L,\fro)$ associated to attachment of pairs of $1$-handles to an equivariant link cobordism. We actually have two pairs of $(\sigma,\mu)$ according to how the $1$-handles change the number of components in $L$.
	\item $\HFLR^-_{O}$ and $\widehat{\HFLR}_{O}$ groups of an oriented skein triple (see Figure~\ref{fig:oriented skein triple}) fit into an exact triangle. 
\end{enumerate}
\end{proposition}

As in \cite[Section~5]{YXHFLR}, we can define $\ord_{u}^O(K,\fro)$ as the torsion order of $\HFLR^-_{O}(K,\fro)$ as an $\F[u]$-module, which gives rise to bounds on the type $A$ and $B$ unknotting numbers using (2). Combining with a direct calculation for the unknot, we obtain another corollary of (2) of Proposition~\ref{intro-prop:summary of OO properties}: for a strongly invertible knot $K$ in $S^3$, $\HFLR^-_{O}(K,\fro)/\Tors$ has rank $2$ as an $\F[u]$-module, while $\HFLR^-_{OO}(K,\fro)/\Tors$ has rank $1$ as an $\F[u_i]$-module, for $i=1,2$. Motivated by the $\tau$-set invariant for links, we extract two real $\tau$-invariants $\tau_{i}^R$ ($i=1,2$) from $\HFLR^-_{OO}(K,\fro)$ and a $\tau^R$-set invariant $\tau_{\min}^R\le \tau_{\max}^R$ from $\HFLR^-_{O}(K,\fro)$. 

\begin{prop}\label{intro-prop:tau invariants}
	For any oriented strongly invertible knot $(K,\fro)$ in $(S^3,\tau_{\std})$, 
	\[\tau^R_1(K,\fro)= \tau^R_2(K,\fro).\]
	Moreover, \[\tau^R_{\max}(K,\fro) = \tau_i^R(K,\fro)+1+n \quad \text{and}\quad \tau^R_{\min}(K,\fro)\le \tau^R_i(K,\fro),\]
where $n=n(K,\fro)$ measures the difference between $u_1$ and $u_2$ actions on $\HFLR^-_{OO}(K,\fro)$. See Section~\ref{sub:Torsion order, real tau-invariant and module structure} for its precise meaning and more detailed discussion on $\tau^R$-invariants.
\end{prop}

Despite these similarities, the $OX$ and $OO$-theories have many differences. We will see this through analysis of the curvature in Section~\ref{sub:Killing the curvature} and examples in Section~\ref{sub:Examples}. More interestingly, suitable minus and hat $OO$ and $OX$-real link Floer groups fit into an exact triangle. 
\begin{thm}\label{intro-thm:exact triangle between OX and OO for knots}
	Let $(K,\fra)$ be a strongly invertible knot with auxiliary data and $(K,\fro)$ be the underlying oriented strongly invertible knot.
	Then there is an exact triangle \[\begin{tikzcd}
	\HFLR^-(K,\fra) \arrow{rr}{(-1,-1/2)} & & \HFLR^-(K,\fra) \arrow{dl}{(0,1/2)}\\
		& \HFLR^-_{O}(K,\fro) \arrow{ul}{(0,0)}\\
	\end{tikzcd}\]
An arrow labeled $(d,s)$ shifts the $(M^R,A^R)$-bigrading by $(d,s)$. These triangles split naturally along real $\spinc$ structures. There is a similar triangle relating $\widehat{\HFLR}_{O}$ and $\widehat{\HFLR}$.
\end{thm}

See Section~\ref{sub:Basic setup} for the definition of bigradings and see Theorem~\ref{thm:spectral sequence relating OO and OX for links} for a generalization to links. This exact triangle also enables us to compare $\tau_{\max}^R$, $\tau_{\min}^R$ with $\tau^R$ defined in \cite[Definition~6.1]{YXHFLR}. 

\begin{cor}\label{intro-cor:bounding tauR_max, tauR_min by tauR}
Let $(K,\fra)$ be a strongly invertible knot with auxiliary data in $S^3$ and $(K,\fro)$ be the underlying oriented strongly invertible knot. Then we have inequalities \[\tau_{\min}^R(K,\fro)+1\le \tau^R(K,\fra)\le \tau_{\max}^R(K,\fro).\]
\end{cor}

One most obvious difference between $OO$ and $OX$-theories is that for each strongly invertible knot, we have two $OX$-theories, essentially depending on a choice of direction, but only one $OO$-theory. One may guess that the $OO$-theory no longer has a direction dependence. However, we shall see in Section~\ref{sub:Torsion order, real tau-invariant and module structure} that the two variables $u_1$, $u_2$ sometimes have non-homotopic actions on $\HFLR^-_{OO}$, which means though both fixed points are labelled by $O$, their roles are sometimes unequal in the real link Floer theory.

Having $OO$ and $OX$-theories in hand, it is then natural to ask about an $XX$-version. In Section~\ref{sub:XX-theory}, we briefly discuss the properties of $XX$-real link Floer homology and its relationship with the $OX$-theory.

After the comparison, we move to define a real link Floer theory $\fullCFLR^-$ for multi-based strongly invertible links (see Definition~\ref{def:based strongly invertible links}) which is in general a multi-graded curved complex over a polynomial ring. This has been shown to be a natural invariant. 

\begin{thm}\label{intro-thm:invariance and naturality of the full real link Floer complex}(\cite[Theorem~2]{GM_real_naturality})
Let $(Y,\tau)$ be a closed real $3$-manifold. Fix a multi-based strongly invertible link $\LL\subset (Y,\tau)$ (see Definition~\ref{def:based strongly invertible links}) and a real $\spinc$ structure $\s^R\in \rspinc(Y,\tau)$. If $\cH$ and $\cH'$ are two strongly $\s^R$-admissible real Heegaard diagrams representing $\LL$, then there is $\Z^{\vert \bfO^f\vert+\frac{1}{2}\vert\bfO^p\vert} \oplus \Z^{\vert \bfX^f\vert+\frac{1}{2}\vert\bfX^p\vert}$-filtered map 

\[\Phi_{\cH\to \cH'}\colon \fullCFLR^-(\cH,\s^R) \to \fullCFLR^-(\cH',\s^R)\] which is a chain homotopy equivalence in the category of curved complexes of $\cR^-(\LL)$-modules. This will be referred to as a \emph{change of diagram map}. It preserves the relative gradings $M^R_{\bfO}$, $M^R_{\bfX}$ as well as $A^R$ when they are well-defined.  
	
If $\cH''$ is yet another strongly $\s^R$-admissible real Heegaard diagram representing $\LL$, then the change of diagram maps satisfy \[\Phi_{\cH\to \cH''}\simeq \Phi_{\cH'\to \cH''} \circ \Phi_{\cH\to \cH'}.\]    
\end{thm}

See Section~\ref{sub:Upgrade to a based link invariant} for terminologies in the statement. The real link Floer complexes can be colored via homomorphism from the base ring (see Definition~\ref{def:coloring on link and colored real link Floer}, and compare~\cite{Zemkequasistabandbasepointmoving,zemke2019link}), and the naturality is inherited by the colored invariants.

Having the naturality in hand, we will make our first attempt in making $\fullCFLR^-$ a TQFT. Motivated by \cite{Zemkequasistabandbasepointmoving,zemke2019link}, we introduce base point actions and quasi-stabilization maps on $\fullCFLR^-$. 

\begin{proposition}\label{intro-prop:base point action}
We have base point actions $\Phi_O^f$ (of degree $(0,1/2)$) and $\Phi_{(O,O')}^p$ (of degree $(1,1)$) on $\fullCFLR^-$ associated to fixed $O$-base points and pairs of $O$-base points, respectively. When the base points are suitably colored, they are chain maps and are well-defined on the level of transitive system. We also have similar actions $\Psi_X^f$ and $\Psi_{(X,X')}^p$.
\end{proposition}

\begin{proposition}\label{intro-prop:quasi-stab}
We have maps $S^{\pm}_{O,X}$, $T^{\pm}_{O,X}$ (and their counterparts with the order of $O$ and $X$ changed) associated to pairs of quasi-stabilizations performed equivariantly away from the fixed set. When the link is suitably colored, they are chain maps and are well-defined on the level of transitive system. See Section~\ref{subsub:Quasi-stabilizations without type change} for distinctions between different versions.  

Moreover, when the complex is suitably colored, we have type-changing quasi-stabilization maps $Q_{X\mapsto O}^+$ and $Q_{O\mapsto X}^-$ associated to quasi-stabilizations occurring on the fixed set (see Figure~\ref{fig:OX-OO stabilization} for an intuition). They are also well-defined on the level of transitive system. When the roles of $O$ and $X$ are interchanged, we have similar maps $P_{O\mapsto X}^+$ and $P_{X\mapsto O}^-$.
\end{proposition}

Motivated by a discussion with Gary Guth and Ciprian Manolescu, we introduce a special coloring $\sigma_s$ and a bounding chain $\bL$ for $\fullCFLR^-(\LL^{\sigma_s})$ (see the end of Section~\ref{sub:Upgrade to a based link invariant} and compare~\cite[Section~4.5]{guth2025real}). The new differential $D=\partial+ \bL\cdot \id$ satisfies $D^2=0$, thus makes $(\fullCFLR^-(\LL^{\sigma_s}),D)$ a chain complex. We will see that  $\sigma_s$ is injective and $\bL$ is a purely combinatorial expression determined by $\LL$, so no information is lost when passing from $(\fullCFLR^-(\LL),\partial)$ to $(\fullCFLR^-(\LL^{\sigma_s}),D)$. This new chain complex admits well-behaved base point actions (see Section~\ref{subsub:Base point actions on Dcomplex}) and quasi-stabilization maps. The latter are even better than the original ones, as we do not need to kill any variable to obtain chain maps (see Section~\ref{subsub:Quasi-stabilization maps on Dcomplex}).

Finally, in Section~\ref{sub:example of full real complex}, we calculate $\fullCFKR^-$ for the minimally pointed trefoil knots, their connected sum and a multi-based knot after type-changing quasi-stabilization. On these simplest examples, we already see distinctions between usual and real link Floer theory as well as the properties discussed in the previous sections of this paper. 

\begin{remark}
Readers familiar with Zemke's work should have noticed that our notation follows his closely. On one hand, we will see in Section~\ref{sub:base point action}-\ref{sub:Relationship between base point actions and quasi-stabilization maps} that $\Phi^p$, $\Psi^p$ as well as the type-$S$ and type-$T$ quasi-stabilization maps share similar properties with their counterparts in usual link Floer theory sharing the same names. On the other hand, we will also see several notable distinctions between the real operators and their usual counterparts. For example, one can compare Lemma~\ref{lem:squares of the base point action} with \cite[Lemma~4.9]{zemke2019link}. 

More importantly, we have two brand-new objects. One is the type-$Q$ and type-$P$ quasi-stabilizations, which have no usual counterpart. Their idea originates from the exact triangle in Theorem~\ref{intro-thm:exact triangle between OX and OO for knots}. Actually, they measure how far the rank fails to double under a fixed point quasi-stabilizations or alternatively, how far a K\"unneth formula is from being true when taking an equivariant connected sum with an unknot described by diagrams in Figure~\ref{fig:std_pieces_for_type-changing_quasi_stab}. The chain complex $(\fullCFLR^-(\LL^{\sigma_s}),D)$ is the other, which removes the curvature obstruction to relating real link Floer complexes of different multi-based links. The idea of a bounding chain comes from Lagrangian Floer homology and the strategy does not apply to the usual link Floer complex.
\end{remark}

\subsection{Organization}
In Section~\ref{sec:A different theory for strongly invertible links}, we construct the $OO$-version of real link Floer homology and analyze its basic properties, thereby proving Theorem~\ref{intro-thm:OO-real link Floer homology}, Proposition~\ref{intro-prop:summary of OO properties} and \ref{intro-prop:tau invariants}.  Section~\ref{sec:Property comparison} is devoted to the comparison between $OO$ and $OX$-theories. The distinction between curvatures is observed in Section~\ref{sub:Killing the curvature}; Theorem~\ref{intro-thm:exact triangle between OX and OO for knots} and  Corollary~\ref{intro-cor:bounding tauR_max, tauR_min by tauR} will be proved in Section~\ref{sub:Exact triangle and a common generalization} and concrete examples will be provided in Section~\ref{sub:Examples}. A brief discussion on $XX$-theory appears in Section~\ref{sub:XX-theory}. Section~\ref{sec:Based link invariants, base point action and quasi-stabilization} focuses on the multi-based theory: Theorem~\ref{intro-thm:invariance and naturality of the full real link Floer complex} is illustrated in Section~\ref{sub:Upgrade to a based link invariant}; Proposition~\ref{intro-prop:base point action} and \ref{intro-prop:quasi-stab} will be detailed and proved in Section~\ref{sub:base point action} and \ref{sub:quasi-stabilization}, respectively; Along the way, we consider properties of $(\fullCFLR^-(\LL^{\sigma_s}),D)$;  The commutative relations between these operators are analyzed in Section~\ref{sub:Relationship between base point actions and quasi-stabilization maps}; Finally, we compute some real knot complexes explicitly in Section~\ref{sub:example of full real complex}. 

\,

\noindent{\bf Note on AI usage.} We used GPT-5.6 Sol to find grammar and typographical problems during the revision process. The corrections were made by hand. Trea was used for coding and preparing Appendix~\ref{app:Calculation results for knots with small crossing numbers}. A more detailed explanation on AI usage is provided at the beginning of Appendix~\ref{app:Calculation results for knots with small crossing numbers}. The author takes full responsibility
for the correctness of all mathematical statements and proofs. 
 
\begin{ack}
The author would like to thank Zhenkun Li for useful discussions during the preparation of this paper and his contribution to computing examples and preparing Appendix~\ref{app:Calculation results for knots with small crossing numbers} using Python programs. The idea of this paper originated from inspiring discussion with Gary Guth and Ciprian Manolescu during the author's visit to Stanford University, for which the author is very grateful. The author is also grateful to Irving Dai for discussions concerning full knot Floer complex and to Ian Zemke for explanations of invariance and naturality in the usual link Floer TQFT.
\end{ack}

\section{A different theory for strongly invertible links}\label{sec:A different theory for strongly invertible links}

In this section, we consider a variation of the real link Floer theory defined in \cite{YXHFLR} for generalized strongly invertible links by requiring a different base-pointing convention on real Heegaard diagrams. Definitions and basic properties will be provided, and we also define some numerical invariants. The relationship between this theory and the one from \cite{YXHFLR} will be analyzed in Section~\ref{sec:Property comparison}.

\subsection{Basic setup}\label{sub:Basic setup}

We first recall some basic notions in the theory of strongly invertible links.

\begin{defn}\label{def:strongly invertible links}
Let $(Y,\tau)$ be a closed oriented real $3$-manifold, i.e., an oriented closed $3$-manifold equipped with an orientation-preserving involution $\tau$ whose fixed set is $1$-dimensional and let $L$ be an equivariant link in $Y$.
\begin{itemize}
        \item $L$ is called \emph{strongly invertible} if for any choice of orientation on $L$, $\tau$ restricts to an orientation-reversing involution on each component of $L$. 
        \item $L$ is called \emph{generalized strongly invertible} if there exists an orientation on $L$ such that $\tau$ restricts to an orientation-reversing involution on $L$. Here, we allow pairs of components to be interchanged in an orientation-reversing way.
    \end{itemize} 
\end{defn}

\begin{defn}\label{def:O-heegaard diagram for s.i.links}
Let $L$ be a generalized strongly invertible link in $(Y,\tau)$. Fix an orientation $\fro$ on $L$ that fits $L$ into the definition. A multi-based real Heegaard diagram $\cH=(\Sigma,\bm\alpha,\bm\beta,\bfO,\bfX, R)$ is called a \emph{$OO$-based real Heegaard diagram} for $(L,\fro)$ if:
    \begin{itemize}
        \item $(\Sigma,\bm\alpha,\bm\beta,\bfO, R)$ and $(\Sigma,\bm\alpha,\bm\beta,\bfX, R)$ are both multi-based real Heegaard diagrams for $(Y,\tau)$. Here, we only require $\bfO$ and $\bfX$ to be fixed as sets by $R$, which is more general than the notion of real Heegaard diagram defined in \cite{guth2025real}.
        \item Each connected component of $\Sigma\setminus \bm\alpha$ contains exactly one $O$-mark and one $X$-mark. The same holds for $\Sigma\setminus \bm\beta$.
        \item Connect arcs from $\bfO$ to $\bfX$ in the complement of $\bm\beta$ and push their interiors into the $\beta$-handlebody. Then, connect arcs from $\bfX$ to $\bfO$ in the complement of $\bm\alpha$ and push their interior into the $\alpha$-handlebody. Performing these simultaneously and equivariantly, we get an oriented equivariant link in $(Y,\tau)$. We require it to be equivariantly isotopic to $(L,\fro)$.
        \item If $l_f$ is the number of components in $L$ that are fixed setwise by $\tau$, then $\vert \fix (R)\cap \bfO\vert =2l_f$ and $\vert \fix (R)\cap \bfX\vert =0$. (This means that on each strongly invertible knot component, we have its two fixed points labeled by $O$.)
    \end{itemize}
    Furthermore, $\cH$ is said to be \emph{minimal} if there are exactly four base points on each strongly invertible component and exactly four base points on each pair of components interchanged by $\tau$.
\end{defn}

\begin{convention}\label{conv:setup of holomorphic structure}
Since our theory closely follows from those in \cite{guth2025real} and \cite{BGX}, we decide not to include basic definitions in real Heegaard Floer theory. We will set up some notation here and refer readers to these papers for details. 
\begin{itemize}
    \item $\mathrm{Sym}^m(\Sigma)$ will denote the $m$-fold symmetric product of $\Sigma$, on which we have an induced involution which we still denote by $R$. For a point $p\in \Sigma$ and a map $\phi$ from $D^2$ to $\mathrm{Sym}^m(\Sigma)$, we use $n_{p}(\phi)$ to denote the intersection number $\# (\im(\phi)\cap \{p\} \times \mathrm{Sym}^{m-1}(\Sigma))$, when the intersection is transverse. 
    
    \item $\T_\alpha$ and $\T_\beta$ are two Lagrangian tori (This can be achieved by using a special class of symplectic forms on $\mathrm{Sym}^m(\Sigma)$, see~\cite{guth2025real}) in $\mathrm{Sym}^m(\Sigma)$ coming from the products $\alpha_1\times \alpha_2\times \ldots\times \alpha_m$, $\beta_1\times \beta_2\times\ldots\times \beta_m$, in which $m=\vert\bm\alpha\vert=\vert\bm\beta\vert$. By definition, $R$ interchanges $\T_\alpha$ and $\T_\beta$ diffeomorphically.

    \item We always fix a generic choice of a symmetric almost complex structure on $\mathrm{Sym}^m(\Sigma)$, so that all the index one moduli spaces (see below) are cut out transversely.
   
    \item We will assume $\T_\alpha$ and $\T_\beta$ intersect transversely, so that $(\T_\alpha\cap \T_\beta)^R$ is a finite set of points. These points will be refer to as generators.
    
    \item For $\xv, \yv \in (\T_\alpha\cap \T_\beta)^R$, we use $\pi_2^R(\xv,\yv)$ to denotes the set of homotopy classes of disks connecting $\xv$ to $\yv$. For $\phi\in \pi_2^R(\xv,\yv) $, $\ind_R(\phi)$ denotes the Fredholm index of the $\bar{\partial}$-operator acting on the space of real sections after suitable Sobolev completion and $\mu_R(\phi)$ denotes its real Maslov index. These two indices coincide on most classes that we will consider. Moreover, $\widehat{\cM}_R(\phi)$ denotes the moduli space of real holomorphic representatives of $\phi$ modulo the natural $\R$-translation on $\C \supset D^2-\{\pm i\}\cong [0,1]\times \R$ and $\#\widehat{\cM}_R(\phi)$ denotes the modulo $2$ count of points in it when it is of dimension zero.

    \item We will use $\rspinc(Y,\tau)$ to denote the set of real $\spinc$ structures on the real manifold $(Y,\tau)$. There is a complete characterization of this in \cite[Section~3.5-3.8]{guth2025real} and a classification theorem was proved in \cite{li2022monopolefloerhomologyreal}. As in \cite[Section~3.7]{guth2025real}, we can assign to each generator $\xv$, a real $\spinc$ structure $\s^R(\xv)$. When we do this assignment, we always use $\bfO$ as the base points of the underlying real $3$-manifold.

    \item We need some real admissibility assumptions on the diagram for the homology theories to be well-defined. In \cite[Section~3.9]{guth2025real}, the authors introduced two notions of admissibility on real Heegaard diagrams. The weak admissibility works once and for all, while the strong admissibility is defined for each real $\spinc$ structure.
    One can generalize the argument from \cite[Lemma~3.33-3.34]{guth2025real} to show that \begin{itemize}
        \item Fix a real $\spinc$ structure $\s^R$. Any real Heegaard diagram is equivariantly isotopic to a $\s^R$-strongly admissible one. 
        \item Any two strongly $\s^R$-admissible (weakly admissible) real Heegaard diagrams can be connected by a sequence of real Heegaard moves such that each intermediate diagram is strongly $\s^R$-admissible (weakly admissible).
    \end{itemize}
    
\end{itemize}
\end{convention}

\begin{prop}\label{prop:existence of real HD for strongly invertible knots and real H moves}
Let $L$ be a generalized strongly invertible link in $(Y,\tau)$ and  $\fro$ be an orientation on $L$ that fits $L$ into the definition. Then there exists a (minimal) $OO$-real Heegaard diagram representing $(L,\fro)$. Moreover, if $\cH$ and $\cH'$ both represent $(L,\fro)$, then they can be connected by a sequence of real Heegaard moves:
\begin{itemize}
    \item $\{1\}$-(de)stabilization (see \cite[Definition~2.13]{BGX}) along the fixed set away from the base points;
    \item $\Z/2$-(de)stabilizations (of index (1,2) or (0,3)) away from the fixed point set and the base points;
    \item real handleslide away from the base points;
    \item equivariant isotopy of $\alpha$ and $\beta$-curves away from the base points.
\end{itemize}
Moreover, for any real $\spinc$ structure $\s^R$ on $(Y,\tau)$, we can choose a strongly $\s^R$-admissible $OO$-real Heegaard diagram $\cH$ representing $(L,\fro)$. If two $OO$-real Heegaard diagrams representing 
$(L,\fro)$ are both strongly $\s^R$-admissible, then they can be connected by a sequence of real Heegaard moves such that each intermediate real Heegaard diagram is also strongly $\s^R$-admissible.
\end{prop}
\begin{proof}
This can be proved in the same way as \cite[Proposition~2.5]{YXHFLR}. We only make a remark about existence. Note that a minimal diagram of $(L,\fra)$ is exactly a real sutured Heegaard diagram for the real sutured manifold $(Y-\nu(L),\gamma,\tau)$ characterized in \cite[Example~2.5]{BGX} with sutures suitably decorated by $O$ and $X$. Thus, the existence follows from \cite[Proposition~2.12]{BGX}.
\end{proof}

Now we fix a closed real $3$-manifold $(Y,\tau)$ and an oriented generalized strongly invertible link $(L,\fro)$ in it as well as a real $\spinc$ structure $\s^R\in \rspinc(Y,\tau)$. Let $\cH=(\Sigma,\bm\alpha,\bm\beta,\bfO,\bfX,R)$ be a strongly $\s^R$-admissible $OO$-real Heegaard diagram representing $(L,\fro)$. We will define an $OO$-version of real link Floer invariant as follows. 

Fix an order of strongly invertible components of $L$ by $1,2,\ldots,l_f$ and an order of pairs of components in $L$ by $1,2,\ldots,l_p$. Label the base points so that the fixed $O$-base points on the $i$-th strongly invertible component are $O^f_{i,1}$ and $O^f_{i,2}$. All other $O$-base points are $(O_1,O_1')$, $\ldots$, $(O_k,O_k')$ so that for $1\le i\le l_p$, $(O_i,O_i')$ lie on the $i$-th pair of components. The $X$-base points are $(X_1,X_1')$, $\ldots$, $(X_{k+l_f},X_{k+l_f}')$, among which $(X_i,X_i')$ lie on the $i$-th pair of components. 

The base ring $\cR(\cH)$ is \[\F[u_{1,1}, u_{1,2},\ldots u_{l_f,1},u_{l_f,2}, U_1,\ldots, U_k, V_1,\ldots V_{k+l_f}].\] We grade the variables as follows: \[M_{\bfO}^R(u_{i,j})=-1, \quad M_{\bfX}^R(u_{i,j})=0,\quad \forall i,j; \]
\[M_{\bfO}^R(U_{i})=-2, \quad M_{\bfX}^R(U_{i})=0,\quad \forall i; \quad M_{\bfO}^R(V_{j})=0, \quad M_{\bfX}^R(V_{j})=-2,\quad \forall j; \] while for any variable, \[A^R(-)=\frac{1}{2}(M_{\bfO}^R(-)-M_{\bfX}^R(-)).\] 
Throughout this paper, when we talk about bigrading, we mean $(M^R_{\bfO},A^R)$ and we will often omit $\bfO$ from the notation.

We define the \emph{$OO$-(minus) full real link Floer complex} $\fullCFLR_{OO}^-(\cH,\s^R)$ to be the module over $\cR(\cH)$ with generators \[\{\xv\in (\T_\alpha\cap \T_\beta)^R| \s^R(\xv)=\s^R \}.\] 
We define an endomorphism on it by \[\partial \xv =\sum_{\yv} \sum_{\phi\in \pi_2^R(\xv,\yv),\mu_R(\phi)=1}\# \widehat{\cM}_R(\phi) \prod_{1\le i\le l_f, 1\le j\le 2}  u_{i,j}^{n_{O_{i,j}^f}(\phi)}\prod_{1\le i\le k} U_i^{n_{O_i}(\phi)} \prod_{1\le j\le k+l_f} V_j^{n_{X_j}(\phi)}\yv\] on generators and extend it linearly over the base ring. 

The strongly $\s^R$-admissible assumption ensures that this is a finite sum (see \cite[Lemma~3.32]{guth2025real}). However, $\partial$ is not a genuine differential, since $\partial^2$ is not zero. We calculate the curvature as follows. Let  $K\subset L$ be a strongly invertible knot. We temporarily label the base points so that the fixed $O$ base points are $O^f_1$ and $O^f_{n+1}$ and along one arc from $O^f_1$ to $O^f_{n+1}$, it reads \[O^f_1, X_1, \ldots, O_{n-1}, X_{n}, O^f_{n+1}\] while along the other, it reads \[O^f_1, X_1', \ldots, O_{n-1}', X_{n}', O^f_{n+1}.\]
In this way, we have degree $(-1,1/2)$ variables $u_1$ and $u_{n}$ account for $O_1$ and $O_{n+1}$, and degree $(-2,-1)$ variables $U_{i}$ $(2\le i\le n)$ and degree $(0,+1)$ variables $V_{j}$ $(1\le j\le n)$ account for pairs $(O_i,O_i')$ and $(X_j,X_j')$.  
Now, we define \[\omega^O_{K}=u_1^2\cdot V_1+ U_2\cdot V_1+ U_2\cdot V_2+\ldots U_{n}\cdot V_{n-1}+ U_{n}\cdot V_n+ u_{n+1}^2\cdot V_n.\] 
If $(K,K')\subset L$, then we label the base points on $K$ so that it reads \[O_1, X_1,\ldots,O_n,X_n\] in it's orientation. Define  \[\omega_{K,K'}=U_1\cdot V_1+ U_2\cdot V_1+ U_2\cdot V_2+\ldots U_{n}\cdot V_{n-1}+ U_{n}\cdot V_n+ V_n\cdot U_1.\]  

Using these and \cite[Lemma~4.14-4.15]{guth2025real}, we can write down the curvature as \[\partial^2= (\sum_{K\subset L, K \text{ strongly invertible}} \omega^O_K+\sum_{(K,K')\subset L, \text{ }\tau(K)=K'}\omega_{K,K'})\cdot \id:=\omega^O_{\cH}\cdot \id.\]

Now, we consider suitable quotients of $\fullCFLR^-_{OO}(\cH,\s^R)$ in which the curvature becomes zero. In this way, we get genuine chain complexes, so the homology can be taken. 

First, consider \[\CFLR_{OO}^-(\cH,\s^R)=\fullCFLR_{OO}^-(\cH,\s^R)/ (V_j)_{1\le i\le l_f+k}.\] One can check that after blocking all $X$-base points, $\CFLR^-_{OO}(\cH,\s^R)$ is indeed a chain complex. Let $\partial^-$ denote the induced differential on it. Then, the homology 
\[\HFLR_{OO}^-(\cH,\s^R)=H_*(\CFLR_{OO}^-(\cH,\s^R),\partial^-)\] 
is an invariant of $(L,\fro)$ and $\s^R$ when it is regarded as an $\F[u_{1,1},u_{1,2},\ldots,u_{l_f,1},u_{l_f,2}, U_1,\ldots U_{l_p}]$-module. We will denote it by $\HFLR_{OO}^-(L,\fro,\s^R)$ and call it the \emph{$OO$-minus version of real link Floer homology} associated to $(L,\fro)$ in the real $\spinc$ structure $\s^R$. 

Then consider \[\CFLR_{O}^-(\cH,\s^R)= \frac{\CFLR_{OO}^-(\cH,\s^R)}{u_{i,1}=u_{i,2}, 1\le i\le l_f},\] which is also a chain complex. We abuse $\partial^-$ for the induced differential. The homology 
\[\HFLR_{O}^-(\cH,\s^R)=H_*(\CFLR_{O}^-(\cH,\s^R),\partial^-)\] 
is an invariant of $(L,\fro)$ and $\s^R$ when it is regarded as an $\F[u_{1},\ldots,u_{l_f}, U_1,\ldots U_{l_p}]$-module. We will denote it by $\HFLR_{O}^-(L,\fra,\s^R)$ and call it the \emph{$O$-minus version of real link Floer homology} associated to $(L,\fro)$ in the real $\spinc$ structure $\s^R$. 

We can further introduce the chain complex \[\widehat{\CFLR}_{O}(\cH,\s^R)=\CFLR_{O}^-(\cH,\s^R)/(u_i, U_j)_{1\le i\le l_f, 1\le j\le l_p} \] with induced differential $\widehat{\partial}$. Then 
\[\widehat{\HFLR}_{O}(\cH,\s^R)=H_{*}(\widehat{\CFLR}_{O}(\cH,\s^R),\widehat{\partial})\] 
is an invariant of $(L,\fro)$ and $\s^R$ as a vector space over $\F$. We will denote it by $\widehat{\HFLR}_O(L,\fro,\s^R)$ and call it the \emph{$O$-hat version of real link Floer homology} associated to $(L,\fro)$ in the real $\spinc$ structure $\s^R$.

A further variation is \[\widetilde{\CFLR}_O(\cH,\s^R)=\CFLR_{O}^-(\cH,\s^R)/(u_i,U_j)_{1\le i\le l_f,1\le j\le k}.\] Its homology $\widetilde{\HFLR}_O(\cH,\s^R)$ is not an invariant of $(L,\fro)$ and $\s^R$, but it appears as a suitable stabilization of $\widehat{\HFLR}_O(L,\fro,\s^R)$ according to the size of the real Heegaard diagram. More precisely, $\widetilde{\HFLR}_O(\cH,\s^R)=\widehat{\HFLR}_O(L,\fro,\s^R)\otimes \F^{2(k-l_p)}$.

Next, we consider relative bigradings on these link Floer theories, which share the same description as the one we characterized for $\HFLR^{\circ}$ in \cite[Section~2.1]{YXHFLR}. 
When $\s^R$ has torsion first Chern class, we can define a relative real Maslov grading $M^R$ on the generators by \[M^R_{\bfO}(\xv,\yv)=\mu_R(\phi)-n_{\bfO}(\phi)\]
for any $\phi\in \pi^R_2(\xv,\yv)$. The torsion assumption ensures that this is well-defined and independent of the choice of $\phi$. if we further assume that $L$ is null-homologous, then we can define \[M^R_{\bfX}(\xv,\yv)=\mu_R(\phi)-n_{\bfX}(\phi)\] and introduce the relative real Alexander grading as their difference \[A^R(\xv,\yv)= \frac{1}{2} (M^R_{\bfO}(\xv,\yv)-M^R_{\bfX}(\xv,\yv)).\] The null-homologous and  torsion assumptions make sure that this is also well-defined. These can be extended over the base ring naturally.
In general, if $c_1(\s^R)$ is non-torsion with divisibility $2n$, we can define a $\Z/n\Z$-valued real Maslov grading on $\CFLR^\circ_{O}$ using the same formula.

To get an invariant for the link $(L,\fro)$ inside $(Y,\tau)$, we take direct sum over all real $\spinc$ structures 
\[\HFLR_{OO}^-(Y,\tau, L,\fro)=\bigoplus_{\s^R\in \rspinc(Y,\tau)} \HFLR_{OO}^-(Y,\tau, L,\fro,\s^R); \]
\[\HFLR_{O}^-(Y,\tau, L,\fro)=\bigoplus_{\s^R\in \rspinc(Y,\tau)} \HFLR_{O}^-(Y,\tau, L,\fro,\s^R);\]
\[\widehat{\HFLR}_O(Y,\tau, L,\fro)=\bigoplus_{\s^R\in \rspinc(Y,\tau)} \widehat{\HFLR}_{O}(Y,\tau, L,\fro,\s^R).\]

For simplicity, we will often omit the real $3$-manifold from notation and write $\HFLR_{*}^{\circ}(L,\fro)$. 

Up to now, we have constructed everything we promised in Theorem~\ref{intro-thm:OO-real link Floer homology}. The promised bigrading on homology is provided by $(M^R,A^R)$. For the proof of invariance,  one need to check each real Heegaard move from Proposition~\ref{prop:existence of real HD for strongly invertible knots and real H moves} induces an isomorphism on $OO$-real link Floer groups. The invariance under $\Z/2$-stabilization of indices $(0,3)$ follows from \cite[Section~2.1]{YXHFLR} and invariance under all other moves follows from \cite{guth2025real} and \cite{GM_real_naturality}. This concludes the proof of Theorem~\ref{intro-thm:OO-real link Floer homology}.


As in Section 3 of \cite{YXHFLR}, the $OO$-real link Floer homologies admit a combinatorial description via real grid diagrams when we restrict to $(Y,\tau)=(S^3,\tau_{\std})\subset (\C^2,\text{conjugation})$. The real grid diagrams are always admissible in any sense. The computer program by Zhenkun Li (\cite{ZhenkunLioythonprgram}) allows us to compute examples with small real grid diagrams. Since the modification to arguments is straightforward, we omit the details. We will use notation such as $\GHR^-_{OO}$, $\GHR^-_{O}$ or $\widehat{\GHR}_{O}$ when we want to stress the combinatorial nature of the result and will use $\HFKR^{\circ}_{O}$ when we focus on strongly invertible knots.

The only part we want to emphasize is the combinatorial description of relative gradings, since in the $OO$-theory, the roles of $\bfO$ and $\bfX$ are no longer symmetric. We define absolute real Maslov gradings by 
\[M_{\bfO}^R(\xv)=\frac{1}{2} M_{\bfO}(\xv)-\frac{1}{4} \vert \xv\cap C \vert+\frac{1}{2}l_{f},\] 
\[M_{\bfX}^R(\xv)=\frac{1}{2} M_{\bfX}(\xv)-\frac{1}{4} \vert \xv\cap C\vert,\] 
Here, $M_{\bfO}$ and $M_{\bfX}$ are the usual Maslov functions given in \cite[Section~4.3]{OSS2015grid}. One can check directly that these take values in $\Z$.
For the real Alexander grading, an absolute lift is given by \[A^R(\xv)=\frac{1}{2}(M_{\bfO}^R(\xv) -M_{\bfX}^R(\xv))-\frac{(n+l_f-2l_p)}{4},\] so that it is still half of the usual Alexander grading. Then, $A^R$ values in $\frac{1}{2}\Z$ as in \cite{YXHFLR}.

Taking the graded Euler characteristic of\textbf{} $\widehat{\HFLR}_{O,*}(L,\fro,*)$, we define an \emph{$O$-real Alexander polynomial} by \[\Delta^R_{O}(L,\fro)(t)=(t-t^{-1})^{-l_p}\sum_{s\in \frac{1}{2}\Z} \sum_{d\in \Z} (-1)^d t^{2s} \dim \widehat{\HFLR}_{O,d}(L,\fro,s).\]
We will see in the next section that this is closely related to $\Delta^R(L,\fra)(t)$ defined in \cite[Section~7]{YXHFLR}, whose definition depends on the choice of auxiliary data, but the resulting polynomial invariant does not.

\subsection{Structural properties in $S^3$}\label{sub:Structural properties in $S^3$}
In this subsection, we consider structural properties of $\HFLR^{\circ}_{O}$ in $(S^3,\tau_{\std})$ via the combinatorial version $\GHR^{\circ}_{O}$, as listed in Proposition~\ref{intro-prop:summary of OO properties}. This will be our background manifold throughout this subsection unless otherwise stated. The discussion follows \cite[Section~4-7]{YXHFLR} closely, so most of the proofs will be omitted.

In the $OO$-theory, the roles of $\bfO$ and $\bfX$ are asymmetric, so most ``symmetric'' properties of $\GHR^{\circ}$ analyzed in \cite[Section~4.1]{YXHFLR} no longer hold in $OO$-theory. The invariants associated to $(K,\fro)$ and $(m(K),m(\fro))$ are expected to be closely related, but there is no explicit formula analogous to \cite[Proposition~4.4]{YXHFLR}. On the other hand, away from the fixed set, we see no difference between a real Heegaard diagram $\cH$ representing a generalized strongly invertible link with auxiliary data $(L_1,\fra)$ and an $OO$-real Heegaard diagram $\cH_{O}$ representing an oriented generalized strongly invertible link $(L_2,\fro)$. Thus, the unknotting maps from \cite[Section~5.1]{YXHFLR}, the saddle maps from \cite[Section~6.2]{YXHFLR} and the oriented skein relation from \cite[Section~7.1]{YXHFLR} still work in $OO$-theory, since their proofs are all based on local pictures in real grid diagrams away from the fixed set. 

As a preliminary result, we analyze the relationship between actions of different variables. 

\begin{prop}\label{prop:U,v action identification}
Let $\cH$ be an $OO$-real grid diagram representing $(K,\fro)$, an oriented strongly invertible knot. Then there are two degree $(-2,-1)$ variables $u_1, u_2$ in $\cR(\cH)$, and for any $1\le j\le k$ ($k$ is the number of paired $O$ base points in $\cH$), the $U_j$-action on $\GCR_{OO}^-(\cH)$ is homotopic to the $u_i^2$-action for $i=1,2$. In particular,  $u_1^2\simeq u_2^2$ as maps $\GCR_{OO}^-(\cH) \to \GCR_{OO}^-(\cH)$. 

More generally, let $\cH$ be any $OO$-real diagram representing an oriented generalized strongly invertible link $(L,\fro)$. If $(O_i,O_i')$ and $O^f_{t,j}$ lie on the same component of $L$, then the $U_i$-action on $\GCR^-_{OO}(\cH)$ is homotopic to the $u_{t,j}^2$-action. If $(O_i,O_i')$ and $(O_j,O_j')$ lie on the same component or the same pair of components of $L$, then the $U_i$ and $U_j$-actions on $\GCR^-_{OO}(\cH)$ are homotopic.    

There are similar relations on $\GCR_{O}^-(\cH)$, since $\GCR^-_{O}(\cH)$ is the quotient of $\GCR^-_{OO}(\cH)$ by $u_{i,1}=u_{i,2}$ for $1\le i\le l_f$. 
\end{prop}
\begin{proof}
This follows from \cite[Lemma~4.6.9]{OSS2015grid} (cf.~\cite[Proposition~3.7]{YXHFLR}).
\end{proof}

In \cite{boyle2025equivariantunknottingnumbersstrongly}, Boyle and Chen introduced three types of equivariant unknotting operations and defined versions of equivariant unknotting numbers for strongly invertible knots. For type $A$ and $B$ crossing changes, we have induced maps on real grid homologies. 

\begin{prop} \label{prop:unknotting maps}
Let $(K_+,\fro_+),(K_-,\fro_-)$ be a pair of oriented strongly invertible knots. When $(K_-,\fro_-)$ is obtained from $(K_+,\fro_+)$ by a $+$ to $-$ type $A$ crossing change, we have $\F[u_1,u_2]$-module maps 
\[C_-^A\colon \GHR_{OO}^-(K_+,\fro_+)\to \GHR_{OO}^-(K_-,\fro_-),\quad C_+^A\colon \GHR_{OO}^-(K_-,\fro_-)\to \GHR_{OO}^-(K_+,\fro_+)\] of bigrading $(M^R,A^R)=(0,0)$, $(-2,-1)$, respectively, so that \[C_+^A\circ C_-^A=u_1^2,\quad C_-^A\circ C_+^A=u_1^2,\] where $u_1^2$ means that module action given by multiplication by $u_1^2$.

When $K_+$ and $K_-$ are related by a type $B$ crossing change, then we have $\F[u_1,u_2]$-module maps \[C_-^B\colon \GHR_{OO}^-(K_+,\fro_+)\to \GHR_{OO}^-(K_-,\fro_-),\quad C_+^B\colon \GHR_{OO}^-(K_-,\fro_-)\to \GHR_{OO}^-(K_+,\fro_+)\] of bigrading $(M^R,A^R)=(-1,-1/2)$, $(-1,-1/2)$, respectively, so that \[C_+^B\circ C_-^B=u_1^2,\quad C_-^B\circ C_+^B=u_1^2.\] 

The same is true with $\GHR_{O}^-$ in place of $\GHR_{OO}^-$ and $\F[u_1,u_2]$ replaced by $\F[u]$ .
\end{prop}
\begin{proof}
The proof of \cite[Proposition~5.2]{YXHFLR} works here verbatim. 
\end{proof}

In \cite[Section~6.1]{YXHFLR}, we defined several versions of equivariant saddle moves and associated maps on collapsed grid homologies to the paired split/merge and split merge moves. The same approach works in $OO$-theory. More precisely, for an oriented generalized strongly invertible link $(L,\fro)$, fix an $OO$-real grid diagram $\cH$ for it. After labeling its base points as in Section~\ref{sub:Basic setup}, we consider the following quotient of $\GCR^-_{O}$ :\[c\GCR^-_{O}(\cH)=\frac{\GCR^-_{O}(\cH)}{u_{1,1}=u_{2,1}=\ldots u_{l_f,1}, u_{1,1}^2=U_1=\ldots=U_{l_p}},\] which has a natural induced differential from $\partial^-$.  The homology \[c\GHR^-_{O}(\cH)=H_*(c\GCR^-_{O}(\cH),\partial^-)\] is called the \emph{$O$-collapsed real grid homology} of $(L,\fro)$, which is an invariant of $(L,\fro)$ in the category of bigraded $\F[u]$-modules. This is because the homotopy type of $\GCR^-_{O}(\cH)$ is actually an invariant of $(L,\fro)$. Note that for a strongly invertible knot, we actually collapse nothing. 

\begin{prop}\label{prop:band maps}
Let $W=\F_{(0,0)}\oplus\F_{(-1,-1)}$. 
\begin{enumerate}
    \item If $L'$ is obtained from $L$ by a paired split move, for any choice of compatible orientations $\fro$ and $\fro'$, there are $\F[u]$-module maps 
    \[\sigma_p\colon c\GHR_{O}^-(L,\fro)\otimes W\to c\GHR_{O}^-(L',\fro')\]
    \[\mu_p\colon c\GHR_{O}^-(L',\fro')\to c\GHR_{O}^-(L,\fro)\otimes W \] with the following properties: \begin{itemize}
        \item $\sigma_p$ is homogeneous of degree $(-1,0)$;
        \item $\mu_p$ is homogeneous of degree $(-1,-1)$;
        \item $\mu_p\circ\sigma_p$ is the multiplication by $u^2$;
        \item $\sigma_p\circ\mu_p$ is the multiplication by $u^2$.
    \end{itemize}
    \item If $L'$ is obtained from $L$ by a paired merge-split move, then there are $\F[u]$-module maps 
    \[\sigma\colon c\GHR_{O}^-(L,\fro)\otimes W\to c\GHR_{O}^-(L',\fro')\otimes W\]
    \[\mu\colon c\GHR_{O}^-(L',\fro')\otimes W\to c\GHR_{O}^-(L,\fro)\otimes W\] with the following properties: \begin{itemize}
        \item $\sigma$ is homogeneous of degree $(-1,-1/2)$;
        \item $\mu$ is homogeneous of degree $(-1,-1/2)$;
        \item $\mu\circ\sigma$ is the multiplication by $u^2$;
        \item $\sigma\circ\mu$ is the multiplication by $u^2$.
    \end{itemize}
\end{enumerate}
\end{prop}
\begin{proof}
The proof of \cite[Proposition~6.1]{YXHFLR} works here verbatim. 
\end{proof}

\begin{thm}\label{thm:O-minus oriented skein triple}
Let $(L_+, L_- ,L_0,\fro)$ be a real oriented skein triple as shown in Figure~\ref{fig:oriented skein triple}. Here, we fix $\fro$ to be a consistent choice of orientations on all three links and omit it from the notation. 

\begin{figure}
    \centering
    \begin{overpic}[width=0.7\textwidth]{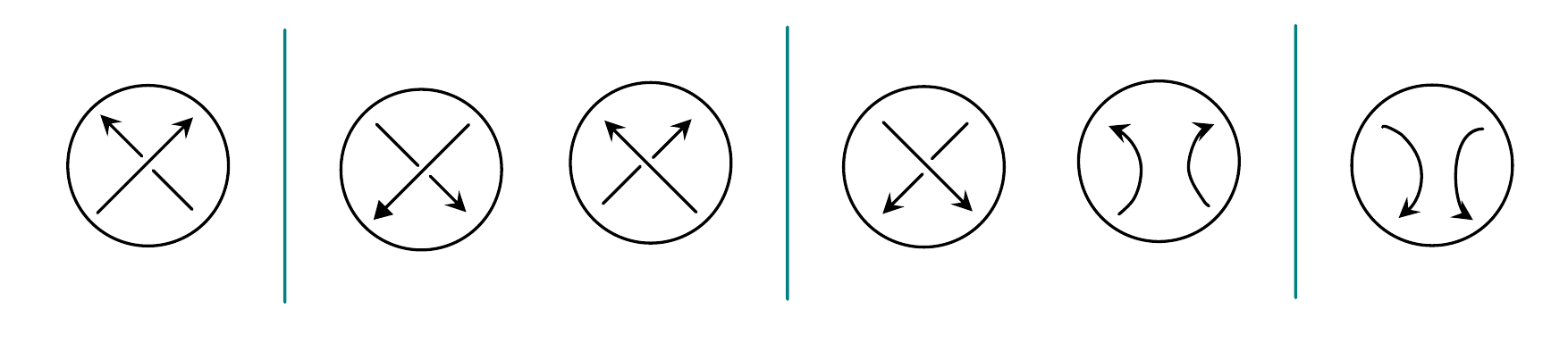}
			\put(16,-2) {$L_+$}
			\put(49,-2) {$L_-$}
            \put(81,-2) {$L_0$}
		\end{overpic}
    \caption{Real oriented skein triple}
    \label{fig:oriented skein triple}
\end{figure}

\begin{itemize}
 \item If $l_p'=l_p+1$, we have a long exact sequence \begin{equation*}
    \to c\GHR^-_{O,d}(L_+,s) \xrightarrow{f^-} c\GHR^-_{O,d}(L_-,s) \xrightarrow{g^-} c\GHR^-_{O,d-1}(L_0,s) \xrightarrow{h^-} c\GHR^-_{O,d-1}(L_+,s) \to 
    \end{equation*}
 \item If $l_p'=l_p$, we have a long exact sequence \begin{equation*}
    \to c\GHR^-_{O,d}(L_+,s) \xrightarrow{f^-} c\GHR^-_{O,d}(L_-,s) \xrightarrow{g^-} (c\GHR_{O}^-(L_0)\otimes W)_{d-1,s} \xrightarrow{h^-} c\GHR^-_{O,d-1}(L_+,s) \to 
    \end{equation*}
    for $W=\F_{(0,1/2)}\oplus \F_{(-1,-1/2)}$.
    \item If $l_p'=l_p-1$, we have a long exact sequence \begin{equation*}
    \to c\GHR^-_{O,d}(L_+,s) \xrightarrow{f^-} c\GHR^-_{O,d}(L_-,s) \xrightarrow{g^-} (c\GHR_{O}^-(L_0)\otimes J)_{d-1,s} \xrightarrow{h^-} c\GHR^-_{O,d-1}(L_+,s) \to 
    \end{equation*}
    for $J=\F_{(0,1)}\oplus \F_{(-1,0)}\oplus \F_{(-1,0)}\oplus \F_{(-2,-1)}$.
\end{itemize}
\end{thm}

\begin{thm}\label{thm:O-hat oriented skein triple}
Let $(L_+, L_- ,L_0,\fro)$ be a real oriented skein triple defined as above.
\begin{itemize}
 \item If $l_p'=l_p+1$, we have a long exact sequence \begin{equation*}
    \to \widehat{\GHR}_{O,d}(L_+,s) \xrightarrow{\hat{f}} \widehat{\GHR}_{O,d}(L_-,s) \xrightarrow{\hat{g}} \widehat{\GHR}_{O,d-1}(L_0,s) \xrightarrow{\hat{h}} \widehat{\GHR}_{O,d-1}(L_+,s) \to 
    \end{equation*}
 \item If $l_p'=l_p$, we have a long exact sequence 
 \begin{equation*}
    \to \widehat{\GHR}_{O,d}(L_+,s) \xrightarrow{\hat{f}} \widehat{\GHR}_{O,d}(L_-,s) \xrightarrow{\hat{g}} (\widehat{\GHR}_{O}(L_0)\otimes W)_{d-1,s} \xrightarrow{\hat{h}} \widehat{\GHR}_{O,d-1}(L_+,s) \to 
    \end{equation*}
    for $W=\F_{(0,1/2)}\oplus \F_{(-1,-1/2)}$.
    \item If $l_p'=l_p-1$, we have a long exact sequence \begin{equation*}
    \to \widehat{\GHR}_{O,d}(L_+,s) \xrightarrow{\hat{f}} \widehat{\GHR}_{O,d}(L_-,s) \xrightarrow{\hat{g}} (\widehat{\GHR}_{O}(L_0)\otimes J)_{d-1,s} \xrightarrow{\hat{h}} \widehat{\GHR}_{O,d-1}(L_+,s) \to 
    \end{equation*}
    for $J=\F_{(0,1)}\oplus \F_{(-1,0)}\oplus \F_{(-1,0)}\oplus \F_{(-2,-1)}$.
\end{itemize}
\end{thm}

\begin{proof}
The proof of \cite[Theorem~7.1]{YXHFLR} works for Theorem~\ref{thm:O-minus oriented skein triple} and \ref{thm:O-hat oriented skein triple} verbatim. 
\end{proof}

\begin{thm}\label{thm:O-skein relation for real Alexander polynomial}
Let $(L_+, L_- ,L_0,\fro)$ be a real oriented skein triple. Then we have \[\Delta^R_O(L_+,\fro)(t)-\Delta^R_O(L_-,\fro)(t)=(t-t^{-1})\cdot\Delta^R_O(L_0,\fro)(t).\]
\end{thm}
\begin{proof}
This is a direct corollary of Theorem~\ref{thm:O-hat oriented skein triple}, using the definition of $\Delta_{O}^R(L,\fro)$. 
\end{proof}

In the real grid homology defined in \cite{YXHFLR}, we also deduced an unoriented skein exact triangle between stabilized hat homology groups, which respects a $\delta^R$-grading (see \cite[Section~7.3]{YXHFLR}). During the proof, we need to interchange the roles of $O$ and $X$ on certain components in order to relate the Floer homology groups we want to. This operation is illegal in $OO$-theory, so such a triangle no longer exists. 

\subsection{Torsion order, real $\tau$-invariant and module structure}\label{sub:Torsion order, real tau-invariant and module structure}

Before defining numerical invariants, we first calculate $\GHR_{O}^
{\circ}$ for the most basic example: the standard strongly invertible unknot.

\begin{example}\label{ex:unknot}
A real grid diagram of size $2$ for the strongly invertible unknot in $S^3$ is shown in Figure~\ref{fig:unknotOOgrid}.
\begin{figure}
\centering
    \begin{overpic}[width=0.2\textwidth]{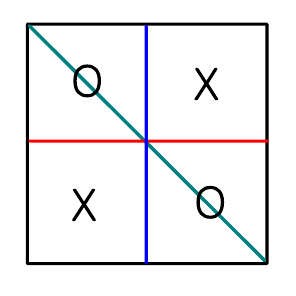}
		\end{overpic}
    \caption{An $OO$-real grid diagram for the strongly invertible unknot.}
    \label{fig:unknotOOgrid}

\end{figure}

It can be calculated easily that \[\GHR_{OO}^-(U,\fro)\cong \F[u_1,u_2]_{(0,0)}/(u_1=u_2),\] \[\GHR_{O}^-(U,\fro)\cong \F[u]_{(0,0)}\oplus \F[u]_{(0,1/2)},\]
\[\widehat{\GHR}_{O}(U,\fro)\cong \F_{(0,0)}\oplus \F_{(0,1/2)}.\]
Here, we write $\F[u]_{(d,s)}$ for a copy of $\F[u]$ with a generator at grading $(M^R,A^R)=(d,s)$ and similarly for $\F_{(d,s)}$.
\end{example}

Recall that for a finitely generated $\F[u]$-module $M$, we can define its \emph{torsion submodule} \[\Tors(M)=\{m\in M|\text{there exists } 0\ne p \in \F[u]\text{ such that } p\cdot m=0\}.\] Then there exists an $r\in \Z_{\ge 0}$ such that $M/\Tors (M)\cong \F[u]^{\oplus r}$, such an $r$ is called the \emph{rank} of $M$. We can also define the \emph{torsion order} of $M$ \[\ord_{u}(M)=\min \{k\in \N| u^k\cdot \Tors(M)=0\}\in \N.\] 

\begin{defn}
Let $(K,\fro)$ be a strongly invertible knot with auxiliary data in $(S^3,\tau)$. We define the \emph{torsion order} $\ord^O_{u}(K,\fro)$ to be $\ord_{u}(\GHR_{O}^-(K,\fro))$.  
\end{defn}

\begin{prop}
If $(K_+,\fro_+)$ and $(K_-,\fro_-)$ are a pair of knots related by a type $A$ or $B$ crossing change, then \[\vert \ord^O_{u}(K_+,\fro_+) -\ord^O_{u}(K_-,\fro_-)\vert \le 2. \]   
\end{prop}
\begin{proof}
This follows from Proposition~\ref{prop:unknotting maps} (see \cite[Proposition~5.8]{YXHFLR}).
\end{proof}

As a corollary, we have the following bounds on type $A$, $B$ equivariant unknotting numbers: 
\begin{cor}\label{cor:torsion order bounds u'}
Let $K$ be a strongly invertible knot in $(S^3,\tau)$, for any choice of $\fro$, we have \[\ord^O_{u}(K,\fro)\le 2\widetilde{u}_A(K), \] \[\ord^O_{u}(K,\fro)\le 2\widetilde{u}_B(K). \] 
\end{cor}

\begin{prop}\label{prop:ranks of minus theories}
Let $(K,\fro)$ be an oriented strongly invertible knot in $S^3$. Then regarding $\GHR^-_{OO}(K,\fro)$ as an $\F[u_1]$-module, it has rank $1$. The same is true with $u_2$ in place of $u_1$. Similarly, regarding $\GHR^-_{O}(K,\fro)$ as an $\F[u]$-module, its free part has rank $2$.

More generally, if $(L,\fro)$ is an oriented generalized strongly invertible link in $S^3$, then as an $\F[u]$-module, $c\GHR^-_{O}(L,\fro)$ has rank $2^{2l_f+l_p-1}$.
\end{prop}

\begin{proof}
These also follow from Proposition~\ref{prop:unknotting maps} (see \cite[Proposition~5.5]{YXHFLR}). 
\end{proof}

\begin{defn}
Let $(K,\fro)$ be an oriented strongly invertible knot. For $j=1,2$, the \emph{real $\tau$-invariant} $\tau^R_j(K,\fro)$ is $-1$ times the maximal integer $i$ for which there is a homogeneous, non-torsion element in $\GHR_{OO}^-(K,\fro)$ whose real Alexander grading is equal to $i/2$, when we regard $\GHR_{OO}^-(K,\fro)$ as an $\F[u_j]$-module.
\end{defn}

\begin{prop}\label{prop:tau_1=tau_2}
For any oriented strongly invertible knot $(K,\fro)$, 
\[\tau^R_1(K,\fro)= \tau^R_2(K,\fro).\]
\end{prop}
\begin{proof}
Fix an $OO$-real grid diagram $\cH$ for $(K,\fro)$. Let $(C,\partial)$ be the chain complex $(\GCR_{OO}^-(\cH),\partial^-)$, so that $C'=C/(u_1-u_2)$ is the chain complex $\GCR_{O}^-(\cH)$. 



Let $x$ be a homogeneous element in $C$ such that $[x]\in H(C)$ realizes $\tau_1^R(K,\fro)$. We claim that there exists some $n\ge 0$ so that $u_1^{2n+1}x-u_2^{2n+1}x$ is a boundary. Otherwise, consider the element $[u_2x]$. If it is not $u_1$-torsion, then $[u_2x]$ and $[x]$ will be two distinct towers in $H(C)$ when it is regarded as an $\F[u_1]$-module, which contradicts Proposition~\ref{prop:ranks of minus theories}. If $[u_2x]$ is a $u_1$-torsion, then there is an element $y\in C$ with $\partial y= u_1^m u_2 x$ for some $m\ge 0$. From Proposition~\ref{prop:U,v action identification}, we know that there is an element $z\in C$ with $\partial z= (u_1^{2}-u_2^{2})x$, then $u_1^{m+2} x=\partial u_2 y+ \partial u_1^m z$, contradicts the assumption on $x$. Thus, the claim holds This argument implies that $[x]$ is also a non-torsion element when $H(C)$ is regarded as an $\F[u_2]$-module. Thus, $\tau^R_{2}(K,\fro)\le \tau_1^R(K,\fro)$. The roles of $u_1$ and $u_2$ are symmetric in the discussion above, so we have the desired equality.
\end{proof}

Actually, we can extract more from the proof above. There is a homogeneous element $x_0\in C$, such that $[x_0]$ realizes both $\tau^R_1(K,\fro)$ and $\tau^R_2(K,\fro)$. Fix such an $x_0$.

At chain level, $u_1-u_2\colon C\to C$ is an injective map, so we have a quasi-isomorphism of chain complexes \[\Cone(u_1-u_2\colon C\to C)\simeq C'\] (see \cite[Lemma~5.2.13]{OSS2015grid}). $\Cone(u_1-u_2\colon C\to C)$ can be written as $C\oplus C$ equipped with differential \[\partial'=\begin{bmatrix}
\partial & u_1-u_2\\
0& \partial
\end{bmatrix}.\]
Note that $(x,y)\in C\oplus C$ is a cycle in the mapping cone if and only if $\partial x=0$ and $\partial y+(u_1+u_2)x=0$. 

Let $n$ be the smallest non-negative integer such that $(u_1^{2n+1}-u_2^{2n+1})x_0$ is a boundary, say $\partial y_0=(u_1^{2n+1}-u_2^{2n+1})x_0$ for some homogeneous element $y_0$. We also know from Proposition~\ref{prop:U,v action identification} that there is a homogeneous element $y_1$, so that $\partial y_1=(u_1^{2n}-u_2^{2n})x_0$. Then $X=(0,x_0)$ and $Y=(u_1^{2n}x_0, y_0+u_2y_1)$ are two cycles in the mapping cone. It is clear that they both represent homogeneous non-torsion elements when $H(C')$ is regarded as an $\F[u]$-module, with $u=u_1\simeq u_2$ action. It is also easy to see that they do not belong to the same tower, since $u^m X-Y$ is never a boundary. Thus, we have found two towers in $H(C')$ by hand, which are the only towers in $H(C')$ by Proposition~\ref{prop:ranks of minus theories}.

The grading in $C'$ is given by $C'_{d,s}=C_{d,s+1/2}\oplus C_{d,s}$, since as a map, $u_1-u_2$ is of bigrading $(-1,-1/2)$. Thus, $X$ lies in bigrading $(M^R(x_0), A^R(x_0))=(M^R(x_0), -\tau_i^R/2)$ and $Y$ lies in bigrading $(M^R(x_0)-2n, A^R(x_0)-n-1/2)=(M^R(x_0)-2n, -\tau_i^R/2-n-1/2)$. 

Mimicking the definition of $\tau$-set invariant for links from collapsed grid homology, we extract the following invariant from $\GHR^-_{O}$. 

\begin{defn}\label{def:real tau set for u_1=u_2}
Let $(K,\fro)$ be an oriented strongly invertible knot. Then we define its \emph{real $\tau_{O}$-set ($\tau^R_{O}$-set)} to be the set of integers $\tau^R_{\min}(K,\fro)\le \tau^R_{\max}(K,\fro)$, defined as follows. Choose a generating set of two elements for $\GHR_{O}^-(K,\fro)/\Tors$ such that each element is homogeneous with respect to $A^R$-grading. Then $\tau^R_O$-set is defined to be $-2$ times the real Alexander grading of the generators. 
\end{defn}

It is easy to check that this is well-defined and is an invariant of $(K,\fro)$ (see~\cite[Corollary~8.3.4]{OSS2015grid}). We claim that the $\tau^R_{O}$-set satisfies $\tau_i^R(K,\fro)+n+1=\tau^R_{\max}(K,\fro)$ and $\tau^R_{\min}(K,\fro)\le \tau^R_i(K,\fro)$. 


From the construction of $X$ and $Y$ above, we know that $\tau^R_{\min}(K,\fro)\le \tau_i^R(K,\fro)$, $\tau^R_{\max}(K,\fro) \le \tau_i^R(K,\fro)+1+n$. Since $Y$ has the first coordinate being non-torsion, we know that there must be a homogeneous tower generator with the same property. 


Let $Z=(x,y)$ be a cycle representative of such a generator. Then $\partial x=0$ and $\partial y+(u_1+u_2)x=0$. By assumption, $[x]$ is a non-torsion element (for both $u_1$ and $u_2$ action) in $H(C)$. (From the proof of Proposition~\ref{prop:tau_1=tau_2}, we see that an element in $H(C)$ is $u_1$-torsion iff it is $u_2$-torsion.) Without loss of generality, we can assume $[x]=u_1^m [x_0]$ for some $m\ge 0$. Thus, $A^R(Z)\le A^R ([x_0])-1/2$, so $Z$ must be the generator of the tower starting at grading $-2\tau^R_{\max}(K,\fro)$, since $\tau^R_{\min}(K,\fro)\le \tau^R_i(K,\fro)=-2A^R(x_0)$. On the other hand, we need $(u_1+u_2)x$ to be a boundary, so $m$ must be bigger than $2n$, which leads to a further refinement $A^R(Z)\le A^R ([x_0])-1/2-n$. Thus, we must have $A^R(Z)= A^R ([x_0])-1/2-n$ and $\tau^R_{\max}(K,\fro) = \tau_i^R(K,\fro)+1+n$.

Thus, we have shown that 
\begin{prop}\label{prop:relationship tau_max, tau_min and tau_i^R}
For an oriented strongly invertible knot $(K,\fro)$ in $(S^3,\tau_{\std})$, we have \[\tau^R_{\max}(K,\fro) = \tau_i^R(K,\fro)+1+n \quad \text{and}\quad \tau^R_{\min}(K,\fro)\le \tau^R_i(K,\fro).\]
Thus, $\tau^R_{\max}(K,\fro)-\tau^R_{\min}(K,\fro)-1$ provides an upper bound on the difference between $u_1$ and $u_2$ actions on the tower in $\HFLR^-_{OO}(K,\fro)$.
\end{prop}


The minus real knot Floer homology defined in \cite{YXHFLR} is a module over $\F[u]$, while the $OO$-minus real knot Floer homology is a module over $\F[u_1,u_2]$. From Proposition~\ref{prop:U,v action identification}, we know that $u_1^2$ and $u_2^2$ actions on $\GHR^-_{OO}$ are the same. So it is natural to ask whether the $u_1$ and $u_2$ actions are the same. From Example~\ref{ex:unknot}, we know that on $\GCR^-_{OO}(U,\fro)$, $u_1$ and $u_2$ actions are homotopic, thus they are equal on $\GHR^-_{OO}(U,\fro)$. But this is not true in general.

\begin{lem}(A modification of \cite[Lemma~7.4.1]{OSS2015grid}) Let $C$ be a chain complex of modules over the polynomial ring over $\F$ with degree $(-1,-1/2)$ variables $\{u_i\}_{1\le i\le t}$ and degree $(-2,-1)$ variables $\{U_i\}_{1\le i\le k}$. If $u_1$ and $u_2$ actions are homotopic on $C$, then we have a quasi-isomorphism from $C\otimes W$ to $\frac{C}{u_1-u_2}$ as chain complexes over $\F[u_1,\ldots,u_t,U_1,\ldots, U_k]$, in which $W=\F_{(0,0)}\otimes \F_{(0,1/2)}$. In particular, $H(C)\otimes W\cong H(C/u_1-u_2)$ as modules over $\F[u_1,\ldots,u_t,U_1,\ldots, U_k]$.
\end{lem}

From this lemma, we can see that if $u_1$ and $u_2$ actions are homotopic on $\GCR^-_{OO}(K,\fro)$, then their actions on $\GHR^-_{OO}(K,\fro)$ are equal (thus $\GHR^-_{OO}(K,\fro)$ can be regarded as a $\F[u]$-module) and $\GHR^-_{O}(K,\fro)$ must take the form $\GHR^-_{OO}(K,\fro)\otimes W$. Thus, if $\GHR^-_{O}(K,\fro)$ does not take the form $P\otimes W$ for some $\F[u]$-module $P$, then $u_1$ and $u_2$ actions on $\GCR^-_{OO}(K,\fro)$ cannot be homotopic. Here, we abuse the notation $\GCR^-_{OO}(K,\fro)$ for $\GCR^-_{OO}(\cH)$ for any choice of $OO$-real Heegaard diagram representing $(K,\fro)$.

In Appendix~\ref{app:Calculation results for knots with small crossing numbers}, we will see from examples that $\GHR^-_{O}(K,\fro)$ rarely takes the form $P\otimes W$. So for most strongly invertible knots, the $u_1$ and $u_2$ actions are not equal on the homology group level thus not homotopic on the chain complex level. This means two $O$-base points play unequal roles in the $OO$-real link Floer homology, which is an evidence that there is still a direction dependence in the $OO$-theory.

Motivated by the propositions and examples, we make the following conjecture.

\begin{conjecture}
The inequality in Proposition~\ref{prop:relationship tau_max, tau_min and tau_i^R} is an equality, so $\tau^R_{\max}(K,\fro)-\tau^R_{\min}(K,\fro)-1$ records the difference between $u_1$ and $u_2$ actions on the tower in $\HFLR^-_{OO}(K,\fro)$ faithfully. 
\end{conjecture}

\subsection{Two spectral sequences}
Before ending this section, we step back to the holomorphic theory and consider two spectral sequences associated to real knot Floer homology as in \cite[Section~4]{YXHFLR}.

The first spectral sequence is a corollary of \cite[Theorem 1.1]{hendricks2025noterealheegaardfloer} by applying it to $(\Sym^m(\Sigma),\T_{\alpha},\T_{\beta},R)$ with $\Sigma$, $\bm\alpha$ and $\bm\beta$ coming from a minimal $OO$-real Heegaard diagram.

\begin{thm}\label{thm:spectral sequence from HFK,strongly invertible case}
For a strongly invertible knot $(K,\fro)$, there is a spectral sequence starting from $\widehat{\HFK}(K)\otimes \F[\theta,\theta^{-1}]\otimes (\F_{(0,0)}\oplus \F_{(1,1)})$ and converging to $\widehat{\HFKR}_{O}(K,\fro)\otimes \F[\theta,\theta^{-1}]$. Moreover, this spectral sequence respects the splitting of $\widehat{\HFK}$ and $\widehat{\HFKR}_{O}$ along the (real) Alexander grading. 
\end{thm}
It was conjectured in \cite{hendricks2025noterealheegaardfloer} and proved in \cite{LO_Real_bordered} and \cite{BGX} that the first differential in a similar spectral sequence in $OX$-theory can be written down explicitly as \[d_1=(1+(\iota_K\tau_K)_*),\] when using a minimally pointed real nice diagram. It would be nice to know whether the differential for this spectral sequence can be written down explicitly in some special cases. As remarked in \cite[Example~1.6]{hendricks2025noterealheegaardfloer}, the spectral sequence collapses immediately when the underlying knot is $\HFK$-thin.

The second one comes from the filtration induced by the real Alexander grading, as in \cite[Section~4.2]{YXHFLR}.

\begin{thm}\label{thm:spectral sequence relating HFKR and HFR}
Let $(K,\fro)$ be an oriented strongly invertible knot. Then there is a spectral sequence \[\widehat{\HFKR}
_{O}(K,\fro) \Longrightarrow \widehat{\HFR}(S^3,\tau) \otimes (\F\oplus\F).\]
\end{thm}

\section{Property comparison}\label{sec:Property comparison}
In this section, we make a comparison between the real link Floer homology defined in \cite{YXHFLR} and the modified real link Floer homology defined above. Unlike in Section~\ref{sub:Structural properties in $S^3$} and \ref{sub:Torsion order, real tau-invariant and module structure}, we will work with generalized strongly invertible links in an arbitrary closed real $3$-manifolds from this point on. But the $3$-manifold will usually be omitted from the notation for simplicity.


In \cite[Definition~2.2]{YXHFLR}, we only characterized generalized strongly invertible links with auxiliary data in real closed $3$-manifold a with connected fixed set. To make the $OX$-theory work in general, we would like to weaken the assumption: we only require a choice of orientation on $L$ that fits it into the definition of ``generalized strongly invertible'' and a labeling of $O$ and $X$ on the fixed points of each strongly invertible component. 

Recall that a choice of orientation $\fro$ on $L$ is a part of $\fra$. Let $\cH$ be a real Heegaard diagram for $(L,\fra)$. Then performing one of the quasi-stabilizations shown in Figure~\ref{fig:OX-OO stabilization} at each fixed $X$-base point, we get $\cH_O$, an $OO$-real Heegaard diagram for $(L,\fro)$. Using such a picture, we will analyze the difference and connection between these two theories.

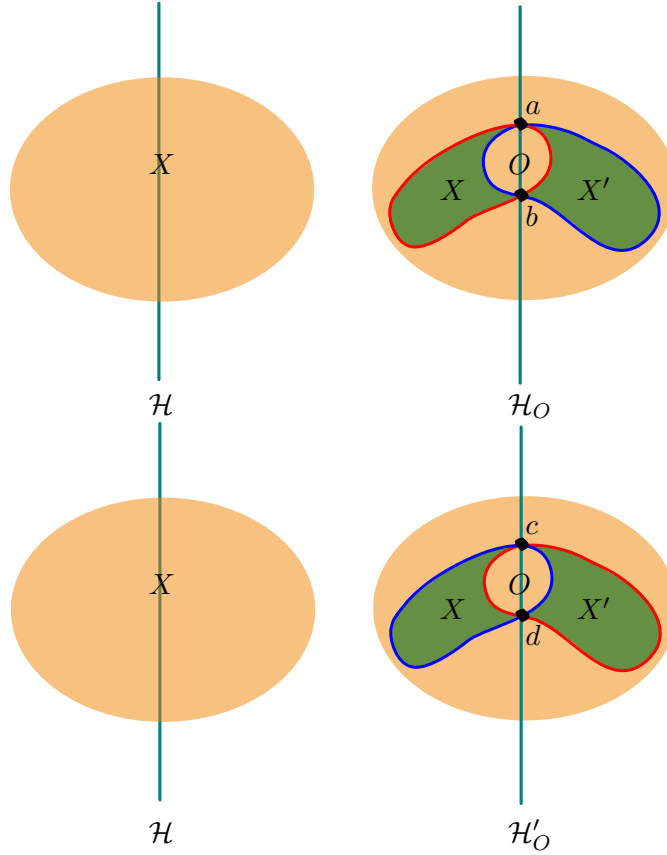
\begin{figure}
    \centering
     \begin{overpic}[width=0.6\textwidth]{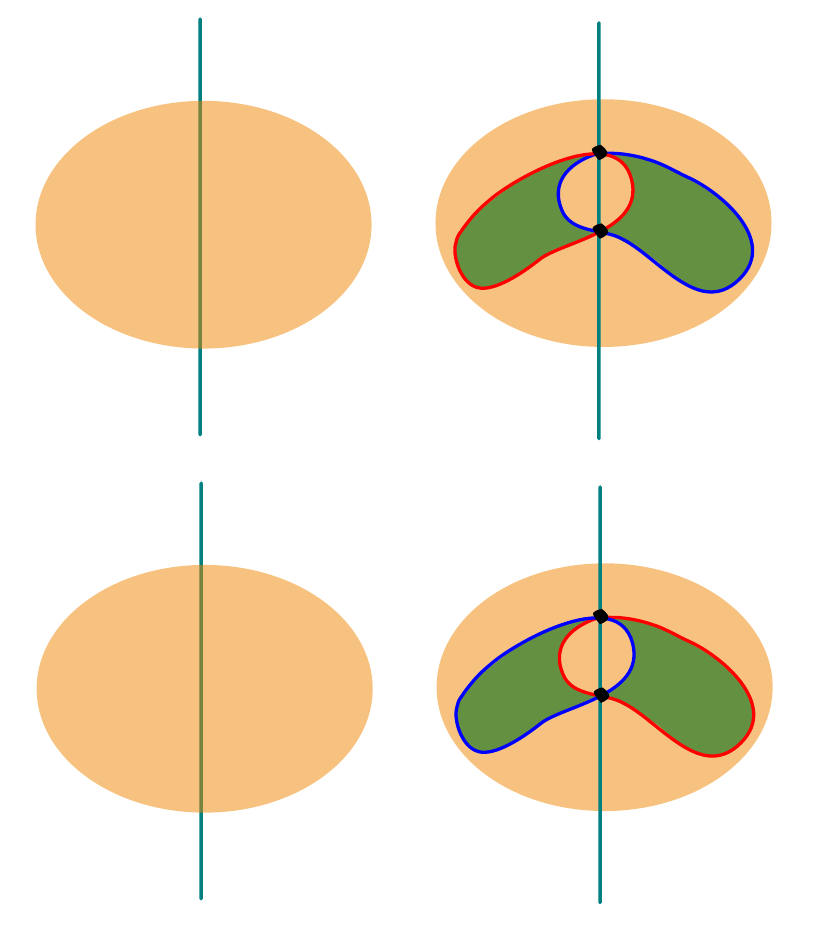}
     \put(20,0) {$\cH$}
     \put(20,50) {$\cH$}
     \put(62,0) {$\cH'_O$}
     \put(62,50) {$\cH_O$}
     \put(64,36) {$c$}
     \put(64,23) {$d$}
     \put(64,85) {$a$}
     \put(64,72) {$b$}
     
     \put(20,29) {$X$}
     \put(20,78) {$X$}
     \put(62,29) {$O$}
     \put(62,78) {$O$}
     
     \put(54,26) {$X$}
     \put(70,26) {$X'$}
     \put(54,75) {$X$}
     \put(70,75) {$X'$}

	\end{overpic}
    \caption{Two kinds of quasi-stabilization from an $OX$-real Heegaard diagram to an $OO$-real Heegaard diagram.}
    \label{fig:OX-OO stabilization}
\end{figure}

\begin{remark}\label{rmk:two local pictures for type changing quasi-stabilization}
	Abstractly, we cannot tell the difference between two quasi-stabilized diagrams shown in Figure~\ref{fig:OX-OO stabilization}. However, when we later consider real link Floer complex or homology as an invariant of a multi-based link embeddings in a real $3$-manifold (see Section~\ref{sec:Based link invariants, base point action and quasi-stabilization}), the two diagrams characterize different embedding of the same based link.
\end{remark}

\subsection{Killing the curvature}\label{sub:Killing the curvature}
In this subsection, we consider the difference in curvature between the two versions of full real link Floer complexes. Since the setups of the two theories are identical on pairs of components, we will focus on the case $L=K$ is a strongly invertible knot.

Label the base points in $\cH$ so that the two fixed base points are $O_1$ and $X_n$. Along the arc from $O_1$ to $X_n$, the base points are \[O_1, X_1, \ldots, X_{n-1}, O_{n}, X_{n}\] while along the arc from $X_n$ to $O_1$, the base points are \[X_n,O_n',X_{n-1}', \ldots,X_1',O_1.\] In this way, we have variables $u_1$ (of degree $(-1,-1/2)$) and $v_n$ (of degree $(0,1/2)$) account for $O_1$ and $X_n$, and variables $U_{i}$ ($2\le i\le n$, each of degree $(-2,-1)$), $V_{j}$ ($1\le j\le n-1$, each of degree $(0,1)$) accounts for pairs $(O_i,O_i')$ and $(X_j,X_j')$. The curvature of $\fullCFLR^-(K,\fra)$ is \[\omega_{K,\fra}=u_1^2\cdot V_1+ U_2\cdot V_1+ U_2\cdot V_2+\ldots U_{n}\cdot V_{n-1}+ U_{n}\cdot v_n^2,\] if $n\ge 2$. The case $n=1$ is a little subtle: if the link consists of the single knot $K$, then $\omega_{K,\fra}=0$; if it is a genuine link, then $K$ contributes a term $\omega_{K,\fra}=u_1^2v_1^2$ to the curvature.

After the quasi-stabilization, $X_n$ moves away from the fixed set and is paired with a new base point $X_n'$ while a new base point $O_{n+1}$ appears on the fixed set. In the base ring, we replace the degree $(0,1/2)$ variable $v_n$ with a degree $(0,1)$ variable $V_n$ and a new degree $(-1,-1/2)$ variable $u_{n+1}$ is introduced. In Section~\ref{sub:Basic setup}, we calculated the curvature of $\fullCFLR^-_{OO}(K,\fro)$: \[\omega^O_{K,\fro}=u_1^2\cdot V_1+ U_2\cdot V_1+ U_2\cdot V_2+\ldots U_{n}\cdot V_{n-1}+ U_{n}\cdot V_n+ u_{n+1}^2\cdot V_n.\]

In order to define homology theories, we need to kill the curvature. In the previous discussion, we always set $v$ and $V$ variables to zero, so the curvature is naturally zero. In the $OO$-theory, we have an extra variation defined by further setting $u_1=u_{n+1}$. Motivated by the coloring used in \cite{zemke2019link}, we make comparison between their curvatures and propose some further variation.

An obvious distinction between the two curvatures is that $\omega_{K,\fra}$ consists of $2n-1$ monomials while $\omega^O_{K,\fro}$ consists of $2n$ monomials. Philosophically, this implies that the two full real link Floer complexes are quite different, since we know from \cite[Section~2]{zemke2019link} that two curved complexes admit a non-trivial morphism between them only if they share the same curvature. The odd and even distinction between $\omega_{K,\fra}$ and $\omega^O_{K,\fro}$ makes it hard for them to be equal, even after coloring. This is partially the reason why the author decided to investigate $OO$ and $OX$-theories separately and expects them to provide us with different information.

In the usual link Floer theory, if along the orientation of the knot, the base points read $O_1, X_1, \ldots O_n, X_n$, then the curvature is \[\omega=U_1V_1+V_1U_2+\ldots U_nV_n+V_n U_1.\] An effective way of killing $\omega$ is to set all $U_i$ variables or all $V_i$ variables to be equal. However, for real theories, this strategy works for $OO$-theory, but not for $OX$-theory. More precisely, by setting $u_1^2=U_2=\ldots=U_{n}=u_{n+1}^2$,  $\omega^O_{K,\fro}$ becomes zero in $\cR(\cH_O)$. But whether we set $u_1^2=U_2=\ldots=U_{n}$ or $V_1=\ldots=V_{n-1}=v_{n}^2$, $\omega_{K,\fra}$ remains non-zero. $2n-2$ terms cancel in pairs and a monomial is left. 

Thus, in the $OO$-theory, we have yet another minus version of real link Floer homology which has no counterpart in $OX$-theory. Let $\cH_O$ be a minimal $OO$-real Heegaard diagram for $(L,\fro)$. Assume it is strongly $\s^R$-admissible for some $\s^R \in \rspinc(Y,\tau)$. Then we can define \[\CFLR_{OO,V}^-(\cH,\s^R)=\fullCFLR(\cH,\s^R)/ (u_{i,1}^2-u_{i,2}^2)_{1\le i\le l_f}.\] The homology \[\HFLR_{OO,V}^-(\cH,\s^R)= H(\CFLR_{OO,V}^-(\cH,\s^R),\partial^-)\] is an invariant of  $(L,\fro)$ and $\s^R$ when it is regarded as an $\F[u_1,u_2,\ldots,u_{l_f}, U_{1},\ldots U_{l_p}, V_{1},\ldots V_{l_p+l_f}]$-module. 

\begin{remark}
We will see in the next subsection that $\HFLR^-$ and $\HFLR^-_O$ can be related by an exact triangle. Motivated by this, one can ask what's the relationship between $\HFLR_{OO,V}^-(K,\fro)$ and $\HFLR^-$ or $H(\fullCFLR^-(K,\fra)/(\bfU))$. Here, $H(\fullCFLR^-(K,\fra)/(\bfU))$ is yet another minus version of real link Floer homology obtained from the full real link Floer complex defined in \cite{YXHFLR}. $\HFLR_{OO,V}^-(K,\fro)$ may contain more information than other homology theories we have discussed, since it admits both $\bfU$ and $\bfV$ actions.
\end{remark}

\subsection{Exact triangle and a common generalization}\label{sub:Exact triangle and a common generalization}

\begin{thm}\label{thm:exact triangle between OX and OO for knots}
Let $(K,\fra)$ be a strongly invertible knot with auxiliary data and $(K,\fro)$ be the underlying oriented strongly invertible knot.
Then there are exact triangles \[\begin{tikzcd}
    \widehat{\HFKR}(K,\fra) \arrow{rr}{(-1,-1/2)} & & \widehat{\HFKR}(K,\fra) \arrow{dl}{(0,1/2)}\\
    & \widehat{\HFKR}_{O}(K,\fro) \arrow{ul}{(0,0)}\\
\end{tikzcd};\]
\[\begin{tikzcd}
    \HFKR^-(K,\fra) \arrow{rr}{(-1,-1/2)} & & \HFKR^-(K,\fra) \arrow{dl}{(0,1/2)}\\
    & \HFKR^-_{O}(K,\fro) \arrow{ul}{(0,0)}\\
\end{tikzcd}.\]
An arrow labeled $(d,s)$ shifts the $(M^R,A^R)$-bigrading by $(d,s)$. These triangles split naturally along real $\spinc$ structures.
\end{thm}

\begin{proof}
Consider the pairs of real Heegaard diagrams set up at the beginning of this section, shown in Figure~\ref{fig:OX-OO stabilization}. For concreteness, we will argue for the pair shown in the upper row. The argument works for the lower one with $c$ in place of $b$ and $d$ in place of $a$. Fix a real $\spinc$ structure $\s^R$ and assume that $\cH$ and $\cH_{O}$ are both strongly $\s^R$-admissible. Without loss of generality, we assume they are minimal. In $\cH_{O}$, there are two new curves $\alpha_O$ and $\beta_O$ intersecting in a pair of points $a$ and $b$. The generators of $\CFKR_O^{\circ}(\cH_O)$ split into two subsets $G_a$ and $G_b$ according to whether they contain $a$ or $b$. We claim that
\begin{enumerate}
    \item the submodule $C_a^-$ generated by $G_a$ is a subcomplex of $\CFKR_O^{-}(\cH_O)$ with quotient complex $C^-_b$ generated by the image of $G_b$;
    \item if we use $G$ to denote the set of generators of $\CFKR^{\circ}(\cH)$, then both $G_a$ and $G_b$ are in one-to-one correspondence with $G$ and both $C_a^-$ and $C_b^-$are naturally identified with $\CFKR^{-}(\cH)$ as chain complexes over $\F[\bfU]$ ($\F$).
\end{enumerate} 
The analogous claims hold for the hat theory. We first prove the theorem assuming the claims. These give rise to mapping cone formulas 
\[\CFKR^{\circ}_O(\cH_O)=\Cone (\partial_b^a: \CFKR^{\circ}(\cH)\to \CFKR^{\circ}(\cH)),\] with the second copy of $\CFKR^{\circ}(\cH)$ having $A^R$ shifted up by $1/2$ according to $M^R(a,b)=0$, $A^R(a,b)=1/2$, which is calculated from the small bigon containing the new $O$-base point.

The mapping cones can be rewritten as short exact sequences
\[0\to \CFKR^{-}(\cH)_{d,s-1/2}\to \CFKR_O^{-}(\cH_O)_{d,s} \to \CFKR^{-}(\cH)_{d,s} \to 0.\]

\[0\to \widehat{\CFKR}(\cH)_{d,s-1/2}\to \widehat{\CFKR}_O(\cH_O)_{d,s} \to \widehat{\CFKR}(\cH)_{d,s} \to 0.\] Then a standard homological algebra lemma leads to the desired exact triangles.

Now we prove the claims for the minus version, as the hat version is similar and easier. To see (1), it suffices to show that there can be no differential from a generator of the form $a\xv$ to a generator of the form $b\yv$. Assume $\cD_O\in \pi^R_2(a\xv,b\yv)$ is such a positive real domain. Let $m$ and $n$ be the multiplicities of $\cD$ at the region containing the new $O$ and the outer orange region in Figure~\ref{fig:OX-OO stabilization}, respectively. Then since the green regions are blocked by $X$-base points, we must have \[0+0-m-n=1,\quad m+n-0-0=-1,\] which is never true for a positive domain. The first statement in (2) is obvious. For the second one, if $\cD_O\in \pi^R_2(b\xv,b\yv)$ ($ \pi^R_2(a\xv,a\yv)$), then we have a similar equation $0+0-m-n=0$ ($m+n-0-0=0$), which tells us that $\cD_O$ can be identified canonically with a domain $\cD \in \pi^R_2(\xv,\yv)$ inside $\cH$. In this way, we can identify the coefficient of $a\yv$ in $\partial_{\cH_O} a\xv$ ($b\yv$ in $\partial_{\cH_O} b\xv$) with the coefficient of $\yv$ in $\partial_{\cH} \xv$ when choosing the almost complex structure on $\cH$ and $\cH_O$ properly. Then the claim is proved.  
\end{proof}

\begin{remark}\label{rmk:differential of OX->OO stabilization}
Based on the proof above, we can say more about the differential $\partial_{\cH_O}$. We have seen that 
\[\partial_{\cH_O}=\begin{bmatrix}
    \partial_{\cH} & p\\
    0 & \partial_{\cH}\\
\end{bmatrix}.\] Now we claim that $p$ takes the form $u_O+p_1$, in which $u_O$ is the degree $-1$ variable associated to the new $O$ base points and $p_1$ is a morphism $\CFKR^-(\cH)\to \CFKR^-(\cH)$, in particular, $u_O$ does not appear in $p_1$. Consider any $\cD_O \in \pi^R_2(b\xv,a\yv)$ that contributes the differential. Using the same notation for multiplicities as above,  we have equations  \[0+0-m-n=-1,\quad m+n-0-0=1.\] Thus either $n=1$, $m=0$ or $n=0$, $m=1$. In the first case, $u_O$ does not appear, so the domain contributes to $p_1$. In the second case, the small bigon of real Maslov index $1$ containing the new $O$ has multiplicity $1$. Since Maslov index is additive under disjoint union and $n=0$ forces this bigon to be isolated, it must be the case $\xv=\yv$. Then we have exactly one copy of $u_0\cdot a\xv$ in $\partial b\xv$, since the small bigon has a unique holomorphic representative. Thus, the claim holds. In Proposition~\ref{prop: differential oftype changing quasi-stab}, we will provide a more comprehensive description on the ``quasi-stabilized'' differential.
\end{remark}

\begin{cor}\label{cor:OO v.s. OX polynomial for knots}
Let $(K,\fra)$ and $(K,\fro)$ be set up as in Theorem~\ref{thm:exact triangle between OX and OO for knots}. Further assume the ambient manifold is $(S^3,\tau_{\std})$. Then we have \[\Delta^R_{O}(K,\fro)=(1+t)\Delta^R(K,\fra)\] when we normalize the real Alexander grading in $\widehat{\HFKR}_{O}(K,\fro)$ by requiring \[A^R_{\max}(L,\fro)+A^R_{\min}(L,\fro)= 1,\] in which \[A^R_{\max}(L,\fro)=\max\{s| \widehat{\HFLR}_{O,*}(L,\fro,s)\ne 0\}, \quad \text{and} \quad A^R_{\min}(L,\fro)=\min\{s| \widehat{\HFLR}_{O,*}(L,\fro,s)\ne 0\}.\]
\end{cor}
\begin{proof}
This follows directly from Theorem~\ref{thm:exact triangle between OX and OO for knots} and its proof.
\end{proof}

In \cite[Proposition~7.2]{YXHFLR}, we have seen that $\Delta^R(K,\fra)(t)=\Delta^R(m(K),m(\fra))(-t)$, 
so \[\Delta^R_{O}(K,\fro)(t)=(1+t)\Delta^R(K,\fra)(t)=(1+t)\Delta^R(m(K),m(\fra))(-t)=(1+t)\Delta^R_{O}(m(K),m(\fro))(-t)/(1-t).\] It is interesting to ask whether there is a more direct relation between $\Delta^R_{O}(K,\fro)$ and $\Delta^R_{O}(m(K),m(\fro))$ or more generally, what's the relationship between $\widehat{\HFLR}_{O}(K,\fro)$ and $\widehat{\HFLR}_{O}(m(K),m(\fro))$. This question is motivated by the fact that we cannot find a general formula as in \cite[Proposition~4.4]{YXHFLR} that is true on small knots we calculated in Appendix~\ref{app:Calculation results for knots with small crossing numbers}.

\begin{cor}\label{cor:bounding tauR_max, tauR_min by tauR}
Let $(K,\fra)$ and $(K,\fro)$ be set up as in Theorem~\ref{thm:exact triangle between OX and OO for knots}. Then we have inequalities \[\tau_{\min}^R(K,\fro)+1\le \tau^R(K,\fra)\le \tau_{\max}^R(K,\fro).\]
\end{cor}
\begin{proof}
This follows from the minus version of exact triangle in Theorem~\ref{thm:exact triangle between OX and OO for knots}.
\end{proof}

Comparing this corollary with Proposition~\ref{prop:relationship tau_max, tau_min and tau_i^R}, one may conjecture that when $(K,\fra)$ and $(K,\fro)$ are set up as in Corollary~\ref{cor:OO v.s. OX polynomial for knots}, we have \[\tau_1^R(K,\fro)=\tau_2^R(K,\fro)=\tau^R(K,\fra).\]

However, this is not true. Actually, from Proposition~\ref{prop:relationship tau_max, tau_min and tau_i^R}, we know that $\tau^R_i\le \tau_{\max}^R-1$, but we can find many pairs of examples from Appendix~\ref{app:Calculation results for knots with small crossing numbers} and \cite[Appendix A]{YXHFLR} with $\tau^R(K,\fra)=\tau_{\max}^R(K,\fro)$. More concretely, $4_1$, $6_1$ with one of their involutions and $5_2$, $6_2$ with all their involutions provide such examples.

Theorem~\ref{thm:exact triangle between OX and OO for knots} and Corollary~\ref{cor:OO v.s. OX polynomial for knots} can be generalized to generalized strongly invertible links. To do this, we need to introduce a common generalization of $OO$ and $OX$-theories.

A \emph{mixed real Heegaard diagram} $\cH$ representing an oriented generalized strongly invertible link $(L,\fro)$ is a real Heegaard diagram satisfying all but the last item in Definition~\ref{def:O-heegaard diagram for s.i.links}. That is, we allow a strongly invertible knot component to have two fixed $O$-base points or a pair of fixed $(O,X)$-base points. When we label the strongly invertible knot components of $L$ by $1,\ldots,l_f$, we can assign $\cH$ a vector $\bm\epsilon\in\{O,X\}^{l_f}$ which records the fixed base point information.

For a mixed real Heegaard diagram $\cH$ with base point information $\bm\epsilon$, we can associate various real link Floer homologies as in $OO$ or $OX$-theory. We set up $\cR(\cH)$ as $\F[\bfu,\bfv,\bfU,\bfV]$, so that we have one $u$ ($v$)-variable for each fixed $O$ ($X$) base point and one $U$ ($V$) for each pair of base points $(O,O')$ ($(X,X')$). As usual, we require $\cH$ to be strongly $\s^R$-admissible. Then, the \emph{$\bm\epsilon$-full real link Floer chain complex} $\fullCFLR^-_{\bm\epsilon}(\cH,\s^R)$ is defined to be the chain complex over $\cR(\cH)$ with generators  \[\{\xv\in (\T_\alpha\cap \T_\beta)^R| \s^R(\xv)=\s^R \}.\] We define an endomorphism on it by 
\begin{align*}
\partial \xv = \sum_{\yv} \sum_{\phi\in \pi_2^R(\xv,\yv),\mu_R(\phi)=1}\# \widehat{\cM}_R(\phi) \prod_{1\le i\le l_f, 1\le j\le 2, \epsilon_i=O}  u_{i,j}^{n_{O_{i,j}^f}(\phi)} 
&\prod_{1\le i\le l_f, \epsilon_i=X}  u_{i}^{n_{O_{i}^f}(\phi)} v_{i}^{n_{X_{i}^f}(\phi)} \\ &\prod_{1\le j\le k} U_i^{n_{O_j}(\phi)} \prod_{1\le j\le k'} V_j^{n_{X_j}(\phi)}\yv    
\end{align*}
on generators and extend linearly over the base ring. Here, $k'-k$ records the number of $O$ in $\bm\epsilon$.

For simplicity, we assume the diagram is minimal in the sense that on each component of $L$, there are at most four base points and there are four base points on component $K$ if and only if $K$ is a strongly invertible knot with two fixed $O$-base points. In this case, $k=l_p$ and $k'=l_p+ \#\{1\le i \le l_f| \epsilon_i=O\}$. The definition can be generalized to stabilized diagrams easily.

Firstly, consider \[\CFLR_{\bm\epsilon,OO}^-(\cH,\s^R)=\fullCFLR^-_{\bm\epsilon}(\cH,\s^R)/ (V_i, v_j)_{1\le i\le k', 1\le j\le k'-k}.\] Since we have blocked all $X$-base points, this is indeed a chain complex. Let $\partial^-$ denote the induced differential on it. Then, the homology 
\[\HFLR_{\bm\epsilon,OO}^-(\cH,\s^R)=H_*(\CFLR_{\bm\epsilon,OO}^-(\cH,\s^R),\partial^-)\] 
is an invariant of $(L,\fra)$ and $\s^R$ when it is regarded as an $\F[\bfu,\bfU]$-module. Here, we have one degree $(-1,-1/2)$ variable $u_i$ for each strongly invertible component with $\epsilon_i=X$, a pair of degree $(-1.-1/2)$ variables $(u_{j,1},u_{j,2})$ for each strongly invertible component with $\epsilon_j=O$ and a degree $(-2,-1)$ variable $U_t$ for each pair of components $(K_t,K_t')$. The auxiliary data $\fra$ includes the orientation $\fro$ as well as the labeling of fixed points by $O$ and $X$ on components $K_i$ with $\epsilon_i=X$ (and the choice of $\bm \epsilon$ is determined by this). This is the \emph{$(\bm\epsilon,OO)$-minus real link Floer homology} associated to $(L,\fra)$ in the real $\spinc$ structure $\s^R$. 

Then consider \[\CFLR_{\epsilon,O}^-(\cH,\s^R)= \frac{\CFLR_{\epsilon,OO}^-(\cH,\s^R)}{(u_{i,1}=u_{i,2})_{1\le i\le l_f,\epsilon_i=O}},\] which is of course also a chain complex. We still abuse $\partial^-$ for the induced differential. Taking homology 
\[\HFLR_{\bm\epsilon,O}^-(\cH,\s^R)=H_*(\CFLR_{\bm\epsilon,O}^-(\cH,\s^R),\partial^-)\] 
is an invariant of $(L,\fra)$ and $\s^R$ when it is regarded as an $\F[u_{1},\ldots,u_{l_f}, U_1,\ldots U_{l_p}]$-module. This is the \emph{$(\bm\epsilon,O)$-minus real link Floer homology} associated to $(L,\fra)$ in the real $\spinc$ structure $\s^R$. 

In Section~\ref{sub:Structural properties in $S^3$} and \cite[Section~6.2]{YXHFLR}, we considered a collapsed version of minus real grid homology for generalized strongly invertible links in $(S^3,\tau_{\std})$. That definition does not depend on grid diagrams or the background real $3$-manifold $S^3$, so it actually works for any generalized strongly invertible link. That is, we can define 
\[c\CFLR_{\bm\epsilon,O}^-(\cH,\s^R)=\frac{\CFLR_{\bm\epsilon,O}^-(\cH,\s^R)}{(u_i-u_1)_{2\le i\le l_f}+ (U_j-u_1^2)_{1\le j\le l_p}},\] in which we abuse $u_i$ for $u_{i,1}=u_{i,2}$ for those strongly invertible components with $\epsilon_i=O$. Taking the homology, we get $c\HFLR_{\bm\epsilon,O}^-(\cH,\s^R)$, which is also an invariant of $(L,\fra)$ and $\s^R$. This is the \emph{$\bm\epsilon$-collapsed minus real link Floer homology} associated to $(L,\fra)$ in the real $\spinc$ structure $\s^R$. 

Lastly, we consider \[\widehat{\CFLR}_{\bm\epsilon,O}(\cH,\s^R)=\CFLR_{\bm\epsilon,O}^-(\cH,\s^R)/(u_i, U_j)_{1\le i\le l_f, 1\le j\le l_p} \] with induced differential $\widehat{\partial}$. Then 
\[\widehat{\HFLR}_{\bm\epsilon,O}(\cH,\s^R)=H_{*}(\widehat{\CFLR}_{\bm\epsilon,O}(\cH,\s^R),\widehat{\partial})\] 
is an invariant of $(L,\fra)$ and $\s^R$ as a vector space over $\F$. This is the \emph{$(\bm\epsilon,O)$-hat real link Floer homology} associated to $(L,\fra)$ in the real $\spinc$ structure $\s^R$.

Note that the collapsed real grid homology can be recovered from this (its generalization to stabilized diagrams, more precisely) by taking $\bm\epsilon$ to consist of all $X$ or all $O$.  More generally, both $OO$ and $OX$-theories can be recovered from the $\bm\epsilon$-decorated theory.

With all these preparations in hand, we can now make explicit the relationship between $OO$ and $OX$-theories for links.


\begin{thm}\label{thm:spectral sequence relating OO and OX for links}
Let $(L,\fra)$ be a generalized strongly invertible link with auxiliary data and $(L,\fro)$ be its underlying generalized strongly invertible link. Write $(L,\fra_{\bm\epsilon})$ for the link with the labeling on $K_i$ forgotten whenever $\epsilon_i=O$.
If $L$ has $l_f$ strongly invertible components, then there is a resolution cube of dimension $l_f$ with each vertex a mixed real link Floer homology group $c\HFLR_{\bm\epsilon,O}^-(L,\fra_{\bm \epsilon})$ ($\epsilon\in \{O,X\}^{l_f}$) calculating the homology $c\HFLR^-_{O}(L,\fro)$. Here, the diagram $\cH_{\bm\epsilon}$ for $(L,\fra_{\bm \epsilon})$ can be constructed from a minimal diagram for $(L,\fra)$ in $OX$-theory (which corresponds to $\bm\epsilon=(X,X\ldots,X)$) by performing quasi-stabilizations from Figure~\ref{fig:OX-OO stabilization} on strongly invertible components according to $\bm\epsilon$.

A similar result holds with $\widehat{\HFLR}_{\bm\epsilon,O}$ in place of $c\HFLR_{\bm\epsilon,O}^-$ and with $c\HFLR^-_{O}$ replaced by $\widehat{\HFLR}_{O}$.

As a corollary, when the background manifold is $(S^3,\tau_{\std})$, we have \[\Delta^R_{O}(L,\fro)=(1+t)^{l_f}\Delta^R(L,\fra)\] after suitable normalization of the real Alexander grading in the $OO$-theory. 
\end{thm}
\begin{proof}
Repeating the construction in Theorem~\ref{thm:exact triangle between OX and OO for knots} on each component is enough.
\end{proof}

\subsection{Examples}\label{sub:Examples}
In Appendix~\ref{app:Calculation results for knots with small crossing numbers}, the polynomial $\Delta_{O}^R$ and homology theories $\GHR^-_{O}$ as well as $\widehat{\GHR}_{O}$ for strongly invertible knots with diagrams of small crossing numbers are calculated. In this subsection, we compare them with properties analyzed above and extract some interesting phenomena from them.

\begin{example}
From Theorem~\ref{thm:exact triangle between OX and OO for knots} and its proof, we saw that for a strongly invertible knot $(K,\fro)$ and any enhancement of $\fro$ to a set of auxiliary data $\fra$, by choosing suitable real Heegaard diagrams, $\widehat{\CFKR}_{O}(K,\fro)$ looks like two copies of $\widehat{\CFKR}(K,\fra)$. Specializing to $(S^3,\tau_{\std})$, the real Alexander polynomial satisfies a similar relation. We also know from \cite{MOS2009knot} that such a quasi-stabilization leads to a doubling of the rank in homology in the  $\widehat{\HFK}$ theory. However, this is no longer true in real link Floer homology. By comparing examples in Appendix~\ref{app:Calculation results for knots with small crossing numbers} and those from \cite[Appendix A]{YXHFLR}, one can see that \[\dim \widehat{\HFKR}_{O}(K,\fro)\le 2\dim \widehat{\HFKR}(K,\fra)\] holds on all these examples as predicted by the exact triangle, but the equality rarely holds. 
\end{example}

\begin{example}\label{ex:torsion order comparison}
In \cite[Section~5.2]{YXHFLR}, we also introduced the notion of torsion order $\ord_{u}(K,\fra)$ as $\ord_{u}(\HFKR^-(K,\fra))$. Though $\HFKR^-(K,\fra)$ and $\HFKR_{O}^-(K,\fro)$ are related by an exact triangle (Theorem~\ref{thm:exact triangle between OX and OO for knots}), the torsion orders may be distinct. More concretely,
\begin{itemize}
    \item for $5_1$ with its unique involution, $\ord_{u}^O=2$ which is greater than $\ord_{u}=1$;
    \item for $5_2$ with one of its involutions, $\ord_{u}^O=1$ which is less than $\ord_{u}=2$.
\end{itemize}
\end{example}

\begin{example}\label{ex: comparison of OO and OX between different knots}
Though there is an exact triangle between $\widehat{\HFKR}$ and $\widehat{\HFKR}_{O}$ and $\Delta_{O}^R$ looks like two copies of $\Delta^R$, for most examples we have calculated in the appendix, $\widehat{\HFKR}_{O}(K,\fro)$ is not $\widehat{\HFKR}(K,\fra)\otimes (\F_{(0,0)}\oplus \F_{(0,1/2)})$ unless $\widehat{\HFKR}(K,\fra)$ is trivial or $(K,\fra)$ is obtained from a symmetric connected sum. This reflects the distinction between $OO$ and $OX$-theories. 

On the other hand, among all the examples we have calculated,  $\widehat{\HFKR}$ ($\HFKR^-$) are the same for $(K_1,\fra_1)$ and $(K_2,\fra_2)$ if and only if $\widehat{\HFKR}_{O}$ ($\HFKR^-_{O}$) coincide on the underlying oriented strongly invertible knots. Four families of examples are \begin{itemize}
    \item unknot, $5_2$, $7_4$, $8_3$ and $8_8$ with one of their involutions;
    \item $3_1$, $4_1$ and $7_3$ with one of their involutions;
    \item $5_2$ and $6_1$ with one of their involutions;
    \item $6_3$ and  $7_7$ with one of their involutions.
\end{itemize}
It is an interesting question to ask whether we can obtain $\widehat{\HFKR}_{O}$ ($\HFKR^-_{O}$) from $\widehat{\HFKR}$ ($\HFKR^-$) in a purely algebraic manner. The converse direction cannot be true since $\widehat{\HFKR}$ ($\HFKR^-$) depends on the choice of a direction while the $\widehat{\HFKR}_{O}$ ($\HFKR^-_{O}$) does not- the direction dependence seen in Section~\ref{sub:Torsion order, real tau-invariant and module structure} disappear after we set $u_1=u_2$.
\end{example}

\begin{example}\label{ex:tau_min, tau_max}
In Proposition~\ref{prop:relationship tau_max, tau_min and tau_i^R} and Corollary~\ref{cor:bounding tauR_max, tauR_min by tauR}, we provided estimates of $\tau^R_{\min}$ and $\tau^R_{\max}$ in terms of $\tau^R_i$ and $\tau^R$, respectively. They both imply the difference $\tau^R_{\max}-\tau^R_{\min}$ is always larger than or equal to one. From the proof of Proposition~\ref{prop:relationship tau_max, tau_min and tau_i^R}, we also know that $\tau^R_{\max}-\tau^R_{\min}=1$ implies that the $u_1$ and $u_2$ actions are homotopic on $\CFKR_{OO}^-$. Through examples, we see that \begin{itemize}
    \item There are knots with $\tau^R_{\max}-\tau^R_{\min}=1$ with $\HFKR_{O}^-$ non-trivial (in the sense that it does not coincide with the unknot). $9_{42}$ with one of its involutions and several other knots obtained equivariant connected sum provide us with such examples.
    \item For most knots in Appendix~\ref{app:Calculation results for knots with small crossing numbers}, $\tau^R_{\max}-\tau^R_{\min}=2$. Moreover, for all knots calculated there, the difference between the bigrading of generators of the two towers takes the form $(M^R,A^R)=(k,\frac{k+1}{2})$ for some $k\in \Z_{\ge 0}$. The author conjectures that this is true in general.
    \item For one involution on $8_{19}$, $\tau^R_{\max}-\tau^R_{\min}=3$. This provides some space for interesting phenomena to occur in the module structure on $\HFKR_{OO}^-$.
\end{itemize}
\end{example}

\begin{remark}
It is worthwhile to compare Example~\ref{ex:tau_min, tau_max} with \cite[Example~8.6]{YXHFLR}. Both examples focus on the only two non-thin knots $8_{19}$ and $9_{42}$ with crossing numbers $\le 9$. In that example, an exotic phenomenon occurred to $9_{42}$, from which the dependence of auxiliary data (essentially a choice of direction on the strongly invertible knot) was seen, while $\widehat{\HFKR}(8_{19},\fra)\cong \widehat{\HFK}(8_{19})$ for either choice of auxiliary data. In the $OO$-theory, there is obviously no direction dependence and the exotica moves to $8_{19}$, on which $\tau^R_{\max}-\tau^R_{\min}$ achieves the maximal value among all examples we have computed. The author expects a more theoretical explanation for these.
\end{remark}

\subsection{XX-theory}\label{sub:XX-theory}
After defining and comparing $OO$ and $OX$-theories, it is natural to ask for an $XX$ version. Intuitively, the $XX$-theory should resemble the $OO$-decorated one. However, careful readers may remember that the roles of $O$ and $X$ are indeed equal on the level of $\fullCFLR^-$, but they become distinguishable after we pass to homology, since $X$'s are now ``blocked''. Thus, requiring each strongly invertible knot component to have two fixed points decorated by $X$ indeed provides us with a new theory. The strategy from \cite{YXHFLR} and previous sections works well for this version, so instead of pursuing details, we will only list its properties in this short section.

To a $XX$-based real Heegaard diagram $\cH$ defined analogously to Definition~\ref{def:O-heegaard diagram for s.i.links}, we can define $\fullCFLR_{XX}^-(\cH,\s^R)$ as in Section~\ref{sub:Basic setup}. Then by blocking $X$-base points, we get two homology theories: 

\begin{itemize}
\item $\HFLR^-_{X}(L,\fro)$, a (bigraded) module over a graded polynomial ring over $\F$ which has one degree $(-2,-1)$ variable for each pair of knot components and for each strongly invertible knot component;
\item $\widehat{\HFLR}_{X}(L,\fro)$, a (bigraded) $\F$-vector space.
\end{itemize}

One should note that, as now we have no $O$-base point appearing on the fixed set, we only have degree $(-2,-1)$ variables in the minus version of the homology. Nevertheless, these homology theories share similar structural properties with the $OO$-ones. 

\begin{prop}
Specializing to $(S^3,\tau_{\mathrm{std}})$, $\HFLR^{\circ}_{X}(L,\fro)$ satisfy the following. 
	\begin{enumerate}
		\item They admit a combinatorial description via real grid diagrams, making them algorithmically computable. 
		\item There are crossing change maps $C_{A}^{\pm}$ ($C_{B}^{\pm}$) on $\HFLR^-_{X}(L,\fro)$ associated to equivariant crossing changes of type $A$ (type $B$). 
		\item There are saddle maps $\sigma$ and $\mu$ on $\HFLR^-_{X}(L,\fro)$ associated to attachment of pairs of $1$-handles to an equivariant link cobordism. We actually have two pairs of $(\sigma,\mu)$ according to how the $1$-handles change the number of components in $L$.
		\item $\HFLR^-_{X}$ and $\widehat{\HFLR}_{X}$ groups of an oriented skein triple fit into an exact triangle. 
	\end{enumerate}

\end{prop}
\begin{proof}
These can be proved using exactly the argument from~\cite[Sections~3,5,6,7]{YXHFLR}.
\end{proof}

Now, we can deduce from (2) and a standard computation for the unknot that for any (oriented) strongly invertible \emph{knot} $(K,\fro)$ in $(S^3,\tau_{\mathrm{std}})$, $\HFLR^-_{X}(K,\fro)$ is a rank one module over $\F[U]$. Then we can define torsion order $\ord^{X}_{U}$ and a $\tau_{X}^R$ invariant, which provide lower bounds on equivariant unknotting numbers and equivariant slice genus. 

As we mentioned above, $OO$ and $XX$-theories are parallel before we distinguish the roles of $O$ and $X$, so the curvature comparison in Section~\ref{sub:Killing the curvature} works for $XX$-theory in place of $OO$-one verbatim. For a counterpart of Theorem~\ref{thm:exact triangle between OX and OO for knots}, we have the following: 
  
\begin{thm}\label{thm:exact triangle between OX and XX for knots}
Let $(K,\fra)$ be a strongly invertible knot with auxiliary data and $(K,\fro)$ be the underlying oriented strongly invertible knot.
Then there are exact triangles \[\begin{tikzcd}
		\widehat{\HFKR}(K,\fra) \arrow{rr}{(0,1/2)} & & \widehat{\HFKR}(K,\fra) \arrow{dl}{(-1,-1/2)}\\
		& \widehat{\HFKR}_{X}(K,\fro) \otimes_{\F[U]} \F[u] \arrow{ul}{(0,0)}\\
	\end{tikzcd};\]
	\[\begin{tikzcd}
		\HFKR^-(K,\fra) \arrow{rr}{(0,1/2)} & & \HFKR^-(K,\fra) \arrow{dl}{(-1,-1/2)}\\
		& \HFKR^-_{X}(K,\fro)  \otimes_{\F} (\F_{(0,0)}\oplus \F_{(-1,-1/2)}) \arrow{ul}{(0,0)}\\
	\end{tikzcd}.\]
The mark $(d,s)$ on an arrow means that the map shifts $(M^R,A^R)$ up by $(d,s)$. These triangles split naturally along real $\spinc$ structures.
\end{thm}

Here, we need a change of base ring $\F[U]\to \F[u]$, $U\mapsto u^2$ in order to make  $\HFKR^-_{X}(K,\fro)$ and $\HFKR^-(K,\fra)$ comparable, which leads to an extra $\F_{(0,0)}\oplus \F_{(-1,-1/2)}$ tensor factor in the hat version. Despite this difference, this can be proved in the same way as Theorem~\ref{thm:exact triangle between OX and OO for knots} and also admits a straightforward generalization to links, as Theorem~\ref{thm:spectral sequence relating OO and OX for links}.

It is obvious that as an $\F[u]$-module, $\HFKR^-_{X}(K,\fro) \otimes_{\F[U]} \F[u]$ has rank two and the two associated $\tau^R$ invariants are just $\tau_{X}^R$ and $\tau_{X}^R+1$. Arguing as for Corollary~\ref{cor:bounding tauR_max, tauR_min by tauR}, we can get $\tau_{X}^R \le \tau^R \le \tau_{X}^R+1$, but we do not pursue this observation further here.

\section{Based link invariants, base point action and quasi-stabilization}\label{sec:Based link invariants, base point action and quasi-stabilization}

In Section~\ref{sub:Exact triangle and a common generalization}, we considered a common generalization of the $OO$ and $OX$-theories for the purpose of generalizing the exact triangle from knots to links. In this section, we continue further in this direction by defining real link Floer theory for multi-based (generalized) strongly invertible links (see Definition~\ref{def:based strongly invertible links}) mimicking the construction in \cite{zemke2019link}. Based on this, we consider base point actions  and quasi-stabilization maps on the real link Floer complex/ homology as well as the relationship between them. The morphisms in the exact triangle from Section~\ref{sub:Exact triangle and a common generalization} and~\ref{sub:XX-theory} turn out to be special cases of the quasi-stabilization maps. Moreover, the contents of this section serve as a motivating special case of a real link Floer TQFT.

\subsection{Upgrade to a based link invariant}\label{sub:Upgrade to a based link invariant}
As a generalization of a multi-based link (see~\cite[Definition~2.1]{zemke2019link}) to the real category, we have the following.

\begin{defn}\label{def:based strongly invertible links}
A \emph{real $3$-manifold with a multi-based (generalized) strongly invertible link} is a triple $(Y,\tau,\LL)$, in which \begin{itemize}
    \item $(Y,\tau)$ is a closed real $3$-manifold;
    \item $\LL=(L,\bfO,\bfX)$ where $L$ is an oriented generalized strongly invertible link inside $(Y,\tau)$ and $(\bfO,\bfX)$ are two disjoint collections of base points on $L$ so that \begin{itemize}
        \item $\tau$ fixes $\bfX$ and $\bfO$ setwise, individually;
        \item every component of $L$ contains at least two base points, and $X$ and $O$ appear alternatively; 
        \item the strongly invertible components $L$ split into three subsets $SI_{OO}\cup SI_{OX}\cup SI_{XX}$, such that \begin{itemize}
            \item if $K\in SI_{OO}$, then both points in $K\cap \fix(\tau)$ are labeled by $O$;
            \item if $K\in SI_{XX}$, then both points in $K\cap \fix(\tau)$ are labeled by $X$;
            \item if $K\in SI_{OX}$, then there is exactly one $X$ and one $O$ on $\fix(\tau)\cap K$.
        \end{itemize} 
    \end{itemize}
\end{itemize}
For simplicity, we shall refer to $\LL=(L,\bfO,\bfX)$ as a multi-based (generalized) strongly invertible link when the background $3$-manifold is clear from context.
\end{defn}


\begin{defn}\label{def:heegaard diagram for based s.i.links}
Let $\LL=(L,\bfO,\bfX)$ be a multi-based (generalized) strongly invertible link in $(Y,\tau)$. We say an (embedded) real Heegaard diagram $\cH=(\Sigma,\bm\alpha,\bm\beta,\bfO,\bfX, \tau)$ represents $\LL$ if: 
    \begin{itemize}
        \item $(\Sigma,\bm\alpha,\bm\beta,\bfX,\tau)$ and $(\Sigma,\bm\alpha,\bm\beta,\bfO, \tau)$ are both multi-based real Heegaard diagrams for $(Y,\tau)$. As in Definition~\ref{def:O-heegaard diagram for s.i.links}, we only require $\bfX$ and $\bfO$ to be fixed as sets by $R$.
        \item When we forget $\tau$, $(\Sigma,\bm\alpha,\bm\beta,\bfO,\bfX)$ is an embedded Heegaard diagram for the multi-based link $(Y,\LL)$ in the sense of \cite[Definition~3.1]{zemke2019link}.
        \item The labels $\bfO$ and $\bfX$ on $L$ are recorded faithfully by the Heegaard diagram.
    \end{itemize} 
\end{defn}

Here, we change the notation for the diagrammatic involution from $R$ to $\tau$ to stress that from now on, we will consider \emph{embedded real Heegaard diagrams}, not just abstract ones. We need this, as we want real link Floer theory for multi-based (generalized) strongly invertible links to be a strong real Heegaard invariant in the sense of \cite{GM_real_naturality}, so it is necessary to record how the link embeds into the $3$-manifold via the Heegaard diagram.  

\begin{prop}\label{prop:existence of real HD for multi-based strongly invertible knots and real H moves}
Let $\LL=(L,\bfO,\bfX)$ be a multi-based generalized strongly invertible link in $(Y,\tau)$. Then there exists a real Heegaard diagram representing $\LL$. Moreover, if $\cH$ and $\cH'$ both represent $\LL$, then they can be connected by a sequence of real Heegaard moves:
\begin{itemize}
    \item $\{1\}$-stabilization along the fixed set away from the base points.
    \item $\Z/2$-(de)stabilizations (of index $(1,2)$) away from the fixed point set and the base points;
    \item real handleslide away from the base points;
    \item equivariant isotopy of $\alpha$ and $\beta$-curves away from the base points.
\end{itemize}
Moreover, for any real $\spinc$ structure $\s^R$ on $(Y,\tau)$, we can choose a strongly $\s^R$-admissible real Heegaard diagram $\cH$ representing $\LL$. If two real Heegaard diagrams representing 
$\LL$ are both strongly $\s^R$-admissible, then they can be connected by a sequence of real Heegaard moves such that each intermediate real Heegaard diagram is also strongly $\s^R$-admissible.
\end{prop}

The argument for \cite[Proposition~2.4]{YXHFLR} still works in this case. Note that due to the based assumption, we no longer need to take care of the $(0,3)$-(de)stabilizations, which makes our lives even easier. 

With real Heegaard diagrams in hand, we define real link Floer groups associated to $\LL$. For the holomorphic data, we will keep using Convention~\ref{conv:setup of holomorphic structure}. To each based link $\LL$, we first associate a graded polynomial ring $\cR^-(\LL)=\F[\bfU_{\bfO}, \bfV_{\bfX}]$ which will serve as the base ring for the real link Floer complex. Note that $\bfO$ inside the data can be naturally divided into two subsets $\bfO^f\cup\bfO^p$ so that $\bfO^f$ is fixed pointwise by $\tau$ and $\bfO^p$ consists of pairs of the form $(O_i,O_i')$, which are interchanged by $\tau$. The number $\vert \bfO^f\vert$ is equal to $2 \vert SI_{OO}\vert +\vert SI_{OX}\vert$. The same remark holds for $\bfX$. We will abuse the notation $\bfO^p$ and  $\bfX^p$ to denote the set of \emph{pairs}. For each base point in $\bfO^f$ ($\bfX^f$), we have one variable $u$ ($v$) in $\bfU_{\bfO}$ ($\bfV_{\bfX}$) while for each pair in $\bfO^p$ ($\bfX^p$), we have one variable $U$ ($V$) in $\bfU_{\bfO}$ ($\bfV_{\bfX}$). As in Section~\ref{sub:Basic setup}, we grade the variables by  \[M_{\bfO}^R(u)=-1, \quad M_{\bfX}^R(u)=0;\quad M_{\bfO}^R(U)=-2, \quad M_{\bfX}^R(U)=0; \]
\[M_{\bfO}^R(v)=0, \quad M_{\bfX}^R(v)=-1;\quad M_{\bfO}^R(V)=0, \quad M_{\bfX}^R(V)=-2; \] and let \[A^R(-)=\frac{1}{2}(M_{\bfO}^R(-)-M_{\bfX}^R(-)).\] As before, we refer to $(M^R_{\bfO},A^R)$ as the bigrading, though the roles of $M^R_{\bfO}$ and $M^R_{\bfX}$ are more symmetric.

Fix a real $\spinc$ structure $\s^R$ on $(Y,\tau)$ and a strongly $\s^R$-admissible real Heegaard diagram $\cH$ representing $\LL$. The \emph{minus full real link Floer complex} associated to the multi-based (generalized) strongly invertible link $\LL$, $\fullCFLR^-(\LL,\s^R)$, is the module over $\cR^-(\LL)$ generated by \[\{\xv\in (\T_\alpha\cap \T_\beta)^R| \s^R(\xv)=\s^R \}.\] It is endowed with an endomorphism defined by \[\partial \xv =\sum_{\yv} \sum_{\phi\in \pi_2^R(\xv,\yv),\mu_R(\phi)=1}\# \widehat{\cM}_R(\phi) \prod_{O_k \in\bfO^f }  u_{k}^{n_{O_k}(\phi)} \prod_{X_t\in \bfX^f} v_t^{n_{X_t}(\phi)} \prod_{(O_i,O_i')\in \bfO^p} U_i^{n_{O_i}(\phi)} \prod_{(X_j,X_j')\in \bfX^p} V_j^{n_{X_j}(\phi)}\yv\] on generators and extended linearly over the base ring. At this stage, it might be better to use $\fullCFLR^-(\cH,\s^R)$  instead of $\fullCFLR^-(\LL,\s^R)$ since we haven't analyzed the diagram dependence of this theory. The abuse of notation will be justified in Theorem~\ref{thm:invariance and naturality of the full real link Floer complex}.

This $\partial$ is not necessarily a genuine differential, but its curvature can be deduced from the existing results above. The curvature of a pair of components or a strongly invertible component in $SI_{OO}$ is calculated in Section~\ref{sub:Basic setup}. The curvature for $K\in SI_{XX}$ is just $\omega^O_K$ with the roles of $X$ and $O$ interchanged, while curvature for $K\in SI_{OX}$ was calculated in \cite{YXHFLR} and cited in Section~\ref{sub:Killing the curvature}. The curvature of this complex $\omega_{\LL}\in \cR^-(\LL)$, characterized by \[\partial^2\xv=\omega_{\LL}\cdot \xv, \quad \forall \xv\] is actually an invariant of $\LL$. It is an easy exercise to see that the curvature does not depend on the real $\spinc$ structure $\s^R$ or the real Heegaard diagram $\cH$.

As in Section~\ref{sub:Basic setup}, the full real link Floer chain complex and derived homology theories admit grading structures under mild assumption on the real $\spinc$ structure. For simplicity, we assume the underlying link $L$ is null-homologous in $Y$. On the set of generators, we define \[M_{\bfO}(\xv,\yv)=\mu^R(\phi)-\sum_{O\in\bfO} n_O(\phi),\quad M_{\bfX}(\xv,\yv)=\mu^R(\phi)-\sum_{X\in\bfX} n_X(\phi),\] for $\phi\in \pi_2^R(\xv,\yv)$. $M_{\bfO}$ is well-defined independent of the choice of $\phi$ with value in $\Z$ when $c_1(\frs^R)$ is torsion; otherwise, it is only well-defined in $\Z/\fro(c_1(\frs^R))\Z$, where $2\fro(c_1(\frs^R))$ denotes the divisibility of its first Chern class of $\frs^R$. And a similar statement holds for $M_{\bfX}$. The difference \[A^R(\xv,\yv)= \frac{1}{2}(M_{\bfO}(\xv,\yv)-M_{\bfX}(\xv,\yv))\] serves as the real Alexander grading.

Besides the grading, we also have a filtration by $\Z^{\vert \bfO^f\vert+\frac{1}{2}\vert\bfO^p\vert} \oplus \Z^{\vert \bfX^f\vert+\frac{1}{2}\vert\bfX^p\vert}$ induced on by the powers of variables in $\bfU_{\bfO}$, $\bfV_{\bfX}$. 

We now provide a more precise statement on the invariance and naturality of $\fullCFLR^-(\LL,\s^R)$ which justifies the notation.

\begin{thm}\label{thm:invariance and naturality of the full real link Floer complex}
Fix $\LL\subset (Y,\tau)$ and $\s^R\in \rspinc(Y,\tau)$ as above. If $\cH$ and $\cH'$ are two strongly $\s^R$-admissible real Heegaard diagrams representing $\LL$, then there is $\Z^{\vert \bfO^f\vert+\frac{1}{2}\vert\bfO^p\vert} \oplus \Z^{\vert \bfX^f\vert+\frac{1}{2}\vert\bfX^p\vert}$-filtered map 
\[\Phi_{\cH\to \cH'}\colon \fullCFLR^-(\cH,\s^R) \to \fullCFLR^-(\cH',\s^R)\] which is a chain homotopy equivalence in the category of curved complexes of $\cR^-(\LL)$-modules. This will be referred to as a \emph{change of diagram map}. It preserves the relative gradings $M^R_{\bfO}$, $M^R_{\bfX}$ as well as $A^R$ when they are well-defined.  

If $\cH''$ is yet another strongly $\s^R$-admissible real Heegaard diagram representing $\LL$, then the change of diagram maps satisfy \[\Phi_{\cH\to \cH''}\simeq \Phi_{\cH'\to \cH''} \circ \Phi_{\cH\to \cH'}.\]    
\end{thm}

\begin{proof}
This is \cite[Theorem~2]{GM_real_naturality}.
\end{proof}

From this proposition, we can deduce that $\fullCFLR^-(\cH,\s^R)$ form a transitive system over the category of curved complexes over $\cR^-(\LL)$ indexed by the set of strongly $\s^R$-admissible real Heegaard diagrams representing $\LL$ in the sense of \cite[Definition~2.15]{zemke2019link}. Thus, we abuse $\fullCFLR^-(\LL,\s^R)$ to denote this transitive system as well as a specific representative of this system. 

Now for different multi-based generalized strongly invertible links, their full link Floer complexes form transitive systems over different categories since their base rings are in general distinct. This makes it hard to relate invariants associated to them. To remedy this, Zemke introduced the notion of colorings on based links (see \cite{Zemkequasistabandbasepointmoving}, \cite{zemke2019link}). He defined colorings as maps $\bfO\cup \bfX\to \bfP$ for some finite set $\bfP$, which induces a ring homomorphism $\cR^-(\LL)\to \cR^-_{\bfP}=\F[\bfU_{\bfP}]$. The dictionary between the base point convention in our paper and Zemke's is \[\bfO\mapsto \bfw, \quad \bfX\mapsto \bfz.\]  However, we find it more convenient to directly define a coloring as a ring homomorphism. 

\begin{defn}\label{def:coloring on link and colored real link Floer}
Let $\LL=(L,\bfO,\bfX)$ be a multi-based generalized strongly invertible link defined in Definition~\ref{def:based strongly invertible links}. A \emph{coloring} of $\LL$ is defined to be a graded homomorphism $\sigma\colon\cR^-(\LL)\to \cR$ for some graded ring $\cR$. This leads to the \emph{colored real link Floer complex} 
\[\fullCFLR^-(\LL^{\sigma},\s^R)=\fullCFLR^-(\LL,\s^R)\otimes_{\cR^-(\LL)} \cR,\] which comes with a differential $\partial^{\sigma}$ induced from $\partial$. When $\sigma(\omega_{\LL})=0$, we take the homology and define the \emph{colored real link Floer homology} \[\fullHFLR^-(\LL^{\sigma},\s^R)=H_*(\fullCFLR^-(\LL^{\sigma})).\]
\end{defn}

We will consider some special colorings which lead to homology theories generalizing $\HFLR^-$ and $\widehat{\HFLR}$ and their $OO$-counterparts. For convenience, we label the strongly invertible components by $1$, $\ldots$, $l_f$ and pairs of components by $1$, $\ldots$, $l_p$ as in Section~\ref{sub:Basic setup}. 

First, consider the ring \[\cR^-_{O}(L)=\F[u_1,\ldots,u_{l_f}, U_1,\ldots U_{l_p}] \] with $u_i$, $U_j$ graded as usual. There is a natural quotient map \[\sigma_-^{O}: \cR^-(\LL)\to \cR^-_{O}(L),\] which kills all the variables in $\bfV_{\bfX}$. For variables in $\bfU_{\bfO}$, $\sigma(u_{O_k})=u_i$ if $O_k$ lies on the $i$-th strongly invertible component; $\sigma(U_{O_k})=u_i^2$ if $(O_k,O_k')$ lies on the $i$-th strongly invertible component;  $\sigma(U_{O_k})=U_i$ if $(O_k,O_k')$ lies on the $i$-th pair of components. It is easy to check that $\sigma_-^{O}(\omega_{\LL})=0$, from which we define the \emph{$O$-minus version of real link Floer homology} $\HFLR^-(\LL^O,\s^R)$ as $\fullHFLR^-(\LL^{\sigma_-^O},\s^R)$. The chain complex $\fullCFLR^-(\LL^{\sigma_-^O},\s^R)$ will be denoted by $\CFLR^-(\LL^O,\s^R)$.

Alternatively, we can interchange the roles of $\bfO$ and $\bfX$ and consider ring \[\cR^-_{X}(L)=\F[v_1,\ldots,v_{l_f}, V_1,\ldots V_{l_p}] \] with $v_i$, $V_j$ graded as usual. Then there is a map \[\sigma_-^{X}: \cR^-(\LL)\to \cR^-_{X}(L),\] which kills all the variables in $\bfU_{\bfO}$. It is again obvious that $\sigma_-^X(\omega_{\LL})=0$, from which we define the \emph{$X$-minus version of real link Floer homology} $\HFLR^{-}(\LL^X,\s^R)$ as $\fullHFLR^-(\LL^{\sigma_-^X},\s^R)$. The chain complex $\fullCFLR^-(\LL^{\sigma_-^X},\s^R)$ will be denoted by $\CFLR^-(\LL^X,\s^R)$.

Next, we consider the natural quotient map \[\widehat{\sigma}: \cR^-(\LL)\to \cR^-(\LL)/(\bfV_{\bfX},\bfU_{\bfO})=\F,\] which kills all the variables in the polynomial ring. From this, we define the \emph{hat version of real link Floer homology} $\widehat{\HFLR}(\LL,\s^R)$ as $\fullHFLR^-(\LL^{\widehat{\sigma}},\s^R)$. The chain complex $\fullCFLR^-(\LL^{\widehat{\sigma}},\s^R)$ will be denoted by $\widehat{\CFLR}(\LL,\s^R)$.

One can deduce the invariance of the colored real link Floer complexes from that of the full real link Floer complex. Thus, each of them forms a transitive system over (possibly curved) complexes of $\mathrm{codomain}(\sigma)$-modules sharing the same index set as $\fullCFLR^-(\LL,\s^R)$. In particular, we have transitive systems $\CFLR^-(\LL^O,\s^R)$, $\CFLR^-(\LL^X,\s^R)$ and $\widehat{\CFLR}(\LL,\s^R)$, all indexed by strongly $\s^R$-admissible diagrams representing $\LL$, over the categories of chain complexes of $\cR^-_{O}(L)$, $\cR^-_{X}(L)$ and $\F$-modules, respectively. Besides having zero curvature which makes taking homology available, these theories have other advantages: The base rings depend only on the underlying oriented link $L$, so they are actually link invariants independent of base point information. Moreover, this theory admits well-defined type-changing quasi-stabilization maps, see Section~\ref{sub:quasi-stabilization}.

On the full real link Floer complex, we have a filtration given by powers of variables in $\bfU_{\bfO}$ and $\bfV_{\bfX}$. In the colored case, we have a similar filtration when the codomain of $\sigma$ is a (graded) polynomial ring. For example, we have $\Z^{l_p+l_f}$-filtration on $\CFLR^-(\LL^O,\s^R)$ and $\CFLR^-(\LL^X,\s^R)$. 

As in Section~\ref{sec:A different theory for strongly invertible links}, we can take direct sum over all real $\spinc$ structures in $\rspinc(Y,\tau)$ and get $\fullCFLR^-(\LL)$ as well as $\HFLR^\circ (\LL)$, which are link invariants. It is worth mentioning that both $OO$ and $OX$-theories of the underlying oriented link $L$ (with any choice of auxiliary data) can be recovered from the multi-based theory by choosing $\bfO$ and $\bfX$ properly. 

Before ending this subsection, we introduce an important coloring as well as a variation of the differential. This originated from a discussion with Gary Guth and Ciprian Manolescu. Consider $\cR^-_{s}(\LL)= \F[\bfu_{\bfO},\bfv_{\bfX}]$, in which we have one $u\in \bfu_{\bfO}$ ($v\in \bfv_{\bfX}$) for each pair in $\bfO^p$ ($\bfX^p$) and each point in $\bfO^f$ ($\bfX^f$). They share the same bi-grading as small $u$ and $v$ in $\cR^-(\LL)$. Now we define \[\sigma_s: \cR^-(\LL)\to \cR^-_{s}(\LL),\] which is the identity on variables associated to base points in $\bfO^f$ or $\bfX^f$ and is $U_O\mapsto u_{O}^2$ ($V_X\mapsto v_{X}^2$) for $(O,O')\in \bfO^p$ ($(X,X')\in \bfX^p$). Note that $\sigma_s$ is an injective ring map and the curvature can be written down explicitly using the previous calculation. First observe that $\sigma_s(\omega_{\LL})$ is zero if and only if $\LL$ is a knot with only one pair of $(O,X)$-base points. Next assume we are not in this case. Then \begin{itemize}
	\item for a pair of components $(K,K')$, $\omega_{K,K'}^s$ reads $u_1^2v_1^2+ u_2^2v_1^2\ldots +u_n^2v_{n-1}^2+u_n^2v_n^2+ u_1^2v_n^2$ ($2n$ terms in total);
	\item for $K\in SI_{OX}$, $\omega_{K}^s$ reads $u_1^2v_1^2+ u_2^2v_1^2\ldots +u_n^2v_{n-1}^2+u_n^2v_n^2$ ($2n-1$ terms in total);
	\item for $K\in SI_{OO}$, $\omega_{K}^s$ reads $u_1^2v_1^2+ u_2^2v_1^2\ldots +u_n^2v_{n-1}^2+u_n^2v_n^2+ u_{n+1}^2v_n^2$ ($2n$ terms in total);
	\item for $K\in SI_{XX}$, $\omega_{K}^s$ reads $u_1^2v_1^2+ u_1^2v_2^2\ldots +u_{n-1}^2v_{n}^2+u_n^2v_{n}^2+u_n^2v_{n+1}^2$ ($2n$ terms in total).
\end{itemize}  
Since we are working in characteristic $2$, each of these has a unique square root. For example, for $K\in SI_{OX}$, $\omega_{K}^s= (u_1v_1+ u_2v_1\ldots +u_nv_{n-1}+u_nv_n)^2$. Following an idea from usual Lagrangian Floer homology \cite{FOOO09a}, we introduce a \emph{bounding chain} $\bL$, which is zero if $\LL$ is a knot with only one pair of $(O,X)$-base points and is the sum over all these square roots otherwise. Then, let $D=\partial + \bL \cdot \id$ be a new endomorphism on $\fullCFLR^-(\LL^{\sigma_s})$, which is a genuine differential, as $D^2=\partial^2+ \bL^2 \cdot \id=0$.

As a direct corollary of Theorem~\ref{thm:invariance and naturality of the full real link Floer complex}, we have the following:
\begin{cor}
Let $\LL\subset (Y,\tau)$ be a multi-based strongly invertible link and let $\s^R$ be a real $\spinc$ structure on $(Y,\tau)$, then $(\fullCFLR^-(\LL^{\sigma_s},\s^R),D)$ is a natural invariant associated to $\LL$. More precisely, for any real Heegaard diagrams $\cH$, $\cH'$ of $\LL$, the change of diagram map $\Phi_{\cH\to \cH'}$ from Theorem~\ref{thm:invariance and naturality of the full real link Floer complex} induces a homotopy equivalence \[\Phi_{\cH\to \cH'}^s\colon (\fullCFLR^-(\cH,\s^R)^{\sigma_s},D) \to (\fullCFLR^-(\cH',\s^R)^{\sigma_s} ,D).\] Moreover, if $\cH''$ is yet another real Heegaard diagram for $\LL$, then \[\Phi_{\cH'\to \cH''}^s\circ \Phi_{\cH\to \cH'}^s= \Phi_{\cH\to \cH''}^s .\] 
\end{cor} 

For simplicity, from now on, we will abbreviate $(\fullCFLR^-(\LL^{\sigma_s},\s^R),D)$ as $(\fullCFLR^-(\LL^s,\s^R),D)$.

\subsection{Base point actions}\label{sub:base point action}

Following \cite{zemke2019link} (see also \cite{Sarkar_2015movingbasept} and \cite{Zemkequasistabandbasepointmoving}), we introduce base point actions on the real link Floer complexes. We present an algebraic definition in this subsection and a more geometric interpretation will be given in Section~\ref{subsub:Compositions}.

As above, we fix a multi-based (generalized) strongly invertible link $\LL=(L,\bfO,\bfX)$ in $(Y,\tau)$, a real $\spinc$ structure $\s^R\in \rspinc(Y,\tau)$ and a strongly $\s^R$-admissible real Heegaard diagram $\cH$ representing $\LL$.

If $O_c\in \bfO^f$ is a base point lying on the fixed set, we define \[\Phi_{O_c}^f\colon \fullCFLR^-(\LL,\s^R)\to \fullCFLR^-(\LL,\s^R)\]
\[\Phi_{O_c}^f(\xv)=u_{c}^{-1}\sum_{\yv}\sum_{\phi\in \pi_2^R(\xv,\yv),\mu^R(\phi)=1} n_{O_c}(\phi)\# \widehat{\cM}_R(\phi) U_{\bfO}^{n_{\bfO}(\phi)} V_{\bfX}^{n_{\bfX}(\phi)} \yv,\] in which $U_{\bfO}^{n_{\bfO}(\phi)}$ and $V_{\bfX}^{n_{\bfX}(\phi)}$ are shorthands for \[\prod_{O_k \in\bfO^f}  u_{k}^{n_{O_k}(\phi)} \cdot \prod_{(O_i,O_i')\in \bfO^p} U_i^{n_{O_i}(\phi)} \quad \text{and}\quad \prod_{X_t\in \bfX^f} v_t^{n_{X_t}(\phi)} \cdot  \prod_{(X_j,X_j')\in \bfX^p} V_i^{n_{X_j}(\phi)},\] respectively.

If $(O_c,O_c') \in \bfO^p$ is a pair of base points off the fixed set, we define \[\Phi_{(O_c,O_c')}^p\colon \fullCFLR^-(\LL,\s^R)\to \fullCFLR^-(\LL,\s^R)\]
\[\Phi_{(O_c,O_c')}^p(\xv)=U_{c}^{-1}\sum_{\yv\in (\T_{\alpha}\cap\T_{\beta})^R}\sum_{\phi\in \pi_2^R(\xv,\yv),\mu^R(\phi)=1} n_{O_c}(\phi)\# \widehat{\cM}_R(\phi) U_{\bfO}^{n_{\bfO}(\phi)} V_{\bfX}^{n_{\bfX}(\phi)} \yv.\]

Similarly, we can define $\Psi^f_{X}$ for $X\in \bfX^f$ and $\Psi^p_{(X,X')}$ for $(X,X') \in \bfX^p$. One can check easily from the definition that $\Phi_{O_c}^f$ is of bigrading $(0,1/2)$ while $\Phi_{(O_c,O_c')}^p$ is of bigrading $(1,1)$.

As for the usual $\Phi_{w}$ and $\Psi_{z}$ maps, these can be regarded as `formal derivatives' of the differential. More precisely, they can be alternatively characterized by  \[\Phi_{O}^f=\frac{d\partial}{du_O}, \quad \Phi_{(O,O')}^p=\frac{d\partial}{dU_O};\]
\[\Psi_{X}^f=\frac{d\partial}{dv_X}, \quad \Phi_{(X,X')}^p=\frac{d\partial}{dV_X},\] where the derivation operators act on the differential. We remind readers to be careful that $\frac{d\partial}{du_O}$ is not the same as $\frac{d}{du_O}\circ \partial$ and similar for other variables.

The base point actions are also well-defined on the colored real link Floer complex $\fullCFLR^-(\LL^{\sigma},\s^R)$ by replacing elements in $U_{\bfO}$ and $V_{\bfX}$ with their images under $\sigma$.

\begin{lem}\label{lem:commutation relation of base point actions and the differential}
On $\fullCFLR^-(\LL,\s^R)$, we have \[\partial\circ \Phi_{O}^f+ \Phi_{O}^f\circ\partial=0; \quad \partial\circ \Psi_{X}^f+ \Psi_{X}^f\circ\partial=0; \]
 \[\partial\circ \Phi_{(O,O')}^p+ \Phi_{(O,O')}^p\circ\partial= V_{X_1}+ V_{X_2}; \quad \partial\circ \Psi_{(X,X')}^p+ \Psi_{(X,X')}^p\circ\partial= U_{O_1}+ U_{O_2}. \] 
Here, $(X_1,X_1')$ and $(X_2,X_2')$ are the pairs of base points adjacent to $(O, O')$; $(O_1,O_1')$ and $(O_2,O_2')$ are characterized similarly. We may have $O_i=O_i'$ which is a fixed base point, in which case, we use $U_{O_1}$ to denote $u_{O_1}^2$ and similarly for $X_i$.
After coloring, if $\fullCFLR^-(\LL^{\sigma},\s^R)$ becomes a chain complex, i.e., $\sigma(\omega_{\LL})=0$, then all the base point actions become chain maps.
\end{lem}
\begin{proof}
All these follow immediately by taking derivatives $\frac{d}{du_{O}}$,  $\frac{d}{dU_{O}}$, etc. of the equations \[\partial^2=\omega_{\LL}\cdot\id, \text{ or } \partial^2=0.\]
\end{proof}

\begin{prop}\label{prop:natuality and invariance of base point action}
The base point actions commute with change of diagram maps characterized in \cite[Section~8]{GM_real_naturality} up to chain homotopy. Thus, we have well-defined base point actions on the transitive chain homotopy type invariants.
\end{prop}
\begin{proof}
This can be proved in exactly the same way as \cite[Lemma~3.2]{Zemkequasistabandbasepointmoving} by observing that \[\Phi_{\cH\to \cH'}\circ \partial_{\cH} +\partial_{\cH'}\circ \Phi_{\cH\to \cH'} =0\] holds on the nose and using the formal differential expression of base point actions. 

\end{proof}

\begin{lem}\label{lem:commutation between base point actions}
Let $(Y,\tau,\LL)$ be a multi-based generalized strongly invertible link. Let $(X_i,X_i')$ or $(O_i,O_i')$ ($i=1,2$) be any pairs of base points on $\LL$, and let $X^f_i$, $O^f_i$ ($i=1,2$) be any fixed base points. Then we have \[\Phi_{(O_1,O_1')}^p\Phi^p_{(O_2,O_2')}+\Phi_{(O_2,O_2')}^p\Phi^p_{(O_1,O_1')}\simeq 0,\quad \Phi_{(O_i,O_i')}^p\Phi^f_{O^f_j}+\Phi_{O^f_j}^f\Phi^p_{(O_i,O_i')}\simeq 0;\]
\[\Psi_{(X_1,X_1')}^p\Psi^p_{(X_2,X_2')}+\Psi_{(X_2,X_2')}^p\Psi^p_{(X_1,X_1')}\simeq 0,\quad \Psi_{(X_i,X_i')}^p\Psi^f_{X_j^f}+\Psi_{X^f_j}^f\Psi^p_{(X_i,X_i')}\simeq 0.\]
Furthermore, we have 
\[\Phi^f_{O_i^f}\Psi^f_{X_j^f}+ \Psi^f_{X^f_j}\Phi^f_{O^f_i}\simeq 0,\quad \Phi^p_{(O_i,O_i')}\Psi^p_{(X_i,X_i')}+ \Psi^p_{(X_i,X_i')}\Phi^p_{(O_i,O_i')}+ N(O_j,X_i)\cdot \id\simeq 0;\]
\[\Phi^f_{O_j^f}\Psi^p_{(X_i,X_i')}+ \Psi^p_{(X_i,X_i')}\Phi^f_{O_j^f}\simeq 0,\quad \Phi^p_{(O_i,O_i')}\Psi^f_{X_j^f}+ \Psi^f_{X_j^f}\Phi^p_{(O_i,O_i')}\simeq 0.\]
Here, $N(O_j,X_i)\in \F$ denotes half of the number of segments in $L\setminus (\bfO\cup \bfX)$ with both ends in $\{O_j,O_j',X_i,X_i'\}$. (Using the symmetry, one can see easily that this number is even, so we can consider half of it modulo $2$.)
\end{lem}
\begin{proof}
All these can be proved by differentiating equations from Lemma~\ref{lem:commutation relation of base point actions and the differential} and noting that \[\frac{d}{dU_O} U_O=1, \quad \frac{d}{du_O} u_O^2=0;\]
\[\frac{d}{dV_X} V_X=1, \quad \frac{d}{dv_X} v_X^2=0,\]  when we work over $\F$.
\end{proof}

\begin{lem}\label{lem:squares of the base point action}
The endomorphisms $\Psi^p_{(X,X')}$ and $\Phi^p_{(O,O')}$ satisfy \[(\Psi^p_{(X,X')})^2\simeq 0 \quad \text{and}\quad (\Phi^p_{(O,O')})^2\simeq 0.\]
When the component of $L$ containing $O$($X$) has more than one $O$-base point, the endomorphisms $\Psi^f_{X}$ and $\Phi^f_{O}$ satisfy \[(\Psi^f_{X})^2\simeq U_{O_a} \quad \text{and}\quad (\Phi^f_{O})^2\simeq V_{X_a},\] in which $(X_a,X_a')$($(O_a,O_a')$) is the pair of base points adjacent to $O(X)\in \fix(\tau)$; when that component lies in $SI_{OX}$ and is minimally pointed, they satisfy \[(\Psi^f_{X})^2\simeq u_{O}^2 \quad \text{and}\quad (\Phi^f_{O})^2\simeq v_{X}^2,\] 
if there is another component in the link; \[(\Psi^f_{X})^2\simeq 0 \quad \text{and}\quad (\Phi^f_{O})^2\simeq 0,\] if $\LL$ is just a minimally pointed strongly invertible knot.
\end{lem}
\begin{proof}
The strategy of \cite[Lemma~4.9]{zemke2019link} applies to our case. We will focus on $\Phi_{(O,O')}^p$ and $\Phi_{O}^f$ and the same argument works for $\Psi^p_{(X,X')}$ and $\Psi^f_{X}$. 

For $\Phi_{(O,O')}^p$, we write $\partial=\sum_{k=0}^\infty \partial_k U_{O}^k$, in which $\partial_k$ is a matrix with entries in $\F[U_{\bfO},V_{\bfX}]$ not involving $U_O$. Then we have \[\Phi_{(O,O')}^p=\sum_{k=0}^\infty k\partial_k U_{O}^{k-1}\quad \text{and} \quad \partial^{2}=\sum_{k=0}^\infty\sum_{m+n=k} \partial_m\partial_n U_{O}^{k}.\]

Using the expression of $\omega_{\LL}$, we know that \[\sum_{m+n=k} \partial_m\partial_n =0\quad \text{for all $k\ge 2$}. \]

Define \[H\colon= \sum_{k=0}^\infty\frac{k(k-1)}{2}\partial_{k} U_{O}^{k-2},\] which provides a null-homotopy of $(\Phi_{(O,O')}^p)^2$. 

For $\Phi_{O}^f$, we can similarly write $\partial=\sum_{k=0}^\infty \partial_k u_{O}^k$, in which $\partial_k$ is a matrix with entries in $\F[U_{\bfO},V_{\bfX}]$ not involving $u_O$. Then we have \[\Phi_{O}^f=\sum_{k=0}^\infty k\partial_k u_{O}^{k-1}\quad \text{and}\quad \partial^{2}=\sum_{k=0}^\infty\sum_{m+n=k} \partial_m\partial_n u_{O}^{k}.\]

Now we have \[\sum_{m+n=k} \partial_m\partial_n =0\quad \text{for all $k\ge 3$},\quad \text{while } \sum_{m+n=2} \partial_m\partial_n =V_{X_a}\] if the component lies in $SI_{OX}$ and is not minimally pointed. When the component is minimally pointed but $L$ has some other components, then \[\sum_{m+n=k} \partial_m\partial_n =0\quad \text{for all $k\ge 3$},\quad \text{while } \sum_{m+n=2} \partial_m\partial_n =v_{X}^2.\] Finally, if $\LL$ is simply a minimally pointed strongly invertible knot, then \[\sum_{m+n=k} \partial_m\partial_n =0\quad \text{for all $k\ge 2$},\] as $\omega_{\LL}=0$ in this case. Define \[H\colon= \sum_{k=0}^\infty\frac{k(k-1)}{2}\partial_{k} u_{O}^{k-2}\] as above.

Then we calculate that \begin{align*}
(\Phi_{O}^f)^2+\partial H+H\partial &= \sum_{m,n\ge 0}^\infty\frac{(m+n)(m+n-1)}{2}\partial_{m}\partial_{n} u_{O}^{m+n-2}\\
&= \sum_{k=0}^\infty\frac{k(k-1)}{2}\partial_{k} u_{O}^{k-2}\sum_{m+n=k}\partial_m\partial_n\\
&= \sum_{m+n=2}\partial_m\partial_n.
\end{align*} 

The result then follows from the equations we derived from $\omega_{\LL}$.
\end{proof}
\subsubsection{Base point actions on $\fullCFLR^-(\LL^s,\s^R)$}\label{subsub:Base point actions on Dcomplex}

In this subsection, we consider the induced base point actions on $\fullCFLR^-(\LL^s,\s^R)$. Before changing the differential from $\partial^{\sigma_s}$ to $D$, $\Psi$ and $\Phi$ maps are naturally inherited by $\fullCFLR^-(\LL^s,\s^R)$. Then one observes that \[D= \partial^{\sigma_s}+\bL \cdot \id,\] with $\bL\in \cR^-_s(\LL)$ living inside the base ring. Thus, for any $\cR^-_s(\LL)$-module homomorphism $f: \fullCFLR^-(\LL^s,\s^R) \to \fullCFLR^-(\LL^s,\s^R)$, $f$ is a chain map for $\partial^{\sigma_s}$ if and only if $f$ is a chain map for $D$. Similarly, when $f$ and $g$ are both chain maps, they are homotopic as chain maps on $(\fullCFLR^-(\LL^s,\s^R), \partial^{\sigma_s})$ if and only if they are homotopic as chain maps on $(\fullCFLR^-(\LL^s,\s^R), D)$. Based on this, all the properties we have seen above remain true on $(\fullCFLR^-(\LL^s,\s^R), D)$. We record them in the following proposition. 

\begin{prop}
For any multi-based strongly invertible link $\LL\subset (Y,\tau)$, we have base point actions $\Phi^f_{O}$, $\Phi^p_{(O,O')}$ and $\Psi^f_{X}$, $\Psi^p_{(X,X')}$ on $\fullCFLR^-(\LL^s,\s^R)$ which are well-defined self-morphisms in the sense that they commute with change of diagram maps (but they are not necessarily chain maps). They satisfy the following properties:
\begin{itemize}
\item $[D,\Phi_{O}^f]=0$, $[D,\Phi_{(O,O')}^p]= v_{X_1}^2+ v_{X_2}^2$, where we use brackets as shorthand for commutators and $X_1$, $X_2$ have the same meaning as in Lemma~\ref{lem:commutation relation of base point actions and the differential}. Similar relations hold for $\Psi^f$ and $\Psi^p$.
\item $[\Phi_{(O_1,O_1')}^p,\Phi^p_{(O_2,O_2')}]$, $[\Phi_{(O_i,O_i')}^p,\Phi^f_{O^f_j}]$,   $[\Phi^f_{O_i^f},\Psi^f_{X_j^f}]$, $[\Phi^p_{(O_i,O_i')},\Psi^f_{X_j^f}]$ and their counterparts with $O$ and $X$ interchanged are all homotopic to zero, while $[\Phi^p_{(O_i,O_i')},\Psi^p_{(X_i,X_i')}] \simeq N(O_j,X_i)\cdot \id$. Here, we borrow notation from Lemma~\ref{lem:commutation between base point actions}.
\item  $(\Psi^p_{(X,X')})^2$ and $(\Phi^p_{(O,O')})^2$ are always homotopic to zero. $(\Psi^f_{X})^2\simeq 0 \simeq (\Phi^f_{O})^2$ when $\LL$ is just a minimally pointed strongly invertible knot while $(\Psi^f_{X})^2\simeq v_{O_a}^2$ and $(\Phi^f_{O})^2\simeq v_{X_a}^2$.
\end{itemize}
\end{prop}

Motivated by the expression of base point action as derivatives of the differential, we introduce the following modification of base point action on  $(\fullCFLR^-(\LL^s), D)$ :\[\Phi^{f,b}_{O}= \Phi^{f}_{O}+ v\cdot \id \quad \Psi^{f,b}_{X}= \Psi^{f}_{X}+ u\cdot \id,\] where $v$ is variable associated to the $X$-basepoint nearest to $O$ and $u$ is variable associated to the $O$-basepoint nearest to $X$. These modified base point actions can be regarded as formal derivatives of the differential $D$. We will see in Lemma~\ref{lem:present f-base point action as composition of quasistab maps} that they admit nice decorated cobordism descriptions.

\subsection{Quasi-stabilizations}\label{sub:quasi-stabilization}

In this section, we consider quasi-stabilization maps between real link Floer complexes of multi-based links with the same underlying oriented strongly invertible link but different base point data. We will first investigate quasi-stabilizations happening in pairs in Section~\ref{subsub:Quasi-stabilizations without type change}, which follow \cite{Zemkequasistabandbasepointmoving} closely and then consider quasi-stabilizations that change the decomposition $SI_{OO}\cup SI_{OX}\cup SI_{XX}$ in Section~\ref{subsub:Type-changing quasi-stabilizations}, which generalize Section~\ref{sub:Exact triangle and a common generalization}. Finally, we move to consider quasi-stabilization maps on $(\fullCFLR^-(\LL^s), D)$, the colored chain complex with shifted differential from Section~\ref{subsub:Quasi-stabilization maps on Dcomplex}.

\subsubsection{Quasi-stabilization without type-change}\label{subsub:Quasi-stabilizations without type change}

Let $\LL=(L,\bfO,\bfX)$ be a multi-based generalized strongly invertible link in a closed real $3$-manifold $(Y,\tau)$. We still fix some $\s^R \in \rspinc(Y,\tau)$.

Consider two new pairs of points $(O,O')$ and $(X,X')$ on $L$ so that $\tau(O)=O'$, $\tau(X)=X'$ and $O$, $X$ lie on the same component of $L\setminus (\bfO\cup \bfX)$. Then $\LL^+=(L,\bfO\cup \{O,O'\}, \bfX\cup \{X,X'\})$ is a multi-based strongly invertible link that fits into Definition~\ref{def:based strongly invertible links}. The four new base points may lie on a single strongly invertible knot component $K\subset L$ or on a pair of components $(K,K')$ interchanged by $\tau$. 
Under these assumptions, we will define quasi-stabilization maps \[S_{O,X}^+\colon\fullCFLR^-(\LL^{\sigma},\s^R)\to \fullCFLR^-((\LL^+)^{\sigma'},\s^R), \quad  S_{O,X}^-\colon\fullCFLR^-((\LL^+)^{\sigma'},\s^R)\to \fullCFLR^-(\LL^{\sigma},s^R),\] if along the short segment in $L\setminus (\bfO \cup \bfX)$ following the orientation of $L$, $X$ appears before $O$. Here, $\sigma$ is a coloring of $\LL$ and $\sigma'$ is an extension of it. (Note that $\cR^-(\LL)$ naturally embeds into $\cR^-(\LL^+)$ as a subring, so it makes sense to talk about extension.) This would be a well-defined chain map if $(O,O')$ is not immediately adjacent to the fixed set in $\LL$, i.e., the nearest $X$-base points in $\LL$ form a pair $(X_a,X_a')$ and $\sigma'(V_X)$ equals $\sigma(V_{X_a})$. 

If instead $X$ appears after $O$ along this arc, then we define maps $S^+_{X,O}$ and $S^-_{X,O}$. 
Interchanging the roles of $\bfX$ and $\bfO$, we define alternative quasi-stabilization maps $T^{\pm}_{O,X}$ as well as $T^{\pm}_{X,O}$ in a similar fashion. They would be chain maps when $(X,X')$ is not adjacent to a fixed $O$ on $\LL$ and the extended coloring $\sigma'$ is chosen so that $(O,O')$ shares the same coloring with the other pair of $O$-base points adjacent to $(X,X')$. Here, we follow~\cite{zemke2019link} and call these \emph{type-$S$} quasi-stabilization maps and \emph{type-$T$} quasi-stabilization maps, respectively.

Suppose that we have fixed a strongly $\s^R$-admissible real Heegaard diagram $\cH=(\Sigma,
\bm\alpha,\bm\beta,\bfO,\bfX,\tau)$ for $(Y,\tau,\LL)$, and $(O,O')$, $(X,X')$ are the pairs of new base points on $L$ labeled using the convention above. Suppose that along the orientation of $L$, $O$ comes after $X$. Write $(O_a,O_a')$, $(X_a,X_a')$ for the adjacent pairs of $X$, $O$-base points and assume that along the orientation of $L$, the base points read $O_a,X,O,X_a$. We always add a prime when we want to consider the $\tau$-image of a base point. By Definition~\ref{def:heegaard diagram for based s.i.links}, we know that the oriented arc from $O_a$ to $X_a$ lies in $U_\alpha$, the $\alpha$-handlebody.

Let $A$ be the component of $\Sigma\setminus \bm\alpha$ containing that arc, so that $B=\tau(A)$ is the component of $\Sigma\setminus \bm\beta$ containing the oriented arc from $X_a'$ to $O_a'$. We will perform the construction from \cite[Section~4]{zemke2019link} in a symmetric way. We first pick a point $p\in A\setminus (\bm\alpha\cup\bm\beta\cup\bfO\cup\bfX\cup \fix(\tau))$ and a new curve $\alpha_s$ in $A$ through $p$ which cuts $A$ into two pieces that contain $X_a$ and $O_a$, respectively. Applying $\tau$, we get $p'=\tau(p)\in B$ and $\beta_s=\tau(\alpha_s)$ passing through $p'$.

Let $\cH_0=(D^2,\alpha_0,\beta'_0,O,X)$ be the piece of Heegaard diagram (cf. Figure~\ref{fig:quasi stab}) that appears in \cite[Section~4.1]{zemke2019link}. On $D^2$, $\alpha_0$ is a properly embedded arc that cuts it into two pieces and separates the $O$ and $X$ base points while $\beta'_0$ is a closed curve which intersects $\alpha_0$ twice and bounds a disk containing both $X$ and $O$. Let $\overline{\cH_0}$ be the diagram $(-D^2, \beta_0, \alpha_0', O', X')$, which is obtained from $\cH_0$ by reversing the orientation on $D^2$, interchanging the roles of $\alpha$ and $\beta$-curves while keeping the positions of $X$ and $O$ unaffected. Now, we form the real Heegaard diagram $\cH_0^R$ by taking the disjoint union $\cH_0\cup \overline{\cH_0}$ and define $R$ to be the natural orientation-reversing involution sending $\cH_0$ to $\overline{\cH_0}$.

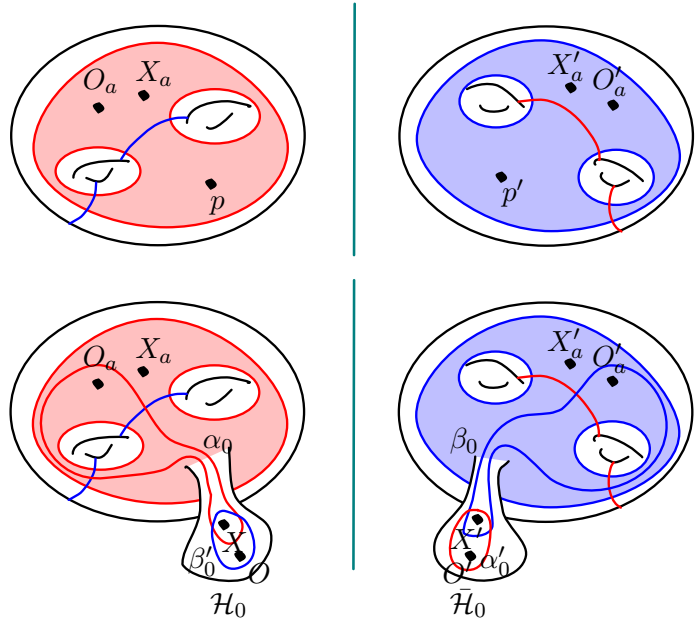
\begin{figure}
    \begin{overpic}[width=0.6\textwidth]{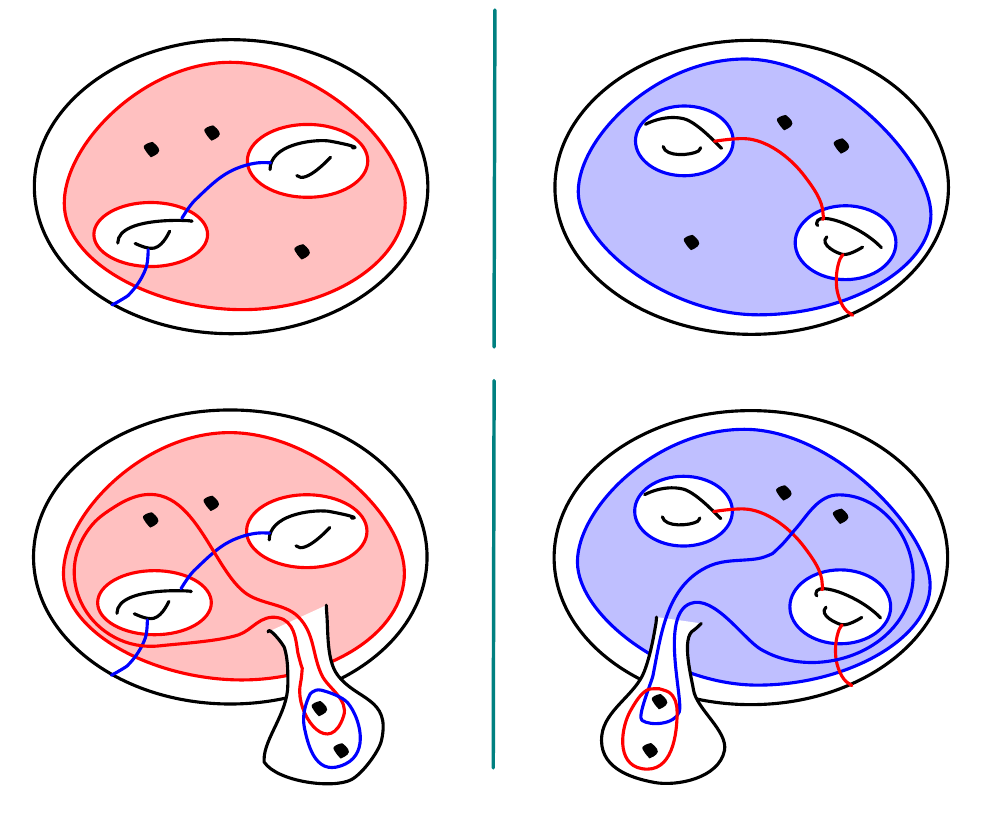}
    	\put(30,0) {$\cH_0$}
    	\put(62,0) {$\bar{\cH}_0$}
    	
    	\put(29,22) {$\alpha_0$}
    	\put(62,22) {$\beta_0$}
    	\put(66,6) {$\alpha'_0$}
    	\put(27,6) {$\beta'_0$}
    	
    	\put(30,54) {$p$}
    	\put(69,55) {$p'$}
    	
    	\put(13,70) {$O_a$}
    	\put(20,71) {$X_a$}
    	
    	\put(75,72) {$X_a'$}
    	\put(81,70) {$O_a'$}
    	
    	\put(13,33) {$O_a$}
    	\put(20,34) {$X_a$}
    	
    	\put(75,35) {$X_a'$}
    	\put(81,33) {$O_a'$}
    	
    	\put(61,4) {$O'$}
    	\put(31.5,8.2) {$X$}
    	\put(62,8.5) {$X'$}
    	\put(35,4) {$O$}
    \end{overpic}
    \caption{A real Heegaard diagram for a type-$S$ quasi-stabilization: the top row shows part of the original real Heegaard diagram and the bottom draws the result of performing a pair of special connected sums at $\{p,p'\}$.} 
    \label{fig:quasi stab}
\end{figure}

Finally, we paste $\cH_0^R$ to $\cH$ equivariantly at $(p,p')$ in a way that $\alpha_0$ is lined up with $\alpha_s$, while $\beta_0'$ is lined up with $\beta_s$. See Figure~\ref{fig:quasi stab} for an illustration. In a word, we are performing the two \emph{special connected sums} defined in~\cite{MOHFandintegersurgeryonlinks} equivariantly between $\cH_0^R$ and $\cH$. The new diagram will be denoted by $\cH^+$ and the two pairs of new curves will inherit labels $(\alpha_0,\beta_0)$ and $(\beta_0',\alpha_0')$. $\cH^+$ is a real Heegaard diagram for the multi-based generalized strongly invertible link $(Y,\tau,\LL^+)$.

There are two pairs of new intersection points in $(\alpha_0\cap \beta_0')\cup (\beta_0\cap \alpha_0')$, distinguished by their gradings $M_{\bfO}$ and  $M_{\bfX}$. We follow the notation in \cite{zemke2019link} and write $\theta^{\bfO}$ ($\xi^{\bfO}$) for the pair with higher (lower) $M_{\bfO}$. $\theta^{\bfX}$ and $\xi^{\bfX}$ are characterized similarly. It is easy to see that \[\theta^{\bfX}=\xi^{\bfO}, \quad \theta^{\bfO}=\xi^{\bfX}.\]

With the diagrams and notations prepared, we now define the type-$S$ quasi-stabilization maps by formulas \[S_{O,X}^+(\xv)=\xv\otimes \theta^{\bfO};\]
\[S_{O,X}^-(\xv\otimes \theta^{\bfO})=0 \text{ and } S_{O,X}^-(\xv\otimes \xi^{\bfO})=\xv \]
on the generators and extend them over $\cR$ equivariantly. 

Similarly, type-$T$ quasi-stabilization maps are defined by \[T_{O,X}^+(\xv)=\xv\otimes \theta^{\bfX};\]
 \[T_{O,X}^-(\xv\otimes \theta^{\bfX})=0 \text{ and } T_{O,X}^-(\xv\otimes \xi^{\bfX})=\xv \]
on the generators and extend it over $\cR$ equivariantly. 

When the new base point $X$ appears after $O$ along the orientation of $L$, the segment in $L\setminus (\bfO\cup \bfX)$ containing them lies in $U_{\beta}$; along it, the base points would read $X_a,O,X,O_a$. In this case, we form a diagram $\cH^+_{\beta}$ using a variation $\cH_{0,\beta}^R$ of $\cH_{0}^R$ with the roles of $\alpha$ and $\beta$-curves interchanged. Then the same formulas lead to maps $S_{X,O}^+$, $S_{X,O}^-$, $T_{X,O}^+$ and $T_{X,O}^-$.

\begin{prop}\label{prop:differential on H+ using a sufficiently stretched almost cplx str}
Let $\cH$ and $\cH^+$ be real Heegaard diagrams considered above representing $(Y,\tau,\LL)$ and $(Y,\tau,\LL^+)$, respectively. Then as modules (over $\F[U_{\bfO},V_{\bfX}]$), we have an isomorphism
\[\fullCFLR^-(\cH^+,\s^R)\cong \fullCFLR^-(\cH,\s^R)\otimes_{\F} \left \langle \theta^{\bfO}, \xi^{\bfO} \right \rangle\otimes_{\F} \F[U_O,V_X].\]

When the symmetric almost complex structure is sufficiently stretched, we have an identification of differentials 
\[\partial_{\cH^+}=
\begin{bmatrix}\partial_{\cH}& U_{O}+U_{O_a} \\
V_{X}+V_{X_a} &\partial_{\cH}\\
\end{bmatrix}.\] 
A similar result holds for $\cH_{\beta}^+$.
\end{prop}

\begin{proof}
This can be proved as ~\cite[Proposition~5.3]{Zemkequasistabandbasepointmoving} by  neck-stretching along the pair of curves along which we perform the special connected sums in the real Heegaard diagram $\cH^+$ or $\cH^+_{\beta}$.
\end{proof}

\begin{cor}\label{cor:T,S stab are chain maps under suitable coloring}
Let $\LL$ and $\LL^+$ be set up as above. Fix a coloring $\sigma$ on $\LL$ and consider an extension $\sigma'$ of $\sigma$, then \begin{enumerate}
    \item assuming $X_a\ne X_a'$, $S^+_{O,X}$ and $S^-_{O,X}$ are chain maps if and only if $\sigma'(V_X)=\sigma'(V_{X_a})$.
    \item assuming $O_a\ne O_a'$, $T^+_{O,X}$ and $T^-_{O,X}$ are chain maps if and only if $\sigma'(U_{O})=\sigma'(U_{O_a})$.
\end{enumerate}
Similar statements hold for $S^{\pm}_{X,O}$ and $T^{\pm}_{X,O}$.
\end{cor}

\begin{thm}\label{thm:invariance and naturality of paired quasi-stab maps}
Assume we have new base points $(O,O')$ and $(X,X')$ on $\LL=(L,\bfO,\bfX)$ as in the first paragraph of this subsection so that we  can consider the quasi-stabilized link $\LL^+$. Fix a coloring $\sigma$ on $\LL$ and extend it to $\sigma'$ on $\LL^+$ satisfying $\sigma'(V_X)=\sigma(V_{X_a})$. Then $S_{O,X}^{\pm}$ are well-defined $\Z^{\vert \bfO^f\vert+\frac{1}{2}\vert\bfO^p\vert} \oplus \Z^{\vert \bfX^f\vert+\frac{1}{2}\vert\bfX^p\vert}$-filtered chain maps on the colored real link Floer complex, which are relative grading-preserving when the grading structure exists. More precisely, the maps $S_{O,X}^{\pm}$ defined using different strongly $\s^R$-admissible diagrams commute with change of diagram maps, thus leading to well-defined morphisms between transitive systems $\fullCFLR^-(\LL^{\sigma},\s^R)$ and $\fullCFLR^-((\LL^+)^{\sigma'},\s^R)$. 

Similar results hold for $S_{X,O}^{\pm}$, $T_{O,X}^{\pm}$ and $T_{X,O}^{\pm}$. 
\end{thm}
\begin{proof}
Using Proposition~\ref{prop:differential on H+ using a sufficiently stretched almost cplx str} and Corollary~\ref{cor:T,S stab are chain maps under suitable coloring}, this can be proved as \cite[Theorem~A]{Zemkequasistabandbasepointmoving}.
\end{proof}

In the classical link Floer theory, quasi-stabilizations are actually one of the basic pieces of the decorated cobordism. We expect this to be true in the real case. We draw several examples in Figure~\ref{fig:examples for dividing set of quasi-stab} when the quasi-stabilization happens on a strongly invertible component. 

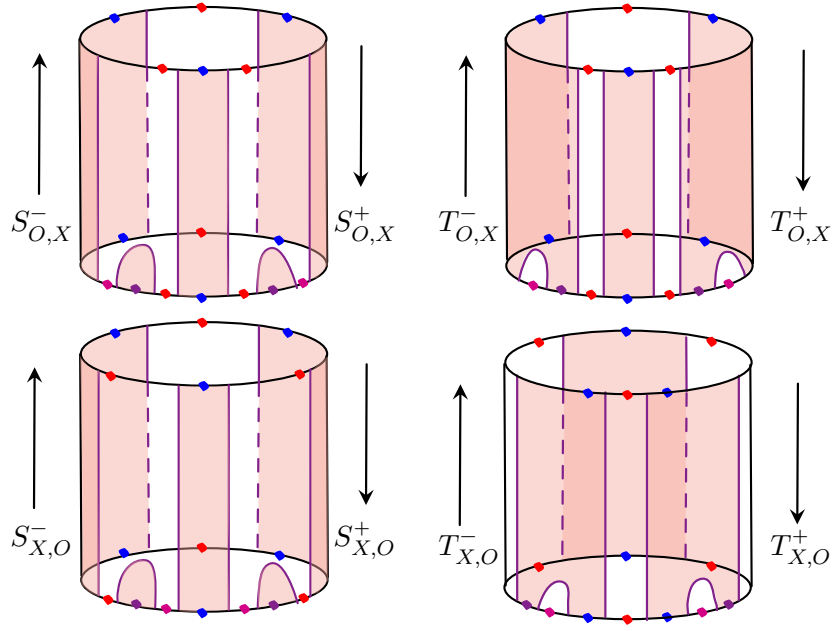
\begin{figure}
	\centering
	\begin{overpic}[width=0.7\textwidth]{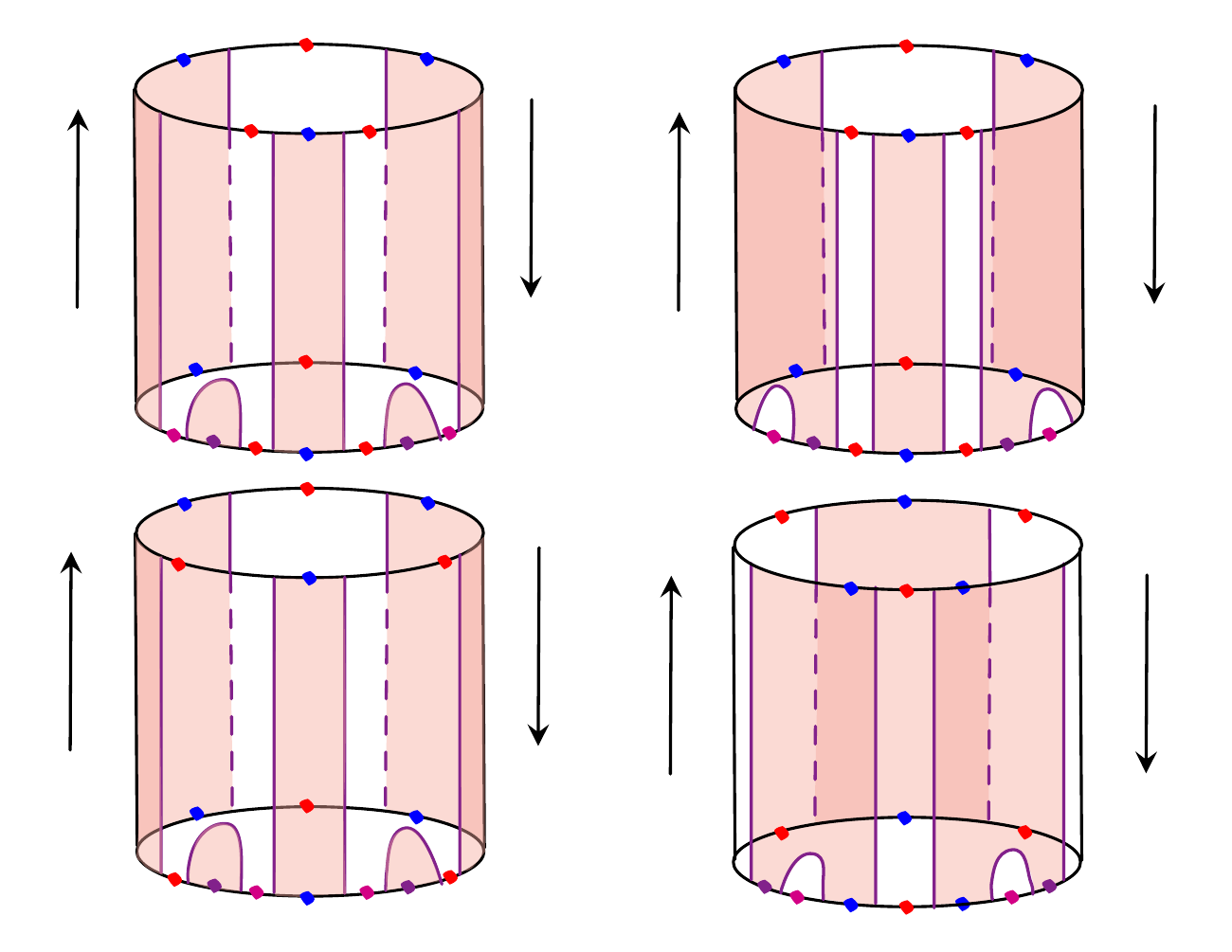}
		
		\put(3,48) {$S_{O,X}^-$}
		\put(40,48) {$S_{O,X}^+$}
		
		\put(52,48) {$T_{O,X}^-$}
		\put(90,48) {$T_{O,X}^+$}
		
		\put(3,12) {$S_{X,O}^-$}
		\put(40,12) {$S_{X,O}^+$}
		
		\put(52,11) {$T_{X,O}^-$}
		\put(90,11) {$T_{X,O}^+$}

	\end{overpic}
	\caption{These are the dividing set interpretations of type-$S$ and $T$ quasi-stabilizations. Here, old and new $X$-base points are marked in red and pink, respectively, while old and new $O$-base points are marked in blue and purple.}
	\label{fig:examples for dividing set of quasi-stab}
\end{figure}

\subsubsection{Type-changing quasi-stabilizations}\label{subsub:Type-changing quasi-stabilizations}

Let $\LL=(L,\bfO,\bfX)\subset (Y,\tau)$, $\s^R \in \rspinc(Y,\tau)$ and let $\cH$ be a real Heegaard diagram representing $\LL$ be as in the previous subsections. We now further assume that $L\cap \fix(\tau)\ne \emptyset$, i.e., $L$ has some strongly invertible components. Let $K\subset L$ be a strongly invertible component. Without loss of generality, we assume there is a fixed $X$-base point $X_0$ on $K$, so that $K\in SI_{XX}\cup SI_{OX}$. Now we replace $X_0$ with a pair of $X$-base points and a fixed $O$-base point. The real Heegaard diagram is modified correspondingly as shown in Figure~\ref{fig:OX-OO stabilization}, according to how $\LL$ embeds in $(Y,\tau)$, which provides us with a real Heegaard diagram $\cH^+$ representing another multi-based strongly invertible link with the same underlying link. We denote this new based link by $\LL^+=(L,\bfO',\bfX')$. It is easy to see that $\LL^+$ depends on $\LL$ (as a link embedding in $(Y,\tau)$), $K\subset L$ (together with modification of embedding, so that it is locally described by one of the diagrams in Figure~\ref{fig:OX-OO stabilization}) and the fixed base point $X$ on $K$, but does not depend on the diagram $\cH$ chosen. Note that when getting from $\LL$ to $\LL^+$, $K$ moves from $SI_{XX}$ to $SI_{OX}$ or from $SI_{OX}$ to $SI_{OO}$.


The curvature difference seen in Section~\ref{sub:Killing the curvature} serves as an essential obstruction for defining quasi-stabilization maps on the full real link Floer complex. Nevertheless, after choosing suitable coloring, we obtain some interesting type-changing quasi-stabilization maps. These are new to the real theory and have no counterparts from \cite{Zemkequasistabandbasepointmoving}.

We will define maps \[Q_{X\mapsto O}^+: \fullCFLR^-(\cH,\s^R)\to \fullCFLR^-(\cH^+,\s^R) \quad Q_{O\mapsto X}^-: \fullCFLR^-(\cH^+,\s^R)\to \fullCFLR^-(\cH,\s^R),\]
which lead to chain maps between properly colored $\fullCFLR^{-}(\LL^{\sigma},\s^R)$ and $\fullCFLR^{-}((\LL^+)^{\sigma'},\s^R)$ as well as chain maps between  $(\fullCFLR^{-}(\LL^{s},\s^R),D)$ and $(\fullCFLR^-((\LL^+)^{s},\s^R),D)$.

Given any real Heegaard diagram $\cH$ for $\LL$, assume it is strongly $\s^R$-admissible as usual. Then we can form $\cH^+$ as a connected sum of $\cH$ with the standard piece $\cH_U$ or $\cH_U'$ shown in Figure~\ref{fig:std_pieces_for_type-changing_quasi_stab}.


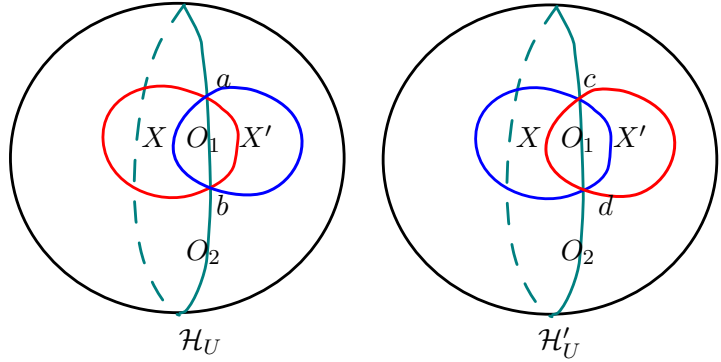
\begin{figure}
	\centering
	\begin{overpic}[width=0.6\textwidth]{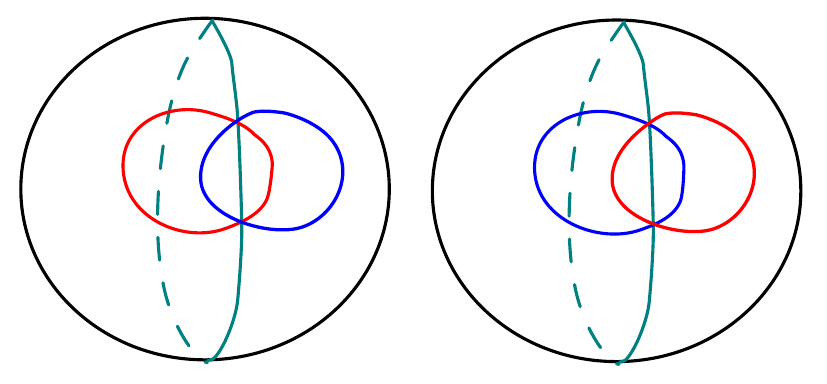}
		\put(25,-2) {$\cH_U$}
		\put(73,-2) {$\cH'_U$}
		
		\put(79,33) {$c$}
		\put(81,16) {$d$}
		
		\put(30,33) {$a$}
		\put(30,16) {$b$}
		
		\put(76,25) {$O_1$}
		\put(76,10) {$O_2$}
		
		\put(26,25) {$O_1$}
		\put(26,10) {$O_2$}
		
		\put(20,25) {$X$}
		\put(33,25) {$X'$}
		
		\put(70,25) {$X$}
		\put(83,25) {$X'$}
		
	\end{overpic}
	\caption{Two standard $OO$-real Heegaard diagrams for the strongly invertible unknot in $(S^3,\tau_{\std})$.}
	\label{fig:std_pieces_for_type-changing_quasi_stab}
\end{figure}

\begin{prop}\label{prop: differential oftype changing quasi-stab}
Let $\cH$ and $\cH^+$ be given as above. If we expand $\partial_{\cH}$ as \[\partial_{\cH}=\sum_{k=0}^{\infty}  v^{k}\partial_{v}^k,\] where $v$ is the variable associated to the chosen fixed $X$-base point $X_0$, then in terms of the basis $\{a\xv, b\xv\}$, \[\partial_{\cH^+}= \begin{bmatrix}
	\sum_{k=0}^{\infty} V^{k} \partial_{v}^{2k}  & \sum_{k=0}^{\infty} V^{k} \partial_{v}^{2k+1} \\
	\sum_{k=0}^{\infty} V^{k+1}  \partial_{v}^{2k+1} & \sum_{k=0}^{\infty} V^{k} \partial_{v}^{2k} \\
\end{bmatrix}  + \begin{bmatrix}
0& u_0+u_1 \\
(u_0+u_1)V & 0\\
\end{bmatrix}\] if $\LL$ is a minimally pointed knot and $u_0$ is the variable associated to the unique old $O$-base point; \[\partial_{\cH^+}= \begin{bmatrix}
\sum_{k=0}^{\infty} \partial_{v}^{2k} V^{k} & \sum_{k=0}^{\infty} \partial_{v}^{2k+1} V^{k} \\
\sum_{k=0}^{\infty} \partial_{v}^{2k+1} V^{k+1} & \sum_{k=0}^{\infty} \partial_{v}^{2k} V^{k}\\
\end{bmatrix}  + \begin{bmatrix}
0& u_1 \\
u_1 V & 0\\
\end{bmatrix},\] otherwise. Here $V$ is the new variable associated to the new pair $(X,X')$ in $\LL^+$.
\end{prop}
\begin{proof}
We will focus on the case $\cH^+=\cH\#\cH_U$, as the other cases can be argued in exactly the same way. We count moduli spaces of real holomorphic curves using a symmetric almost complex structure on $\cH^+$ that is sufficiently stretched along the neck between $\cH$ and $\cH_U$. The cylindrical reformulation of real Heegaard Floer homology will be used, see~\cite{GM_real_naturality} or \cite{BGX}. In \cite[Section~7]{BGX}, we performed the counts for classes with multiplicity zero at the regions containing base points on real sutured diagrams. Now we sketch an argument that works for link diagrams and take base points into consideration. See~\cite{GMX_Functoriality} for the argument at the $3$-manifold level, from which the author develops the following proof. 

A class of pseudo-holomorphic disks in $\cH^+=\cH\#\cH_U$ can be written as $\phi \# \phi_0$, where $\phi$ is a disk in $\cH$ and $\phi_0$ is a disk in $\cH_U$. Consider $\phi \# \phi_0 \in \pi_2^R(z_1\xv ,z_2 \yv )$, where $\xv$, $\yv$ are real generators on $\cH$ and $z_i\in \{a,b\}$. It can be computed easily that 
\[\mu_R(\phi_0)= n_{O_1}(\phi_0)+n_{O_2}(\phi_0);\] 
\begin{align*}
	\mu_R(\phi\# \phi_0)&= \mu_R(\phi) +\mu_R(\phi_0)-n_{O_2}(\phi_0)\\
	 &= \mu_R(\phi) +n_{O_1}(\phi_0).
\end{align*}
Hence, for index 1 curves, either $ \mu_R(\phi) = 0$ and $n_{O_1}(\phi_0) = 1$ or $\mu_R(\phi) = 1$ and $n_{O_1}(\phi_0) = 0$.

Now, choose a sequence of neck lengths $T_i$ limiting to infinity, and consider a sequence $u_i$ of $J(T_i)$-holomorphic curves in the class $\phi\# \phi_0$. We can extract a subsequence converging to a collection broken curves $\cU$ into $\Sigma \times [0,1] \times \R$ representing $\phi$, broken curves $\cU_0$ into $\Sigma_U \times [0,1] \times \R$ representing $\phi_0$, and a collection $\cV$ mapping into the connected sum region $(S^1 \times \R) \times [0,1] \times \R$. Then as \cite[Proposition~7.2]{BGX}, we can rule out ghost components from in $\cU$ and $\cU_0$ and prove that \begin{itemize}
	\item $\cU$ consists of a single embedded curve;
	\item $\cV$ consists of constant curves;
	\item $\cU_0$ consists of an index 1 curve in the former case and an index $n_{O_2}(\phi_0)$ curve in the latter case.
\end{itemize} 

In the case that  $\mu_R(\phi) = 0$, $\phi$ is constant and $\phi_0$ is one of the two bigons which covers $O_1$. This contributes $u_1$ and $u_1V$ to the second term in the differential. 

When $\mu_R(\phi) = 1$, there are several cases to consider. First, assume that $\phi$ is not the class of a cylindrical $\{1\}$-boundary degeneration.  In this case, if $n_{O_2}(\phi)=2k$ is even, then $z_1=z_2$, and we splice in symmetric disk bubbles (the total count of possible curves obtained this way is 1). In this case, $n_X(\phi_0)=k$ and $n_{X_0}(\phi)=2k$ leading to the two diagonal entries $\sum_{k=0}^{\infty} V^{k} \partial_{v}^{2k}$ in the first matrix.
If $n_{O_2}(\phi) = 2k+1$ is odd, then $\{z_1,z_2\}=\{a,b\}$. In this case, $u^{-1}(X_0)$ consists of $2k$ points which are swapped by the involution and a single point fixed by the involution. For those pairs of points, we splice in $2k$ disks with boundary an entire $\alpha$ or $\beta$-curve, while to the fixed point, we splice in a disk with an $\alpha$-$\beta$-boundary. Since $\phi$ and $\phi_0$ match at the connected sum point, $n_{X_0}(\phi)= 2k+1$, but $n_X(\phi_0)$ needs a more careful discussion:\begin{itemize}
\item When $z_1=a$ and $z_2=b$, $\phi_0$ takes the form $B_{O_2}+ B_{X}+B_{X'}+ k (2B_{O_2}+ B_{X}+B_{X'})$;
\item When $z_1=b$ and $z_2=a$, $\phi_0$ takes the form $B_{O_2}+ k (2B_{O_2}+ B_{X}+B_{X'})$. 
\end{itemize} 
Here, we denote the small bigon containing $P$ in $\cH_U$ by $B_P$ for $P=O_1,X,X'$ or $O_2$. Summing over all $k\ge 0$, we get the off-diagonal entries in the first matrix.

Finally, we consider the case that $\phi$ is a $\{1\}$-boundary degeneration with real index 1. Such curves exist only when $\vert\bfX\vert= \vert\bfO\vert=1$ and these curves have domain $[\Sigma]$. As noted in \cite{GM_real_naturality}, bubbles of this kind appear in pairs: indeed, if $u: (S, \partial S) \to (\Sigma \times \R \times \R, C \times \{1/2\} \times \R)$ is a $\{1\}$-boundary degeneration, so is $\overline{\tau}\circ u \circ \iota$, where $\iota$ is the source involution. For trivial orbit disk bubbles, the doubling map
\begin{align*}
	\Pi_2(\xv,M^R) \to \pi_2^R(\xv)
\end{align*}
is two-to-one to its image, in which $\Pi_2(\xv,M^R)$ denotes the set of homotopy class of boundary degenerations with boundary in $M^R$ based at $\xv$ regarded as a point in $M^R\cap \T_{\alpha}$. However, the two symmetrized classes can be distinguished with an extra piece of data. The punctures on a holomorphic curve
\begin{align*}
	u: (S, \partial S, \sigma) \to (\Sigma \times [0,1] \times \R, \bm\alpha \cup \bm\beta \times \partial [0,1] \times \R, \tau\times \text{reflection} \times \id),
\end{align*}
are labeled as positive or negative. This induces an orientation on those components $c_0$ of $\mathrm{Fix}(\iota)$ which have boundary approaching one positive puncture and one negative puncture. 

Boundary degenerations naturally induce an orientation on the fixed point set of the source. Therefore, if $p \in C_0 \subset C$ for some connected component of $C$, and $\mathfrak o_p$ is an orientation of the connected component of $S^\sigma$ containing $p$, define $\pi_2^R(\xv, \mathfrak o_p)$ to be the class of $\{1\}$-boundary degenerations at $p$ equipped with a choice of orientation compatible with $\mathfrak o_p$. The doubling map
\begin{align*}
	\Pi_2(\xv,M^R) \to \pi_2^R(\xv, \frak o_p)
\end{align*}
is now a bijection. 

Note that $B_{O_1}$ and $B_{O_1}+B_{X}+B_{X'}$ induce different orientation on the fixed set. 
Thus, for each class in $\pi_2^R(\xv, \frak o_p)$, there is only one class in $\pi_2^R(a, b) \cup \pi_2^R(b, a) $ which induces an orientation of the fixed point set which is compatible. These give rise to the terms $u_0$ and $u_0V$ in the minimally pointed case. This concludes the analysis of all moduli spaces and the proof of the proposition.
\end{proof}

Fix any coloring $\sigma$ on $\LL$. We extend it to one on $\LL^+$ by $V \mapsto \sigma(v)^2$ and leave $u_1$ unspecified for now. Performing a basis change from $\{a\xv, b\xv\}$ to $\{a\xv+v\cdot b\xv, b\xv\}$, the differential can be rewritten as follows. 
\begin{cor}\label{cor:type-changing quasi-stab differential in new basis}
When the coloring is extended to $\LL^+$ in a way that $\sigma'(V)= \sigma(v)^2$, in terms of the new basis $\{a\xv+v\cdot b\xv, b\xv\}$, we have 
\[\partial_{\cH^+}=\begin{bmatrix}
	\partial_{\cH}+\sigma'(u_0+u_1)\sigma(v) & * \\
0 & \partial_{\cH}+\sigma'(u_0+u_1)\sigma(v) \\
\end{bmatrix} \]if $\LL$ is a minimally pointed knot and $u_0$ is the variable associated to the unique old $O$-base point; \[\partial_{\cH^+}=\begin{bmatrix}
\partial_{\cH}+\sigma'(u_1)\sigma(v) & * \\
0 & \partial_{\cH}+\sigma'(u_1)\sigma(v) \\
\end{bmatrix} \] otherwise.
Here, $*= \sum_{k=0}^{\infty} \sigma (v)^{2k} \partial^{2k+1} + \sigma'(u_0+u_1)\sigma(v)$ in the former case and  $*= \sum_{k=0}^{\infty} \sigma (v)^{2k} \partial^{2k+1} + \sigma'(u_1)\sigma(v)$ in the latter.
\end{cor} 

As for the type-$T$ and $S$ stabilization maps, these \emph{type-$Q$} quasi-stabilization maps have simple expressions on these special Heegaard diagrams. On generators, they are defined by 
\[Q^+_{X\mapsto O} (\xv)=a\xv+v \cdot b\xv \]
\[Q^-_{O\mapsto X} (a\xv)=v\cdot \xv ,\quad Q^-_{O\mapsto X} (b\xv)=\xv\] or \[Q^+_{X\mapsto O} (\xv)=d\xv+ v \cdot c\xv \]
\[Q^-_{O\mapsto X} (d\xv)= v\cdot \xv ,\quad Q^-_{O\mapsto X} (c\xv)=\xv\] and extended linearly to the whole colored complex.

Similarly, we can introduce \emph{type-$P$} quasi-stabilization maps 
\[P_{O\mapsto X}^+: \fullCFLR^-(\cH,\s^R)\to \fullCFLR^-(\cH^+,\s^R) \quad P_{O\mapsto X}^-: \fullCFLR^-(\cH^+,\s^R)\to \fullCFLR^-(\cH,\s^R);\]

These maps are defined when there is a fixed base point $O_0\in \bfO^f$ on $\LL$ and we form $\LL^+$ by replacing $O_0$ with a fixed $X$-base point and a pair of $O$-base points, which is exactly one of the pictures in Figure~\ref{fig:OX-OO stabilization} with the roles of $O$ and $X$ interchanged.

\begin{thm}\label{thm:invariance and naturality of type changing quasi-stab maps}
Assume $\LL$ and $\LL^+$ are set up as above and we have already extended $\sigma$ to 
$\sigma'$ by $V\mapsto \sigma(v)^2$. Then $Q_{O\mapsto X}^{-}$ and $Q_{X\mapsto O}^+$ are well-defined chain maps between transitive systems $\fullCFLR^{-}(\LL^\sigma,\s^R)$ and $\fullCFLR^{-}((\LL^+)^{\sigma'},\s^R)$ when one of the following holds: \begin{enumerate}
	\item $\LL$ is a knot with a single pair of base points and $\sigma'(u_1)=\sigma(u_0)$;
	\item there is more than one pair of base points on $\LL$ and $\sigma'(u_1)=0$;
	\item $\sigma(V_{\bfX})=\sigma'(V_{\bfX'})=0$.
\end{enumerate} 
When the coloring is chosen properly, these are filtered by powers of variables in the codomain of $\sigma$ and are relative grading-preserving. Similar results hold for $P_{O\mapsto X}^{+}$ and $P_{X\mapsto O}^-$.
\end{thm}

Using Corollary~\ref{cor:type-changing quasi-stab differential in new basis}, one checks directly that when $\sigma$ and $\sigma'$ satisfy one of the conditions above, then $Q_{X\mapsto O}^+$ and $Q_{O\mapsto X}^-$ are chain maps (i.e., $\partial_{\cH^+}\circ Q_{X\mapsto O}^+= Q_{X\mapsto O}^+ \circ \partial_{\cH}$ and $\partial_{\cH}\circ Q_{O\mapsto X}^-= Q_{O\mapsto X}^- \circ \partial_{\cH^+}$). To see that they are well-defined on the level of transitive system, we need to show that the definition is independent of the Heegaard diagram $\cH$ chosen. To be precise, we will consider a type-$Q$ quasi-stabilization that can be represented by a fixed point connected sum with $\cH_U$, since the other cases are essentially the same. What we need to show is that the squares \begin{equation}\label{eq:invariance of type-changing quasi-stab}
	\begin{tikzcd}
		\fullCFLR^-(\cH,\s^R)^\sigma  \arrow[r,"^{Q_{X\mapsto O}^+}"] \arrow[d,"\Phi_{\cH\to \cH'}"] & \fullCFLR^-(\cH\# \cH_U,\s^R)^{\sigma'} \arrow[d,"\Phi_{\cH\#\cH_U\to \cH'\#\cH_U}"] \\
		\fullCFLR^-(\cH',\s^R)^\sigma    \arrow[r,"^{Q_{X\mapsto O}^+}"] & \fullCFLR^-(\cH'\# \cH_U,\s^R)^{\sigma'} \\	 
	\end{tikzcd}\quad \quad\begin{tikzcd}
		\fullCFLR^-(\cH\# \cH_U,\s^R)^{\sigma'} \arrow[r,"^{Q_{O\mapsto X}^-}"] \arrow[d,"\Phi_{\cH\#\cH_U\to \cH'\#\cH_U}"]  & \fullCFLR^-(\cH,\s^R)^\sigma \arrow[d,"\Phi_{\cH\to \cH'}"]  \\
		\fullCFLR^-(\cH'\# \cH_U,\s^R)^{\sigma'} \arrow[r,"^{Q_{O\mapsto X}^-}"]  & \fullCFLR^-(\cH',\s^R)^\sigma \\	 
	\end{tikzcd}
\end{equation}
commute for any pair of diagrams $\cH$ and $\cH'$ representing $\LL$.

Proposition~\ref{prop:existence of real HD for multi-based strongly invertible knots and real H moves} allows us to reduce to the case that $\cH$ and $\cH'$ differ by a single real Heegaard move as the independence of the choice of a symmetric almost complex structure is tautological. The commutativity of squares when $\Phi_{\cH\to \cH'}$ is a $\{1\}$ or $\Z/2$-stabilization is also obvious. Thus, we will focus on the case that $\Phi_{\cH\to \cH'}$ is a real rectangle counting map.

Let $(\Sigma,\bm \alpha',\bm \alpha,\bm \beta,\bm \beta', \bfO,\bfX,\tau)$ be the real Heegaard quadruple which underlies the real rectangle counting map $\Phi_{\cH\to \cH'}=\Psi_{\bm \beta\to \bm\beta'}^{\bm \alpha\to \bm \alpha'}$. Then $\Phi_{\cH\#\cH_U\to \cH'\#\cH_U}$ is provided by counting rectangles in $(\Sigma ,\bm \alpha',\bm \alpha,\bm \beta,\bm \beta',\bfO,\bfX,\tau)$ $\# (\Sigma_U, \alpha_U',\alpha_U,\beta_U,\beta_U', \{O_1,O_2\}, \{X,X'\},\tau_U)$ where the latter diagram is shown in Figure~\ref{fig:triple_H_u} and the connected sum identifies $O_2$ with the distinguished $X$-base point $X_0$.

\begin{figure}
	\centering
	\begin{overpic}[width=0.9\textwidth]{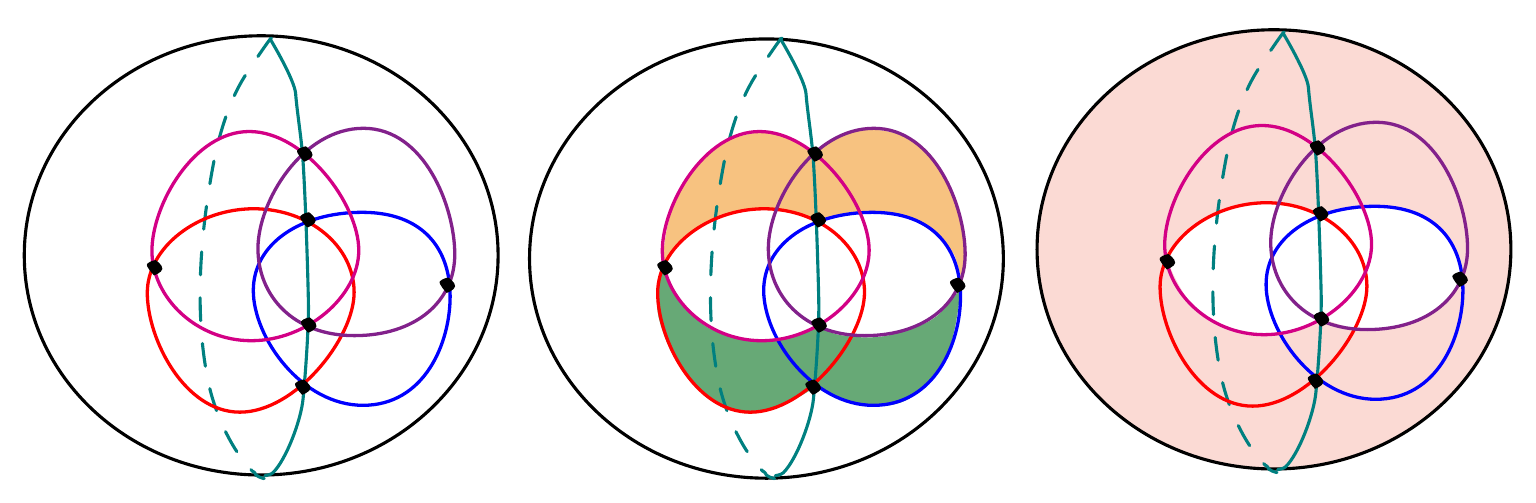}
		\put(27,16) {$\Theta_{\beta_U,\beta'_U}$}
		\put(5,17) {$\Theta_{\alpha'_U,\alpha_U}$}
		
		\put(28.5,21) {$\beta_U$}
		\put(8,21) {$\alpha_U$}
		
		\put(28,7) {$\beta'_U$}
		\put(8,7) {$\alpha'_U$}
		
		\put(19,14.5) {$O_1$}
		\put(18,4) {$O_2$}
		
		\put(13.5,14) {$X$}
		\put(24,14) {$X'$}
		
		\put(20,24) {$a$}
		\put(20,19) {$a'$}
		
		\put(20.5,12) {$b$}
		\put(21.5,7.5) {$b'$}
		
		\put(52,14.5) {$O_1$}
		\put(51,4) {$O_2$}
	
		\put(46.5,14) {$X$}
		\put(57,14) {$X'$}
		
		\put(85,14.5) {$O_1$}
		\put(84,4) {$O_2$}
		
		\put(79.5,14) {$X$}
		\put(90,14) {$X'$}
		
	\end{overpic}
	\caption{The quadruple diagram associated to $\cH_U$ and real rectangles on it.}
	\label{fig:triple_H_u}
\end{figure}
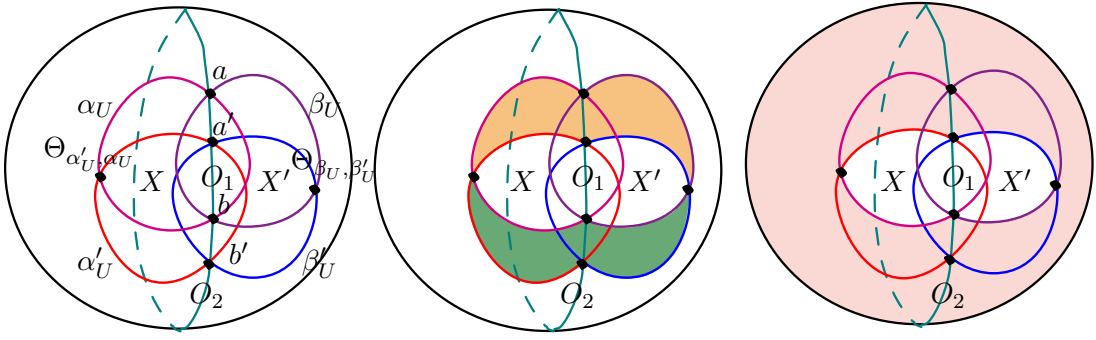

\begin{prop}\label{prop:quasi-stab quadruples}
Let $\cH$, $\cH'$ and $\cH_U$ be set up as above. Write \[\Phi_{\cH\to \cH'}=\sum_{k=0}^{\infty}  v^k \Phi_{v}^{k},\] where $v$ is the variable associated to the chosen fixed $X$-base point $X_0$. Then in terms of the basis $\{a\xv, b\xv\}$ and $\{a'\xv, b'\xv\}$, \[\Phi_{\cH\#\cH_U\to \cH'\#\cH_U}= \begin{bmatrix}
		\sum_{k=0}^{\infty} V^{k} \Phi_{v}^{2k}  & \sum_{k=0}^{\infty} V^{k}\Phi_{v}^{2k+1}  \\
		\sum_{k=0}^{\infty} V^{k+1}  \Phi_{v}^{2k+1} & \sum_{k=0}^{\infty}  V^{k} \Phi_{v}^{2k}\\
	\end{bmatrix}.\]
Here $V$ is the new variable associated to the new pair $(X,X')$ in $\LL^+$.
\end{prop}

Assuming Proposition~\ref{prop:quasi-stab quadruples}, we now finish the proof of Theorem~\ref{thm:invariance and naturality of type changing quasi-stab maps}. Let $\xv$ be any generator on the diagram $\cH$, then \[\Phi_{\cH\#\cH_U\to \cH'\#\cH_U}^s (Q_{X\mapsto O}^+(\xv))= \Phi_{\cH\#\cH_U\to \cH'\#\cH_U}^s(a\xv+ v\cdot b\xv)= a'\Phi_{\cH\to \cH'}^s(\xv) + b'\Phi_{\cH\to \cH'}^s(\xv);\]
\[Q_{X\mapsto O}^+(\Phi_{\cH\to \cH'}^s(\xv))= a'\Phi_{\cH\to \cH'}^s(\xv) + v\cdot b'\Phi_{\cH\to \cH'}^s(\xv),\]
so $\Phi_{\cH\#\cH_U\to \cH'\#\cH_U}^s \circ Q_{X\mapsto O}^+= Q_{X\mapsto O}^+\circ \Phi_{\cH\to \cH'}^s$. Similarly, one can compute 
\[\Phi_{\cH\to \cH'}^s (Q_{O\mapsto X}^-(a\xv))=Q_{O\mapsto X}^-(\Phi_{\cH\#\cH_U\to \cH'\#\cH_U}^s(a\xv))= 0;\]

\[\Phi_{\cH\to \cH'}^s (Q_{O\mapsto X}^-(b\xv))= 
Q_{O\mapsto X}^-(\Phi_{\cH\#\cH_U\to \cH'\#\cH_U}^s(b\xv))= \Phi_{\cH\to \cH'}^s(\xv),\] which imply $\Phi_{\cH\to \cH'}^s\circ Q_{O\mapsto X}^-= Q_{O\mapsto X}^-\circ \Phi_{\cH\#\cH_U\to \cH'\#\cH_U}^s.$ Thus, diagrams in \eqref{eq:invariance of type-changing quasi-stab} commute. The graded and filtered version can be easily verified.

To prove Proposition~\ref{prop:quasi-stab quadruples}, we need the following calculation of real Maslov index. The argument is modeled on the proof of an analogous statement in $3$-manifold case, which will appear in~\cite{GMX_Functoriality}. Here, we sketch an argument for the link case. The author thanks Gary Guth for his help in improving this argument.

\begin{lemma}\label{lem:local-real-maslov}
Consider the real quadruple shown in the left of Figure~\ref{fig:triple_H_u}. If $\psi_0 \in \pi_2^R(\Theta_{\alpha'_U,\alpha_U}, z_1, \Theta_{\beta_U,\beta'_U}, z_2)$ is a class of symmetric rectangles, with $z_1\in \{a,b\}$ and  $z_2\in \{a',b'\}$, 
	\begin{align*}
		\mu_R(\psi_0) = n_{O_1}(\psi_0) + n_{O_2}(\psi_0).
	\end{align*}
\end{lemma}
\begin{proof}
The formula can be proved following \cite[Lemma 5.8]{Zemke_2026}. This can be checked directly on the class whose domain is a single component of $S^2 \smallsetminus (\alpha_U \cup \alpha'_U \cup \beta_U \cup \beta'_U)$ with multiplicity one. Both sides of the formula are unchanged by splicing in any of the bigons of the diagram, and any two classes with the same corners are related by such splicings. 
\end{proof}

As a direct corollary, we know that there are only two index 0 rectangles with $n_{O_1}(\psi_0) = n_{O_2}(\psi_0) = 0$, shown in orange and green in the middle frame of Figure~\ref{fig:triple_H_u}. Besides, the rectangles with $n_{O_1}(\psi_0) = 0$ and $n_{O_2}(\psi_0) = 1$ have index 1 for either choice of corners.

\begin{proof}[Proof of Proposition~\ref{prop:quasi-stab quadruples}]
	
Rectangles taking $a\xv$ to $a'\yv$ or $b\xv$ to $b'\yv$ can be analyzed in exactly the way as in Proposition~\ref{prop: differential oftype changing quasi-stab}. These rectangles contribute 
the two diagonal terms $\sum_{k=0}^{\infty} V^{k} \Phi_{v}^{2k}$.

It remains to consider classes of rectangles in 
$$\pi_2^R(\Theta_{\alpha'_U,\alpha_U} \bm \Theta_{\bm\alpha',\bm\alpha}, z_1 \xv, \Theta_{\beta_U,\beta'_U} \bm{\Theta}_{\bm\beta,\bm\beta'}, z_2\yv)$$ for $(z_1,z_2)=(a,b')$ or $(b,a')$. As in Proposition~\ref{prop: differential oftype changing quasi-stab}, a class of rectangle in the connected sum diagram can be written as $\psi \# \psi_0$, where $\psi$ is a rectangle in $\cH$ and $\psi_0$ is a rectangle in $\cH_U$. Such classes $\psi \# \psi_0$ have $n_{X_0}(\psi) = n_{O_2}(\psi_0) = 2k+1  \equiv 1 \mod 2$. Then we consider a sequence neck length $T_i$ limiting to infinity and a sequence of $J(T_i)$-holomorphic rectangles in class $\psi\# \psi_0$. After extracting a converging subsequence, we get a limit broken curves which is the union of three collections $\cU$, $\cV$ and $\cU_0$, belonging to $\Sigma\times \Diamond$, $S^1\times \R\times \Diamond$ and $\Sigma_U\times \Diamond$, respectively. Again, we can rule out ghost components from $\cU$ and $\cU_0$ and prove that $\cV$ must be constant.

Then, a gluing argument identifies 
\begin{align*}
	\cM^{J(T)}_R(\psi\#\psi_0) = \cM^{J}_R(\psi)\times_{\rho^{X_0}} \cM^{J_0}_R(\psi_0).
\end{align*}
Here, $\rho^{X_0}: \cM_R(\psi) \ra \Sym^{2k+1}(\Diamond)^R$ as $(\pi_{\Diamond}\circ u)((\pi_{\Sigma}\circ u)^{-1}(X_0))$. 
Note that \begin{align*}
	0= \mu_R(\psi\# \psi_0)&= \mu_R(\psi) +\mu_R(\psi_0)-n_{O_2}(\phi_0)\\
	&= \mu_R(\psi) +n_{O_1}(\psi_0).
\end{align*}

Then the non-negativity of indices and multiplicities of holomophic curves implies $\mu_R(\psi)=n_{O_1}(\psi_0)=0$. Thus, the space $\cM^{J}_R(\psi)$ is zero dimensional. Now the remaining problem is to analyze the union of moduli spaces $\bigcup_{u\in \cM(\psi)} \cM^{J_0}_R(\psi_0, \rho^{X_0}(u))$. As in \cite[Section~7]{BGX}, it suffice consider $\cM^{J_0}_R(\psi_0, \bm d)$ for a generic $\bm d \in \Sym^{2k+1}(\Diamond)^R$. We claim that
\begin{align*}
	\sum_{
		\substack{
			\psi_0 \in \pi_2^R(\Theta_{\alpha'_U,\alpha_U},a, \Theta_{\beta_U,\beta'_U},b')\\
			n_{O_2}(\psi_0) = 2k+1 \\
			n_{O_1}(\psi_0) = 0
		}
	}
	\#\cM^{J_0}_R(\psi_0, \bm d) \equiv 1 \mod 2.
\end{align*} A similar result holds with the roles of $a,b$ interchanged.

A standard cobordism argument shows that this count is independent of the choice of $\bm d$ (see~\cite[Lemma~7.3]{BGX}). Choose a path $\bm d_t$ in $\Sym^{2k+1}(\Diamond)^R$ avoiding the diagonal so that 
\begin{enumerate}
	\item $\bm d_1 = \bm d$;
	\item as $t \to \infty$, points in $\bm d_t$ limit to $\infty$ in the $\alpha$-$\beta$ cylindrical end of $\Diamond$;
	\item identifying the $\alpha$-$\beta$ end of $\Diamond$ with $[0,1] \times (-\infty, 0]$, the points of $\bm d_t$ are spaced out by distance at least $t$ as $t$ gets large.
	\item The $[0,1]$ component of each $\bm d_t$ approaches a fixed $s_0$ in $(0,1)$.
\end{enumerate}
If $\bm \cD = \{\bm d_t\}_{t \in [1,\infty)}$, consider the 1-dimensional moduli space
\begin{align*}
	\cM_{R}(\bm \cD) = \coprod_{
		\substack{
			\psi_0 \in \pi_2^R(\Theta_{\alpha'_U,\alpha_U},a, \Theta_{\beta_U,\beta'_U},b')\\
			n_{O_2}(\psi_0) = 2k+1 \\
			n_{O_1}(\psi_0) = 0
		}
	} \cM^{J_0}_R(\psi_0, \bm \cD).
\end{align*}
The ends of this moduli space correspond to 
\begin{enumerate}
	\item Curves at $t = 1$;
	\item Degenerations at finite $t$;
	\item Limiting curves as $t \to \infty$.
\end{enumerate}
A standard argument shows that ends of the second kind appear in canceling pairs. Hence, it follows that the choice of divisor depends only on its path component in $\Sym^k(\Diamond)^R \smallsetminus \mathrm{Diag}(\Diamond)$. It will be apparent that the count is independent of the connected component as well.

By the Maslov index calculations of Lemma~\ref{lem:local-real-maslov}, it must be that the limiting curves consist of $|\bm d|-1$ flow lines on $(S^2, \alpha_U, \beta_U)$ representing (symmetric pairs) of index 2 classes in $\pi_2^R(a,a)$ or $\pi_2^R(b,b)$ which match a fixed point $d$ in $\{1/2\} \times \R$. The remaining configuration is an index 1 rectangle with $n_{O_2}(\psi_0) = 1$, which breaks into an index 1 bigon $\phi_0$ with $n_{O_2}(\phi_0) = 1$ and an index 0 rectangle $\psi_0^0$ with $n_{O_2}(\psi_0^0) = 0$. 

Write $\cR_0$, $\cB_0$ and $\cS_0$ for the sets of classes appearing in the limit:
\begin{align*}
	\cR_0 &= \{\, \psi_0^0 \in \pi_2^R(\Theta_{\alpha'_U,\alpha_U}, b, \Theta_{\beta_U,\beta'_U}, b') \; : \; n_{O_1}(\psi_0^0) = n_{O_2}(\psi_0^0) = 0 \,\},\\
	\cB_0 &= \{\, \phi_0 \in \pi_2^R(a,b) \; :  \; n_{O_1}(\phi_0) = 0, \; n_{O_2}(\phi_0) = 1 \,\},\\
	\cS_0 &= \{\, \phi \in \pi_2^R(a,a) \; : n_{O_1}(\phi) = 0, \; n_{O_2}(\phi) = 1 \,\}.
\end{align*}
Appealing to the gluing results of \cite[Proposition A.1]{Lipshitz2005ACR}, it follows that the ends at $t = \infty$ correspond to the (zero-dimensional) space
\begin{align}\label{eqn:stab-rect-count}
	\left( \coprod_{\psi_0^0 \in \cR_0} \cM_R(\psi_0^0) \right) \times
	\left( \coprod_{\phi_0 \in \cB_0} \cM_R(\phi_0, d) \right) \times
	\left( \coprod_{\phi \in \cS_0} \cM_R(\phi, d) \right)^{|\bm d| - 1}.
\end{align}
The third term of equation \eqref{eqn:stab-rect-count} also contributes a single point to the mod 2 count according to \cite[Lemma 6.4]{HolomorphicdiskslinkinvariantsandthemultivariableAlexanderpolynomial}. Since $\bm \cD$ limits to $a \in \alpha_U \cap \beta_U$, the index 1 rectangle configuration must break into one of the two index 1 bigons in the complement of $O_1$ with boundary in $\alpha_U \cup \beta_U$, from which it follows that the second term in equation \eqref{eqn:stab-rect-count} is also $1 \text{ mod } 2$. Finally, there is a single index 0 rectangle with $n_{O_2}(\psi_0^0) = n_{O_1}(\psi_0^0) = 0$ which has the correct asymptotics. Hence, the total count is odd and path component independent as we claimed.

Finally, as in Proposition~\ref{prop: differential oftype changing quasi-stab}, we need to keep track of $n_{X}(\psi_0)$. Assume that $n_{O_2}(\psi_0)=2k+1$, as above. When $\psi_0$ has vertices at $a$ and $b'$, $n_X(\psi_0)=k+1$, while when $\psi_0$ has vertices at $b$ and $a'$, $n_X(\psi_0)=k$. See the right frame of Figure~\ref{fig:triple_H_u} for an example with $n_{O_2}(\psi_0)=n_X(\psi_0)=1$. Thus, we have the two summations off-diagonal as shown. This concludes the proof of Proposition~\ref{prop:quasi-stab quadruples}.
\end{proof}

\subsubsection{Quasi-stabilization maps on $\fullCFLR^-(\LL^s,\s^R)$}\label{subsub:Quasi-stabilization maps on Dcomplex}

In this subsection, we focus on how the various quasi-stabilization maps act on the chain complex $(\fullCFLR^-(\LL^s,\s^R), D)$. By a coloring of $\LL^s$, we mean a ring homomorphism with domain $\cR^-_s(\LL)$. We first deal with type-$S$ and type-$T$ quasi-stabilizations. 

\begin{prop}\label{prop:differential of paired quasi-stab on D complex}
Let $\cH^+$ and $\cH$ be set up as in  Proposition~\ref{prop:differential on H+ using a sufficiently stretched almost cplx str}. Fix some coloring $\sigma$ on $\LL^s$ and let $\sigma'$ be an extension to $\LL^{+,s}$. When the symmetric almost complex structure is sufficiently stretched, we have an identification of differentials 
\[D_{\cH^+}^{\sigma'}=
\begin{bmatrix}
	D_{\cH}^{\sigma}& \sigma'(U_{O}+U_{O_a}) \\
	0 &D_{\cH}^{\sigma}\\
\end{bmatrix},\] when we are performing a type-$S$ quasi-stabilization and $\sigma'(V_X)=\sigma (V_{X_a})$, \[D_{\cH^+}^{\sigma'}=
\begin{bmatrix}
	D_{\cH}^{\sigma}& 0 \\
	\sigma'(V_{X}+V_{X_a}) &D_{\cH}^{\sigma}\\
\end{bmatrix},\]
when we are performing a type-$T$ quasi-stabilization and $\sigma'(U_O)=\sigma (U_{O_a})$. 

\end{prop}
\begin{proof}
This follows from Proposition~\ref{prop:differential on H+ using a sufficiently stretched almost cplx str} and the definition of $D$ immediately.
\end{proof}

As a corollary of this and Theorem~\ref{thm:invariance and naturality of paired quasi-stab maps}, we have the following:
\begin{cor}
Under the same coloring assumption as in Proposition~\ref{prop:differential of paired quasi-stab on D complex}, $S^{\pm}_{O,X}$ are well-defined morphisms between transitive systems  $(\fullCFLR^-(\LL^s,\s^R), D)^{\sigma}$ and $(\fullCFLR^-(\LL^{+,s},\s^R), D)^{\sigma'}$. Similar results hold for $S_{X,O}^{\pm}$, $T_{O,X}^{\pm}$ and $T_{X,O}^{\pm}$. 
\end{cor}

Next, we move to type-changing quasi-stabilizations $P$ and $Q$. For concreteness, we will again focus on type-$Q$ maps associated to local changes shown on the left of Figure~\ref{fig:OX-OO stabilization}. Note that now $\cR^-_s(\LL)$ is naturally a subring of $\cR^-_s(\LL^+)$ by identifying the variables associated to the old $X_0$ and the new pair $(X,X')$. Then it follows from Corollary~\ref{cor:type-changing quasi-stab differential in new basis}, in the new basis $\{a\xv+v\cdot b\xv,b\xv\}$, we have \[D_{\cH^+}= \begin{bmatrix}
 D_{\cH} & * \\
 0 & D_{\cH}\\
\end{bmatrix}. \] Here $*$ is the same as in Corollary~\ref{cor:type-changing quasi-stab differential in new basis}.

Combining this calculation with Proposition~\ref{prop:quasi-stab quadruples}, we see that 
\begin{thm}\label{thm: invariance and naturality of P,Q on D complex}
$Q_{X\mapsto O}^+$ and $Q_{O\mapsto X}^-$ are well-defined morphisms between $(\fullCFLR^-(\LL^s,\s^R), D)$ and $(\fullCFLR^-(\LL^{+,s},\s^R), D)$ after including $\cR^-_s(\LL)$ as a subring of $\cR^-_s(\LL^+)$. A similar result holds for $P_{O\mapsto X}^+$ and $P_{X\mapsto O}^-$.  
\end{thm}
\begin{proof}
This follows from the definition of type-$P$, $Q$ quasi-stabilizations and the upper-triangular shape that $D_{\cH^+}$ takes.
\end{proof}

In \cite{zemke2019link}, the dividing set on link cobordism satisfies that two base points lie in same connected component of the complement of the dividing set if and only if they share the same color. In most cases from Theorem~\ref{thm:invariance and naturality of type changing quasi-stab maps}, we set some variables to zero to make $Q^{\pm}$ chain maps. These are not valid colorings in Zemke's setting (he colored base points, but we generalize it to ring homomorphisms) and the author does not find any reasonable dividing set for them. However, when we move from $\partial$ to $D$, $Q^{\pm}$ and $P^{\pm}$ are naturally chain maps, so they admit interesting dividing set interpretations. Examples are provided in Figure~\ref{fig:dividing sets for P and Q}.

\begin{figure}
	\centering
	\begin{overpic}[width=0.7\textwidth]{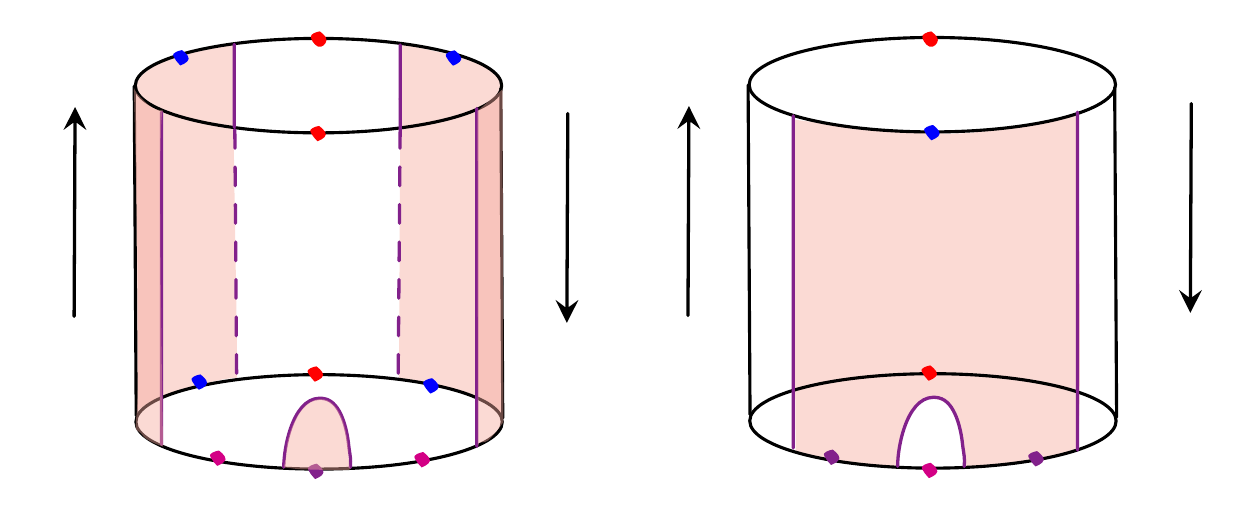}

		\put(0,12) {$Q_{O\mapsto X}^-$}
		\put(41,12) {$Q_{O\mapsto X}^+$}
		\put(91,12) {$P_{O\mapsto X}^+$}
		\put(51,12) {$P_{X\mapsto O}^-$}

	\end{overpic}
	\caption{These are the dividing set interpretations of type-$Q$, $P$ quasi-stabilizations on $\fullCFLR^-(\LL^s,\s^R)$. The coloring convention is the same as in Figure~\ref{fig:examples for dividing set of quasi-stab}.}
	\label{fig:dividing sets for P and Q}
\end{figure}
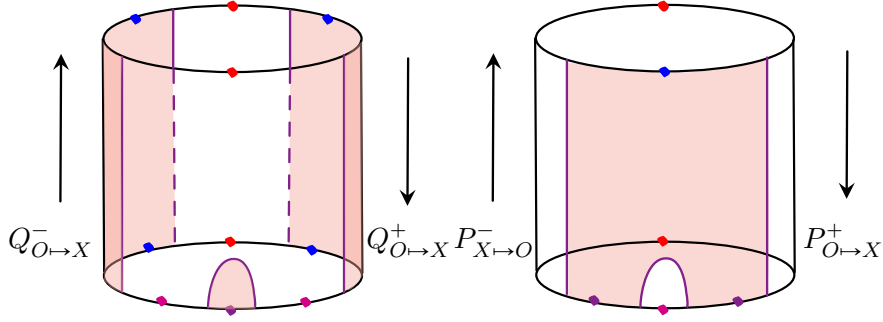

In Section~\ref{subsub:Compositions}, we will see that after coloring the chain complex suitably, $\Psi^p$ and $\Phi^p$ can be expressed as compositions of type-$T$ and $S$ quasi-stabilization maps, while $\Psi^f$ and $\Phi^f$ can be expressed as compositions of type-$P$ and $Q$ quasi-stabilization maps, leading to a geometric interpretation of the base point actions in terms of decorated cobordism.

\subsection{Relationship between base point actions and quasi-stabilization maps}\label{sub:Relationship between base point actions and quasi-stabilization maps}

In this subsection, we analyze the relationship between base point actions and quasi-stabilization maps and further investigate their properties. Some of these results follow \cite[Section~4.3]{zemke2019link} closely, but additional new phenomena appear, as in Lemma~\ref{lem:commutation relation of base point actions and the differential} and~\ref{lem:commutation between base point actions}. As we remarked in Section~\ref{subsub:Base point actions on Dcomplex}, the homotopy relation in $(\fullCFLR^-(\LL^s,\s^R), \partial^{\sigma_s})$ is the same as in $(\fullCFLR^-(\LL^s,\s^R), D)$. In the following, $\simeq$ means homotopy for both $\partial$ and $D$ unless otherwise stated.

\subsubsection{Compositions}\label{subsub:Compositions}
\begin{lem}\label{lem:T simequ Psi S etc}
Let $\LL=(L,\bfO,\bfX)$ be a multi-based strongly invertible link and new base points $(O,O')$, $(X,X')$ be set up as in Section~\ref{subsub:Quasi-stabilizations without type change}, with adjacent pairs $(O_a,O_a')$, $(X_a,X_a')$. Let $\sigma$ be a coloring on $\LL$ and $\sigma'$ be its extension to $\LL^+$. If both \begin{enumerate}
    \item $O_a\ne O_a'$ and $\sigma'(U_{O})=\sigma'(U_{O_a})$;
    \item $X_a\ne X_a'$ and $\sigma'(V_{X})=\sigma'(V_{X_a})$ 
\end{enumerate}
hold, then \[T_{O,X}^+\simeq \Psi^p_{(X,X')}S_{O,X}^+ \quad \text{and}\quad T_{O,X}^-\simeq S_{O,X}^-\Psi^p_{(X,X')}.\]
Similarly, \[S_{O,X}^+\simeq \Phi^p_{(O,O')}T_{O,X}^+ \quad \text{and}\quad S_{O,X}^-\simeq T_{O,X}^-\Phi^p_{(O,O')}.\]
\end{lem}
\begin{proof}
In view of Proposition~\ref{prop:differential on H+ using a sufficiently stretched almost cplx str}, this can be proved in the same way as~\cite[Lemma~4.10]{zemke2019link}.    
\end{proof}

Following \cite[Lemma~9.3]{Zemkequasistabandbasepointmoving}, we can interpret $\Psi^p$ and $\Phi^p$ using the quasi-stabilization maps.

\begin{lem}\label{lem:present p-base point action as composition of quasistab maps}
Consider $\LL$, $(O,O')$, $(X,X')$ and $(O_a,O_a')$, $(X_a,X_a')$ as above. If the base points read $O_a,X,O,X_a$ along the orientation on $\LL$ and (1) from Lemma~\ref{lem:T simequ Psi S etc} is true, then \[\Phi^p_{(O,O')}\simeq S_{O,X}^+S_{O,X}^- \quad \Phi^p_{(O,O')}\simeq S_{X_a,O}^+S_{X_a,O}^-.\]  

If the base points read $X_a,O,X,O_a$ along the orientation on $\LL$ and (2) from Lemma~\ref{lem:T simequ Psi S etc} is true, then \[\Psi^p_{(X,X')}\simeq T_{O,X}^+T_{O,X}^- \quad \Psi^p_{(X,X')}\simeq T_{X,O_a}^+T_{X,O_a}^-.\]  
\end{lem}
\begin{proof}
This follows from the expression of $\partial_{\cH^+}$ in terms of $\partial_{\cH}$ in Proposition~\ref{prop:differential on H+ using a sufficiently stretched almost cplx str} and the definition of quasi-stabilization maps in Section~\ref{subsub:Quasi-stabilizations without type change}.
\end{proof}

For the fixed point actions, we have a similar characterization. 

\begin{lem}\label{lem:present f-base point action as composition of quasistab maps}
Let $\LL$ be a multi-based generalized strongly invertible link and $K$ be a component of $L$ in $SI_{OX}\cup SI_{XX}$. Fix a base point $X^*$ on $K\cap \bfX^f$ and consider the link $\LL^+$ given by a type-$Q$ quasi-stabilization at $X^*$, on which we have a new fixed base point $O^*$. Then we have  \[\Phi^{f,b}_{O^*}\simeq Q_{X^*\mapsto O^*}^+Q_{O^*\mapsto X^*}^- \] as endomorphisms whenever the coloring makes the type-$Q$ quasi-stabilization well-defined.

Similarly, if $K$ belongs to $SI_{OO}\cup SI_{OX}$ and $O^*$ is a fixed base point on $K$, we can form $\LL^+$ by performing the local change in Figure~\ref{fig:OX-OO stabilization} with the roles of $O$ and $X$ interchanged. Call the new $X$-base point $X^*$. Then we have 
\[\Phi^{f,b}_{X^*}\simeq P_{O^*\mapsto X^*}^+P_{X^*\mapsto O^*}^- \] as endormorphisms whenever the coloring makes the type-$P$ quasi-stabilization well-defined.
\end{lem}
\begin{proof}
We will prove the first statement. The second follows from this by interchanging the roles of $O$ and $X$. For a diagrammatic generator $\xv$, it is true that \[\frac{d}{du}(\partial \xv)= \frac{d\partial}{du} (\xv).\] Then from Proposition~\ref{prop: differential oftype changing quasi-stab}, we calculate that \[Q_{X^*\mapsto O^*}^+Q_{O^*\mapsto X^*}^-(a\xv)= v\cdot Q_{X^*\mapsto O^*}^+(\xv)= v (a\xv+ v\cdot b\xv)= \Phi^{f,b}_{O^*}(a\xv);\]
\[Q_{X^*\mapsto O^*}^+Q_{O^*\mapsto X^*}^-(b\xv)= \cdot Q_{X^*\mapsto O^*}^+(\xv)=  (a\xv+ v\cdot b\xv)= \Phi^{f,b}_{O^*}(b\xv).\] These conclude the proof. 

\end{proof}

In Figure~\ref{fig:base point action as composition of quasi-stab}, we provide the decorated cobordism interpretations of Lemma~\ref{lem:present p-base point action as composition of quasistab maps} and \ref{lem:present f-base point action as composition of quasistab maps}. Those for $\Phi^p$ or $\Psi^p$ always record the coloring faithfully, while the one for $\Phi^f$ ($\Psi^f$) is only faithful when we work with $(\fullCFLR^-(\LL^s),D)$.
\begin{figure}
    \centering
    \begin{overpic}[width=0.7\textwidth]{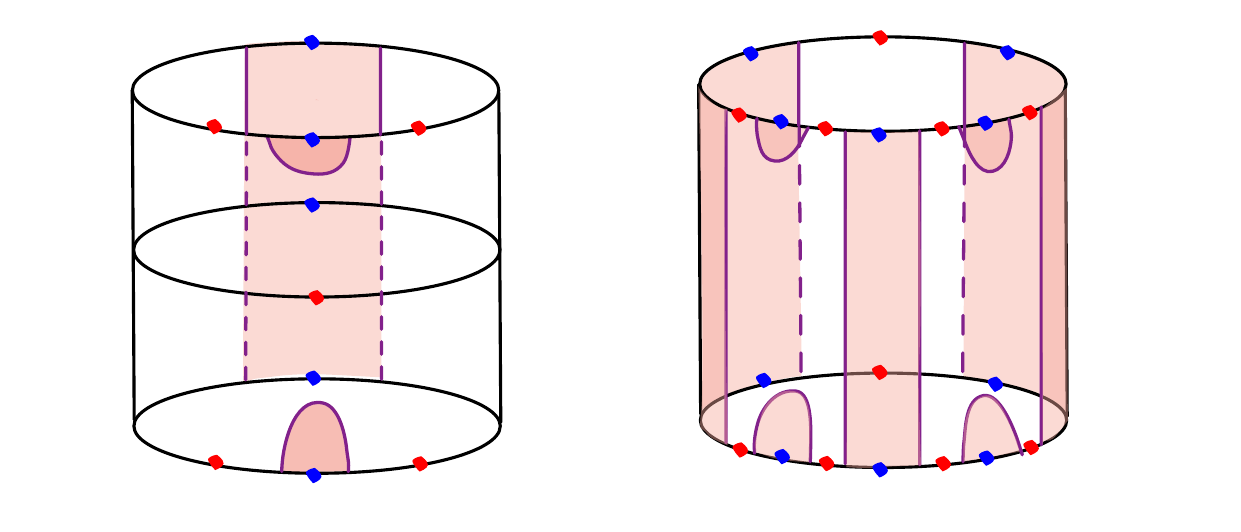}
    	\put(24,-1) {$\Phi_{O^*}^{f,b}$}
    	\put(70,-1) {$\Phi_{(O,O')}^p$}
    \end{overpic}
    \caption{The decorated cobordism interpretation of Lemma~\ref{lem:present p-base point action as composition of quasistab maps} and \ref{lem:present f-base point action as composition of quasistab maps}.}
    \label{fig:base point action as composition of quasi-stab}
\end{figure}

\begin{lem}\label{lem:composition of type T,S ,P, Qquasi-stab maps}
Type-$T$ and $S$ quasi-stabilization maps satisfy 
\[S_{O,X}^-S_{O,X}^+\simeq0,\quad T_{O,X}^-T_{O,X}^+\simeq0,\quad T_{O,X}^-S_{O,X}^+\simeq\id ,\quad S_{O,X}^-T_{O,X}^+\simeq\id \] and \[ S_{O,X}^+T_{O,X}^- +T_{O,X}^+S_{O,X}^-\simeq\id,\] when they are well-defined chain maps.

\noindent For type-$Q$ and $P$, we have 
\[Q_{O\mapsto X}^-Q_{X\mapsto O}^+\simeq 0, \quad P_{X\mapsto O}^-P_{O\mapsto X}^+\simeq 0\] when they are well-defined chain maps.
\end{lem}
\begin{proof}
These follow directly from the definitions in Section~\ref{sub:quasi-stabilization}.
\end{proof}


\subsubsection{Commutative relations}\label{subsub:Commutative relations}

Before analyzing commutativity between base point actions and quasi-stabilization maps, we first generalize the notion of base point actions. 

Suppose that $(A,A')$ is a pair of disjointly embedded subarcs of $L$ with endpoints in $\bfO$ so that $A'=\tau(A)$, with distinct endpoints. We define \[\Psi_{A,A'}:=\sum_{(X,X')\in (A\cup A')\cap \bfX^p}  \Psi_{(X,X')}^p,\] which is a chain map if and only if the two pairs of endpoints have the same color. (See Lemma~\ref{lem:commutation relation of base point actions and the differential}.)

Next, consider an arc $A\subset L$ satisfying $\tau(A)=A$ and having both endpoints in $\bfO$. $A$ necessarily lies on some strongly invertible component $K\subset L$. We distinguish three cases. \begin{enumerate}
    \item $A\subset K$ is a proper subarc and $\{O^f\}=A\cap \fix(\tau)\subset \bfO$. Let $(O,O')$ be the endpoints of $A$. We define \[\Psi_{A}:=\sum_{(X,X')\in A\cap \bfX^p} \Psi_{(X,X')}^p,\] which is a chain map if and only if $\sigma(u_{O^f})^2=\sigma (U_O)$. 
    \item $A\subset K$ is a proper subarc and $\{X^f\}=A\cap \fix(\tau)\subset \bfX$. Let $(O,O')$ be endpoints of $A$. We define \[\Psi_{A}:=\sum_{(X,X')\in A\cap \bfX^p} \Psi_{(X,X')}^p +\Psi_{X^f}^f,\] which is a chain map if and only if $\sigma(U_{O_a})=\sigma (U_O)$ in which $(O_a,O_a')$ is the pair of $O$-base points adjacent to $X^f$. 
    \item $A=K$, so that it has no endpoints. We define \[\Psi_{A}:=\sum_{(X,X')\in K\cap \bfX^p} \Psi_{(X,X')}^p +\sum_{X\in K\cap \bfX^f}\Psi_{X}^f.\] 
   This is a chain map if and only if 
   \begin{itemize}
    \item $\sigma(u_{O^f})^2=\sigma(U_{O_a})$, where $O^f$ and $X^f$ are the two fixed base points on $K$ and $(O_a,O_a')$ are the pair of $O$-base points adjacent to $X^f$, when $K\in SI_{OX}$; 
    \item $\sigma(u_{O^f_1})^2=\sigma(u_{O_2^f})^2$, where $O^f_i$($i=1,2$) are the two fixed base points on $K$, when $K\in SI_{OO}$; 
    \item $\sigma(U_{O_1})=\sigma(U_{O_2})$, where $(O_i,O_i')$($i=1,2$) are the two pairs of $O$-base points nearest to the fixed set, when $K\in SI_{XX}$; 
   \end{itemize}
\end{enumerate}

\begin{lem}\label{lem:commutative relation between Psi_A and quasi-stabilization}
Setup $\LL$, $(O,O')$, $(X,X')$ and $(O_a,O_a')$, $(X_a,X_a')$ as in Section~\ref{sub:quasi-stabilization}. Let $(A,A')$ be a pair of subarcs of $\LL$ for which $\Psi_{A,A'}$ is well-defined. Color $\LL$ so that $\Psi_{A,A'}$ is a chain map. Then, we have that 
\[ \Psi_{A,A'}S^{\circ}_{O,X}+S^{\circ}_{O,X}\Psi_{A,A'}\simeq 0\quad \text{and}\quad \Psi_{A,A'}T^{\circ}_{O,X}+T^{\circ}_{O,X}\Psi_{A,A'}\simeq 0, \] for $\circ=+,-$. 

Let $A$ be a subarc of $\LL$ for which $\Psi_{A}$ is well-defined. Color $\LL$ so that $\Psi_{A}$ is a chain map. Then, we have that 
\[ \Psi_{A}T^{\circ}_{O,X}+T^{\circ}_{O,X}\Psi_{A}\simeq 0; \] 
\[\Psi_{A}S^{\circ}_{O,X}+S^{\circ}_{O,X}\Psi_{A}\simeq 0, \] for $\circ=+,-$. 
\end{lem}

\begin{proof}
This follows from the proof of \cite[Lemma~4.14]{zemke2019link}. 
\end{proof}

We can define $\Phi_{A}$ and $\Phi_{A,A'}$ using arcs with endpoints in $\bfX$ and base point actions from $\bfO$. An analogue of Lemma~\ref{lem:commutative relation between Psi_A and quasi-stabilization} holds for them.

\begin{lem}
Setup $\LL$, $\LL^+$ as in Section~\ref{subsub:Type-changing quasi-stabilizations}. Assume the coloring satisfies one of the conditions from Theorem~\ref{thm:invariance and naturality of type changing quasi-stab maps}. Then we have \[\Phi_{A,A'} Q_{X\mapsto O}^{+} + Q_{X\mapsto O}^{+} \Phi_{A,A'}\simeq 0 \quad \text{and} \quad \Phi_{A} Q_{X\mapsto O}^{+} + Q_{X\mapsto O}^{+}\Phi_{A}\simeq 0;\] 
\[\Phi_{A,A'} Q_{O\mapsto X}^{-} + Q_{O\mapsto X}^{-} \Phi_{A,A'}\simeq 0 \quad \text{and} \quad \Phi_{A} Q_{O\mapsto X}^{-}+ Q_{O\mapsto X}^{-}\Phi_{A}\simeq 0.\] 
Similar expressions are true between type-$P$ quasi-stabilization maps and $\Psi$-actions.
\end{lem}

\begin{proof}
Since we have assumed that the colorings satisfy  Theorem~\ref{thm:invariance and naturality of type changing quasi-stab maps}, $Q_{X\mapsto O}^+$ and $ Q_{O\mapsto X}^-$ are chain maps, i.e., \begin{equation}\label{eq:Q commutes with differential}
\partial_{\cH^+} Q_{X\mapsto O}^+ + Q_{X\mapsto O}^+ \partial_{\cH} =0 \quad \text{and}  \quad \partial_{\cH} Q_{O\mapsto X}^- + Q_{O\mapsto X}^- \partial_{\cH^+}=0
\end{equation} as maps between $\fullCFLR^-((\LL^+)^{\sigma'},\s^R)$ and $\fullCFLR^-(\LL^{\sigma},\s^R)$. Note that $\Phi_{A}$ and $\Phi_{A,A'}$ can be written as a formal differential of $\partial$ since each of $\Phi^p_{(O,O')}$ or $\Phi^f_{O}$ can. Then the result follows from differentiating Equations~\eqref{eq:Q commutes with differential}.  
\end{proof}

Note that whether we are working with $(\fullCFLR^-(\LL^s,\s^R), \partial^{\sigma_s})$ or $(\fullCFLR^-(\LL^s,\s^R), D)$, we need the coloring assumption from Theorem~\ref{thm:invariance and naturality of type changing quasi-stab maps}. This is because base point actions are derivatives of $\partial$, but not that of $D$.

We extract some interesting cases from Lemma~\ref{lem:commutative relation between Psi_A and quasi-stabilization}.

\begin{lem}\label{lem:special commu relation between Psi and quasi-stabilization}
Setup $\LL$, $(O,O')$, $(X,X')$ and $(O_a,O_a')$, $(X_a,X_a')$ as in Section~\ref{sub:quasi-stabilization}.\begin{enumerate}
    \item For any $O^f\in \bfO^f$ or $(O_i,O_i')\in \bfO^p$, we have \[S_{O,X}^{\circ}\Phi_{O^f}^f+\Phi_{O^f}^fS_{O,X}^{\circ}\simeq0 \quad \text{and}\quad S_{O,X}^{\circ}\Phi_{O_i}^p+\Phi_{O_i}^p S_{O,X}^{\circ}\simeq0. \]
    \item If $(X_i,X_i')\in \bfX^p$ does not lie on the pair of segments $([X_a,O_a],[X_a',O_a'])$, then \[S_{O,X}^{\circ}\Psi_{(X_i,X_i')}^p+\Psi_{(X_i,X_i')}^p S_{O,X}^{\circ}\simeq0. \]  For any $X\in \bfX^f$,
    \[S_{O,X}^{\circ}\Psi_{X}^f+\Psi_{X}^fS_{O,X}^{\circ}\simeq0.\]
    \item As we have assumed $X_a\ne X_a'$ when defining $S_{O,X}^{\pm}$, it is always true that \[S_{O,X}^{+}\Psi_{(X_a,X_a')}^p\simeq (\Psi_{(X,X')}^p +\Psi_{(X_a,X_a')}^p) S_{O,X}^{+} ;\] \[ S_{O,X}^{-} (\Psi_{(X,X')}^p +\Psi_{(X_a,X_a')}^p)\simeq \Psi_{(X_a,X_a')}^pS_{O,X}^{-}.\]
\end{enumerate}
\end{lem}

\begin{lem}
Consider $\LL$ and $\LL^+$ from Section~\ref{subsub:Quasi-stabilizations without type change} so that $(X_a,O,X,O_a)$ are four consecutive base points on $\LL^+$. Suppose that $O_a\ne O_a'$, $X_a\ne X_a'$ and we have chosen a coloring $\sigma'$ with $\sigma'(U_{O})=\sigma'(U_{O_a})$ and $\sigma'(V_{X})=\sigma'(V_{X_a})$. Then, \[S_{O,X}^{-}\Psi_{(X,X')}^p S_{O,X}^{+}\simeq \id \simeq S_{O,X}^{-}\Psi_{(X_a,X_a')}^p S_{O,X}^{+}\quad \text{and}\quad T_{O,X}^{-}\Phi_{(O,O')}^p T_{O,X}^{+}\simeq \id \simeq T_{O,X}^{-}\Phi_{(O_a,O_a')}^p T_{O,X}^{+}.\]
\end{lem}

\begin{proof}
We will prove the statement for type-$S$ quasi-stabilization maps, the case of $T$ follows similarly. We know from Lemma~\ref{lem:composition of type T,S ,P, Qquasi-stab maps} that $S^{-}_{O,X}T^{+}_{O,X}\simeq \id$ and from Lemma~\ref{lem:T simequ Psi S etc} that $T^{+}_{O,X}\simeq \Psi_{(X,X')}^p S^+_{O,X}$. These two combine to give $S_{O,X}^{-}\Psi_{(X,X')}^p S_{O,X}^{+}\simeq \id$. 

On the other hand, we know from Lemma~\ref{lem:special commu relation between Psi and quasi-stabilization} that \[S_{O,X}^{-}\Psi_{(X,X')}^p S_{O,X}^{+}\simeq S_{O,X}^{-}\Psi_{(X_a,X_a')}^p S_{O,X}^{+}+ S_{O,X}^{-} S_{O,X}^{+}\Psi_{(X_a,X_a')}^p.\] Then the result follows from $S_{O,X}^{-} S_{O,X}^{+}\simeq 0$ (cf. Lemma~\ref{lem:composition of type T,S ,P, Qquasi-stab maps}) which leads to \[S_{O,X}^{-}\Psi_{(X,X')}^p S_{O,X}^{+}\simeq  S_{O,X}^{-}\Psi_{(X_a,X_a')}^p S_{O,X}^{+}.\]
\end{proof}

\begin{prop}\label{prop:commutation relation between quasi-stab at different pairs of points}
Let $\LL$ be a multi-based generalized strongly invertible link. Consider four pairs of new base points $(X_i,X_i')$, $(O_i,O_i')$ ($i=1,2$), so that the eight short segments of the form $(X_i,O_i)$, $(O_i',X_i')$ follow the orientation of $\LL$ and contain no other base point. Assume their relative positions also satisfy the assumption in the beginning of~Section~\ref{sub:quasi-stabilization}. Then we have \[S_{O_1,X_1}^{\circ_1}S_{O_2,X_2}^{\circ_2}\simeq S_{O_2,X_2}^{\circ_2} S_{O_1,X_1}^{\circ_1},\]
for $\circ_{1},\circ_{2}\in \{+,-\}$.
\end{prop}
\begin{proof}
This follows from the argument in \cite[Section~8]{Zemkequasistabandbasepointmoving}.
\end{proof}
We have similar commutating relations between $S_{X_1,O_1}^{\circ_1}$ and $S_{X_2,O_2}^{\circ_2}$ when the order of the new base points changes as well as their type-$T$ counterparts.

\begin{prop}
Let $\LL=(L,\bfO,\bfX)$ be a multi-based generalized strongly invertible link. \begin{enumerate}
\item  For any distinct base points $X_1, X_2\in \bfX^f$, we have \[Q_{X_1\mapsto O_1}^+ Q_{X_2\mapsto O_2}^+ \simeq Q_{X_2\mapsto O_2}^+ Q_{X_1\mapsto O_1}^+\] whenever the links are colored so that all the maps are well-defined chain maps. Here, $O_1$ and $O_2$ denote the new fixed base points introduced during the quasi-stabilization.
\item Let $(X_1, O_1,X_1')$ be a triple of base points such that we can perform a quasi-stabilization $Q^-_{O_1\mapsto X_1}$ and $X_2 \in \bfX^f$ be a fixed $X$-base point. Then \[Q_{O_1\mapsto X_1}^- Q_{X_2\mapsto O_2}^+ \simeq Q_{X_2\mapsto O_2}^+ Q_{O_1\mapsto X_1}^-\] holds whenever the links are colored so that all the maps are well-defined chain maps.
  \item Let $(X_i, O_i,X_i')$ ($i=1,2$) be triples of base points such that we can perform quasi-stabilizations $Q^-_{O_i\mapsto X_i}$. Assume $(X_1,X_1')$ and $(X_2,X_2')$ are distinct pairs in $\bfX^p$, then \[Q_{O_1\mapsto X_1}^- Q_{O_2\mapsto X_2}^- \simeq Q_{O_2\mapsto X_2}^- Q_{O_1\mapsto X_1}^-\] holds whenever the links are colored so that all the maps are well-defined chain maps.
\end{enumerate}
\end{prop}
\begin{proof}
All these relations can be checked directly at the chain level using diagrams that bear local pictures like Figure~\ref{fig:OX-OO stabilization} for both quasi-stabilizations.
\end{proof}
We have similar results for type-$P$ quasi-stabilizalition maps but we omit the statements for conciseness. Next we consider commutative relations between different types of quasi-stabilization.

\begin{prop}\label{prop:far commutativity of P and Q}
Let $\LL=(L,\bfO,\bfX)$ be a multi-based generalized strongly invertible link. \begin{enumerate}
\item  For any fixed base points $X_1\in \bfX^f$ and $O_2\in \bfO^f$ we have \[Q_{X_1\mapsto O_1}^+ P_{O_2\mapsto X_2}^+ \simeq P_{O_2\mapsto X_2}^+ Q_{X_1\mapsto O_1}^+\] whenever the links are colored so that all the maps are well-defined chain maps. Here, $O_1$ and $X_2$ denote the new base points introduced during the quasi-stabilization.
\item Let $(X_1,O_1,X_1')$ be a triple of base points such that we can perform a quasi-stabilization $Q^-_{O_1\mapsto X_1}$ and $O_2 \in \bfO^f$ be a fixed $O$-base point. Suppose that $X_1$ is not adjacent to $O_2$. Then \[Q_{O_1\mapsto X_1}^- P_{O_2\mapsto X_2}^+ \simeq P_{O_2\mapsto X_2}^+ Q_{O_1\mapsto X_1}^-\] holds whenever the links are colored so that all the maps are well-defined chain maps.
\item Let $(X_1, O_1,X_1')$ and $(O_2, X_2,O_2')$  be triples of base points such that we can perform quasi-stabilizations $Q^-_{O_1\mapsto X_1}$ and $P_{X_2\mapsto O_2}^-$. Assume $(X_1,X_1')$ and $(O_2,O_2')$ are not adjacent on $\LL$, then \[Q_{O_1\mapsto X_1}^- P_{X_2\mapsto O_2}^- \simeq P_{X_2\mapsto O_2}^- Q_{O_1\mapsto X_1}^-\] holds whenever the links are colored so that all the maps are well-defined chain maps.
\end{enumerate}
\end{prop}

\begin{prop}\label{prop:far commutativity between T and S at different pairs of points}
Setup $\LL$, $(X_i,X_i')$, and $(O_i,O_i')$ ($i=1,2$) as in Proposition~\ref{prop:commutation relation between quasi-stab at different pairs of points}. Assume their relative position also satisfies the assumption in the beginning of Section~\ref{sub:quasi-stabilization} and $O_2$ is not adjacent to $X_1$. Then we have\[T_{O_1,X_1}^{\circ_1}S_{O_2,X_2}^{\circ_2}\simeq S_{O_2,X_2}^{\circ_2} T_{O_1,X_1}^{\circ_1},\]
for $\circ_{1},\circ_{2}\in \{+,-\}$. A similar relation holds between $T_{X_1,O_1}^{\circ_1}$ and $S_{X_2,O_2}^{\circ_2}$ when the order of $O$ and $X$ is changed. 
\end{prop}

\begin{proof}
This follows from the argument in~\cite[Proposition~4.20]{zemke2019link}. 
\end{proof}

Similar to Proposition~\ref{prop:far commutativity of P and Q} and \ref{prop:far commutativity between T and S at different pairs of points}, we have ``far commutativity'' between type-changing quasi-stabilizations $P$, $Q$ and type-keeping quasi-stabilizations $T$, $S$ when the real link Floer complex is colored so that they are all well-defined.

\subsection{Examples}\label{sub:example of full real complex}
In the classical knot Floer theory, the full knot complex and its local equivalence class are very powerful tools for investigating knot concordance. In this subsection, we provide the first examples of full  real knot Floer complexes along with an illustration of base point actions and type-changing quasi-stabilization. From this, we will observe 
\begin{itemize}
	\item how real knot Floer complexes differ from the usual ones structurally;
	\item how the non-commutativity of equivariant connected sum is reflected by the real knot Floer complexes;
	\item how type-changing quasi-stabilization affects the real knot Floer complex, which also reflects the non-existence of a K\"unneth principle. 
\end{itemize} 
For the last two items, one should compare \cite[Section~6.5]{YXHFLR}, in which we provided related examples at the homology level.

In the left frame of Figure~\ref{fig:LHT_and_RHT}, one sees a minimal real Heegaard diagram for the left-hand trefoil equipped with its unique strong inversion. Interchanging the roles of $\alpha$ and $\beta$-curves, we get a diagram for the right-hand trefoil.

\begin{figure}
	\centering
	\begin{overpic}[width=0.6\textwidth]{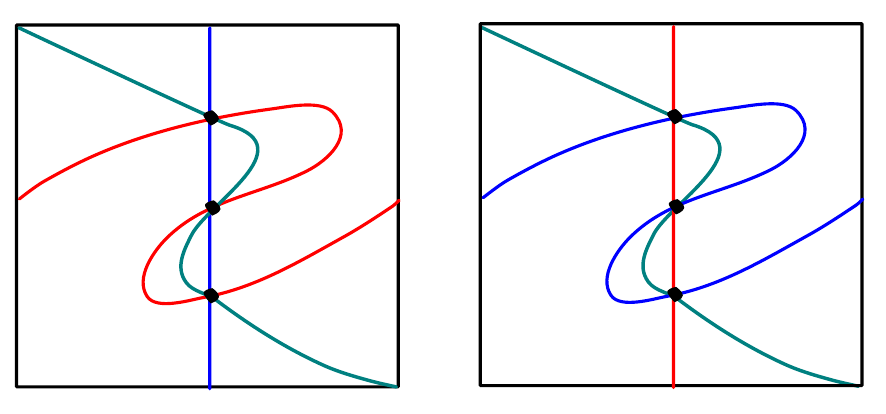}
		\put(25,34.5) {$l_1$}
		\put(25,20) {$l_2$}
		\put(20.5,9) {$l_3$}
		\put(27.5,28) {$O$}
		\put(19,15) {$X$}
		
		\put(77,34.5) {$r_1$}
		\put(77,20) {$r_2$}
		\put(72.5,9) {$r_3$}
		\put(79.5,28) {$O$}
		\put(70.5,15) {$X$}

	\end{overpic}
	\caption{Real Heegaard diagrams for left-hand and right-hand trefoils.}
	\label{fig:LHT_and_RHT}
\end{figure}
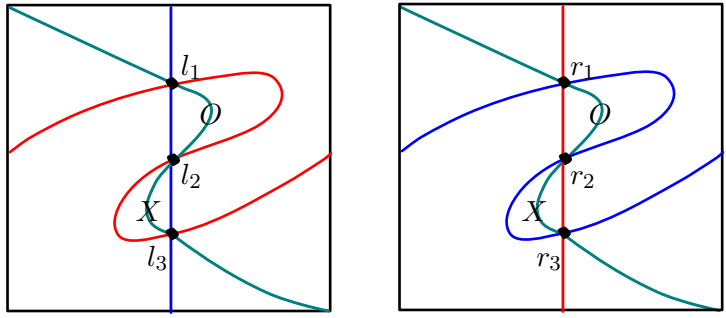

Using these genus one diagrams, we can calculate $\fullCFKR^-(\bar{3}_1)$ and $\fullCFKR^-(3_1)$ directly, which are shown in Diagram~\eqref{eq:full real complex for trefoils}. For these knots, the real knot complexes have no dependence on auxiliary data, so we omit them from the notation for simplicity. 


\begin{equation}\label{eq:full real complex for trefoils}
	\begin{tikzcd}
		l_1  & l_2 \arrow{d}{v} \arrow{l}{u}  \\
		& l_3  \\
	\end{tikzcd}\quad \quad \begin{tikzcd}
		r_3  \arrow{d}{v}  \\
		r_2 & r_1 \arrow{l}{u}   \\
	\end{tikzcd}
\end{equation}
From these diagrams, $\Phi^f$ and $\Psi^f$ actions can be read off directly. For $\bar{3}_1$, \[\Phi^{f}(l_1)= \Phi^{f}(l_3)=0 \quad \Phi^{f}(l_2)=l_1;\quad \Psi^{f}(l_1)= \Psi^{f}(l_3)=0 \quad \Psi^{f}(l_2)=l_3.\] For $3_1$, \[\Phi^{f}(r_2)= \Phi^{f}(r_3)=0 \quad \Phi^{f}(r_1)=r_2;\quad \Psi^{f}(r_1)= \Psi^{f}(r_2)=0 \quad \Psi^{f}(r_3)=r_2.\]

Now we perform an equivariant connected sum by identifying the $O$-base point of the left-hand trefoil with the $X$-base point of the right-hand trefoil. The full knot complex $\fullCFKR^-(\bar{3_1} \# 3_1)$ is shown on the left of Diagram~\eqref{eq:full real complex for connected sum of trefoils}. If we change the roles of left-hand trefoil and right-hand trefoil in the connected sum, we will get the complex shown on the right instead. In these diagrams, each horizontal arrow represents a differential of the form $y\to ux$, each vertical arrow represents a differential of the form $y\to vx$ while the diagonal arrows represent a differential of the form $y\to x$. By calculating their homology after setting $u$ or $v$ equal to zero, one can see that the chain complexes are not homotopy equivalent, as we have seen in \cite[Appendix~A]{YXHFLR} via suitably colored homologies.

\begin{equation}\label{eq:full real complex for connected sum of trefoils}
\begin{tikzpicture}
		\node(P0) at(0,5) {$l_1r_3$};
		\node(P1) at(2.5,5) {$l_2r_3$};
		\node(P2) at(0,2.5) {$l_1r_2$};
	    \node(P3) at(2.5,2.85) {$l_2r_2+l_3r_3$};
	    \node(P4) at(2.5,2.5) {$l_1r_2+l_2r_2$};
	    \node(P5) at(2.2,2.15) {$l_2r_2$};
	    \node(P6) at(5,2.5) {$l_2r_1$};
	    \node(P7) at(2.5,0) {$l_3r_2$};
	    \node(P8) at(5,0) {$l_3r_1$};
	    \draw[->] (P1)--(P0);
	    \draw[->] (P1)--(P3);
	    \draw[->] (P0)--(P2);
	    \draw[->] (P6)--(P4);
	    \draw[->] (P5)--(P2);
	    \draw[->] (P5)--(P7);
	    \draw[->] (P3)--(P2);
	    \draw[->] (P4)--(P7);
	    \draw[->] (P5)--(P7);
	    \draw[->] (P6)--(P8);
	    \draw[->] (P8)--(P7);
	    \draw[->] (P1)--(P2);  
	  
	  \node(P10) at(7,5) {$l_1r_3$};
	  \node(P11) at(9.5,5) {$l_2r_3$};
	  \node(P12) at(7,2.5) {$l_1r_2$};
	  \node(P13) at(9.5,2.85) {$l_2r_2+l_3r_3$};
	  \node(P14) at(9.5,2.5) {$l_1r_2+l_2r_2$};
	  \node(P15) at(9.2,2.15) {$l_2r_2$};
	  \node(P16) at(12,2.5) {$l_2r_1$};
	  \node(P17) at(9.5,0) {$l_3r_2$};
	  \node(P18) at(12,0) {$l_3r_1$};	  
	  \draw[->] (P11)--(P10);
	  \draw[->] (P11)--(P13);
	  \draw[->] (P10)--(P12);
	  \draw[->] (P16)--(P14);
	  \draw[->] (P15)--(P12);
	  \draw[->] (P15)--(P17);
	  \draw[->] (P13)--(P12);
	  \draw[->] (P14)--(P17);
	  \draw[->] (P15)--(P17);
	  \draw[->] (P16)--(P18);
	  \draw[->] (P18)--(P17);
	  \draw[->] (P16)--(P17);
	\end{tikzpicture} 
  \end{equation}
Again, $\Psi^f$ and $\Phi^f$ can be read off from the diagrams. 
\begin{remark}
In the usual knot Floer theory, the full knot Floer complexes satisfy a nice K\"unneth principle with respect to connected sum (see \cite{ConnectedsumsandinvolutiveknotFloerhomology} for example). This is not true in the real case, even on small knots such as the left and right-hand trefoils, as we have seen above. Moreover, though the real complexes of left and right-hand trefoils are the same as the usual complexes, in particular, take the form of standard stair case complexes, a seemingly unavoidable diagonal arrow appears in their connected sum. This is in great contrast with $\fullCFK^-(\bar{3_1} \# 3_1)$, which is a direct sum of squares and a single staircase complex, as $\bar{3_1} \# 3_1$ is $\HFK$-thin (cf.~\cite{CablesofthinknotsandborderedHeegaardFloerhomology}). 
\end{remark}

Finally, we perform a type-$Q$ quasi-stabilization to the minimally pointed left-hand trefoil. At the diagram level, this amounts to connected summing the left diagram in Figure~\ref{fig:LHT_and_RHT} with the left diagram in Figure~\ref{fig:std_pieces_for_type-changing_quasi_stab} by identifying a disk centered at $X$ with a disk centered at $O_2$ in an orientation-reversing fashion. Using Proposition~\ref{prop: differential oftype changing quasi-stab}, we see that 
\[\partial al_1= (u_1+u_0)V \cdot bl_1 \quad \partial al_2= u_0 \cdot al_1+ V\cdot bl_3+ (u_1+u_0)V\cdot bl_2 \quad \partial al_3= (u_1+u_0)V\cdot bl_3;\]
\[\partial bl_1= (u_1 +u_0) \cdot al_1 \quad \partial bl_2= u_0 \cdot bl_1+al_3 +(u_1 +u_0)\cdot al_2\quad \partial bl_3=(u_1 +u_0) \cdot al_3.\]
Taking derivatives with respect to the variables $u_0$, $u_1$ and $V$, we see that \[\Phi^f_{O_0}(al_1)= V\cdot bl_1\quad \Phi^f_{O_0}(al_2)= al_1+ V\cdot bl_1 \quad \Phi^f_{O_0}(al_3)= V\cdot bl_3;\] 
\[\Phi^f_{O_0}(bl_1)=al_1\quad \Phi^f_{O_0}(bl_2)= bl_1+ al_2 \quad \Phi^f_{O_0}(bl_3)= V\cdot al_3;\] 
\[\Phi^f_{O_1}(al_1)= V\cdot bl_1\quad \Phi^f_{O_1}(al_2)= V\cdot bl_1 \quad \Phi^f_{O_1}(al_3)= V\cdot bl_3;\] 
\[\Phi^f_{O_1}(bl_1)=  al_1\quad \Phi^f_{O_1}(bl_2)= bl_2 \quad \Phi^f_{O_1}(bl_3)= al_3;\] 
\[\Phi^p_{(X,X')}(al_1)= (u_0+u_1)\cdot bl_1\quad \Phi^p_{(X,X')}(al_2)= bl_3+ (u_0+u_1)\cdot bl_1 \quad \Phi^p_{(X,X')}(al_3)= (u_0+u_1)\cdot bl_3;\] 
\[\Phi^p_{(X,X')}(bl_1)=  \Phi^p_{(X,X')}(bl_2)= \Phi^p_{(X,X')}(bl_3)=0.\]
From these, one can also see that in the $OO$-theory, the roles of the two $O$-base points are not equivalent, as we remarked in Section~\ref{sub:Torsion order, real tau-invariant and module structure}. 

Now, we color the complex by $\sigma_s$ and compute that
\[D al_1= (u_1+u_0)v^2 \cdot bl_1 + (u_1+u_0)v \cdot al_1 \quad 
D bl_1= (u_1 +u_0) \cdot al_1 + (u_1+u_0)v \cdot bl_1;\]
\[D al_2= u_0 \cdot al_1+ v^2\cdot bl_3+ (u_1+u_0)v^2\cdot bl_2 + (u_1+u_0)v \cdot al_2 \quad D bl_2= u_0 \cdot bl_1+al_3 +(u_1 +u_0)\cdot al_2+ (u_1+u_0)v \cdot bl_1;\]
\[D al_3= (u_1+u_0)v^2\cdot bl_3 + (u_1+u_0)v \cdot al_3\quad  D bl_3=(u_1 +u_0) \cdot al_3 + (u_1+u_0)v \cdot bl_3.\]

Recall that $Q_{X\mapsto O}^+$ is given by $l_i\mapsto al_i+vbl_i$, while $Q_{O\mapsto X}^-$ is given by $al_i\mapsto vl_i$, $bl_i\mapsto l_i$. Then one can verify that $D$ commutes with the type-Q quasi-stabilization maps and $\partial$ commutes with them after setting $u_0=u_1$ or $v$ and $V$ both to zero. Readers can also try to verify Lemma~\ref{lem:present f-base point action as composition of quasistab maps} and \ref{lem:composition of type T,S ,P, Qquasi-stab maps} on this example. As a sanity check, one can compute that $\partial^2= (u_1^2+u_0^2)V \cdot \id $ and $D^2=0$.  One can also set $u=u_1$ ($u=u_1=0$) and compare the resulting homology $\HFKR_{O}^-$  ($\widehat{\HFKR}_O$) with the first example in Appendix~\ref{app:Calculation results for knots with small crossing numbers}.

\appendix
\section{$OO$-real grid homology for strongly invertible knots with small crossing numbers}\label{app:Calculation results for knots with small crossing numbers}

\maketitle
\end{comment}

{\bf Introduction}. This is prepared jointly by Zhenkun Li and Yonghan Xiao. The first author developed the computer program and made the table shown below, while the second author collected examples of strongly invertible knots and constructed real grid diagrams for them.

The appendix includes a list of small crossing strongly invertible knots and their real grid homology. We cover all involutions on knots with at most $7$ crossings, using the planar diagrams from \cite[Appendix]{Lobb2021ArefinementofKhovanovhomology} and some other strongly invertible knots that admit a real grid diagram of size less than $17$. The computer program \cite{ZhenkunLioythonprgram} is able to calculate minus version of real grid homology for strongly invertible knots with diagrams of size $\le 12$, hat version of real grid homology for strongly invertible knots with diagrams of size $\le 12$ and the real Alexander polynomial for strongly invertible knots with diagrams of size $\le 16$.

{\bf AI usage}. Here we explain how the appendix was jointly prepared by the human and an AI coding agent.
\begin{itemize}
	\item The translation of the mathematical construction into a computer-realizable algorithm was carried out by the human authors.
	\item The Python code was written by an AI coding agent and then reviewed and verified by the human authors.
	\item Knot data was collected by human and grid diagrams were also constructed and input into the python program by hand.
	\item The knot homologies and polynomials were output by the python program via JSON files.
	\item The AI coding agent wrote a python program with JSON files as input and LaTex code as output. This was used to generate the table below. 
	\item It has been checked by human that the data in the table agree with the original output from json files.
\end{itemize}

{\bf Code trustworthiness}.
The scripts that compute real invariants from grid diagrams were developed by Zhenkun Li with assistance from the coding agent, while the grid diagrams of various knots were collected by Yonghan Xiao. We would like to be transparent about the extent to which these scripts should be trusted.
\begin{itemize}
	\item The core functionality---building the set of generators for the real chain complex, finding all domains between generators, and computing the bi-grading of each generator---was written by hand and tested against one example: a size-6 grid diagram of the trefoil, which has 76 generators and 267 domains.
	\item The computations for the polynomial and the real homology (hat and minus versions) were generated by the coding agent. The following validations have been applied:
	\begin{itemize}
		\item Direct comparison on the trefoil knot, the only non-trivial example computable by hand.
		\item Verification that $d^2 = 0$ for all examples.
		\item An intermediate step of the computation produces the real homology of the tilde version, whose dimension is divisible by a high power of $2$ (e.g., $64$ for any strongly invertible knot on a size-14 grid). This divisibility holds for all examples we tested.
		\item Symmetry of the polynomials and homologies for small-crossing knots.
	\end{itemize}
\end{itemize}

{\bf Notation}.
\begin{itemize}
	\item For knot names, we use the names from Rolfsen table following \cite{knotinfo}, but we do not distinguish a knot with its mirror.
	\item For the hat version of the homology, each summand is of the form $(2a,m)^d$ where $a$ is the real Alexander grading, $m$ is the real Maslov grading, and $d$ is the dimension of the corresponding graded space.
	\item For the minus version, each summand is of the form $\bigg(U^o_{(2a,m)}\bigg)^d$. Here again $(a,m)$ encodes the grading information. The "power" of $U$ represents the $U$-torsion order, and $d$ represents the number of generators in the corresponding bi-grading and has the corresponding order. The first two factors are always of the form $U^{\infty}_{(2a,m)}$ representing the generator of the infinite $U$-towers.
\end{itemize}
We remind readers that the notation here is a little bit different from that in the main text, where we write the grading as $(M^R,A^R)$ instead of $(2A^R,M^R)$.
\vspace{0.2in}
\setlength{\tabcolsep}{3pt}
\begin{longtable}{m{0.25\textwidth} m{0.3\textwidth} m{0.45\textwidth}}
  \hline
  \centering \textbf{Knot} & \centering \textbf{Grid Diagram} & \centering \textbf{Real Invariants} \tabularnewline
  \hline
  \endfirsthead
  \hline
  \centering \textbf{Knot} & \centering \textbf{Grid Diagram} & \centering \textbf{Real Invariants} \tabularnewline
  \hline
  \endhead
  \hline
  \endfoot
  \centering $3_1^{OO}$ & \centering \includegraphics[width=0.25\textwidth]{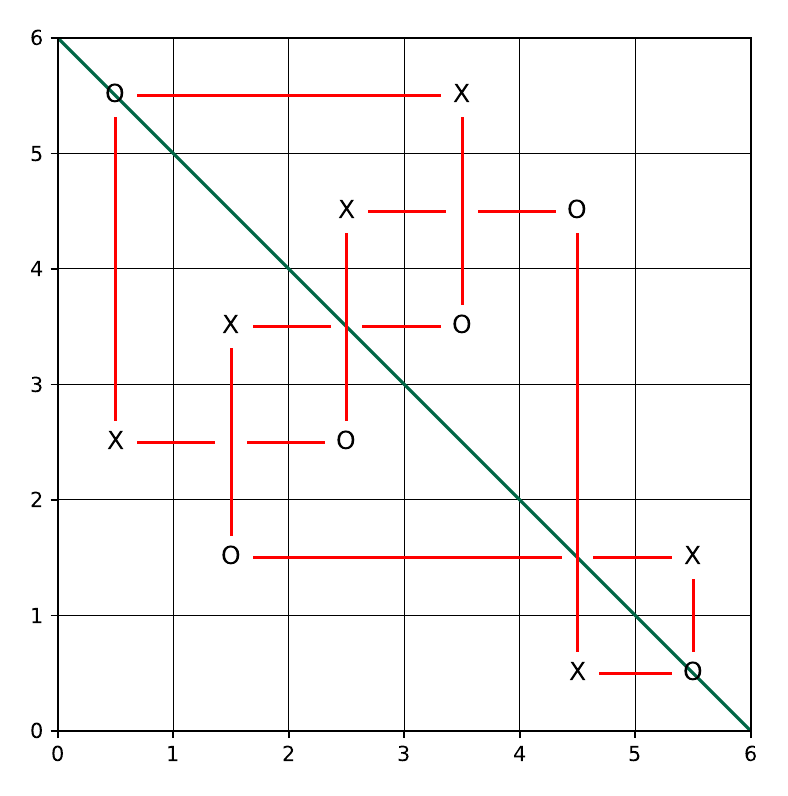} & \parbox[c]{\linewidth}{\centering\tiny
\textcolor{blue}{Polynomial Invariant}\\
$t^{-1} + 2 - t^{2}$\\[0.1in]
\textcolor{blue}{Real grid homology - hat version}\\
$(-1,0)\oplus (0,0)^{2}\oplus (2,1)$\\[0.1in]
\textcolor{blue}{Real grid homology - minus version}\\
$U^{\infty}_{(0,0)}\oplus U^{\infty}_{(2,1)}\oplus U_{(0,0)}$
} \tabularnewline
  \hline
  \centering $4_1^{OO}$ & \centering \includegraphics[width=0.25\textwidth]{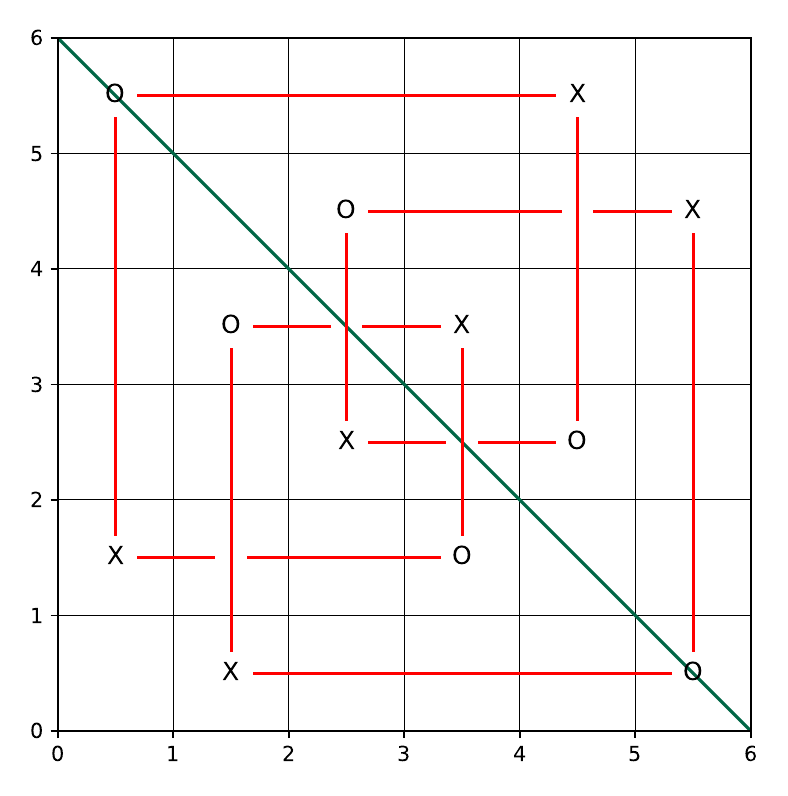} & \parbox[c]{\linewidth}{\centering\tiny
\textcolor{blue}{Polynomial Invariant}\\
$-t^{-1} + 2t + t^{2}$\\[0.1in]
\textcolor{blue}{Real grid homology - hat version}\\
$(-1,-1)\oplus (1,0)^{2}\oplus (2,0)$\\[0.1in]
\textcolor{blue}{Real grid homology - minus version}\\
$U^{\infty}_{(-1,-1)}\oplus U^{\infty}_{(1,0)}\oplus U_{(2,0)}$
} \tabularnewline
  \hline
  \parbox[c]{\linewidth}{\centering $(4_1^{OO})'$\\{\tiny Second symmetry}} & \centering \includegraphics[width=0.25\textwidth]{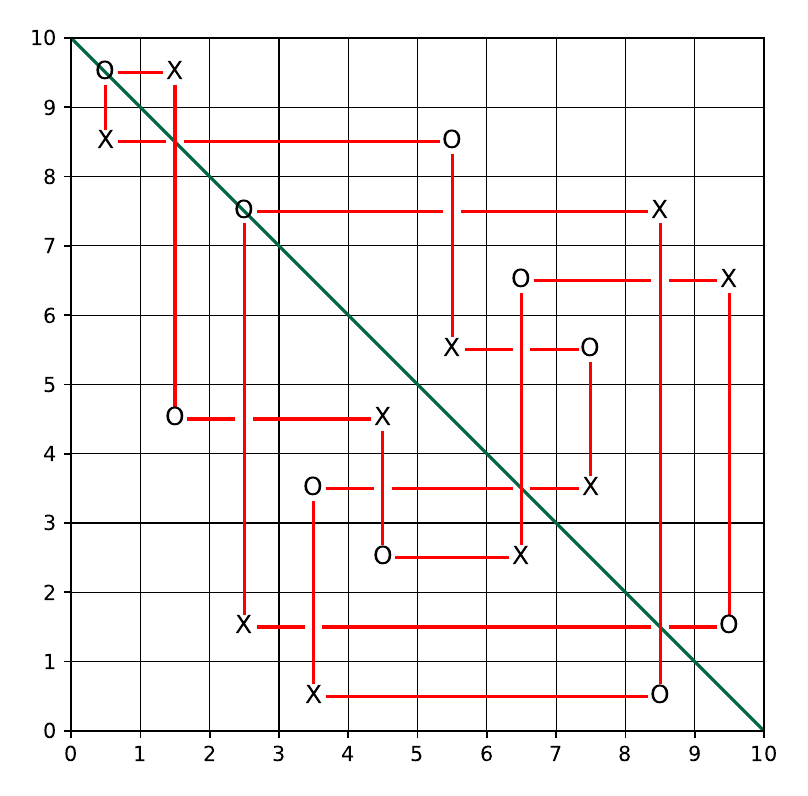} & \parbox[c]{\linewidth}{\centering\tiny
\textcolor{blue}{Polynomial Invariant}\\
$t^{-1} + 2 - t^{2}$\\[0.1in]
\textcolor{blue}{Real grid homology - hat version}\\
$(-1,0)\oplus (0,0)^{2}\oplus (2,1)$\\[0.1in]
\textcolor{blue}{Real grid homology - minus version}\\
$U^{\infty}_{(0,0)}\oplus U^{\infty}_{(2,1)}\oplus U_{(0,0)}$
} \tabularnewline
  \hline
  \centering $5_1^{OO}$ & \centering \includegraphics[width=0.25\textwidth]{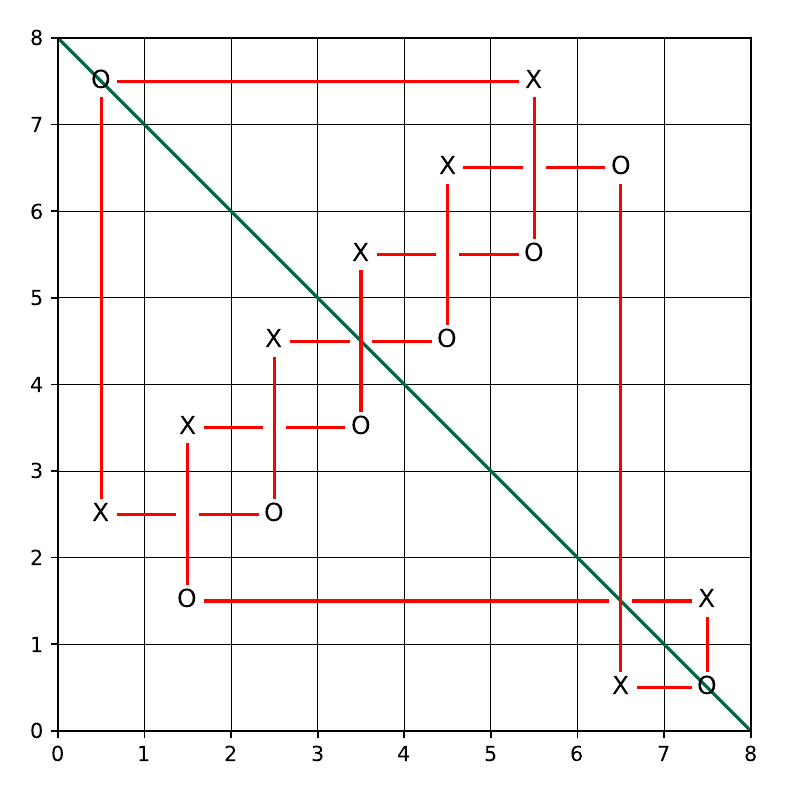} & \parbox[c]{\linewidth}{\centering\tiny
\textcolor{blue}{Polynomial Invariant}\\
$t^{-2} + 2t^{-1} - 2t + t^{3}$\\[0.1in]
\textcolor{blue}{Real grid homology - hat version}\\
$(-1,0)^{2}\oplus (-2,0)\oplus (1,1)^{2}\oplus (3,2)$\\[0.1in]
\textcolor{blue}{Real grid homology - minus version}\\
$U^{\infty}_{(1,1)}\oplus U^{\infty}_{(3,2)}\oplus U_{(-1,0)}\oplus U^{2}_{(1,1)}$
} \tabularnewline
  \hline
  \centering $5_2^{OO}$ & \centering \includegraphics[width=0.25\textwidth]{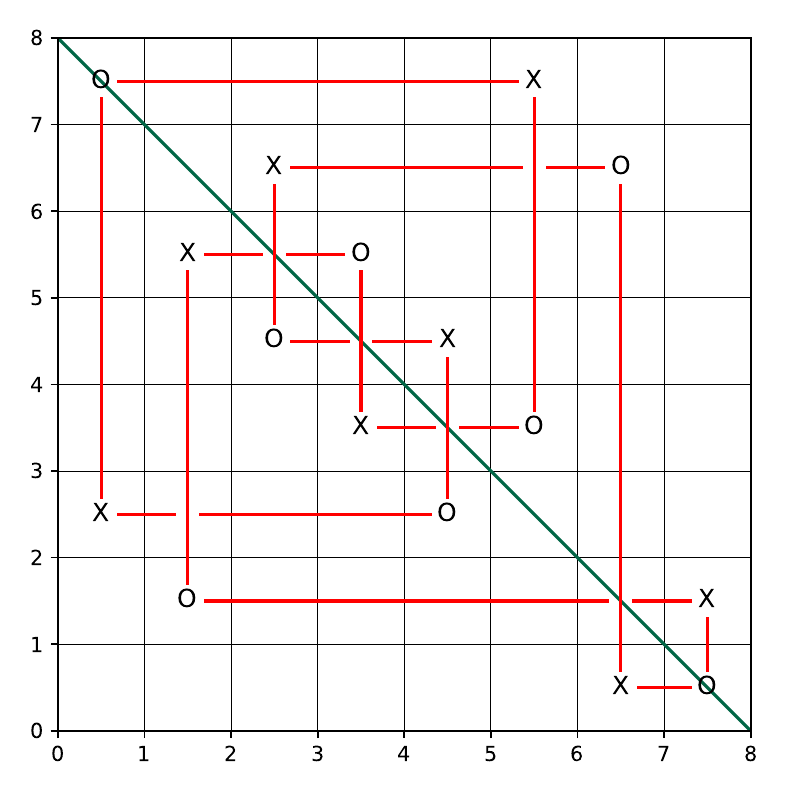} & \parbox[c]{\linewidth}{\centering\tiny
\textcolor{blue}{Polynomial Invariant}\\
$1 + t$\\[0.1in]
\textcolor{blue}{Real grid homology - hat version}\\
$(0,0)\oplus (1,0)$\\[0.1in]
\textcolor{blue}{Real grid homology - minus version}\\
$U^{\infty}_{(0,0)}\oplus U^{\infty}_{(1,0)}$
} \tabularnewline
  \hline
  \parbox[c]{\linewidth}{\centering $(5_2^{OO})'$\\{\tiny Second symmetry}} & \centering \includegraphics[width=0.25\textwidth]{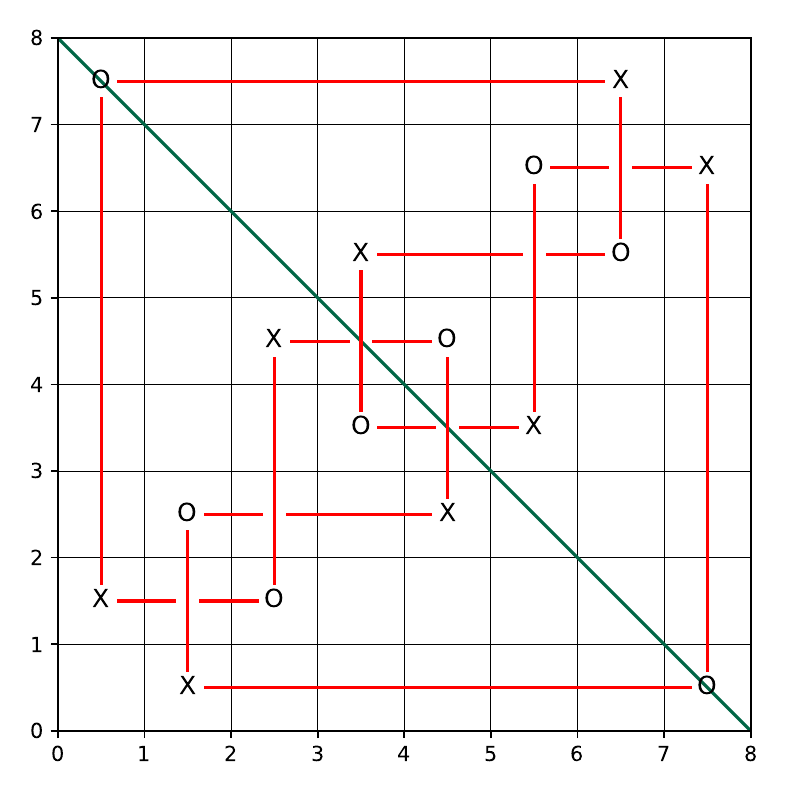} & \parbox[c]{\linewidth}{\centering\tiny
\textcolor{blue}{Polynomial Invariant}\\
$-2t^{-1} - 1 + 3t + 2t^{2}$\\[0.1in]
\textcolor{blue}{Real grid homology - hat version}\\
$(-1,-1)^{2}\oplus (0,-1)\oplus (1,0)^{3}\oplus (2,0)^{2}$\\[0.1in]
\textcolor{blue}{Real grid homology - minus version}\\
$U^{\infty}_{(-1,-1)}\oplus U^{\infty}_{(1,0)}\oplus U_{(0,-1)}\oplus \bigg(U_{(2,0)}\bigg)^{2}$
} \tabularnewline
  \hline
  \centering $6_1^{OO}$ & \centering \includegraphics[width=0.25\textwidth]{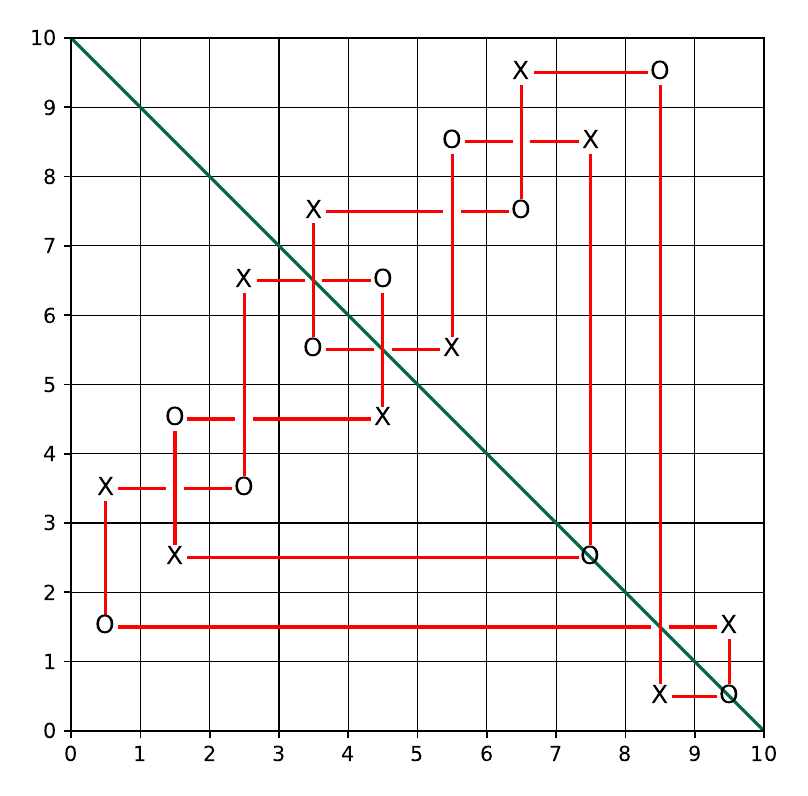} & \parbox[c]{\linewidth}{\centering\tiny
\textcolor{blue}{Polynomial Invariant}\\
$-2t^{-1} - 1 + 3t + 2t^{2}$\\[0.1in]
\textcolor{blue}{Real grid homology - hat version}\\
$(-1,-1)^{2}\oplus (0,-1)\oplus (1,0)^{3}\oplus (2,0)^{2}$\\[0.1in]
\textcolor{blue}{Real grid homology - minus version}\\
$U^{\infty}_{(-1,-1)}\oplus U^{\infty}_{(1,0)}\oplus U_{(0,-1)}\oplus \bigg(U_{(2,0)}\bigg)^{2}$
} \tabularnewline
  \hline
  \parbox[c]{\linewidth}{\centering $(6_1^{OO})'$\\{\tiny Second symmetry}} & \centering \includegraphics[width=0.25\textwidth]{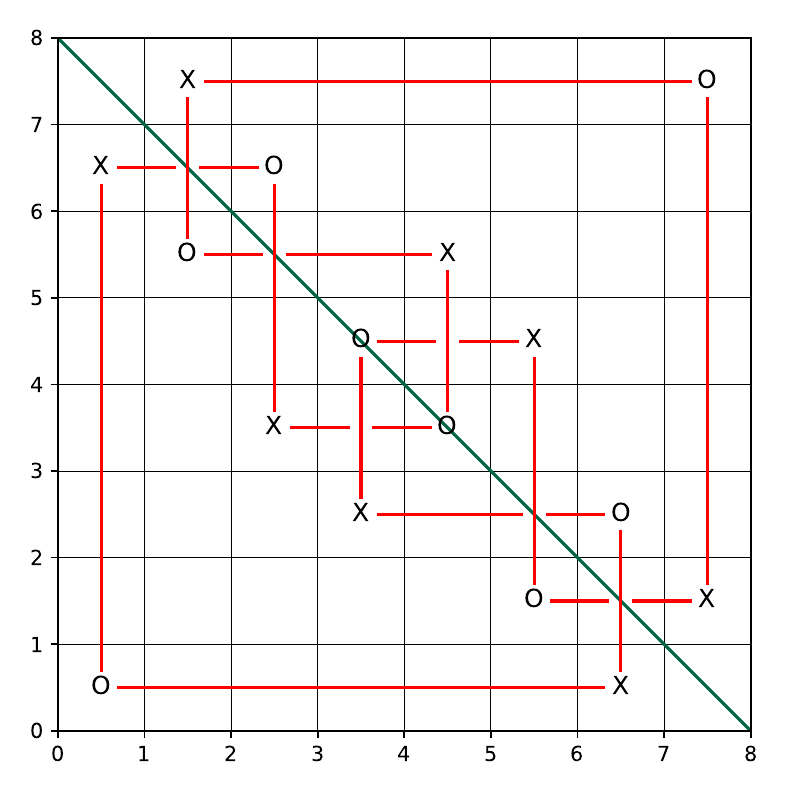} & \parbox[c]{\linewidth}{\centering\tiny
\textcolor{blue}{Polynomial Invariant}\\
$1 + t$\\[0.1in]
\textcolor{blue}{Real grid homology - hat version}\\
$(0,0)\oplus (1,0)$\\[0.1in]
\textcolor{blue}{Real grid homology - minus version}\\
$U^{\infty}_{(0,0)}\oplus U^{\infty}_{(1,0)}$
} \tabularnewline
  \hline
  \centering $6_2^{OO}$ & \centering \includegraphics[width=0.25\textwidth]{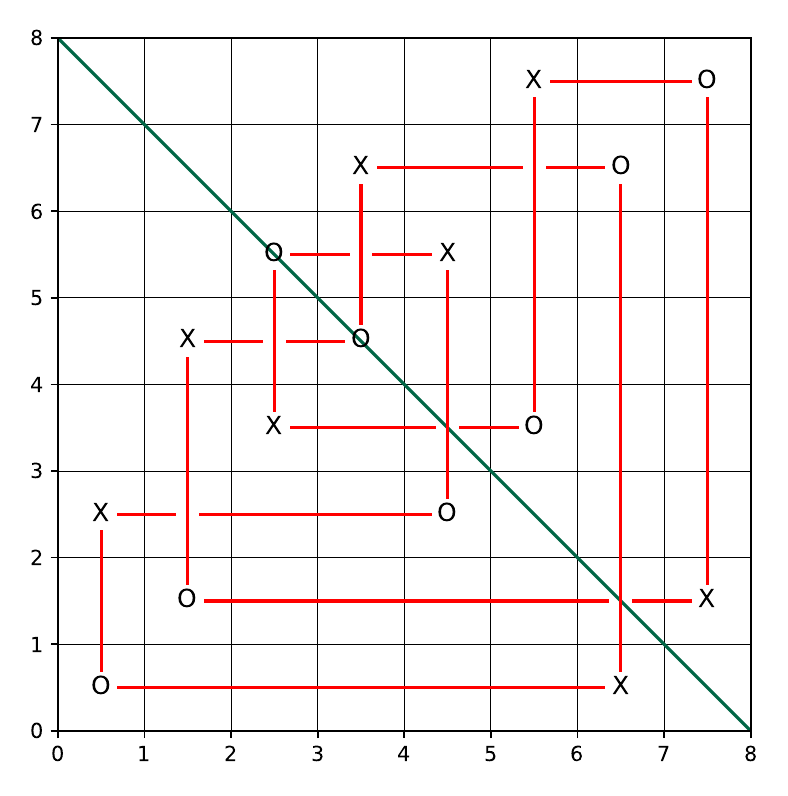} & \parbox[c]{\linewidth}{\centering\tiny
\textcolor{blue}{Polynomial Invariant}\\
$-t^{-2} + 4 + 2t - 2t^{2} - t^{3}$\\[0.1in]
\textcolor{blue}{Real grid homology - hat version}\\
$(-2,-1)\oplus (0,0)^{4}\oplus (1,0)^{2}\oplus (2,1)^{2}\oplus (3,1)$\\[0.1in]
\textcolor{blue}{Real grid homology - minus version}\\
$U^{\infty}_{(0,0)}\oplus U^{\infty}_{(2,1)}\oplus \bigg(U_{(1,0)}\bigg)^{2}\oplus U_{(3,1)}\oplus U^{2}_{(0,0)}$
} \tabularnewline
  \hline
  \parbox[c]{\linewidth}{\centering $(6_2^{OO})'$\\{\tiny Second symmetry}} & \centering \includegraphics[width=0.25\textwidth]{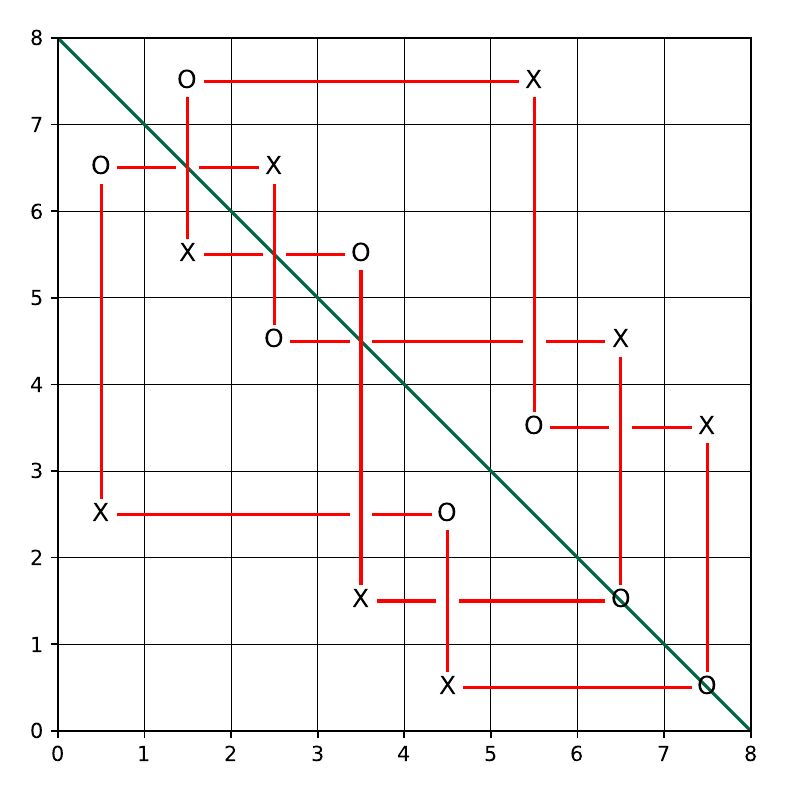} & \parbox[c]{\linewidth}{\centering\tiny
\textcolor{blue}{Polynomial Invariant}\\
$t^{-2} - 2 + 2t^{2} + t^{3}$\\[0.1in]
\textcolor{blue}{Real grid homology - hat version}\\
$(-2,-2)\oplus (0,-1)^{2}\oplus (2,0)^{2}\oplus (3,0)$\\[0.1in]
\textcolor{blue}{Real grid homology - minus version}\\
$U^{\infty}_{(-2,-2)}\oplus U^{\infty}_{(0,-1)}\oplus U_{(3,0)}\oplus U^{2}_{(2,0)}$
} \tabularnewline
  \hline
  \centering $6_3^{OO}$ & \centering \includegraphics[width=0.25\textwidth]{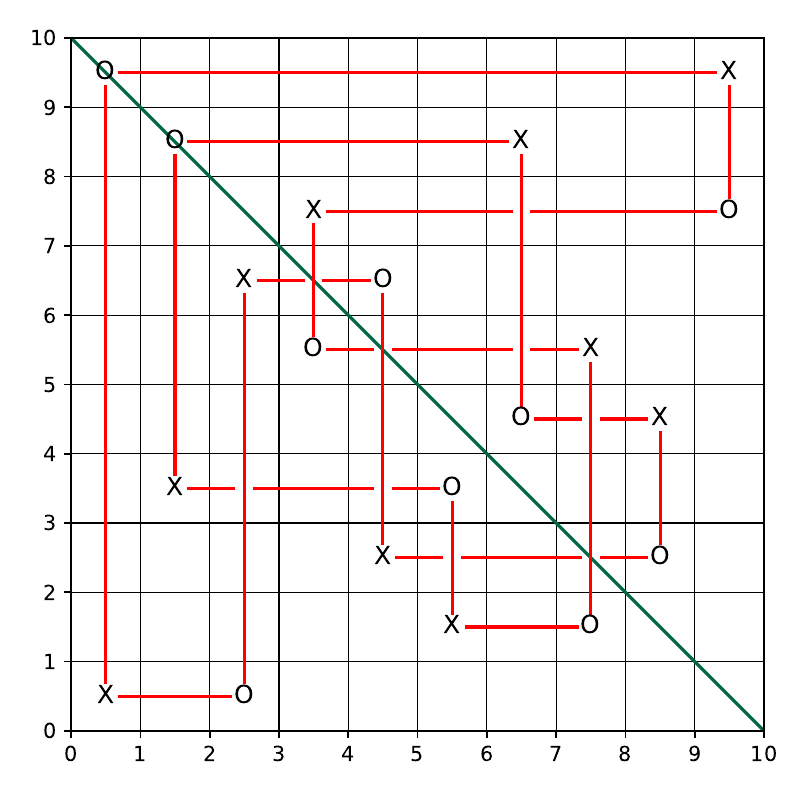} & \parbox[c]{\linewidth}{\centering\tiny
\textcolor{blue}{Polynomial Invariant}\\
$-t^{-2} - 2t^{-1} + 2 + 4t - t^{3}$\\[0.1in]
\textcolor{blue}{Real grid homology - hat version}\\
$(-1,-1)^{2}\oplus (-2,-1)\oplus (0,0)^{2}\oplus (1,0)^{4}\oplus (3,1)$\\[0.1in]
\textcolor{blue}{Real grid homology - minus version}\\
$U^{\infty}_{(-1,-1)}\oplus U^{\infty}_{(1,0)}\oplus U_{(-1,-1)}\oplus \bigg(U_{(1,0)}\bigg)^{2}\oplus U^{2}_{(3,1)}$
} \tabularnewline
  \hline
  \centering $7_1^{OO}$ & \centering \includegraphics[width=0.25\textwidth]{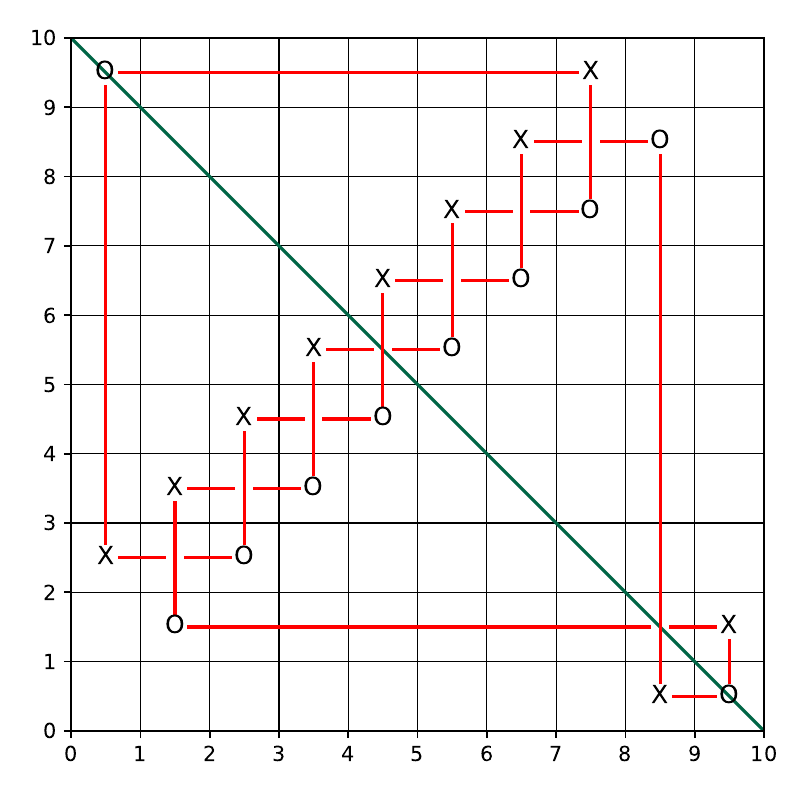} & \parbox[c]{\linewidth}{\centering\tiny
\textcolor{blue}{Polynomial Invariant}\\
$t^{-3} + 2t^{-2} - 2 + 2t^{2} - t^{4}$\\[0.1in]
\textcolor{blue}{Real grid homology - hat version}\\
$(-2,0)^{2}\oplus (-3,0)\oplus (0,1)^{2}\oplus (2,2)^{2}\oplus (4,3)$\\[0.1in]
\textcolor{blue}{Real grid homology - minus version}\\
$U^{\infty}_{(2,2)}\oplus U^{\infty}_{(4,3)}\oplus U_{(-2,0)}\oplus U^{2}_{(0,1)}\oplus U^{2}_{(2,2)}$
} \tabularnewline
  \hline
  \centering $7_2^{OO}$ & \centering \includegraphics[width=0.25\textwidth]{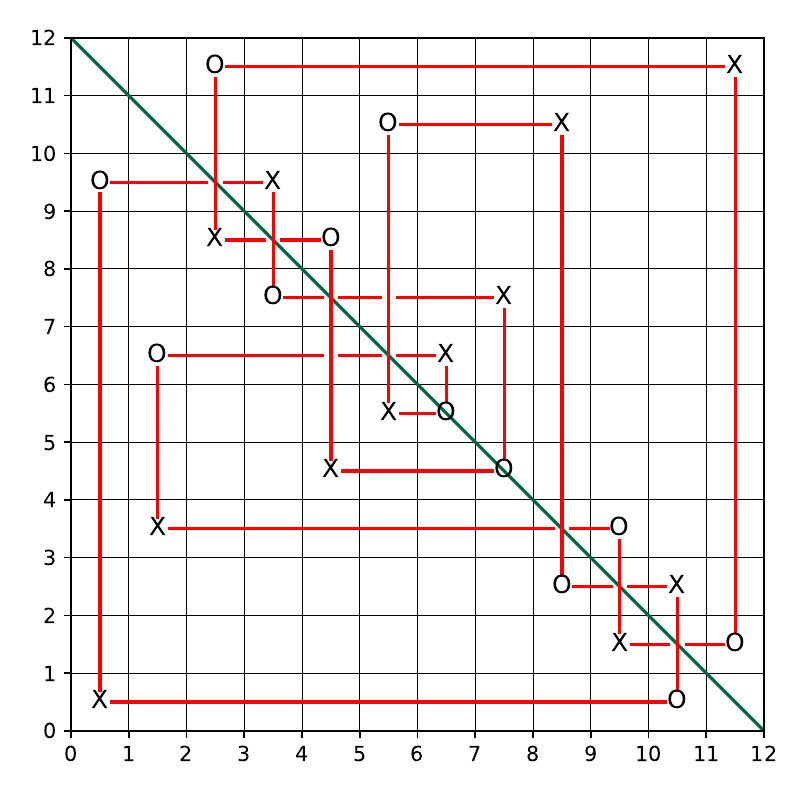} & \parbox[c]{\linewidth}{\centering\tiny
\textcolor{blue}{Polynomial Invariant}\\
$t^{-1} + 2 - t^{2}$\\[0.1in]
\textcolor{blue}{Real grid homology - hat version}\\
$(-1,0)\oplus (0,0)^{2}\oplus (2,1)$\\[0.1in]
\textcolor{blue}{Real grid homology - minus version}\\
$U^{\infty}_{(0,0)}\oplus U^{\infty}_{(2,1)}\oplus U_{(0,0)}$
} \tabularnewline
  \hline
  \parbox[c]{\linewidth}{\centering $(7_2)^{OO}$\\{\tiny Also twisted${}_5^{OO}$}} & \centering \includegraphics[width=0.25\textwidth]{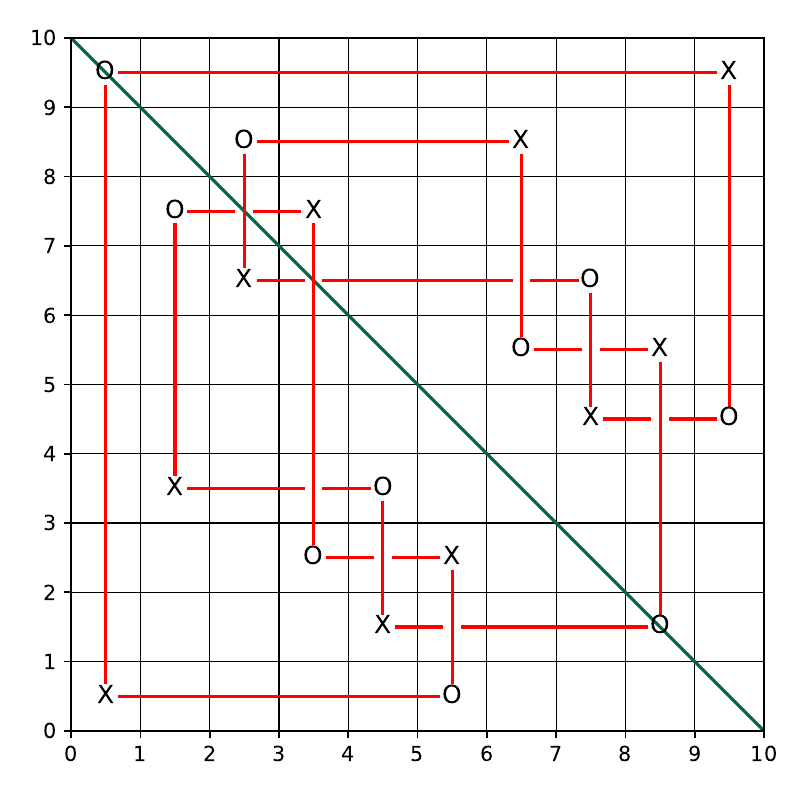} & \parbox[c]{\linewidth}{\centering\tiny
\textcolor{blue}{Polynomial Invariant}\\
$3t^{-1} + 4 - 2t - 3t^{2}$\\[0.1in]
\textcolor{blue}{Real grid homology - hat version}\\
$(-1,0)^{3}\oplus (0,0)^{4}\oplus (1,1)^{2}\oplus (2,1)^{3}$\\[0.1in]
\textcolor{blue}{Real grid homology - minus version}\\
$U^{\infty}_{(0,0)}\oplus U^{\infty}_{(2,1)}\oplus \bigg(U_{(0,0)}\bigg)^{3}\oplus \bigg(U_{(2,1)}\bigg)^{2}$
} \tabularnewline
  \hline
  \centering $7_3^{OO}$ & \centering \includegraphics[width=0.25\textwidth]{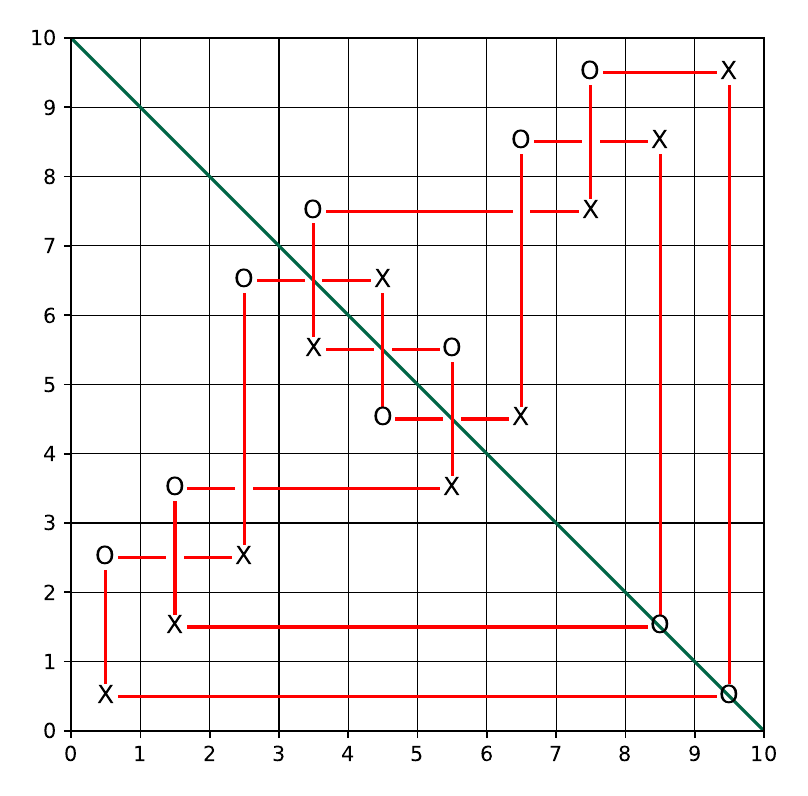} & \parbox[c]{\linewidth}{\centering\tiny
\textcolor{blue}{Polynomial Invariant}\\
$t^{-1} + 2 - t^{2}$\\[0.1in]
\textcolor{blue}{Real grid homology - hat version}\\
$(-1,0)\oplus (0,0)^{2}\oplus (2,1)$\\[0.1in]
\textcolor{blue}{Real grid homology - minus version}\\
$U^{\infty}_{(0,0)}\oplus U^{\infty}_{(2,1)}\oplus U_{(0,0)}$
} \tabularnewline
  \hline
  \parbox[c]{\linewidth}{\centering $(7_3^{OO})'$\\{\tiny Second symmetry}} & \centering \includegraphics[width=0.25\textwidth]{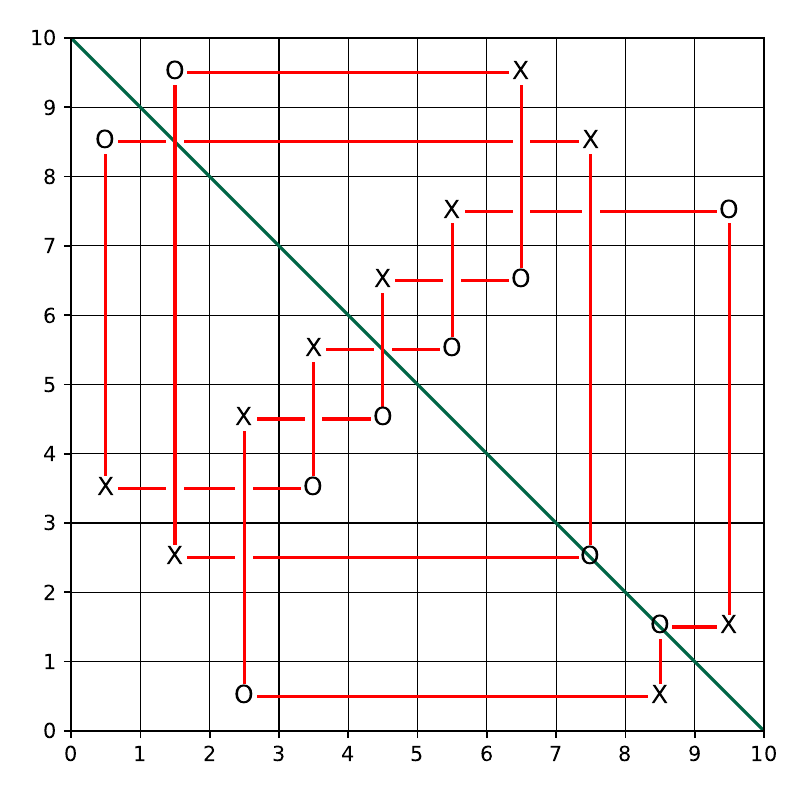} & \parbox[c]{\linewidth}{\centering\tiny
\textcolor{blue}{Polynomial Invariant}\\
$2t^{-2} + 3t^{-1} - 2 - 4t + t^{2} + 2t^{3}$\\[0.1in]
\textcolor{blue}{Real grid homology - hat version}\\
$(-1,0)^{3}\oplus (-2,0)^{2}\oplus (0,1)^{2}\oplus (1,1)^{4}\oplus (2,2)\oplus (3,2)^{2}$\\[0.1in]
\textcolor{blue}{Real grid homology - minus version}\\
$U^{\infty}_{(1,1)}\oplus U^{\infty}_{(3,2)}\oplus \bigg(U_{(-1,0)}\bigg)^{2}\oplus \bigg(U_{(1,1)}\bigg)^{2}\oplus U_{(3,2)}\oplus U^{2}_{(1,1)}$
} \tabularnewline
  \hline
  \centering $7_4^{OO}$ & \centering \includegraphics[width=0.25\textwidth]{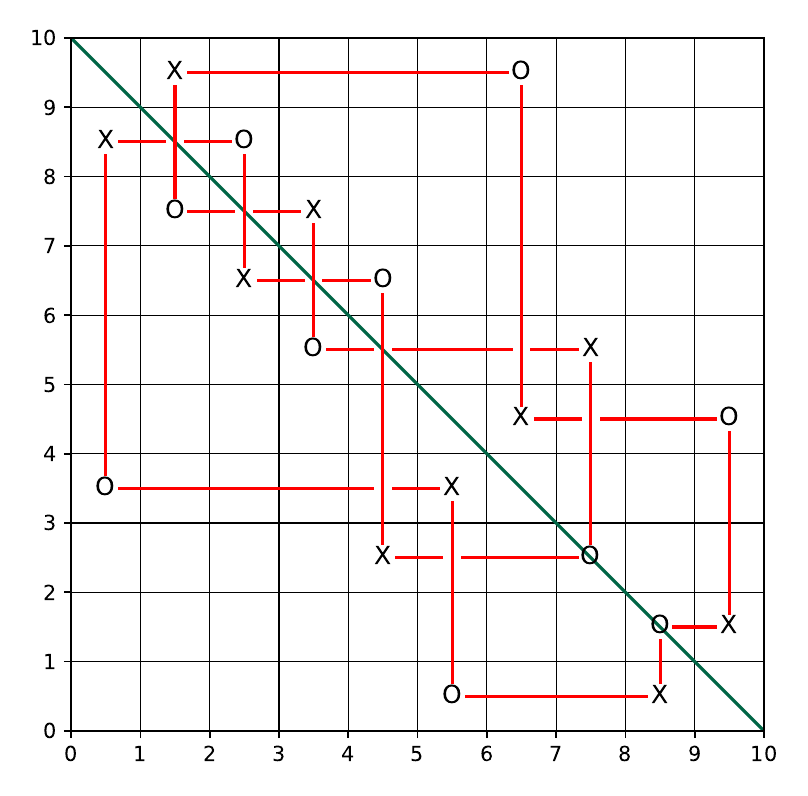} & \parbox[c]{\linewidth}{\centering\tiny
\textcolor{blue}{Polynomial Invariant}\\
$1 + t$\\[0.1in]
\textcolor{blue}{Real grid homology - hat version}\\
$(0,0)\oplus (1,0)$\\[0.1in]
\textcolor{blue}{Real grid homology - minus version}\\
$U^{\infty}_{(0,0)}\oplus U^{\infty}_{(1,0)}$
} \tabularnewline
  \hline
  \parbox[c]{\linewidth}{\centering $(7_4^{OO})'$\\{\tiny Second symmetry}} & \centering \includegraphics[width=0.25\textwidth]{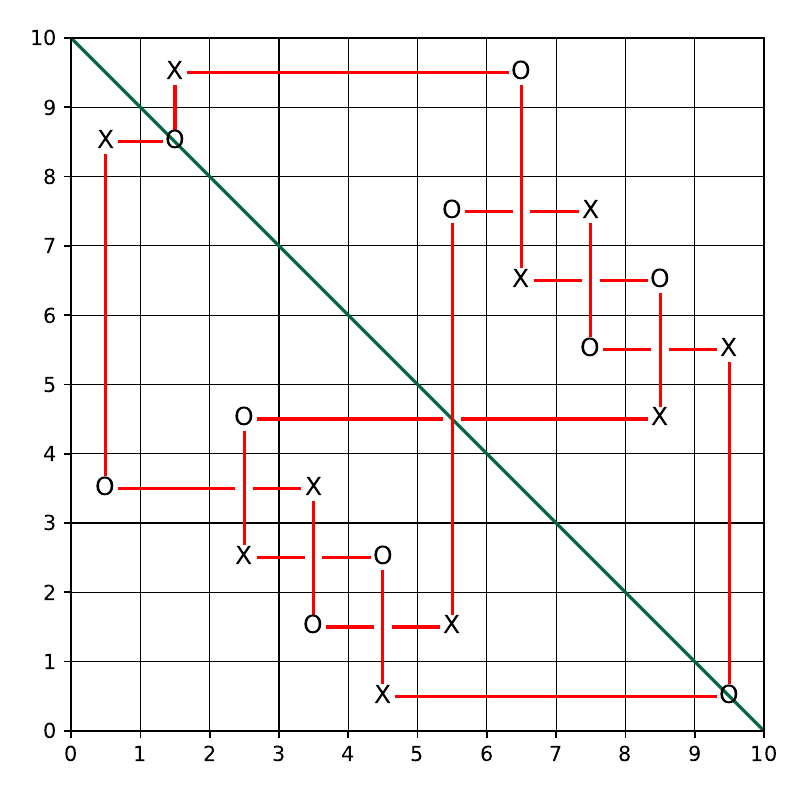} & \parbox[c]{\linewidth}{\centering\tiny
\textcolor{blue}{Polynomial Invariant}\\
$2t^{-1} + 3 - t - 2t^{2}$\\[0.1in]
\textcolor{blue}{Real grid homology - hat version}\\
$(-1,0)^{2}\oplus (0,0)^{3}\oplus (1,1)\oplus (2,1)^{2}$\\[0.1in]
\textcolor{blue}{Real grid homology - minus version}\\
$U^{\infty}_{(0,0)}\oplus U^{\infty}_{(2,1)}\oplus \bigg(U_{(0,0)}\bigg)^{2}\oplus U_{(2,1)}$
} \tabularnewline
  \hline
  \centering $7_5^{OO}$ & \centering \includegraphics[width=0.25\textwidth]{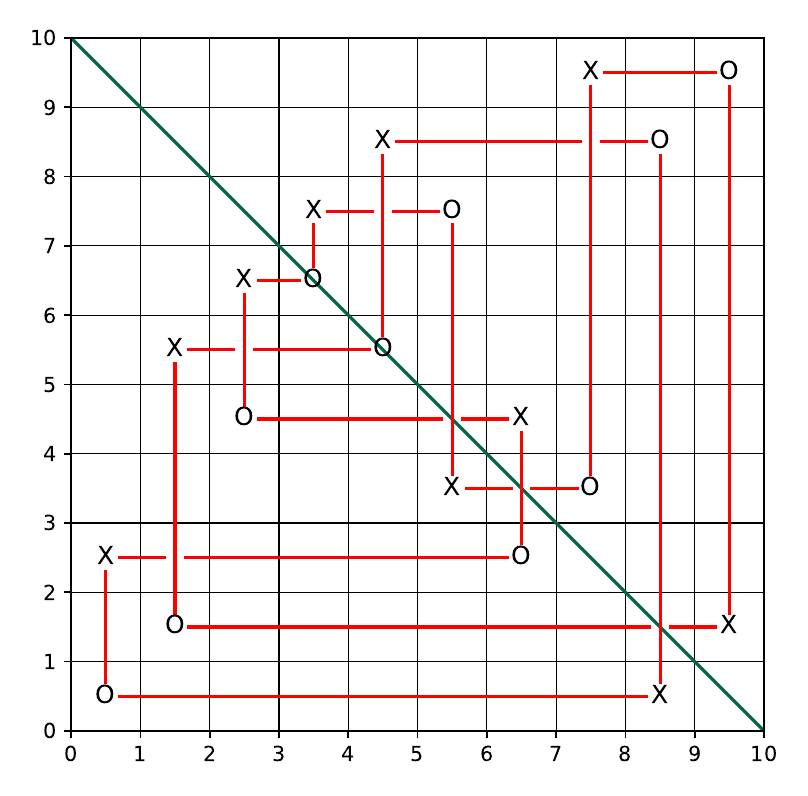} & \parbox[c]{\linewidth}{\centering\tiny
\textcolor{blue}{Polynomial Invariant}\\
$2t^{-1} + 3 - t - 2t^{2}$\\[0.1in]
\textcolor{blue}{Real grid homology - hat version}\\
$(-1,0)^{2}\oplus (0,0)^{3}\oplus (1,1)\oplus (2,1)^{2}$\\[0.1in]
\textcolor{blue}{Real grid homology - minus version}\\
$U^{\infty}_{(0,0)}\oplus U^{\infty}_{(2,1)}\oplus \bigg(U_{(0,0)}\bigg)^{2}\oplus U_{(2,1)}$
} \tabularnewline
  \hline
  \parbox[c]{\linewidth}{\centering $(7_5^{OO})'$\\{\tiny Second symmetry}} & \centering \includegraphics[width=0.25\textwidth]{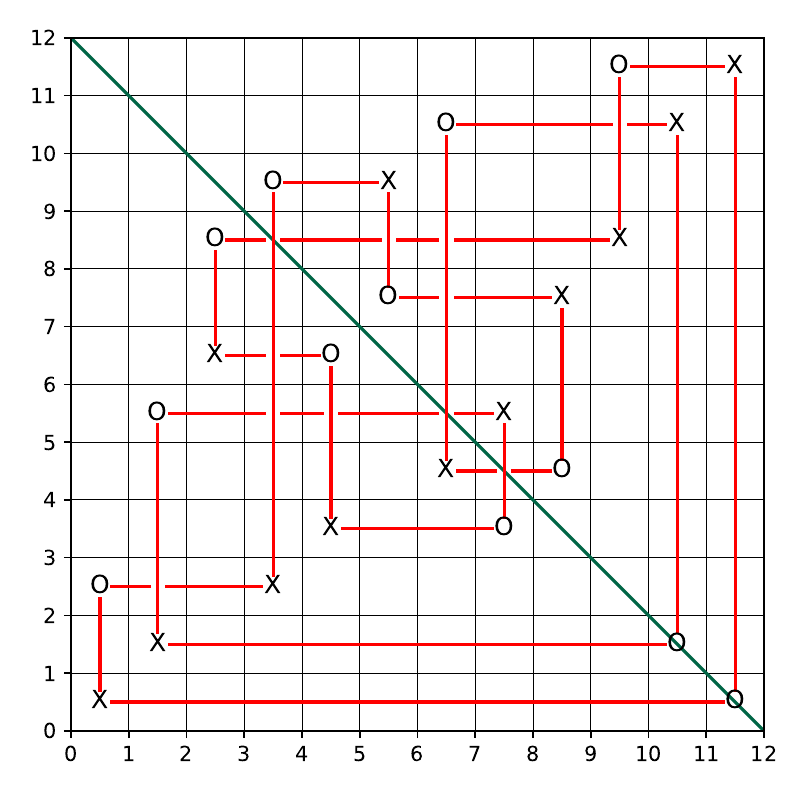} & \parbox[c]{\linewidth}{\centering\tiny
\textcolor{blue}{Polynomial Invariant}\\
$2t^{-2} + 4t^{-1} - 1 - 5t + 2t^{3}$\\[0.1in]
\textcolor{blue}{Real grid homology - hat version}\\
$(-1,0)^{4}\oplus (-2,0)^{2}\oplus (0,1)\oplus (1,1)^{5}\oplus (3,2)^{2}$\\[0.1in]
\textcolor{blue}{Real grid homology - minus version}\\
$U^{\infty}_{(1,1)}\oplus U^{\infty}_{(3,2)}\oplus \bigg(U_{(-1,0)}\bigg)^{2}\oplus U_{(1,1)}\oplus \bigg(U^{2}_{(1,1)}\bigg)^{2}\oplus U^{2}_{(3,2)}$
} \tabularnewline
  \hline
  \centering $7_6^{OO}$ & \centering \includegraphics[width=0.25\textwidth]{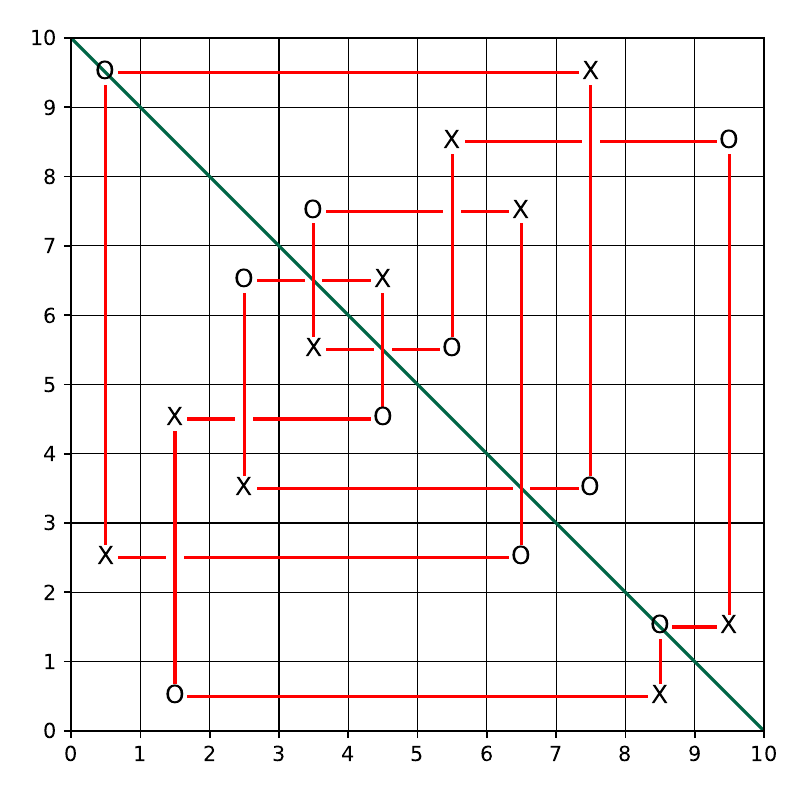} & \parbox[c]{\linewidth}{\centering\tiny
\textcolor{blue}{Polynomial Invariant}\\
$-t^{-2} - 2t^{-1} + 2 + 4t - t^{3}$\\[0.1in]
\textcolor{blue}{Real grid homology - hat version}\\
$(-1,-1)^{2}\oplus (-2,-1)\oplus (0,0)^{2}\oplus (1,0)^{4}\oplus (3,1)$\\[0.1in]
\textcolor{blue}{Real grid homology - minus version}\\
$U^{\infty}_{(-1,-1)}\oplus U^{\infty}_{(1,0)}\oplus U_{(-1,-1)}\oplus \bigg(U_{(1,0)}\bigg)^{2}\oplus U^{2}_{(3,1)}$
} \tabularnewline
  \hline
  \parbox[c]{\linewidth}{\centering $(7_6^{OO})'$\\{\tiny Second symmetry}} & \centering \includegraphics[width=0.25\textwidth]{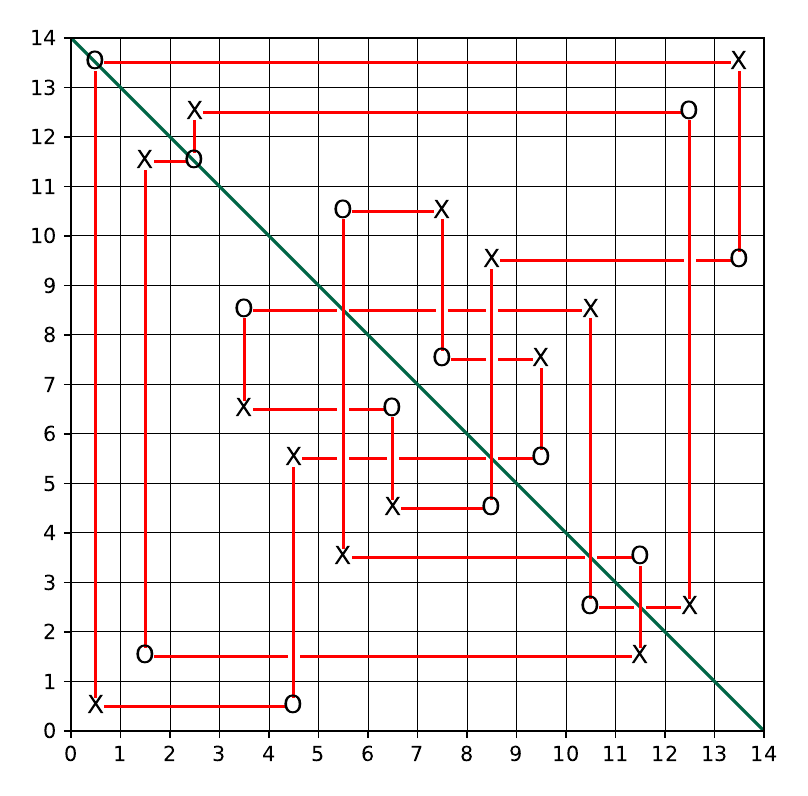} & \parbox[c]{\linewidth}{\centering\tiny
\textcolor{blue}{Polynomial Invariant}\\
$t^{-2} - 2 + 2t^{2} + t^{3}$
} \tabularnewline
  \hline
  \centering $7_7^{OO}$ & \centering \includegraphics[width=0.25\textwidth]{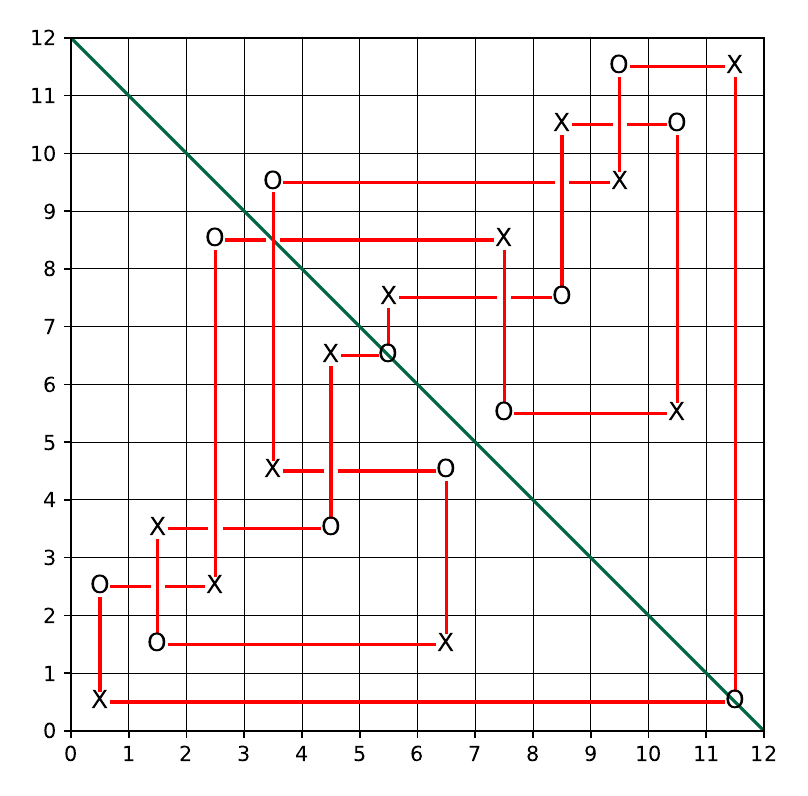} & \parbox[c]{\linewidth}{\centering\tiny
\textcolor{blue}{Polynomial Invariant}\\
$-t^{-2} - 2t^{-1} + 2 + 4t - t^{3}$\\[0.1in]
\textcolor{blue}{Real grid homology - hat version}\\
$(-1,-1)^{2}\oplus (-2,-1)\oplus (0,0)^{2}\oplus (1,0)^{4}\oplus (3,1)$\\[0.1in]
\textcolor{blue}{Real grid homology - minus version}\\
$U^{\infty}_{(-1,-1)}\oplus U^{\infty}_{(1,0)}\oplus U_{(-1,-1)}\oplus \bigg(U_{(1,0)}\bigg)^{2}\oplus U^{2}_{(3,1)}$
} \tabularnewline
  \hline
  \parbox[c]{\linewidth}{\centering $(7_7^{OO})'$\\{\tiny Second symmetry}} & \centering \includegraphics[width=0.25\textwidth]{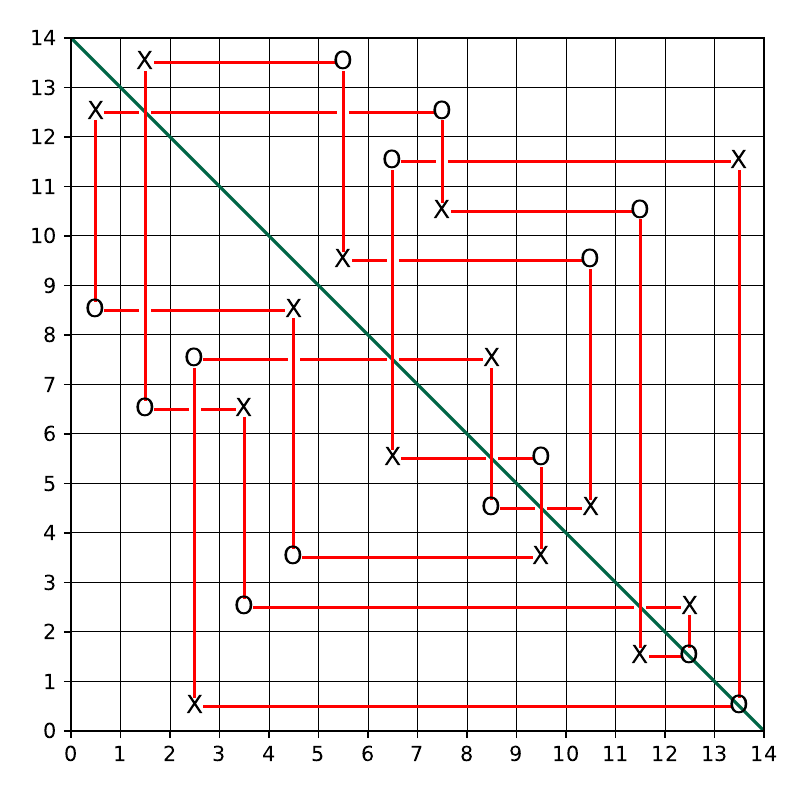} & \parbox[c]{\linewidth}{\centering\tiny
\textcolor{blue}{Polynomial Invariant}\\
$-t^{-2} - 2t^{-1} + 2 + 4t - t^{3}$
} \tabularnewline
  \hline
  \parbox[c]{\linewidth}{\centering $8_1^{OO}$\\{\tiny Also twisted${}_6^{OO}$}} & \centering \includegraphics[width=0.25\textwidth]{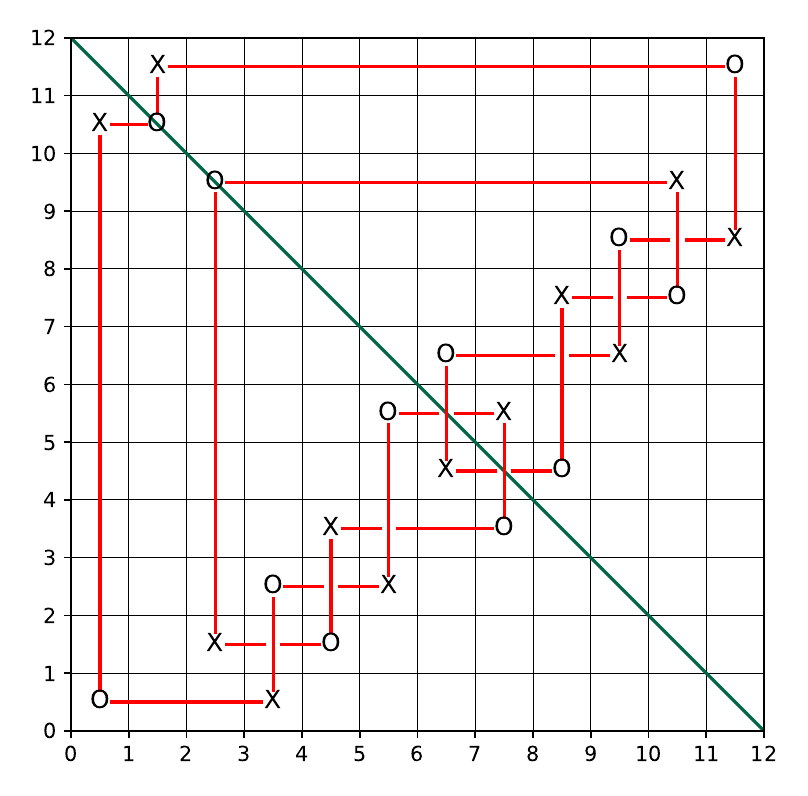} & \parbox[c]{\linewidth}{\centering\tiny
\textcolor{blue}{Polynomial Invariant}\\
$-3t^{-1} - 2 + 4t + 3t^{2}$\\[0.1in]
\textcolor{blue}{Real grid homology - hat version}\\
$(-1,-1)^{3}\oplus (0,-1)^{2}\oplus (1,0)^{4}\oplus (2,0)^{3}$\\[0.1in]
\textcolor{blue}{Real grid homology - minus version}\\
$U^{\infty}_{(-1,-1)}\oplus U^{\infty}_{(1,0)}\oplus \bigg(U_{(0,-1)}\bigg)^{2}\oplus \bigg(U_{(2,0)}\bigg)^{3}$
} \tabularnewline
  \hline
  \centering $8_2^{OO}$ & \centering \includegraphics[width=0.25\textwidth]{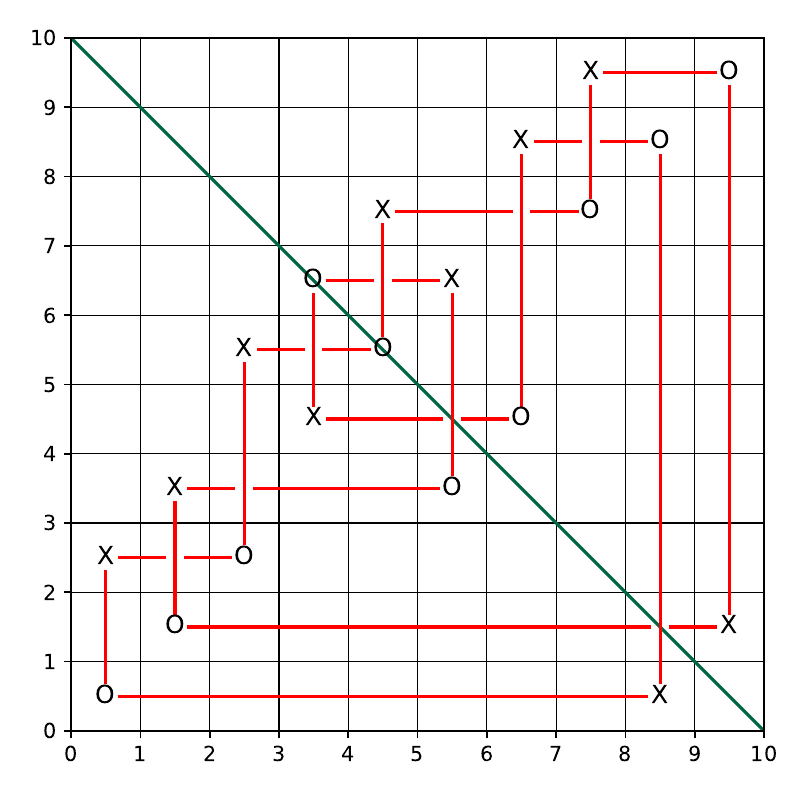} & \parbox[c]{\linewidth}{\centering\tiny
\textcolor{blue}{Polynomial Invariant}\\
$-t^{-3} + 4t^{-1} + 2 - 4t - 2t^{2} + 2t^{3} + t^{4}$\\[0.1in]
\textcolor{blue}{Real grid homology - hat version}\\
$(-1,0)^{4}\oplus (-3,-1)\oplus (0,0)^{2}\oplus (1,1)^{4}\oplus (2,1)^{2}\oplus (3,2)^{2}\oplus (4,2)$\\[0.1in]
\textcolor{blue}{Real grid homology - minus version}\\
$U^{\infty}_{(1,1)}\oplus U^{\infty}_{(3,2)}\oplus \bigg(U_{(0,0)}\bigg)^{2}\oplus \bigg(U_{(2,1)}\bigg)^{2}\oplus U_{(4,2)}\oplus U^{2}_{(-1,0)}\oplus U^{2}_{(1,1)}$
} \tabularnewline
  \hline
  \centering $8_3^{OO}$ & \centering \includegraphics[width=0.25\textwidth]{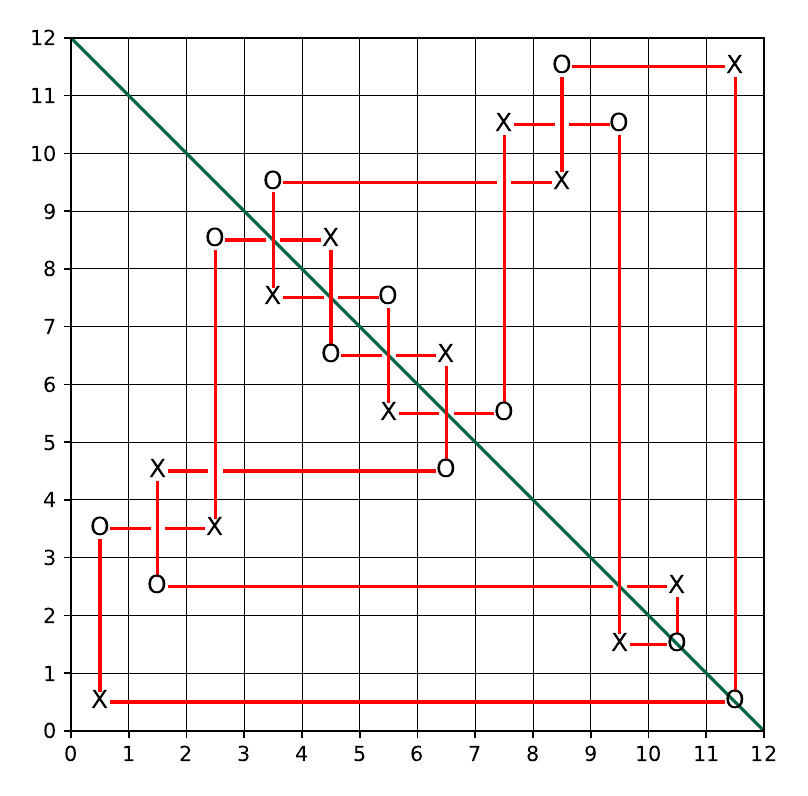} & \parbox[c]{\linewidth}{\centering\tiny
\textcolor{blue}{Polynomial Invariant}\\
$1 + t$\\[0.1in]
\textcolor{blue}{Real grid homology - hat version}\\
$(0,0)\oplus (1,0)$
} \tabularnewline
  \hline
  \centering $8_8^{OO}$ & \centering \includegraphics[width=0.25\textwidth]{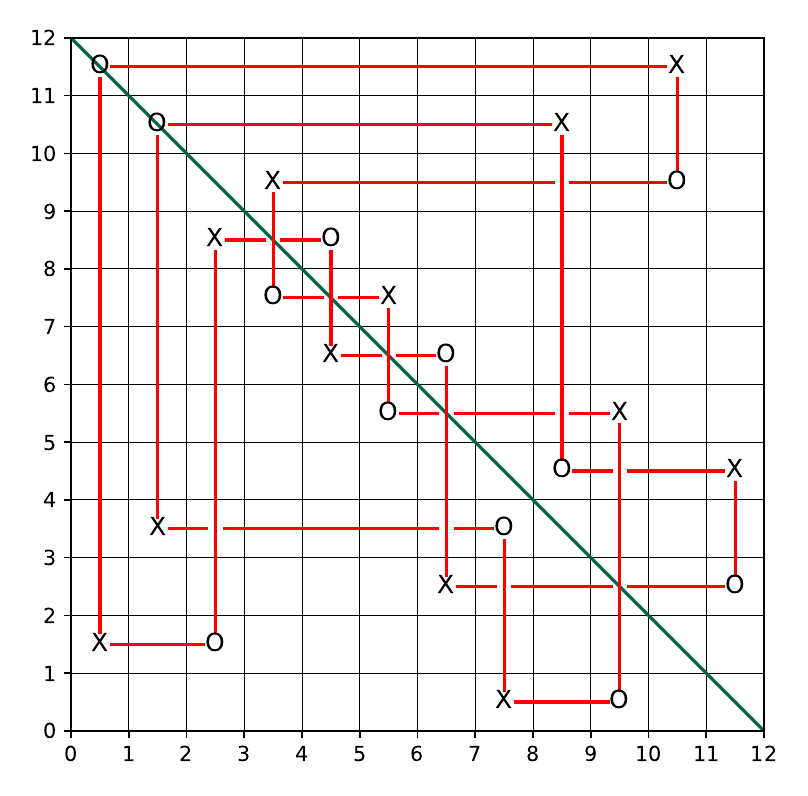} & \parbox[c]{\linewidth}{\centering\tiny
\textcolor{blue}{Polynomial Invariant}\\
$1 + t$\\[0.1in]
\textcolor{blue}{Real grid homology - hat version}\\
$(0,0)\oplus (1,0)$\\[0.1in]
\textcolor{blue}{Real grid homology - minus version}\\
$U^{\infty}_{(0,0)}\oplus U^{\infty}_{(1,0)}$
} \tabularnewline
  \hline
  \parbox[c]{\linewidth}{\centering $(8_8^{OO})'$\\{\tiny Second symmetry}} & \centering \includegraphics[width=0.25\textwidth]{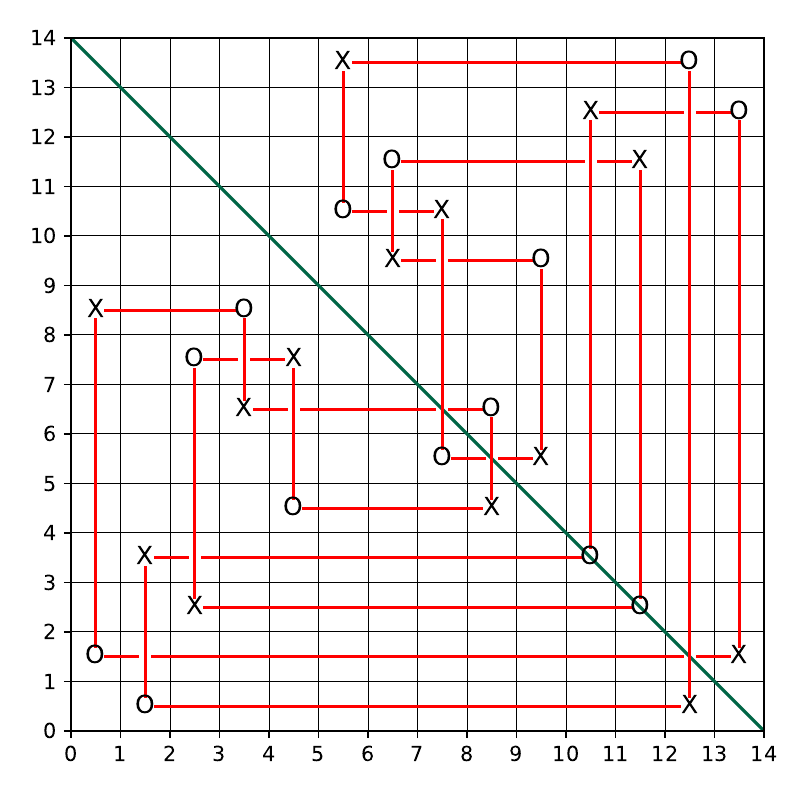} & \parbox[c]{\linewidth}{\centering\tiny
\textcolor{blue}{Polynomial Invariant}\\
$-2t^{-2} + 7 + 3t - 4t^{2} - 2t^{3}$
} \tabularnewline
  \hline
  \centering $8_9^{OO}$ & \centering \includegraphics[width=0.25\textwidth]{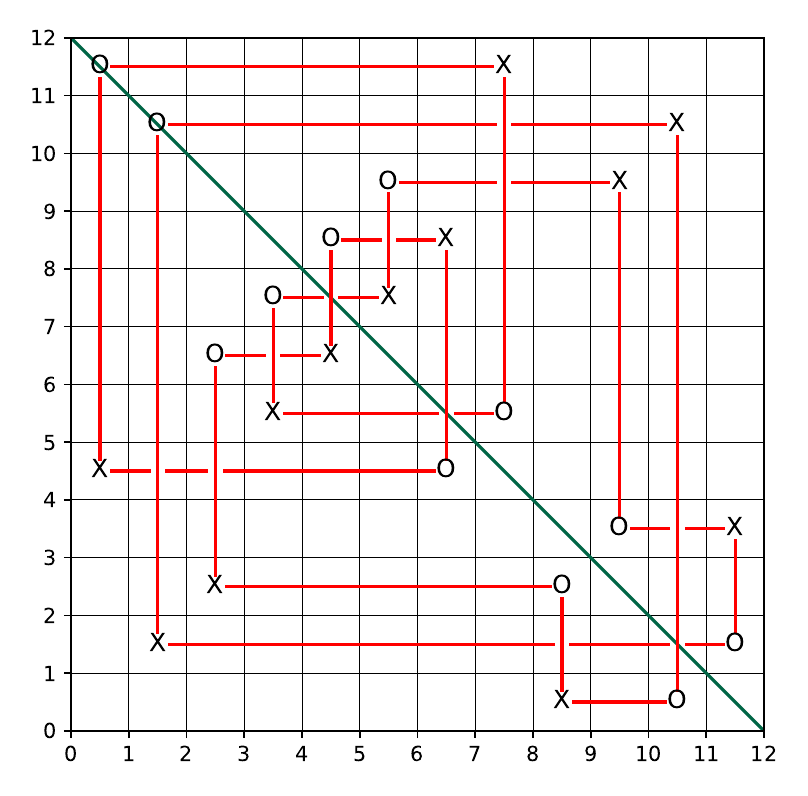} & \parbox[c]{\linewidth}{\centering\tiny
\textcolor{blue}{Polynomial Invariant}\\
$t^{-3} - 4t^{-1} + 6t + 2t^{2} - 2t^{3} - t^{4}$\\[0.1in]
\textcolor{blue}{Real grid homology - hat version}\\
$(-1,-1)^{4}\oplus (-3,-2)\oplus (1,0)^{6}\oplus (2,0)^{2}\oplus (3,1)^{2}\oplus (4,1)$\\[0.1in]
\textcolor{blue}{Real grid homology - minus version}\\
$U^{\infty}_{(-1,-1)}\oplus U^{\infty}_{(1,0)}\oplus \bigg(U_{(2,0)}\bigg)^{2}\oplus U_{(4,1)}\oplus U^{2}_{(-1,-1)}\oplus \bigg(U^{2}_{(1,0)}\bigg)^{2}\oplus U^{2}_{(3,1)}$
} \tabularnewline
  \hline
  \centering $8_{12}^{OO}$ & \centering \includegraphics[width=0.25\textwidth]{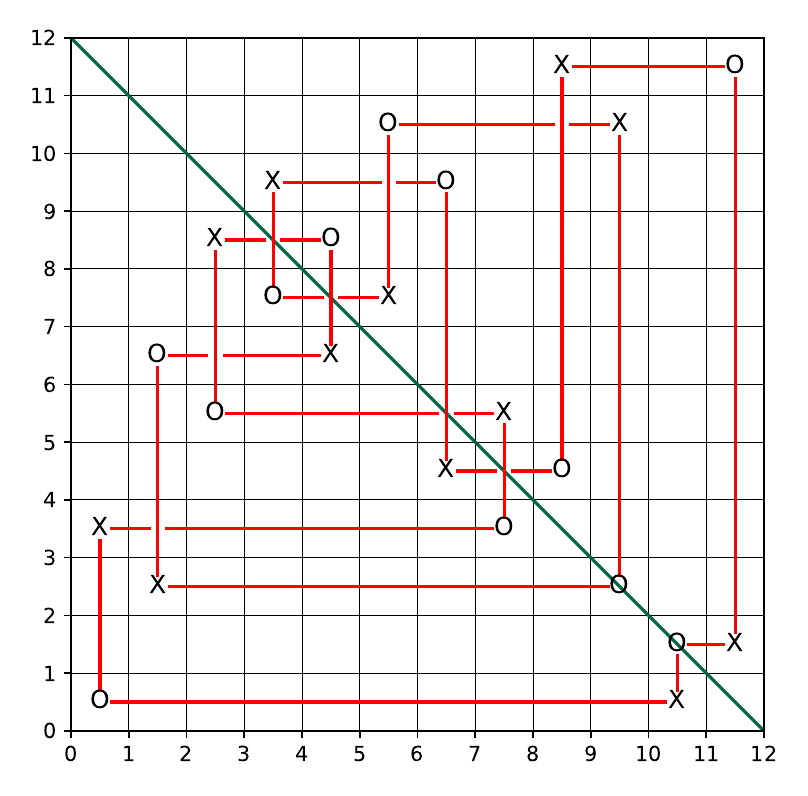} & \parbox[c]{\linewidth}{\centering\tiny
\textcolor{blue}{Polynomial Invariant}\\
$t^{-2} - 2 + 2t^{2} + t^{3}$\\[0.1in]
\textcolor{blue}{Real grid homology - hat version}\\
$(-2,-2)\oplus (0,-1)^{2}\oplus (2,0)^{2}\oplus (3,0)$\\[0.1in]
\textcolor{blue}{Real grid homology - minus version}\\
$U^{\infty}_{(-2,-2)}\oplus U^{\infty}_{(0,-1)}\oplus U_{(3,0)}\oplus U^{2}_{(2,0)}$
} \tabularnewline
  \hline
  \centering $8_{18}^{OO}$ & \centering \includegraphics[width=0.25\textwidth]{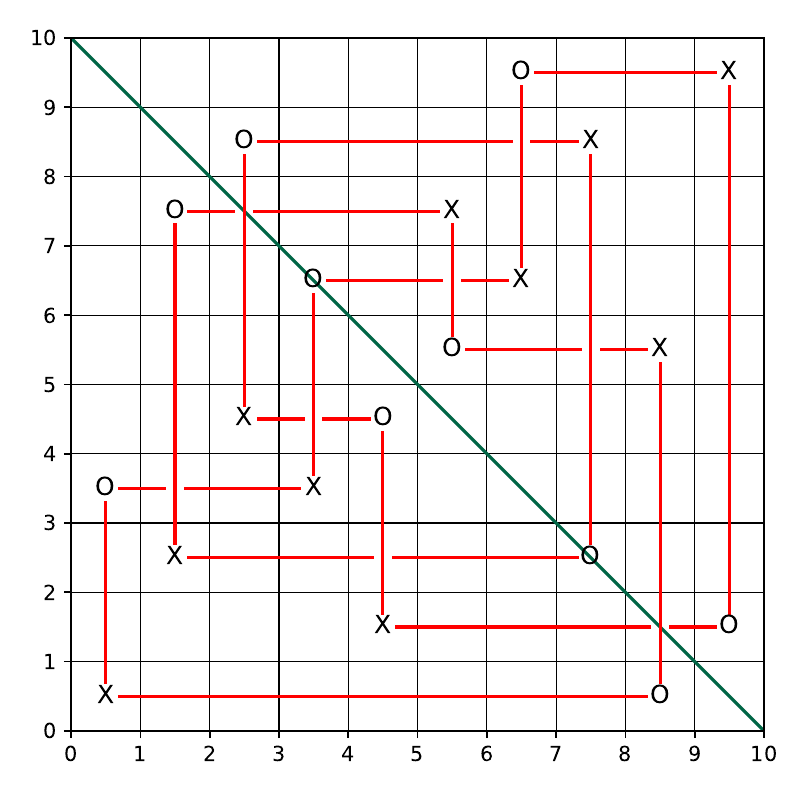} & \parbox[c]{\linewidth}{\centering\tiny
\textcolor{blue}{Polynomial Invariant}\\
$t^{-3} - 3t^{-1} + 1 + 5t + t^{2} - 2t^{3} - t^{4}$\\[0.1in]
\textcolor{blue}{Real grid homology - hat version}\\
$(-1,-1)^{3}\oplus (-3,-2)\oplus (0,0)\oplus (1,0)^{5}\oplus (2,0)^{2}\oplus (2,1)\oplus (3,1)^{2}\oplus (4,1)$\\[0.1in]
\textcolor{blue}{Real grid homology - minus version}\\
$U^{\infty}_{(-1,-1)}\oplus U^{\infty}_{(1,0)}\oplus U_{(1,0)}\oplus \bigg(U_{(2,0)}\bigg)^{2}\oplus U_{(3,1)}\oplus U_{(4,1)}\oplus U^{2}_{(-1,-1)}\oplus U^{2}_{(1,0)}$
} \tabularnewline
  \hline
  \centering $8_{19}^{OO}$ & \centering \includegraphics[width=0.25\textwidth]{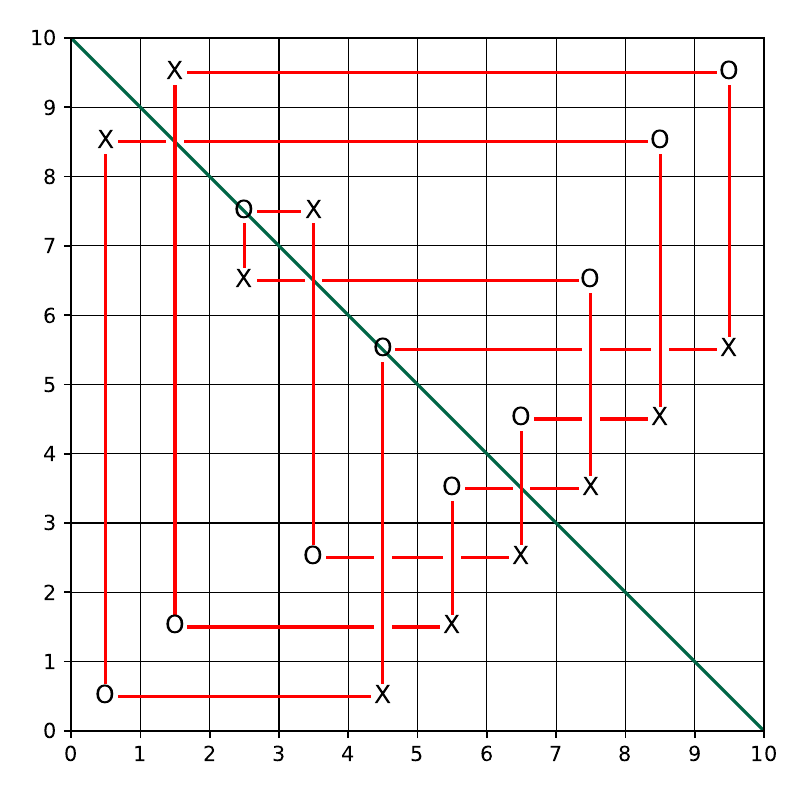} & \parbox[c]{\linewidth}{\centering\tiny
\textcolor{blue}{Polynomial Invariant}\\
$t^{-3} + 2t^{-2} + t^{-1} - 1 - t + t^{2} - t^{4}$\\[0.1in]
\textcolor{blue}{Real grid homology - hat version}\\
$(-1,0)\oplus (-2,0)^{2}\oplus (-3,0)\oplus (0,1)\oplus (1,1)\oplus (2,2)\oplus (4,3)$\\[0.1in]
\textcolor{blue}{Real grid homology - minus version}\\
$U^{\infty}_{(1,1)}\oplus U^{\infty}_{(4,3)}\oplus U_{(-1,0)}\oplus U_{(-2,0)}\oplus U^{2}_{(2,2)}$
} \tabularnewline
  \hline
  \centering $8_{20}^{OO}$ & \centering \includegraphics[width=0.25\textwidth]{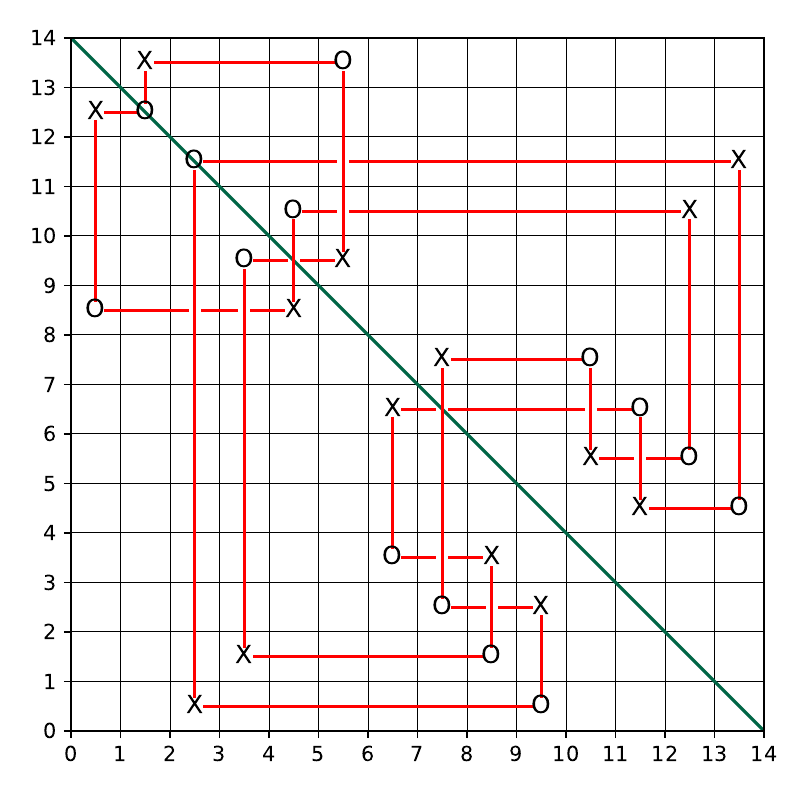} & \parbox[c]{\linewidth}{\centering\tiny
\textcolor{blue}{Polynomial Invariant}\\
$-t^{-2} - t^{-1} + 3 + 3t - t^{2} - t^{3}$
} \tabularnewline
  \hline
  \centering $9_3^{OO}$ & \centering \includegraphics[width=0.25\textwidth]{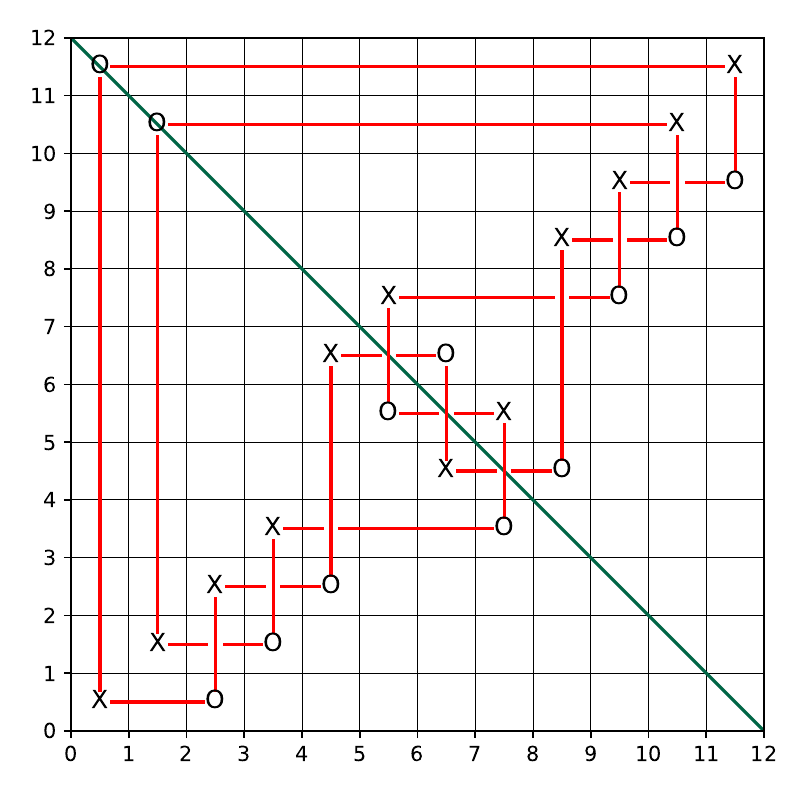} & \parbox[c]{\linewidth}{\centering\tiny
\textcolor{blue}{Polynomial Invariant}\\
$t^{-2} + 2t^{-1} - 2t + t^{3}$\\[0.1in]
\textcolor{blue}{Real grid homology - hat version}\\
$(-1,0)^{2}\oplus (-2,0)\oplus (1,1)^{2}\oplus (3,2)$\\[0.1in]
\textcolor{blue}{Real grid homology - minus version}\\
$U^{\infty}_{(1,1)}\oplus U^{\infty}_{(3,2)}\oplus U_{(-1,0)}\oplus U^{2}_{(1,1)}$
} \tabularnewline
  \hline
  \centering $9_6^{OO}$ & \centering \includegraphics[width=0.25\textwidth]{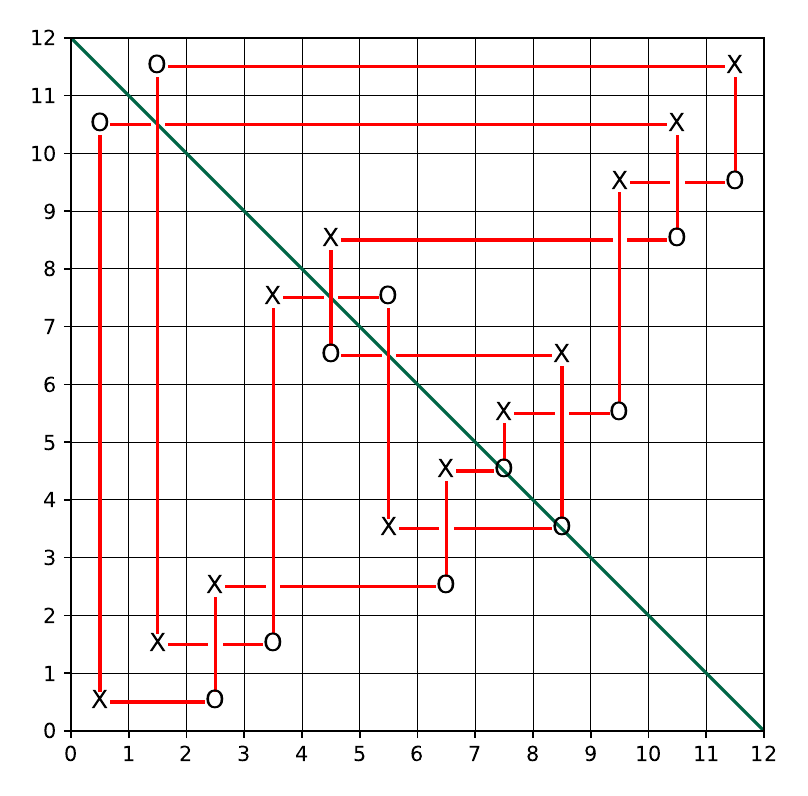} & \parbox[c]{\linewidth}{\centering\tiny
\textcolor{blue}{Polynomial Invariant}\\
$2t^{-2} + 3t^{-1} - 2 - 4t + t^{2} + 2t^{3}$\\[0.1in]
\textcolor{blue}{Real grid homology - hat version}\\
$(-1,0)^{3}\oplus (-2,0)^{2}\oplus (0,1)^{2}\oplus (1,1)^{4}\oplus (2,2)\oplus (3,2)^{2}$\\[0.1in]
\textcolor{blue}{Real grid homology - minus version}\\
$U^{\infty}_{(1,1)}\oplus U^{\infty}_{(3,2)}\oplus \bigg(U_{(-1,0)}\bigg)^{2}\oplus \bigg(U_{(1,1)}\bigg)^{2}\oplus U_{(3,2)}\oplus U^{2}_{(1,1)}$
} \tabularnewline
  \hline
  \centering $9_9^{OO}$ & \centering \includegraphics[width=0.25\textwidth]{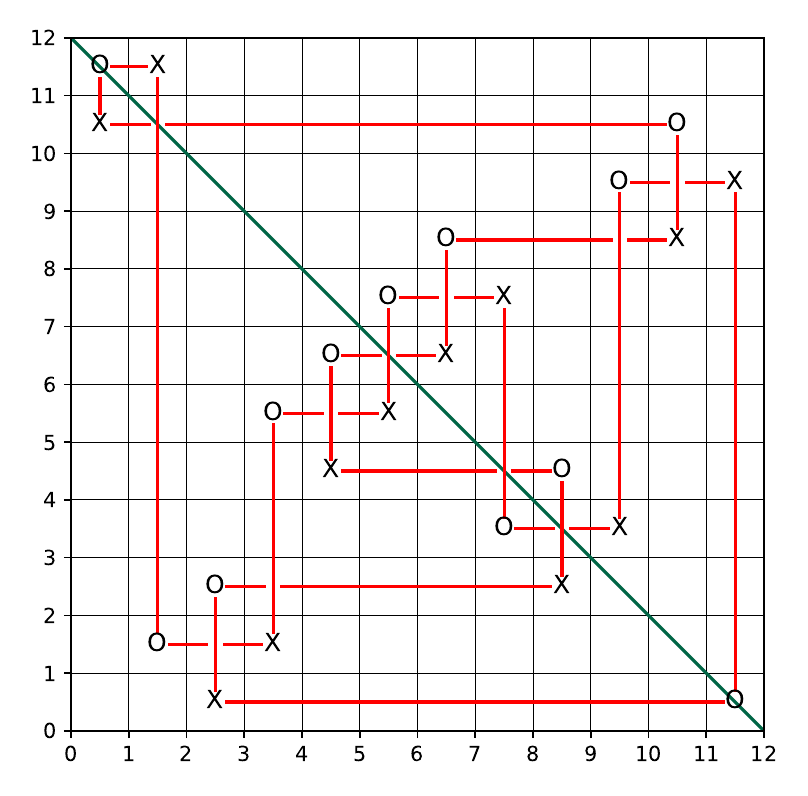} & \parbox[c]{\linewidth}{\centering\tiny
\textcolor{blue}{Polynomial Invariant}\\
$2t^{-2} + 4t^{-1} - 1 - 5t + 2t^{3}$\\[0.1in]
\textcolor{blue}{Real grid homology - hat version}\\
$(-1,0)^{4}\oplus (-2,0)^{2}\oplus (0,1)\oplus (1,1)^{5}\oplus (3,2)^{2}$\\[0.1in]
\textcolor{blue}{Real grid homology - minus version}\\
$U^{\infty}_{(1,1)}\oplus U^{\infty}_{(3,2)}\oplus \bigg(U_{(-1,0)}\bigg)^{2}\oplus U_{(1,1)}\oplus \bigg(U^{2}_{(1,1)}\bigg)^{2}\oplus U^{2}_{(3,2)}$
} \tabularnewline
  \hline
  \centering $9_{28}$ & \centering \includegraphics[width=0.25\textwidth]{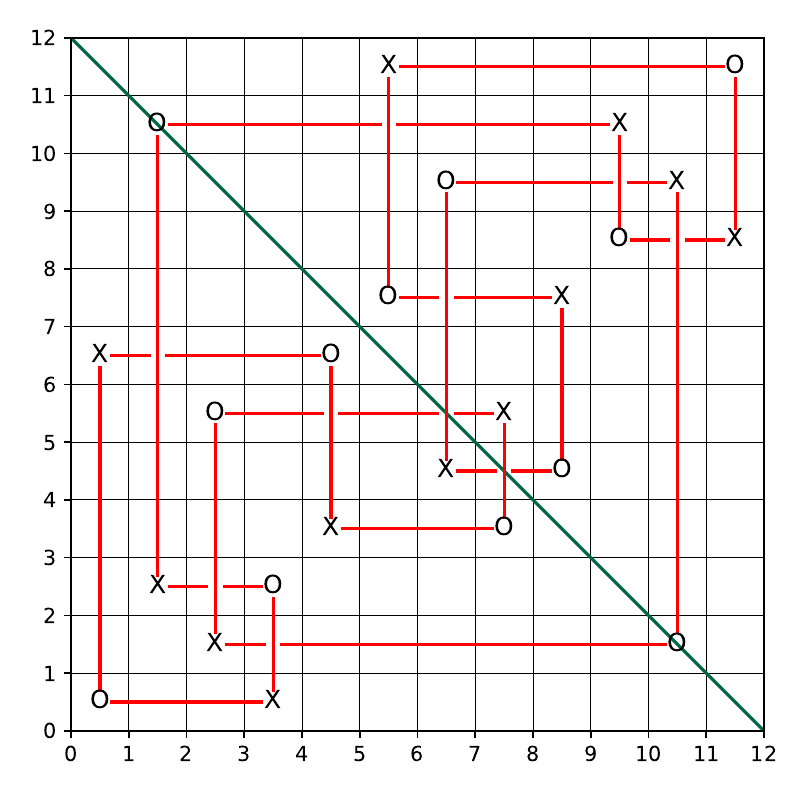} & \parbox[c]{\linewidth}{\centering\tiny
\textcolor{blue}{Polynomial Invariant}\\
$t^{-3} + 2t^{-2} - 3t^{-1} - 5 + 3t + 5t^{2} - t^{4}$\\[0.1in]
\textcolor{blue}{Real grid homology - hat version}\\
$(-1,-1)^{3}\oplus (-2,-2)^{2}\oplus (-3,-2)\oplus (0,-1)^{5}\oplus (1,0)^{3}\oplus (2,0)^{5}\oplus (4,1)$\\[0.1in]
\textcolor{blue}{Real grid homology - minus version}\\
$U^{\infty}_{(-2,-2)}\oplus U^{\infty}_{(0,-1)}\oplus U_{(-2,-2)}\oplus \bigg(U_{(0,-1)}\bigg)^{3}\oplus \bigg(U_{(2,0)}\bigg)^{3}\oplus U^{2}_{(2,0)}\oplus U^{2}_{(4,1)}$
} \tabularnewline
  \hline
  \centering $9_{42}^{OO}$ & \centering \includegraphics[width=0.25\textwidth]{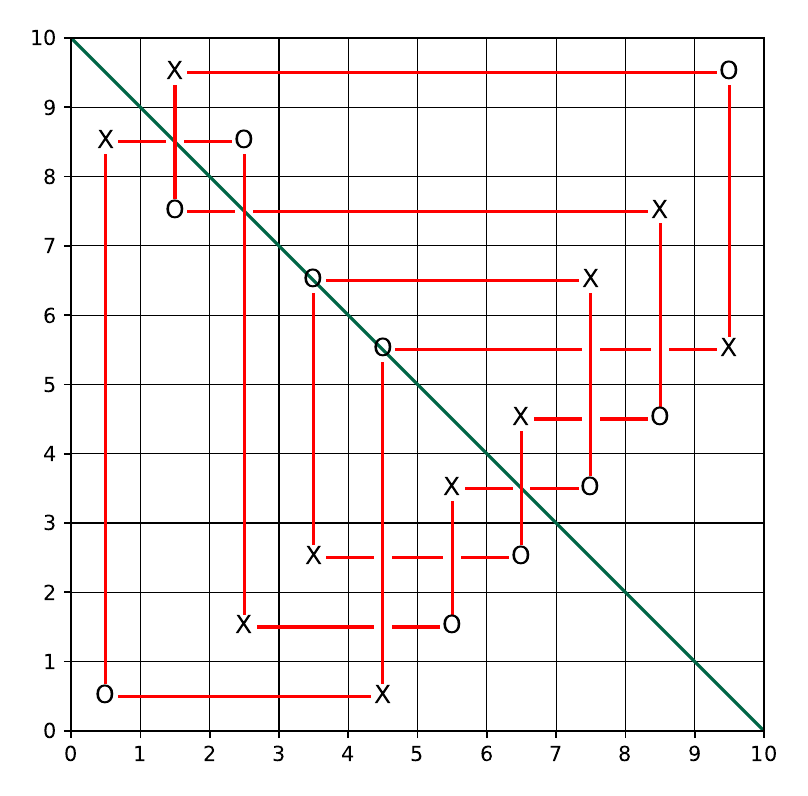} & \parbox[c]{\linewidth}{\centering\tiny
\textcolor{blue}{Polynomial Invariant}\\
$-t^{-2} - t^{-1} + 3 + 3t - t^{2} - t^{3}$\\[0.1in]
\textcolor{blue}{Real grid homology - hat version}\\
$(-1,-1)\oplus (-2,-1)\oplus (0,0)^{3}\oplus (1,0)^{3}\oplus (2,1)\oplus (3,1)$\\[0.1in]
\textcolor{blue}{Real grid homology - minus version}\\
$U^{\infty}_{(0,0)}\oplus U^{\infty}_{(1,0)}\oplus U_{(-1,-1)}\oplus \bigg(U_{(1,0)}\bigg)^{2}\oplus U_{(3,1)}$
} \tabularnewline
  \hline
  \centering $9_{46}^{OO}$ & \centering \includegraphics[width=0.25\textwidth]{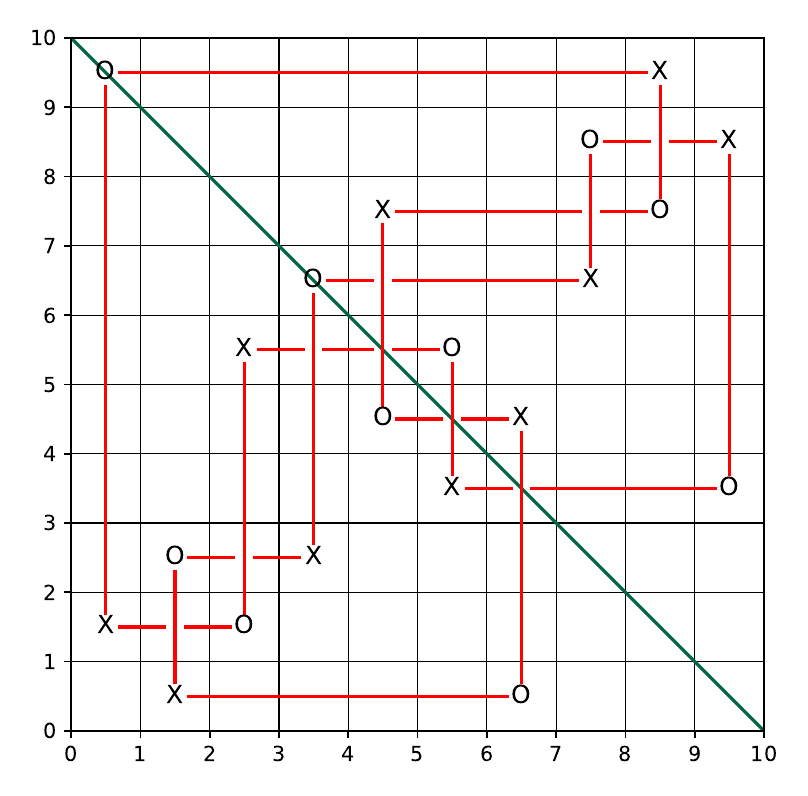} & \parbox[c]{\linewidth}{\centering\tiny
\textcolor{blue}{Polynomial Invariant}\\
$-2t^{-1} - 1 + 3t + 2t^{2}$\\[0.1in]
\textcolor{blue}{Real grid homology - hat version}\\
$(-1,-1)^{2}\oplus (0,-1)\oplus (1,0)^{3}\oplus (2,0)^{2}$\\[0.1in]
\textcolor{blue}{Real grid homology - minus version}\\
$U^{\infty}_{(-1,-1)}\oplus U^{\infty}_{(1,0)}\oplus U_{(0,-1)}\oplus \bigg(U_{(2,0)}\bigg)^{2}$
} \tabularnewline
  \hline
  \centering $10_{35}^{OO}$ & \centering \includegraphics[width=0.25\textwidth]{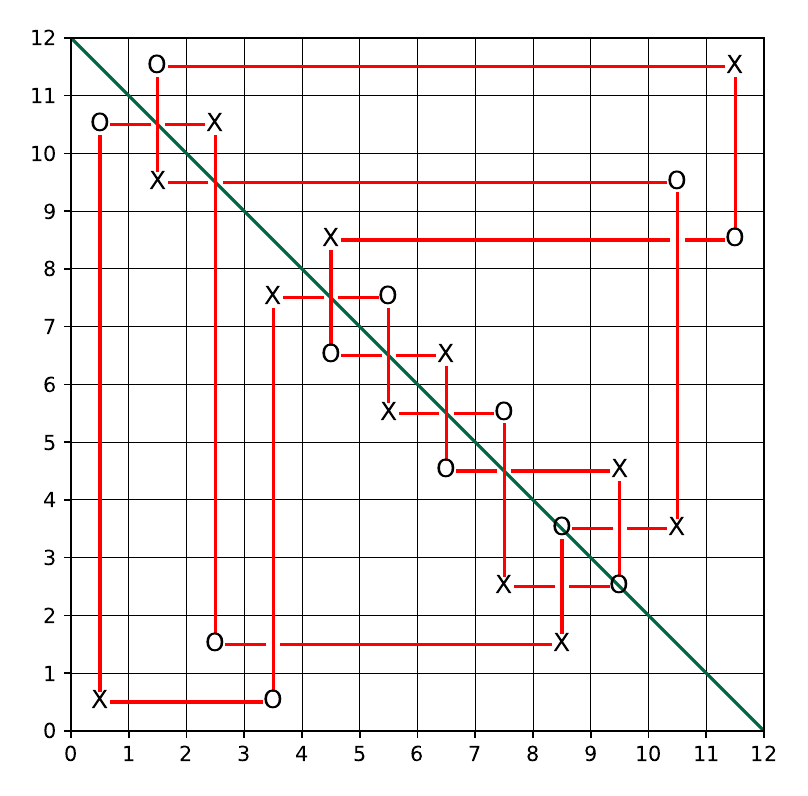} & \parbox[c]{\linewidth}{\centering\tiny
\textcolor{blue}{Polynomial Invariant}\\
$-2t^{-1} - 1 + 3t + 2t^{2}$\\[0.1in]
\textcolor{blue}{Real grid homology - hat version}\\
$(-1,-1)^{2}\oplus (0,-1)\oplus (1,0)^{3}\oplus (2,0)^{2}$\\[0.1in]
\textcolor{blue}{Real grid homology - minus version}\\
$U^{\infty}_{(-1,-1)}\oplus U^{\infty}_{(1,0)}\oplus U_{(0,-1)}\oplus \bigg(U_{(2,0)}\bigg)^{2}$
} \tabularnewline
  \hline
  \centering $10_{75}^{OO}$ & \centering \includegraphics[width=0.25\textwidth]{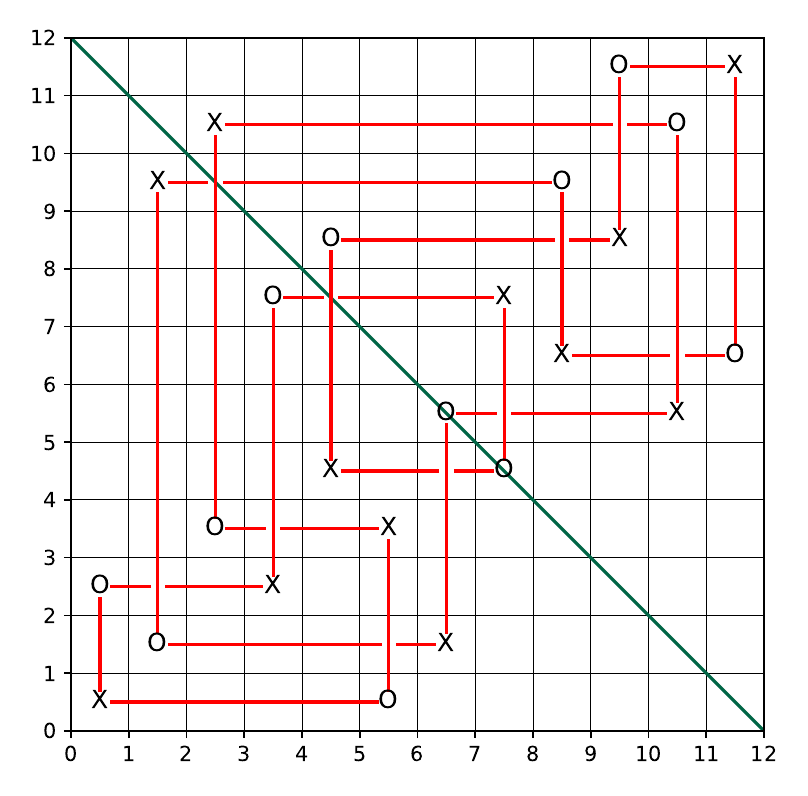} & \parbox[c]{\linewidth}{\centering\tiny
\textcolor{blue}{Polynomial Invariant}\\
$t^{-3} - 6t^{-1} - 2 + 8t + 4t^{2} - 2t^{3} - t^{4}$\\[0.1in]
\textcolor{blue}{Real grid homology - hat version}\\
$(-1,-1)^{6}\oplus (-3,-2)\oplus (0,-1)^{2}\oplus (1,0)^{8}\oplus (2,0)^{4}\oplus (3,1)^{2}\oplus (4,1)$\\[0.1in]
\textcolor{blue}{Real grid homology - minus version}\\
$U^{\infty}_{(-1,-1)}\oplus U^{\infty}_{(1,0)}\oplus \bigg(U_{(0,-1)}\bigg)^{2}\oplus \bigg(U_{(2,0)}\bigg)^{4}\oplus U_{(4,1)}\oplus U^{2}_{(-1,-1)}\oplus \bigg(U^{2}_{(1,0)}\bigg)^{2}\oplus U^{2}_{(3,1)}$
} \tabularnewline
  \hline
  \centering $10_{99}$ & \centering \includegraphics[width=0.25\textwidth]{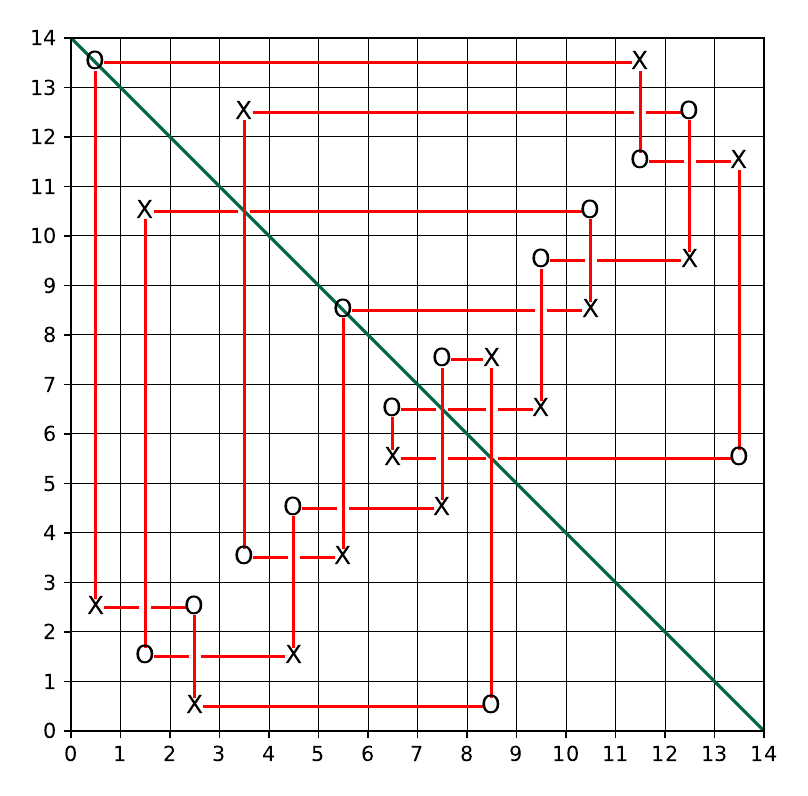} & \parbox[c]{\linewidth}{\centering\tiny
\textcolor{blue}{Polynomial Invariant}\\
$2t^{-2} + 3t^{-1} - 2 - 4t + t^{2} + 2t^{3}$
} \tabularnewline
  \hline
  \centering $10_{129}^{OO}$ & \centering \includegraphics[width=0.25\textwidth]{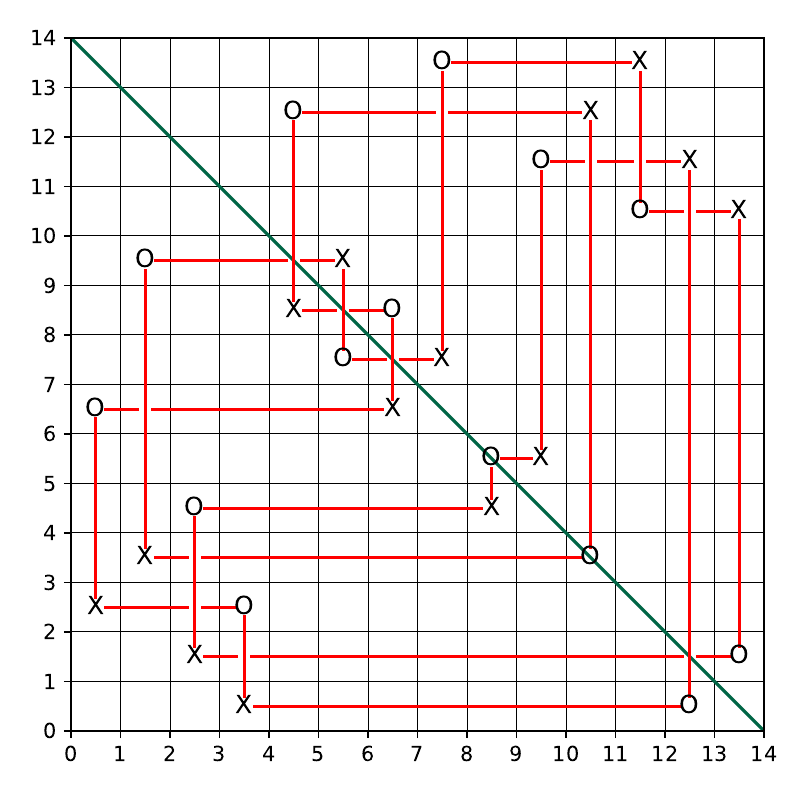} & \parbox[c]{\linewidth}{\centering\tiny
\textcolor{blue}{Polynomial Invariant}\\
$-2t^{-1} - 1 + 3t + 2t^{2}$
} \tabularnewline
  \hline
  \centering $3_1 \# 3_1$ & \centering \includegraphics[width=0.25\textwidth]{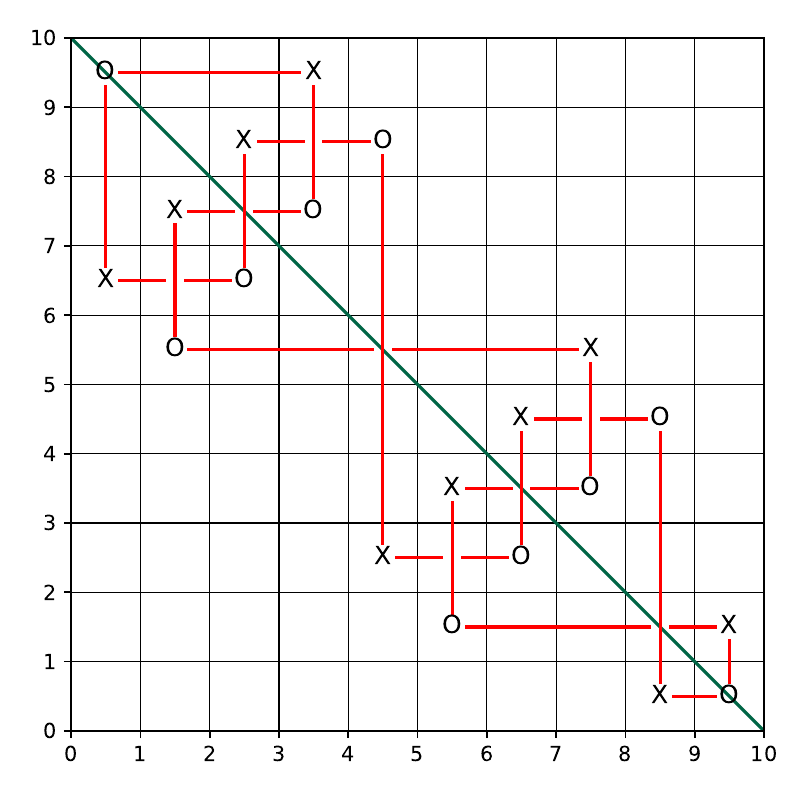} & \parbox[c]{\linewidth}{\centering\tiny
\textcolor{blue}{Polynomial Invariant}\\
$t^{-2} + 3t^{-1} + 1 - 3t - t^{2} + t^{3}$\\[0.1in]
\textcolor{blue}{Real grid homology - hat version}\\
$(-1,0)^{3}\oplus (-2,0)\oplus (0,0)\oplus (1,1)^{3}\oplus (2,1)\oplus (3,2)$\\[0.1in]
\textcolor{blue}{Real grid homology - minus version}\\
$U^{\infty}_{(1,1)}\oplus U^{\infty}_{(3,2)}\oplus U_{(-1,0)}\oplus U_{(0,0)}\oplus U_{(2,1)}\oplus U^{2}_{(1,1)}$
} \tabularnewline
  \hline
  \centering $3_1 \# 4_1$ & \centering \includegraphics[width=0.25\textwidth]{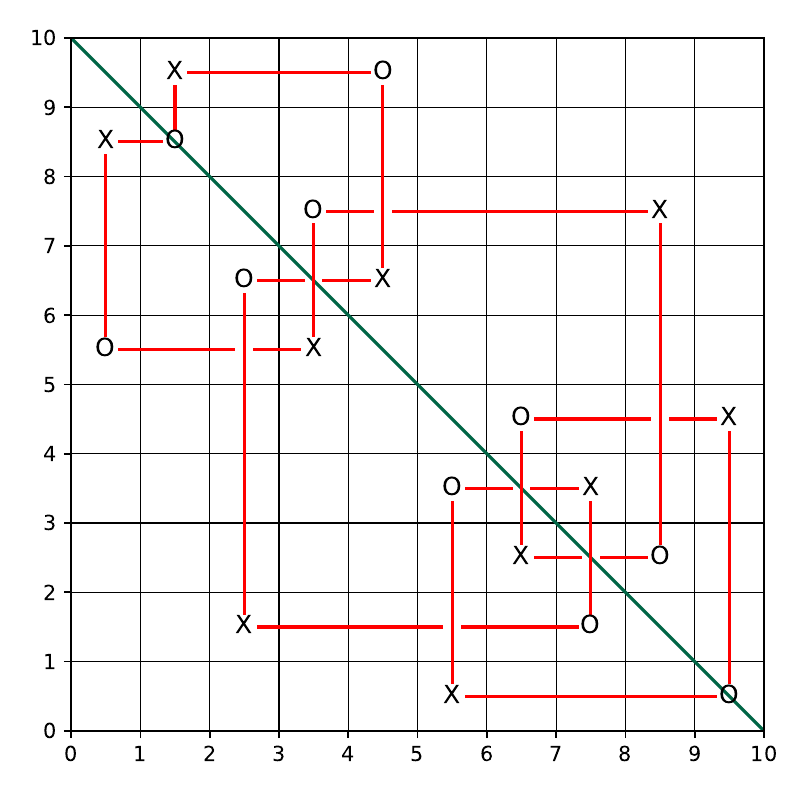} & \parbox[c]{\linewidth}{\centering\tiny
\textcolor{blue}{Polynomial Invariant}\\
$-t^{-2} - t^{-1} + 3 + 3t - t^{2} - t^{3}$\\[0.1in]
\textcolor{blue}{Real grid homology - hat version}\\
$(-1,-1)\oplus (-2,-1)\oplus (0,0)^{3}\oplus (1,0)^{3}\oplus (2,1)\oplus (3,1)$\\[0.1in]
\textcolor{blue}{Real grid homology - minus version}\\
$U^{\infty}_{(0,0)}\oplus U^{\infty}_{(1,0)}\oplus U_{(-1,-1)}\oplus \bigg(U_{(1,0)}\bigg)^{2}\oplus U_{(3,1)}$
} \tabularnewline
  \hline
  \centering $(3_1 \# 4_1)'$ & \centering \includegraphics[width=0.25\textwidth]{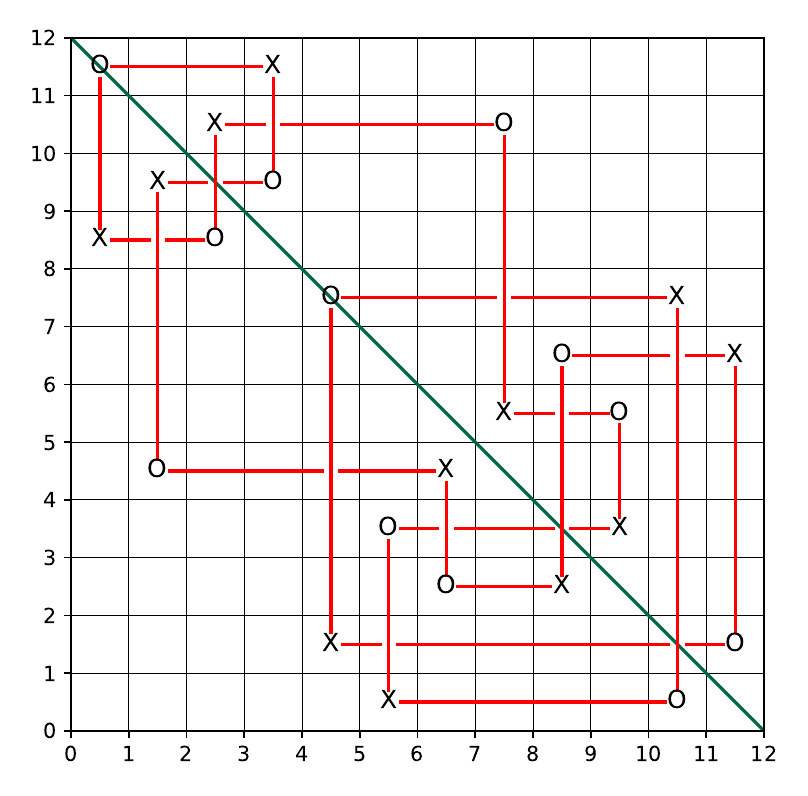} & \parbox[c]{\linewidth}{\centering\tiny
\textcolor{blue}{Polynomial Invariant}\\
$t^{-2} + 3t^{-1} + 1 - 3t - t^{2} + t^{3}$\\[0.1in]
\textcolor{blue}{Real grid homology - hat version}\\
$(-1,0)^{3}\oplus (-2,0)\oplus (0,0)\oplus (1,1)^{3}\oplus (2,1)\oplus (3,2)$\\[0.1in]
\textcolor{blue}{Real grid homology - minus version}\\
$U^{\infty}_{(1,1)}\oplus U^{\infty}_{(3,2)}\oplus U_{(-1,0)}\oplus U_{(0,0)}\oplus U_{(2,1)}\oplus U^{2}_{(1,1)}$
} \tabularnewline
  \hline
  \centering $3_1 \# 5_1$ & \centering \includegraphics[width=0.25\textwidth]{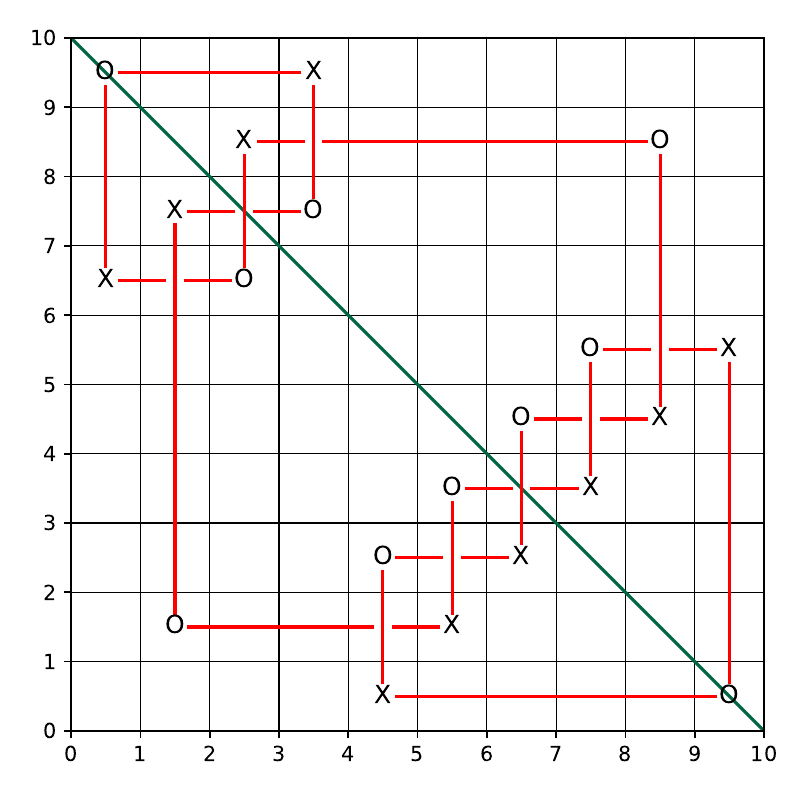} & \parbox[c]{\linewidth}{\centering\tiny
\textcolor{blue}{Polynomial Invariant}\\
$t^{-3} + 3t^{-2} + t^{-1} - 4 - 2t + 3t^{2} + t^{3} - t^{4}$\\[0.1in]
\textcolor{blue}{Real grid homology - hat version}\\
$(-1,0)\oplus (-2,0)^{3}\oplus (-3,0)\oplus (0,1)^{4}\oplus (1,1)^{2}\oplus (2,2)^{3}\oplus (3,2)\oplus (4,3)$\\[0.1in]
\textcolor{blue}{Real grid homology - minus version}\\
$U^{\infty}_{(2,2)}\oplus U^{\infty}_{(4,3)}\oplus U_{(-1,0)}\oplus U_{(-2,0)}\oplus \bigg(U_{(1,1)}\bigg)^{2}\oplus U_{(3,2)}\oplus U^{2}_{(0,1)}\oplus U^{2}_{(2,2)}$
} \tabularnewline
  \hline
  \centering $3_1 \# 5_2$ & \centering \includegraphics[width=0.25\textwidth]{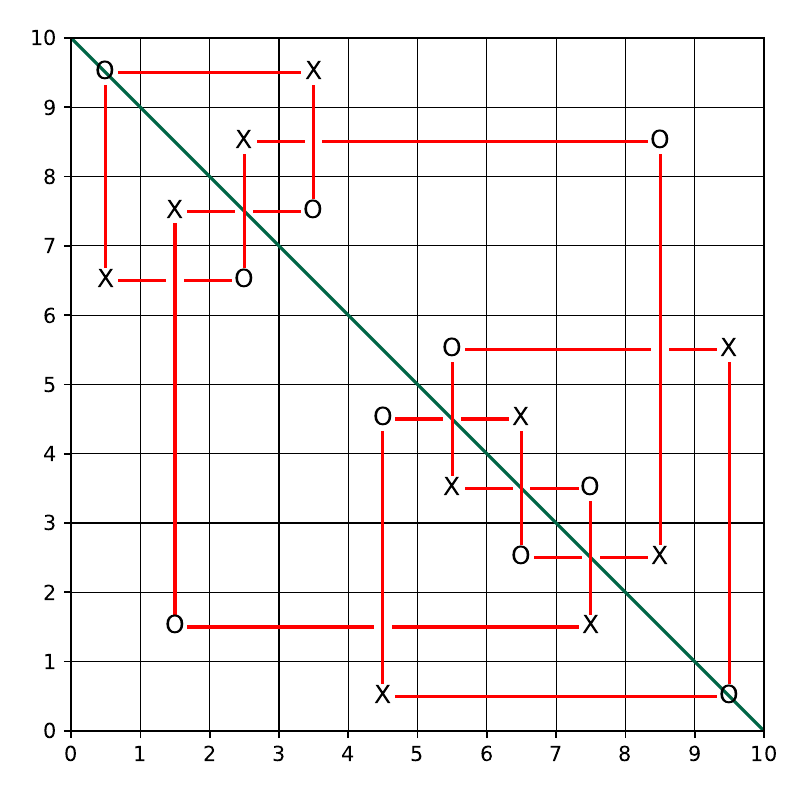} & \parbox[c]{\linewidth}{\centering\tiny
\textcolor{blue}{Polynomial Invariant}\\
$t^{-1} + 2 - t^{2}$\\[0.1in]
\textcolor{blue}{Real grid homology - hat version}\\
$(-1,0)\oplus (0,0)^{2}\oplus (2,1)$\\[0.1in]
\textcolor{blue}{Real grid homology - minus version}\\
$U^{\infty}_{(0,0)}\oplus U^{\infty}_{(2,1)}\oplus U_{(0,0)}$
} \tabularnewline
  \hline
  \centering $3_1 \# 6_1$ & \centering \includegraphics[width=0.25\textwidth]{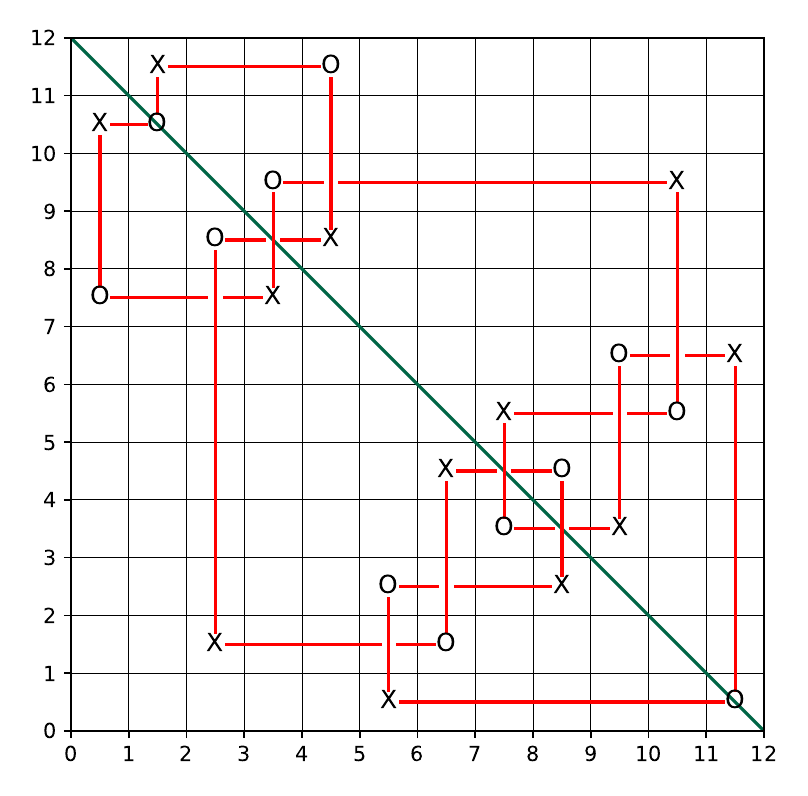} & \parbox[c]{\linewidth}{\centering\tiny
\textcolor{blue}{Polynomial Invariant}\\
$-2t^{-2} - 3t^{-1} + 4 + 6t - t^{2} - 2t^{3}$\\[0.1in]
\textcolor{blue}{Real grid homology - hat version}\\
$(-1,-1)^{3}\oplus (-2,-1)^{2}\oplus (0,0)^{4}\oplus (1,0)^{6}\oplus (2,1)\oplus (3,1)^{2}$\\[0.1in]
\textcolor{blue}{Real grid homology - minus version}\\
$U^{\infty}_{(-1,-1)}\oplus U^{\infty}_{(1,0)}\oplus \bigg(U_{(-1,-1)}\bigg)^{2}\oplus \bigg(U_{(1,0)}\bigg)^{4}\oplus U_{(3,1)}\oplus U^{2}_{(3,1)}$
} \tabularnewline
  \hline
  \centering $4_1 \# 4_1^{OO}$ & \centering \includegraphics[width=0.25\textwidth]{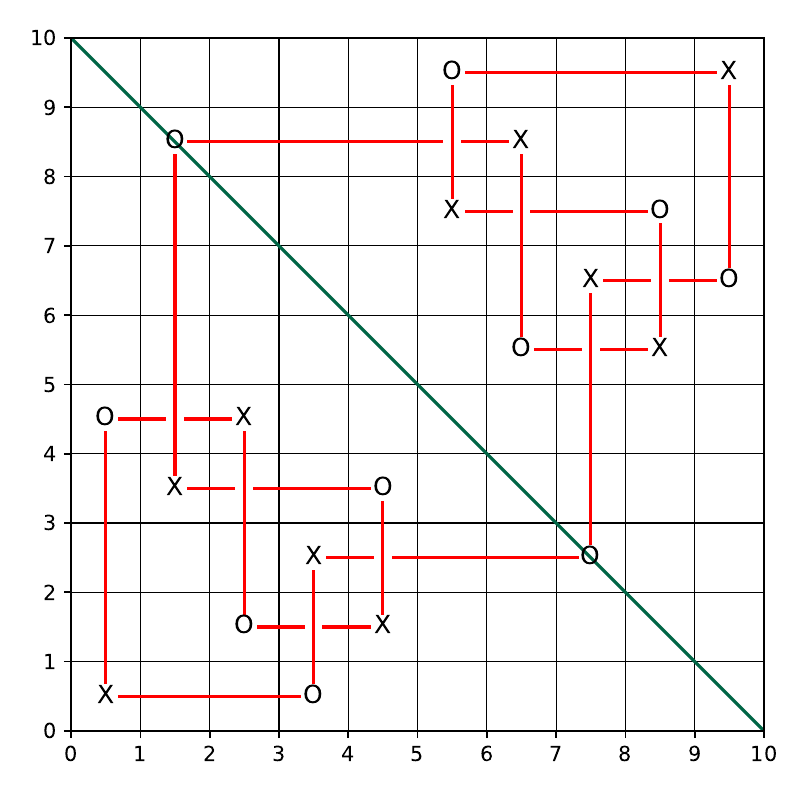} & \parbox[c]{\linewidth}{\centering\tiny
\textcolor{blue}{Polynomial Invariant}\\
$-t^{-2} - t^{-1} + 3 + 3t - t^{2} - t^{3}$\\[0.1in]
\textcolor{blue}{Real grid homology - hat version}\\
$(-1,-1)\oplus (-2,-1)\oplus (0,0)^{3}\oplus (1,0)^{3}\oplus (2,1)\oplus (3,1)$\\[0.1in]
\textcolor{blue}{Real grid homology - minus version}\\
$U^{\infty}_{(0,0)}\oplus U^{\infty}_{(1,0)}\oplus U^{2}_{(0,0)}\oplus U^{2}_{(1,0)}\oplus U^{2}_{(2,1)}\oplus U^{2}_{(3,1)}$
} \tabularnewline
  \hline
  \centering $(4_1 \# 4_1)'^{OO}$ & \centering \includegraphics[width=0.25\textwidth]{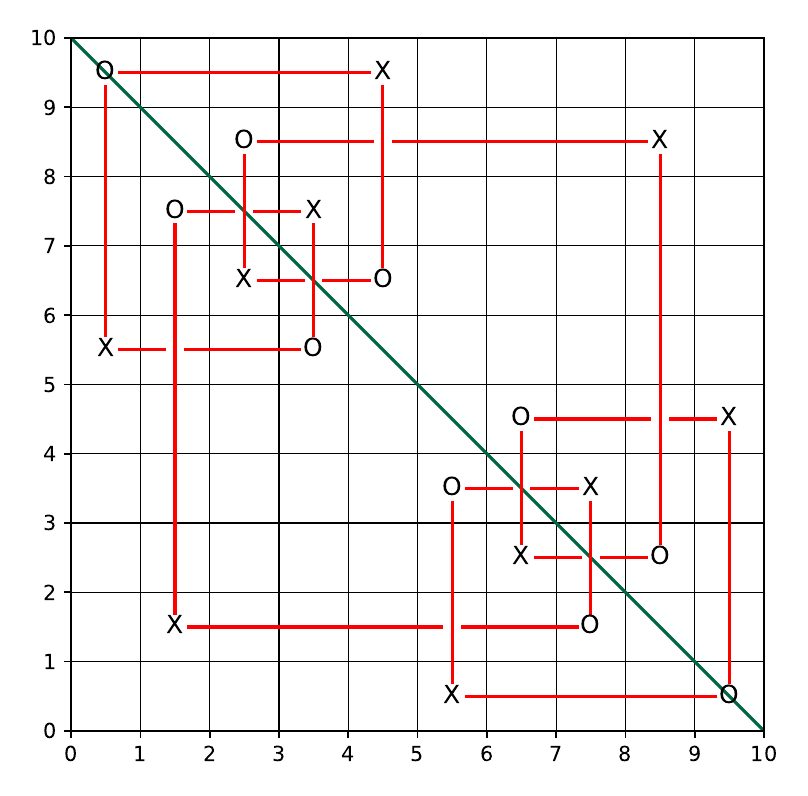} & \parbox[c]{\linewidth}{\centering\tiny
\textcolor{blue}{Polynomial Invariant}\\
$t^{-2} - t^{-1} - 3 + t + 3t^{2} + t^{3}$\\[0.1in]
\textcolor{blue}{Real grid homology - hat version}\\
$(-1,-1)\oplus (-2,-2)\oplus (0,-1)^{3}\oplus (1,0)\oplus (2,0)^{3}\oplus (3,0)$\\[0.1in]
\textcolor{blue}{Real grid homology - minus version}\\
$U^{\infty}_{(-2,-2)}\oplus U^{\infty}_{(0,-1)}\oplus U_{(0,-1)}\oplus U_{(2,0)}\oplus U_{(3,0)}\oplus U^{2}_{(2,0)}$
} \tabularnewline
  \hline
  \centering $5_1 \# 5_1^{OO}$ & \centering \includegraphics[width=0.25\textwidth]{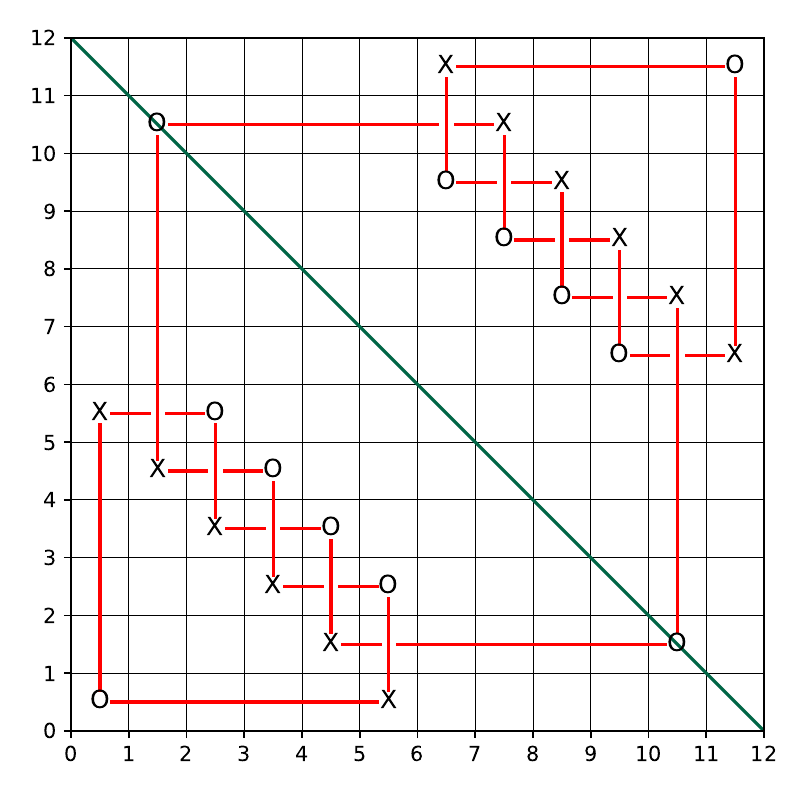} & \parbox[c]{\linewidth}{\centering\tiny
\textcolor{blue}{Polynomial Invariant}\\
$t^{-4} + t^{-3} - t^{-2} - t^{-1} + 1 + t - t^{2} - t^{3} + t^{4} + t^{5}$\\[0.1in]
\textcolor{blue}{Real grid homology - hat version}\\
$(-1,-3)\oplus (-2,-3)\oplus (-3,-4)\oplus (-4,-4)\oplus (0,-2)\oplus (1,-2)\oplus (2,-1)\oplus (3,-1)\oplus (4,0)\oplus (5,0)$\\[0.1in]
\textcolor{blue}{Real grid homology - minus version}\\
$U^{\infty}_{(-3,-4)}\oplus U^{\infty}_{(-4,-4)}\oplus U^{2}_{(0,-2)}\oplus U^{2}_{(1,-2)}\oplus \bigg(U^{2}_{(2,-1)}\bigg)^{0}\oplus \bigg(U^{2}_{(3,-1)}\bigg)^{0}\oplus U^{2}_{(4,0)}\oplus U^{2}_{(5,0)}$
} \tabularnewline
  \hline
  \centering $5_2 \# 5_2^{OO}$ & \centering \includegraphics[width=0.25\textwidth]{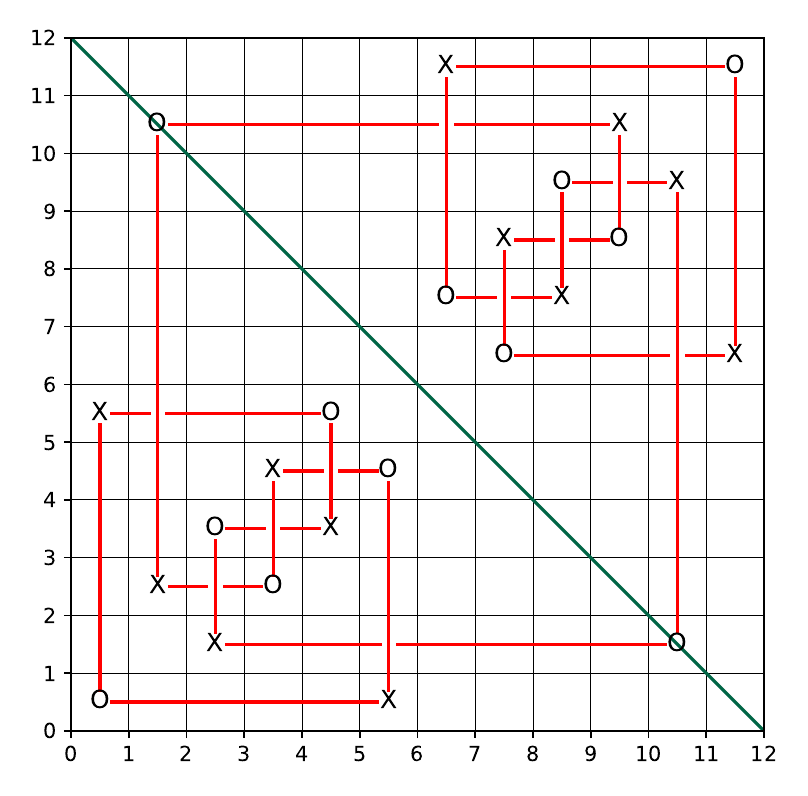} & \parbox[c]{\linewidth}{\centering\tiny
\textcolor{blue}{Polynomial Invariant}\\
$2t^{-2} + 2t^{-1} - 3 - 3t + 2t^{2} + 2t^{3}$\\[0.1in]
\textcolor{blue}{Real grid homology - hat version}\\
$(-1,-2)^{2}\oplus (-2,-2)^{2}\oplus (0,-1)^{3}\oplus (1,-1)^{3}\oplus (2,0)^{2}\oplus (3,0)^{2}$\\[0.1in]
\textcolor{blue}{Real grid homology - minus version}\\
$U^{\infty}_{(-1,-2)}\oplus U^{\infty}_{(-2,-2)}\oplus U^{2}_{(0,-1)}\oplus U^{2}_{(1,-1)}\oplus \bigg(U^{2}_{(2,0)}\bigg)^{2}\oplus \bigg(U^{2}_{(3,0)}\bigg)^{2}$
} \tabularnewline
  \hline
\end{longtable}

\bibliographystyle{plain}
\bibliography{bibliography}
\end{document}